\documentclass[11pt, a4paper, reqno]{amsart}
\usepackage{mathrsfs, mathpazo}
\usepackage{amsfonts, amsbsy}
\usepackage{dsfont}
\usepackage{amssymb, mathabx, mathtools}
\usepackage{xcolor}
\usepackage{enumitem}
\setlist[enumerate]{label=(\roman*)}

\usepackage[colorlinks,citecolor=blue,urlcolor=blue]{hyperref}

\usepackage{amsmath}
\usepackage{framed}
\usepackage{mathrsfs,mathtools}

\usepackage{esint}

\usepackage{todonotes}

\newcommand{\todozb}[1]{\todo[inline,color=green!40]{from ZB: #1}}

\newcommand{\todoan}[1]{\todo[inline,color=blue!20]{from AN: #1}}

\newcommand\adda[1]{{\color{blue}{#1}}}

\newcommand\linspan{\mathrm{linspan}}

\newcommand{\embed}{\hookrightarrow}
\newcommand{\bis}{{\prime\prime}}

\usepackage{fouridx}

\newcommand{\duality}[4]{\fourIdx{}{#4}{}{#3}{\big\langle #1, #2 \big\rangle}}

\renewcommand{\P}{\boldsymbol{\Pi}}

\usepackage{colortbl}
\hypersetup{urlcolor=midnightblue, citecolor=red}

\theoremstyle{remark}

\newcommand{\red}[1]{{\color{red} #1}}

\newcommand\coma[1]{{\color{red} #1}}

\newcommand{\eps}{\varepsilon}

\newcommand{\bn}{\mbox{\boldmath{$n$}}}
\newcommand{\bv}{\mbox{\boldmath{$v$}}}
\newcommand{\bu}{\mbox{\boldmath{$u$}}}

\newcommand{\bb}{\mbox{\boldmath{$b$}}}

\newcommand{\bX}{\mbox{\boldmath{$X$}}}
\newcommand{\by}{\mbox{\boldmath{$y$}}}

\newcommand{\bA}{\mbox{\boldmath{$A$}}}

\newcommand{\bG}{\mbox{\boldmath{$G$}}}

\newcommand{\bw}{\mbox{\boldmath{$\mathrm{w}$}}}
\def\d{\mathrm{d}}

\DeclareMathOperator{\diver}{div}

\newcommand\dela[1]{}

\selectfont

\theoremstyle{plain}
\newtheorem{theorem}{\textbf{Theorem}}[section]
\newtheorem{lemma}[theorem]{\textbf{Lemma}}

\newtheorem{proposition}[theorem]{\textbf{Proposition}}

\theoremstyle{definition}

\newtheorem{remark}[theorem]{\textbf{Remark}}

\newtheorem{definition}[theorem]{\textbf{Definition}}
\newtheorem{assumption}[theorem]{\textbf{Assumption}}
\newtheorem{example}[theorem]{\textbf{Example}}
\numberwithin{equation}{section}
\numberwithin{figure}{section}
\begin{document}

\title[]{On a stochastic Cahn-Hilliard-Brinkman model}

\dela{
\author{
Z. Brze\'zniak$^{1,\ast}$%
\thanks{Corresponding author: Z. Brze\'zniak, e-mail: zdzislaw.brzezniak@york.ac.uk}
and
A. Ndongmo Ngana$^{1}$
and
T. Tachim Medjo$^{2}$
}
}

\author{Z. Brze\'zniak$^{1,\ast}$, A. Ndongmo Ngana$^{1}$ and T. Tachim Medjo$^{2}$}

\dedicatory{\vspace{-10pt}\normalsize{$^{1}$ Department of Mathematics, University of York
Heslington, York, YO105DD, UK \\
$^{2}$ Department of Mathematics and Statistics, Florida International University, MMC, Miami, FL $33199$, USA
}}

\setcounter{footnote}{0}
\renewcommand{\thefootnote}{\fnsymbol{footnote}}
\footnotetext[1]{Corresponding author. E-mail address: zdzislaw.brzezniak@york.ac.uk (Z. Brze\'zniak).
}
\renewcommand{\thefootnote}{\arabic{footnote}
}



	
\date{} 

\maketitle

\begin{abstract}
In this paper, we consider a stochastic version of the Cahn-Hilliard-Brinkman model in a smooth two- or three-dimensional domain with dynamical boundary conditions. The system describes creeping two-phase flows and is basically a coupling of the Brinkman equation for the velocity field that governs the flow through the porous media coupled with  convective Cahn-Hilliard equations for the phase field, both with two independent sources of randomness given by general multiplicative-type Wiener noises in the Cahn-Hilliard equations. The existence of a weak solution, both in the probabilistic and PDEs sense,  is proved. Our construction of a solution is based on the classical Faedo-Galerkin approximation, the Yosida approximation and uses a  compactness method.
Our paper is the first attempt to generalize the paper \cite{Colli+Knopf+Schimperna+Signor_2024} to a stochastic setting.
\end{abstract}	

\section{Introduction}
In this paper, we are interested in the Cahn-Hilliard-Brinkman equations with random perturbation filling a bounded open domain 
$\mathcal{O}\subset \mathbb{R}^d$, $d=2$ or $d=3$, with a boundary $\Gamma$ which is a $C^2$-class riemannian submanifold of $\mathbb{R}^d$. 
More precisely, given a time horizon $T$, we consider the following stochastic system, 
\begin{subequations}\label{Eq1.1}
\begin{align}
-\diver(2 \nu(\phi) D\bu) + \lambda (\phi) \bu + \nabla p = \mu \nabla \phi \quad & \text{in} \quad (0,T) \times \mathcal{O},  \label{eq1.1a} 
 \\
\diver \bu=0 \quad &\text{in} \quad (0,T) \times \mathcal{O}, \label{eq1.2a}
    \\
\partial_t \phi + \diver(\phi \bu) - \diver(M_\mathcal{O}(\phi) \nabla \mu)= F_1(\phi)\d W \quad &\text{in} \quad (0,T) \times \mathcal{O},\label{eq1.3a}
	    \\
\mu= -\eps \Delta \phi + \frac1\eps F'(\phi) \quad &\text{in} \quad (0,T) \times \mathcal{O}, \label{eq1.4a}                                          
\end{align}
\end{subequations}
\begin{subequations}\label{Eq1.2}
\begin{align}
\partial_t \varphi + \diver_\Gamma(\varphi \bu) - \diver_\Gamma (M_\Gamma(\varphi) \nabla_\Gamma \theta)= F_2(\varphi)\d W_\Gamma  \quad &\text{on} \quad (0,T) \times \Gamma, \label{eq1.5a}
\\
\theta= -\eps_\Gamma \Delta_\Gamma \varphi + \frac{1}{\eps_\Gamma} G'(\varphi) + \eps \partial_{\bn} \phi \quad &\text{on} \quad (0,T) \times \Gamma, \label{eq1.6a}
  \\
 K \partial_{\bn} \phi= \varphi - \phi \quad &\text{on} \quad (0,T) \times \Gamma, \label{eq1.7a} 
\end{align}
\end{subequations}
together with the following boundary and initial conditions
\begin{subequations}\label{Eq1.3}
\begin{align}
 M_\mathcal{O} (\phi) \partial_{\bn} \mu= \bu \cdot \bn= 0 \quad &\text{on} \quad (0,T) \times \Gamma, \label{eq1.8a} 
 \\
[2 \nu(\phi) (D \bu) \bn + \gamma(\varphi) \bu]_\tau= [- \varphi \nabla_\Gamma \theta]_\tau \quad &\text{on} \quad (0,T) \times \Gamma, \label{eq1.9a}                                           
\end{align}
\end{subequations}
\begin{subequations}
\begin{align}
\phi(0)= \phi_0 \quad \text{in} ~ \mathcal{O} \quad \text{and} \quad \varphi(0)= \varphi_0 \quad &\text{on} \quad \Gamma\coloneq\partial \mathcal{O}. \label{eq1.10a}
\end{align}
\end{subequations}
Here we assume that $W$ is a cylindrical Wiener process on a separable Hilbert space $U$ and $F_1(\phi)$, for every suitable $\phi$,  is a Hilbert-Schmidt operator from $U$ to $H^1(\mathcal{O})$. Moreover,  $W_\Gamma$ is a cylindrical Wiener process on another separable Hilbert space $U_\Gamma$, independent of $W$, and $F_2(\varphi)$, for every suitable $\varphi$, is a Hilbert-Schmidt operator from $U_\Gamma$ to $H^1(\Gamma)$. In \eqref{Eq1.1}-\eqref{Eq1.2}, the unknown $\bu= \bu(t,x)$, $p= p(t,x)$, $\phi= \phi(t,x)$, and $\varphi= \varphi(t,x)$ represent, respectively, the velocity vector field, the pressure of the fluid, the phase-field in the bulk and the phase-field on the boundary $\Gamma$.
$\bn$ is the unit outward normal vector to the boundary $\Gamma$ and $D\bu$ is the rate of strain tensor defined by 
$$
D_{ij}\bu=\frac{1}{2} \left(\frac{\partial \bu_i}{\partial x_j} + \frac{\partial \bu_j}{\partial x_i}\right).
$$
$\nu(\phi)$ and $\lambda(\phi)$ are the kinematic viscosity and the permeability coefficients, respectively, whereas $\mu$ and $\theta$ are the chemical potentials in the bulk and on the boundary, respectively. The parameters $\eps$ and $\eps_\Gamma$ are positive real numbers which are related to the thickness of the diffuse interface in the bulk and on the surface, respectively. These parameters are fixed through the paper. For any vector $\bv: \Gamma \to \mathbb{R}^d$, $\bv_{\bn}\coloneq (\bv \cdot \bn) \bn$ is its normal component, while $\bv_\tau\coloneq \bv - \bv_{\bn}$ is its tangential component.

\noindent
The evolution problem \eqref{eq1.1a}–\eqref{eq1.10a} does not represent a gradient flow of the local free energy functional 
\begin{equation}\label{energy-functional}
E(\phi,\varphi)= \int_{\mathcal{O}}\left( \frac{\eps}{2} \lvert \nabla \phi \rvert^2 + \frac{1}{\eps} F(\phi) \right) \d x + \int_{\Gamma} \left(\frac{\eps_\Gamma}{2} \lvert \nabla_\Gamma \varphi \rvert^2 + \frac{1}{\eps_\Gamma} G(\varphi) \right) \d S + \frac{\eps [K]}{2} \int_\Gamma (\varphi - \phi)^2 \, \d S,
\end{equation}
in the $\mathcal{H}= H^{-1}(\mathcal{O}) \times H^{-1}(\Gamma)$ metric. Indeed, as one can observe, the bulk-surface Cahn-Hilliard system \eqref{eq1.3a}–\eqref{eq1.10a} is coupled to the velocity equation \eqref{eq1.1a}, and consequently, the overall system no longer constitutes a pure gradient flow. Nevertheless, the system \eqref{eq1.1a}–\eqref{eq1.10a} remains thermodynamically consistent with respect to the free energy functional 
$E$ in the sense that the chemical potentials $\mu$ and $\theta$  are directly linked, through equations \eqref{eq1.4a} and \eqref{eq1.6a}, to the variational derivatives (or subdifferentials) of $E$ with respect to the metric space $\mathcal{H}$. Here $[K]= 0$ if $K=0$ and $[K]= 1/K$ if $K>0$. 
However, we only consider the case $K>0$ in this paper.
\newline
\dela{
In the subsequel analysis of the problem \eqref{eq1.1a}-\eqref{eq1.10a}, the potentials $F$ and $G$ in the free energy functional $E$ are allowed to be of regular type, i.e. they are of the form 
   \begin{equation}\label{regular-potential}
     F_{\text{pol}}(s)= \frac{\alpha}{4} (s^2 - \beta^2)^2, \; \; s \in \mathbb{R}.
   \end{equation}
Potential $F_{\text{pol}}$ can be seen as an approximation of the Flory-Huggins logarithmic potential introduced in \cite{Cahn+Hilliard_1958}:
   \begin{equation}\label{Eq-singular potential}
     F_{\log}(s)= \frac{\Theta}{2}[(1+s)\log(1+s) - (1-s)\log(1-s)] - \frac{\Theta_c}{2} s^2, \; \; s \in (-1,1),
   \end{equation}
where $\Theta$ and $\Theta_c$ are positive constants denoting, respectively, the absolute temperature and the critical temperature of the mixture. The above two constants satisfy the physical relation $0< \Theta < \Theta_c$, where the phase separation occurs. $F_{\log}$ is referred to as a singular potential since its derivative $F_{\log}^\prime$ blows up at the pure phases $\pm 1$. The parameter $\alpha>0$ in \eqref{regular-potential} is related to $\Theta$ and $\Theta_c$, whereas $\pm \beta$ are the two minima of $F_{\log}$. However, the problem \eqref{eq1.1a}-\eqref{eq1.10a} can also be analyzed in the case of singular potential, but due to the size of the paper, we have not consider this issue here and we postpone that question to a subsequent paper. Let us point out that the potential $F_{\log}$ ensures the existence of physical solutions, i.e. solutions such that $\phi,\,\varphi \in[-1,1]$, and this is most coherent with the physical derivation of the Cahn-Hilliard model itself in relation to its thermodynamical consistency.
}
When studying a two-phase fluid mixture, the equation governing the phase field variable $\phi$ (which is generally the Cahn-Hilliard equation) is coupled with an equation describing the averaged velocity $\bu$ of the fluid mixture. A classic choice for modelling this dynamic is the Navier-Stokes system, to which a capillary force, commonly known as the Korteweg force, is added. The latter can be represented as $\mu \nabla \phi$, where $\mu$ is the chemical potential in the bulk phases. In the context of diffuse-interface models, one of the most commonly used models for describing the motion of two incompressible viscous fluids with matched (constant) densities is the H-model.  Initially introduced in \cite{Hohenberg+Halperin_1997} and later rigorously derived in \cite{Gurtin+Polignone+Vinal_1996}, this model involves a system of partial differential equations (PDEs) consisting of the incompressible Navier-Stokes equation coupled to a convective Cahn-Hilliard equation.  In recent years, several papers have investigated the Navier-Stokes-Cahn-Hilliard equations (both in deterministic and stochastic contexts) from the perspectives of mathematical theory and numerical analysis, see for example \cite{Abels+ Garcke_2018,Gal+ Grasselli+Wu_2019,Gurtin+Polignone+Vinal_1996,Deugoue+Boris+Tachim_2021,Abels+Garcke+Grun_2012,Giorgini+Ngana+Medjo+Temam_2023,Giorgini+Temam_2020,Aristide+Deugoue+Tachim_2021,Scarpa_2021} and the reference therein.

\noindent
In general, the H model is used to describe the dynamics of phase transition phenomena for fluids with the same constant density (matched densities). However, another improvement scheme during a seminar work was proposed by the researchers in \cite{Abels+Garcke+Grun_2012}, and their model takes into account the variation of density (unmatched densities). This model is usually called the AGG model. Relevant studies on AGG model's include those by \cite{Abels+Depner+Garcke_2013,Giorgini_2021,Abels+Garcke+Giorgini_2024} and the reference therein. Note that in the literature, both the AGG model and the model H have mostly been supplemented with the classical no-slip boundary condition for the velocity field and homogeneous Neumann boundary conditions for the phase field $\phi$ and the chemical potential $\mu$ in the bulk. 
Note that the Euler equation \cite{Bardos_1972} with a class of dynamic boundary conditions have become extremely popular, and remain so, because they sit at the crossroads of mathematics, physics, and engineering. Furthermore, recently, the H and AGG models with a class of dynamic boundary conditions have become extremely popular, since they more accurately describe the dynamics of phase transition phenomena between bulk and surface. Among these models, we have the Allen-Cahn equation \cite{Cherfils+Gatti+Miranville_2013}, the Cahn-Hilliard equation \cite{Knopf+Lam+Liu+Metzger_2021,Knopf+Lam_2020}, the Cahn-Hilliard-Brinkman equation \cite{Bosia+Conti+Grasselli_2015,Colli+Knopf+Schimperna+Signor_2024}, and the Navier-Stokes-Cahn-Hilliard equations \cite{Giorgini+Knopf_2022}, all subjected with dynamic boundary conditions. In particular, in \cite{Giorgini+Knopf_2022}, the authors proved the existence of global weak solutions both in two and three dimensions as well as the uniqueness of weak solutions for the deterministic Navier-Stokes-Cahn-Hilliard equations with dynamic boundary conditions and in the case of matched densities. We should point out that for the deterministic system \eqref{eq1.1a}-\eqref{eq1.10a}, that is in the absence of the noises terms, the authors in \cite{Colli+Knopf+Schimperna+Signor_2024} established the existence of a weak solution in the case $K\geq 0$ with regular potentials, and the uniqueness of solution in dimension two. Furthermore, they proved the existence of global weak solutions in the case $K \geq 0$ even when the potentials are now singular. Let us point out that in comparison to the model derived and studied in \cite{Giorgini+Knopf_2022} in the case $K=0$, the model investigated in \cite{Colli+Knopf+Schimperna+Signor_2024} take into account the following aspect: The phase-fields $\phi$ and $\varphi$ are not merely coupled through the trace relation $\phi\lvert_\Gamma= \varphi$, but rather through a more general Robin-type coupling condition  $K \partial_{\bn} \phi= \varphi - \phi$ on $\Gamma$, where $K\geq 0$ (cf. \ref{eq1.7a}), and this general condition also reduces to the trace relation when $K=0$. 

\noindent
The present paper is a generalization of the results of  \cite{Colli+Knopf+Schimperna+Signor_2024} under the presence of random forces. More precisely, with the view to capture the randomness component which may affect the phase-field evolutions, we introduce a Wiener type noise in the Cahn-Hilliard equations themselves. This type of approach was first considered in \cite{Cook_1970} with the use of Wiener noise in the well-known stochastic version of the Cahn-Hilliard-Cook model, which was then confirmed as the most accurate description of phase separation, as showed in the contribution work \cite{Milchev+Heermann+Binder_1988} and the references cited therein. In continuation of these fundamental works, many other stochastic models describing the dynamics of incompressible fluids have also been the subject of intensive investigation in the mathematical literature. For such stochastic models, we refer the reader, for instance, to \cite{Feireisl+Petcu_2019,Deugoue+Boris+Tachim_2021,Aristide+Deugoue+Tachim_2021,Scarpa_2021} in the case of regular potentials like $F_{\text{pol}}$ (see \eqref{regular-potential}); \cite{Aristide+Deugoue+Tachim_2024,Andrea+Grasselli+Scarpa_2024,Scarpa_2021} in the case of singular potentials like $F_{\log}$; and also to \cite{Tachim_2019,Tachim_2021} for modified models. For analysis of stochastic Navier-Stokes equations in unbounded domains, see \cite{Brzezniak+Moty_2013,Brzezniak+Motyl+Ondrejat_2017}.

\noindent
In this paper, we are interested in studying the stochastic Cahn-Hilliard-Brinkman system \eqref{eq1.1a}-\eqref{eq1.10a} with smooth potential (i.e. the potentials $F$ and $G$ in the free energy functional $E$ are allowed to be of class $C^2$) using a variational approach.
We prove the existence of martingale solutions to the problem \eqref{eq1.1a}-\eqref{eq1.10a} in both two and three dimension in the case $K>0$. 
The proofs of the main results rely on a double approximation through a Faedo-Galerkin discretization in spacial space an Yosida regularization on the nonlinearity. 
We derive then uniform estimates on the solutions by applying the It\^o formula to the energy functional $E(\phi,\varphi)$ (see \eqref{energy-functional}) and by using the properties of $\nu,\,\lambda,\,M_{\mathcal{O}},\,M_{\Gamma},\,\gamma,\,F_1,$ and $F_2$. 
Finally, we pass to the limit by a stochastic compactness argument. The main challenges on the mathematical side in the present paper are the following: for the deterministic system, the integration in space of equations \eqref{eq1.3a} and \eqref{eq1.5a} together with the boundary conditions \eqref{eq1.7a} and \eqref{eq1.8a} yields the conservation of the phase field in the bulk $\phi$ and the phase field on the boundary $\varphi$ during the evolution, and this conservation of mass is then crucial to control the spatial mean of the chemical potentials and also to simplify the mathematical analysis of the system. However, in the stochastic case, the presence of noise in the equations \eqref{eq1.3a} and \eqref{eq1.5a} deteriorates the system's mass conservation, making it more difficult to control the chemical potentials $\mu$ and $\theta$. 
\dela{
The second main issue is to obtain satisfactory estimates on the phase field $\varphi$ on the surface $\Gamma$, for instance in space $L^\infty(0,T;L^2(\Gamma))$, due to the fact that we have no mass conservation, the presence of the stochastic term $F_2(\varphi)\d W_\Gamma$ in equation \eqref{eq1.5a}, and the impossibility to control the term $\int_\Gamma \varphi \bu \cdot \nabla_\Gamma \varphi \,\d S$ if we apply the It\^o formula to the functional $\Phi(x)= \Vert x\Vert_{L^2(\Gamma)}^2$.}

\noindent
\noindent
In a subsequent paper, by taking advantage of the present analysis, we aim to study \eqref{eq1.1a}-\eqref{eq1.10a} in the case of singular potential \eqref{Eq-singular potential}. From the mathematical perspective, the approach to proving the existence of a solution in the case \(K=0\) differs from the deterministic work \cite{Colli+Knopf+Schimperna+Signor_2024} due to the presence of a proliferation terms in equations \eqref{eq1.3a} and \eqref{eq1.5a}. However, we believe that the approach we propose in this first part of the work for solving system \eqref{eq1.1a}-\eqref{eq1.10a} in the case \(K>0\) can also be applied, but because of the length of the paper, we postpone its investigation in a subsequent paper. 

\noindent
Outline of this paper. Section \ref{sect2} is devoted to recalling some basic notation and abstract framework. 
In Section \ref{sect3}, we introduce some auxiliary operators, the general assumptions, and we formulate our first main result. 
Sections \ref{sect4} and \ref{sect_Uniform_estimates_delta} are devoted to the proofs of several auxiliary results, such as the compactness and tightness criteria, which play a crucial role in establishing the main existence theorem for martingale solutions 
in the case $K>0$, see Section \ref{eqn-thm-first_main_theorem}.
Additional supporting results related to the proofs of our main theorems are provided in the Appendix.
\section{Deterministic setting of the Cahn-Hilliard-Brinkman equations}\label{sect2}
\subsection{Notation and setting}	
Throughout the paper, we use the symbol $C$ to denote any arbitrary positive constant depending only on the data of the problem, whose value may be updated throughout the proofs. When we want to specify the dependence of $C$ on specific quantities, we will indicate them through a subscript. We denote by $\bu_i$ the $i$-th component of a vector field $\bu$. 
\newline
Next, we introduce the necessary definitions of functional spaces that are frequently used in this work. For any (real) Banach space $X$, its (topological) dual is denoted by $X'$ and the duality pairing between $X'$ and $X$ by \duality{\cdot}{\cdot}{X}{X^\prime} or simply \duality{\cdot}{\cdot}{}{}  if there is no confusion. Furthermore, we write $(\cdot,\cdot)_X$ to denote the corresponding inner product.
Given $p \geq 1$, we denote by $L_w^2(0,T;X)$ the Banach space $L^2(0,T;X)$ endowed with weak topology.  In particular,  $v_n \to v$ in $L_w^p(0,T;X)$ if and only if for all $\varphi \in L^q(0,T;X^\prime)$, $p^\prime$ being the conjugate of $p$,
      \begin{equation*}
        \lim_{n \to \infty} \int_0^T \duality{\varphi(s)}{v_n(s) - v(s)}{X}{X^\prime} \, \d s=0.
      \end{equation*}
For two separable Hilbert spaces $\mathbb{K}$ and $\mathbb{X}$, $\mathscr{T}_2(\mathbb{K},\mathbb{X})$ will denote the space of all Hilbert-Schmidt operators $\mathbb{A}: \mathbb{K} \to \mathbb{X}$ with norm $\Vert \mathbb{A} \Vert_{\mathscr{T}_2(\mathbb{K},\mathbb{X})}$, i.e.
   \begin{equation}\label{eqn-Hibert-Schmidt-norm}
     \Vert \mathbb{A} \Vert_{\mathscr{T}_2(\mathbb{K},\mathbb{X})}^2\coloneq \sum_{i=1}^{\text{dim}§\,\mathbb{K}} \Vert A e_i \Vert_{\mathbb{X}}^2,
   \end{equation}
where $(e_i)_{i\geq 1}$ is an ONB of $\mathbb{K}$. Note that any operator in $\mathscr{T}_2(\mathbb{K},\mathbb{X})$ is compact and $\mathscr{T}_2(\mathbb{K},\mathbb{X})$ is a separable Hilbert space endowed with a scalar product
\[
(\mathbb{A},\mathbb{B})_{\mathscr{T}_2(\mathbb{K},\mathbb{X})}
= \sum_{k=1}^{\text{dim}§\,\mathbb{K}} (\mathbb{A} e_k, \mathbb{B} e_k)_{\mathbb{X}}, \; \mathbb{A},\,\mathbb{B} \in \mathscr{T}_2(\mathbb{K},\mathbb{X}).
\]
The Borel $\sigma$-field $\mathcal{B}(\mathscr{T}_2(\mathbb{K},\mathbb{X}))$ is the $\sigma$-field generated by $\mbox{top}(\mathscr{T}_2(\mathbb{K},\mathbb{X}))$, i.e. the topology on $\mathscr{T}_2(\mathbb{K},\mathbb{X})$ which is induced by the norm \eqref{eqn-Hibert-Schmidt-norm} on $\mathscr{T}_2(\mathbb{K},\mathbb{X})$.
\newline
Let $1\leq p \leq\infty$ and $I$ be a subset of $[0,\infty)$. The functions spaces $L^p(I;X)$ consist of all Bochner measurable $p$-measurable functions defined on $I$ with values in $X$. We introduce the space $W^{1,p}(I;X)$ to consist of all functions $u \in L^p(I;X)$ with the vector-valued distributional derivative $\partial_t u \in L^p(I;X)$. If $p=2$, we simply write $H^1(I;X)= W^{1,2}(I;X)$. \newline
The set of continuous functions $f: I\to X$ is denoted by $C(I;X)$. \newline 
Given the Banach space $X$, $C_w([0,T];X)$ is the space of weakly continuous $f:  [0,T] \to X$. More precisely,
\begin{align*}
C_w([0,T];X)\coloneq \{f: & [0,T] \to X; \; \; \mbox{ for all} \; \; \varphi \in X^\prime,  \\
& [0,T] \ni t \mapsto \langle f(t), \varphi \rangle \in \mathbb{R} \mbox{ is continuous}\}.
\end{align*}
$C_w([0,T];X)$ is endowed with the locally convex topology induced by the family $\mathscr{P}$ of semi-norms given by
      \begin{align*}
         \mathscr{P}&\coloneq \{\mathscr{P}_{\varphi}: \; \varphi \in X^\prime\},
          \\
       \mathscr{P}_{\varphi}(f)&\coloneq \sup_{t \in [0,T]} \lvert \langle f(t), \varphi \rangle \rvert, \; \; f \in  C_w([0,T];X).
     \end{align*}
Assume that $d=2$ or $d=3$, and  $\mathcal{O}\subset \mathbb{R}^d$ is a bounded open domain. Then, define for every $t \in(0,T]$,
\begin{align*}
 Q_t&\coloneq(0,t) \times \mathcal{O},
 \\
 \Sigma_t&\coloneq (0,t) \times \partial \mathcal{O}.
\end{align*}
For any $p \in [1,\infty)$ and $s \in \mathbb{R}$, we denote by $L^p(\mathcal{O})$ and $W^{s,p}(\mathcal{O})$ the usual Lebesgue and Sobolev spaces of scalar functions, respectively. If $p=2$, we simply write $W^{s,2}(\mathcal{O})\coloneq H^s(\mathcal{O})$. 
\dela{
We denote by $H_0^1(\mathcal{O})$ the closure of $\mathcal{C}_c^\infty(\mathcal{O})$ in $H^1(\mathcal{O})$. 
}
We use the notation $\mathbb{L}^p(\mathcal{O})$, $\mathbb{W}^{s,p}(\mathcal{O})$, and $\mathbb{H}^s(\mathcal{O})$, to denote the spaces $L^p(\mathcal{O},\mathbb{R}^d)$, $W^{s,p}(\mathcal{O},\mathbb{R}^d)$, and $H^s(\mathcal{O},\mathbb{R}^d)$, respectively, with $d=2,3$. The Lebesgue spaces and Sobolev spaces on the boundary $\Gamma$ are defined analogously.
\newline
We denote by $\lvert \cdot \rvert$ and $(\cdot,\cdot)$ the norm and the scalar product, respectively, on $\mathbb{L}^2(\mathcal{O})$, $L^2(\mathcal{O})$, and $H$, see \eqref{eqn-spaces} for the definition of the space $H$.
We will also denote, for the sake of brevity, by $\lvert \cdot \rvert$ the norm on $L^2(\mathcal{O},\mathbb{R}^{d \times d})$. 
Next, we denote by $(\cdot,\cdot)_\Gamma$ and $\lvert \cdot \rvert_\Gamma$ the inner product and the norm in $L^2(\Gamma)$ and $\mathbb{L}^2(\Gamma)$, respectively.
\newline
The space $V$, as defined in point \eqref{eqn-spaces} below, is endowed with the norm and scalar product inherited from the space $\mathbb{H}^1(\mathcal{O})$, 
\begin{equation*}
(\bu,\bv)_V\coloneq (\bu,\bv) + (\nabla \bu,\nabla \bv), \quad \bu,\,\bv \in V,
\end{equation*}
and the norm induced by the scalar product $(\cdot,\cdot)_V$ in $V$ is then given by
\[
\Vert \bv \Vert_V^2= \lvert \bv \rvert^2 + \lvert \nabla \bv \rvert^2, \; \; \bv \in V.
\]

The averages of $u$ over $\mathcal{O}$ and $v$ over $\Gamma$ are defined respectively by:
\begin{align*}
       \langle u \rangle_{\mathcal{O}}=
          \begin{cases}
             \frac{1}{\lvert \mathcal{O} \rvert} \duality{u}{1}{H^1(\mathcal{O})}{(H^1(\mathcal{O}))^\prime} &\mbox{ if } u \in (H^1(\mathcal{O}))^\prime, 
               \\
             \fint_{\mathcal{O}} u \, \d x &\mbox{ if } u \in L^1(\mathcal{O}), 
          \end{cases}
\\
\langle v \rangle_{\Gamma}=
\begin{cases}
\frac{1}{\lvert \Gamma \rvert} \duality{v}{1}{H^1(\Gamma)}{(H^1(\Gamma))^\prime} &\mbox{ if } v \in (H^1(\Gamma))^\prime,
\\
\fint_{\Gamma} v \, \d S &\mbox{ if } v \in L^1(\Gamma).
\end{cases}
\end{align*}
Here $1$ represents the constant function assuming $1$ in $\mathcal{O}$ and on $\Gamma$, respectively, while $\lvert \mathcal{O} \rvert$ and $\lvert \Gamma \rvert$ stand for the $d$-dimensional and $d-1$-dimensional Lebesgue measures of $\mathcal{O}$ and $\Gamma$, respectively.
\newline
If $\bv: \mathcal{O} \to \mathbb{R}^d$, then we put 
\begin{align*}
\lvert \nabla \bv \rvert_{L^2(\mathcal{O},\mathbb{R}^{d \times d})}^2 
&\coloneq \sum_{i=1}^d \lvert D_i \bv \rvert_{L^2(\mathcal{O},\mathbb{R}^d)}^2
= \sum_{i,j=1}^d \lvert D_i \bv_j \rvert_{L^2(\mathcal{O})}^2,
\\
\lvert D \bv \rvert_{L^2(\mathcal{O},\mathbb{R}^{d \times d})}
&\coloneq \sum_{i,j=1} \left \lvert \frac{1}{2} (D_j \bv_i + D_i \bv_j )\right \rvert_{L^2(\mathcal{O})}^2.
\end{align*}
We set
       \begin{equation*}
         \mathcal{C}_{0,\sigma}^\infty(\mathcal{O})= \{\bv \in \mathcal{C}_0^\infty(\mathcal{O},\mathbb{R}^d): ~ \diver \bv= 0 \mbox{ in } \mathcal{O}\},
       \end{equation*}
and for the mathematical setting of \eqref{eq1.1a}–\eqref{eq1.10a} we introduce the Hilbert spaces
\begin{equation}\label{eqn-spaces}
\begin{aligned}
H(\diver,\mathcal{O})&= \{\bv \in \mathbb{L}^2(\mathcal{O}); ~ \diver \bv \in L^2(\mathcal{O})\},
\\
H &= \text{the closure of} ~ \mathcal{C}_{0,\sigma}^\infty(\mathcal{O}) ~ \mbox{ in } ~ \mathbb{L}^2(\mathcal{O}),
   \\
   V &= \mathbb{H}^1(\mathcal{O}) \cap H.
\end{aligned}
\end{equation}
\dela{
$
It can be shown, see e.g. \cite{Temam_2001}, that $V = \mathbb{H}_0^1(\mathcal{O}) \cap H$.

On the other hand, why we cannot prove \eqref{eq4.34} directly without approximation?  }

The space $H$ can also be characterised as \cite[Theorem 1.1/1.4]{Temam_2001}:
   \begin{align*}
      H&= \{\bv \in H(\diver,\mathcal{O}):  \; \bu \cdot \bn=0 \mbox{ on } \Gamma \}.
    \end{align*}
For every $k \in[0,\infty)$, we introduce the following standard scale of Hilbert spaces
\[\mathcal{H}^k\coloneq H^k(\mathcal{O}) \times H^k(\Gamma),\]
which are endowed with the standard inner product
  \begin{equation*}
   ((u, f), (v, g))_{\mathcal{H}^k}\coloneq (u,v)_{H^k(\mathcal{O})} + (f, g)_{H^k(\Gamma)} \; \;  (u, f), \; (v, g) \in \mathbb{H}^k
  \end{equation*}
and the induced norm 
\[
\Vert (u,f) \Vert_{\mathcal{H}^k}^2= ((u,f),(u,f))_{\mathcal{H}^k}, \; \; (u,f) \in \mathcal{H}^k.
\]
Consequently, $(\mathcal{H}^k, (\cdot,\cdot)_{\mathcal{H}^k}, \Vert \cdot \Vert_{\mathcal{H}^k})$ is a Hilbert space. Note that for $k=0$, 
\[
\mathbb{H}\coloneq L^2(\mathcal{O}) \times L^2(\Gamma)= \mathcal{H}^0.
\]
\newline
Subsequently, we introduce the product:
       \begin{equation*}
         ((u,f),(v,g))_{\mathbb{H}}\coloneq (u,v) + (f,g)_\Gamma, \;\; (u,f),\, (v,g) \in \mathbb{H}.
       \end{equation*}
We set $\mathbb{V}= H^1(\mathcal{O}) \times H^1(\Gamma)$ and endow the space $\mathbb{V}$ with the inner product $(\cdot,\cdot)_{\mathbb{V}}\coloneq(\cdot,\cdot)_{\mathcal{H}^1}$ and its induced norm. This means that $(\mathbb{V}, (\cdot,\cdot)_{\mathbb{V}}, \|\cdot\|_{\mathbb{V}})$ is also a Hilbert space. Since the space $\mathbb{V}$ is densely and continuously embedded in $\mathbb{H}$, then by identifying $\mathbb{H}$ with its dual space $\mathbb{H}^\prime$, we have the following Gelfand chain, where each space is densely and compactly embedded into the next one, 
\[
\mathbb{V} \hookrightarrow \mathbb{H} \cong \mathbb{H}^\prime \hookrightarrow \mathbb{V}^\prime.
\]
\dela{
We set $\mathbb{V}= H^1(\mathcal{O}) \times H^1(\Gamma)$ and we introduce the bulk-surface product space
  \begin{align*}
    \mathbb{V}_K &\coloneq \begin{cases}
                        \mathbb{V}  &\mbox{ if }  K>0, \\
                        \{(v, g ) \in \mathbb{V}: ~ g= v 
                        _\Gamma ~ \mbox{ a.e. on } \Gamma\}  &\mbox{ if } K= 0.
                      \end{cases}
  \end{align*}
For arbitrary $K\geq 0$, we endow the space $\mathbb{V}_K$ with the inner product $(\cdot,\cdot)_{\mathbb{V}_K}\coloneq(\cdot,\cdot)_{\mathcal{H}^1}$ and its induced norm. This means that $(\mathbb{V}_K, (\cdot,\cdot)_{\mathbb{V}_K}, \|\cdot\|_{\mathbb{V}_K})$ is also a Hilbert space. Since the space $\mathbb{V}_K$ is densely and continuously embedded in $\mathbb{H}$, then by identifying $\mathbb{H}$ with its dual space $\mathbb{H}^\prime$, we have the following Gelfand chain, where each space is densely and compactly embedded into the next one, 
\[
\mathbb{V}_K \hookrightarrow \mathbb{H} \cong \mathbb{H}^\prime \hookrightarrow \mathbb{V}_K^\prime.
\]
}
For two Banach spaces $X_i, \,i=1,2$, we understand that $(\prod_{i=1}^2 X_i )^\prime \coloneq \prod_{i=1}^2 X_i^\prime$. 
So,
\begin{align*}
\duality{ Z^\ast}{Z}{\mathbb{V}}{\mathbb{V}^\prime}
= \duality{Z_1^\ast}{ Z_1}{V_1}{V_1^\prime} + \duality{ Z_2^\ast}{Z_2}{V_\Gamma}{V_\Gamma^\prime}, \; \;  Z\in \mathbb{V}, \; Z^\ast \in \mathbb{V}^\prime,
\end{align*}
having set, for the sake of brevity, 
   \begin{equation*}
     V_1 \coloneq H^1(\mathcal{O}) \mbox{ and } V_\Gamma \coloneq H^1(\Gamma).
   \end{equation*}
The dual norm on $\mathbb{V}^\prime$ is given by
\[
\Vert (u,f) \Vert_{\mathbb{V}^\prime}
\coloneq \sup \{\vert \duality{(u,f)}{(v,g)}{\mathbb{V}}{\mathbb{V}^\prime} \rvert \;: \; (v,g) \in \mathbb{V}  \mbox{ with }  \Vert (v,g) \Vert_{\mathbb{V}}= 1 \}, \; \; (u,f) \in \mathbb{V}^\prime.
\] 
We conclude this subsection by recalling the following Korn inequality, cf. \cite{Abels_2012,Colli+Knopf+Schimperna+Signor_2024}, and the bulk-surface Poincar\'e inequality, see \cite[Lemma A.1]{Knopf+Liu_2021}\dela{\cite[Chapter II, Section 1.4]{Temam2}}, respectively, and which will be used throughout this work.
\dela{
We conclude this subsection by reporting the following definition on uniform $C^m$-regularity condition borrowed from Adams's book \cite[Section 4.10]{Adams_1975}, the Korn inequality, cf. \cite{Abels_2012,Colli+Knopf+Schimperna+Signor_2024}, and the bulk-surface Poincar\'e inequality, see \cite[Lemma A.1]{Knopf+Liu_2021}\dela{\cite[Chapter II, Section 1.4]{Temam2}}, respectively, and which will be used throughout this work.
Note that we use Adam's book but restrict to the case of bounded set.
\begin{definition}\label{Uniform-regularity condition}
Fix $n,\;m\geq 1$ and let $\mathcal{B}_1=\{x\in \mathbb{R}^n: \vert x\rvert <1\}$. An open and bounded set $\mathcal{O}\subset \mathbb{R}^n$ satisfies the $C^m$-regularity condition if there exists a finite set $\{U_i\}$ and a $C^m$-class diffeomorphism
$\boldsymbol{\psi}_i: U_i \to \mathcal{B}_1$ such that
\begin{trivlist}
\item[(i)] each $U_i$ is an open subset of $\mathbb{R}^n$ and $\partial \mathcal{O} \subset \bigcup_i U_i$;
\item[(ii)] for some $\delta>0$, $\mathcal{O}_\delta \subset \bigcup_{i=1}^\infty \boldsymbol{\psi}_i^{-1} (\{x\in \mathbb{R}^n: \;\vert x\rvert <1/2\})$; 
\item[(iii)] for each $i$, $\boldsymbol{\psi}_i(U_i \cap \mathcal{O})=\{x\in \mathcal{B}_1:\; x_n>0\}$; 
\item[(iv)] there is a constant $N>0$ such that for every $\alpha$, with $0<\lvert\alpha\rvert\leq m$, and every $i$, we have
            \begin{align*}
                \lvert D^\alpha \psi_{i}(x) \rvert \leq N, \; x \in U_i, \\
                 \lvert D^\alpha \Psi_{i}(y) \rvert \leq N, \;y \in \mathcal{B}_1.
            \end{align*}
\end{trivlist}
\end{definition}
We report the following Korn inequality, see \cite[Appendix]{Abels_2012} for its proof.
}
\begin{lemma}
Assume $\mathcal{O}$ is of class $C^2$. Then, there exists $C_{\text{KN}}>0$ such that 
    \begin{equation}\label{Korn-inequality}
      \lvert \nabla \bv \rvert \leq C_{\text{KN}}(\mathcal{O}) (\lvert D \bv \rvert + \lvert \bv \rvert_\Gamma), \; \; \bv \in \mathbb{H}^1(\mathcal{O}).
    \end{equation}
\end{lemma}
The following lemma is a direct consequence of the bulk surface Poincar\'e inequality derived in \cite[Lemma A.1]{Knopf+Liu_2021}.
\begin{lemma}\label{Poincare-inequality}
Suppose $\mathcal{O}$ is of class $C^2$. Then, there exists $C_{\text{PC}}>0$ depending only on $\mathcal{O}$ such that
     \begin{equation}\label{bulk-surface-Poincare-inequality}
       \left \lvert v - \fint_\Gamma v \, \d S \right \rvert \leq C_{\text{PC}} \lvert \nabla_\Gamma v \rvert_\Gamma, \; \; v \in H^1(\Gamma). 
     \end{equation}
\end{lemma}
\dela{
\begin{lemma}\label{Poincare-inequality}
There exists a positive constant $C_P>0$ depending only on $\mathcal{O}$ such that
     \begin{equation}
       \vert v \rvert \leq C_P (\lvert \nabla v \rvert + \lvert v \rvert_{\Gamma}), \; \; v \in V_1.
     \end{equation}
\end{lemma}
}
\subsection{Setting of bilinear operators and their properties}
Let $\P: \mathbb{L}^2(\mathcal{O}) \to H$ be the orthogonal projection usually called the Helmhotz-Leray projection. It is known, see \cite[Remark 1.6]{Temam_2001}, that $\P$ maps continuously the Hilbert space $H^1(\mathcal{O})$ into itself. \newline
From now on, we denote by $\bA$ the Stokes operator in the case of Navier boundary conditions, which is defined by, see, e.g., \cite[Appendix A]{Abels_2012},
   \begin{equation}\label{eqn-Stokes operator}
    \begin{aligned}
     D(\bA) &\coloneq \{\bv \in H^2(\mathcal{O},\mathbb{R}^d) \cap H:\, [2(D\bv \bn) + \bv]_\tau= 0 \mbox{ on }  \Gamma\},  
       \\
     \bA \bu &\coloneq -\P \diver (2 D \bv), \;\; \bu \in  D(\bA).
    \end{aligned}
   \end{equation}
Notice that $\bA$ is a symmetric operator since
     \begin{equation*}
        (\bA \bv,\bw)= 2 \int_{\mathcal{O}} D\bv:D\bw\, \d x + \int_{\Gamma} \bv \cdot \bw\, \d S= (\bv, \bA \bw), \; \; \bv,\,\bw \in D(\bA).
    \end{equation*}
Moreover, $\bA$ is a positive and self-adjoint operator on $H$ with compact inverse and $V= (H,D(\bA))_{\frac12,2}$.
\newline
Now and throughout this paper, we assume that $\Gamma$ is a $C^2$-class riemannian submanifold of $\mathbb{R}^d$ endowed with the metric induced from 
$\mathbb{R}^d$.  
Thus, we have the scale $H^s(\Gamma)$, $s \in \mathbb{R}$ of Sobolev spaces, see e.g. \cite{Taylor_1981}.
\newline
We denote by $\Delta_\Gamma: L^2(\Gamma) \to L^2(\Gamma)$ the corresponding Laplace-Beltrami operator on $L^2(\Gamma)$ defined by
   \begin{equation}\label{eqn-Laplace-Beltrami operator}
       \begin{aligned}
          D(\Delta_\Gamma) &\coloneq H^2(\Gamma),  
            \\
           \Delta_\Gamma u &\coloneq  \diver_\Gamma \nabla_\Gamma u, \;\; u \in D(\Delta_\Gamma),
       \end{aligned}
   \end{equation}
where $\diver_\Gamma$ denotes the surface divergence, $\nabla_\Gamma \coloneq (\partial_{\tau_1},\ldots,\partial_{\tau_{d-1}})$ the Riemannian gradient on $\Gamma$, and $\partial_{\tau_i}$ the derivative along the $i$-th tangential direction $\tau_i$, $i= 1,\ldots,d-1$. 
The operator $I-\Delta_\Gamma$ has a compact inverse.
\newline
Let us introduce the trilinear form $b$ defined by 
\begin{equation*}
  b: V \times V_1 \times V_1 \ni (\bu,\phi,\psi) \mapsto  
    \sum_{i,j = 1}^{d} \int_{\mathcal{O}} \bu_i(x) \frac{\partial \phi(x)}{\partial x_i} \psi(x) \,\d x \in \mathbb{R}.
\end{equation*}
Using the density of $\mathcal{C}_0^\infty(\mathcal{O})$ in $V_1$ together with the Gagliardo-Nirenberg inequality, we can prove as in 
\cite[Lemmas II.1.3 \& II.1.6]{Temam_2001} that there exists $C>0$ depending only on $\mathcal{O}$ s.t.
    \begin{equation*}
      \begin{aligned}
         \vert b(\bu,\phi,\psi) \rvert 
         \leq C \Vert \bu \Vert_{\mathbb{L}^4(\mathcal{O})} \Vert \phi \Vert_{L^4(\mathcal{O})} \lvert \nabla \psi \rvert 
         \leq C \lvert \bu \rvert^{\frac{1}{d}} \Vert \bu \Vert_{V}^{1 - \frac{1}{d}} \vert \phi \rvert^{\frac{1}{d}} \Vert \phi \Vert_{V_1}^{1 - \frac{1}{d}} \Vert \psi \Vert_{V_1}, \; \; \bu \in V, \; \psi,\; \phi \in V_1.
     \end{aligned}
    \end{equation*}
Moreover,
      \begin{equation}\label{B1-first-Property}
        b(\bu,\phi,\phi)= 0, \; \; \bu \in V, \; \phi \in V_1.
      \end{equation}
The previous inequality implies the following fact. There exists a continuous bilinear form 
$$B_1: V \times V_1 \to V_1^\prime$$ such that
                              \begin{equation}
                                 \duality{B_1(\bu,\phi)}{\psi}{V_1}{V_1^\prime}= b(\bu,\phi,\psi), \; \; \bu \in V, \; \psi, \; \phi \in V_1.
                              \end{equation}
In particular, $B_1$ satisfies 
      \begin{equation}
       \duality{B_1(\bu,\phi)}{\psi}{V_1}{V_1^\prime}=- b(\bu,\psi,\phi), \; \; \bu \in V, \; \psi, \; \phi \in V_1
     \end{equation}
and
    \begin{equation}\label{B1-Property}
       \Vert B_1(\bu,\phi) \Vert_{V_1^\prime} 
         \leq C \lvert \bu \rvert^{\frac{1}{d}} \Vert \bu \Vert_V^{1 - \frac{1}{d}} \vert \phi \rvert^{\frac{1}{d}} \Vert \phi \Vert_{V_1}^{1 - \frac{1}{d}}, \; \; \bu \in V, \; \phi \in V_1.
     \end{equation}
Subsequently, if $\bu \in V$ and $\phi \in V_1$, we have the following equality, which is to be understood in the weak sense: $$\diver(\phi \bu)= \bu \cdot \nabla \phi + \phi \underbrace{\diver \bu}_{=0}= \bu \cdot \nabla \phi.$$
Hence, for any test function $\psi \in V_1$, 
\begin{align*}
\int_{\mathcal{O}} \diver(\phi \bu) \psi\, \d x
= \int_{\mathcal{O}} [\bu \cdot \nabla \phi] \psi\, \d x 
= - \int_{\mathcal{O}} [\bu \cdot \nabla \psi] \phi\, \d x 
=- b(\bu,\psi,\phi).
\end{align*}
The following equality holds in the weak sense: 
\[\diver_\Gamma (\varphi \bu) \coloneq \nabla_\Gamma(\varphi \bu)= \varphi \nabla_\Gamma \cdot \bu + \nabla_\Gamma \varphi \cdot \bu.
\]
Thus, in particular, for all $\bu \in V, \; \varphi, \, \psi \in V_\Gamma$, we have
\begin{equation*}
    \begin{aligned}
      &\int_\Gamma \diver_\Gamma (\varphi \bu) \psi \, \d S
      = \int_\Gamma \varphi \psi \nabla_\Gamma \cdot \bu \, \d S + \int_\Gamma \nabla_\Gamma \varphi \cdot \bu \psi \, \d S \\
      &= \int_\Gamma \varphi \psi \nabla_\Gamma \cdot \bu \, \d S - \int_\Gamma \varphi \psi \nabla_\Gamma \cdot \bu \, \d S - \int_\Gamma \varphi \bu \cdot \nabla_\Gamma \psi\, \d S 
      = - \int_\Gamma \varphi \bu \cdot \nabla_\Gamma \psi\, \d S.
    \end{aligned}
\end{equation*}
From now on, we introduce a trilinear form $\tilde{b}$ defined by
\begin{equation*}
   \tilde{b}: V \times V_\Gamma \times V_\Gamma  \ni (\bu,\varphi,\psi) \mapsto - \int_\Gamma \varphi \bu \cdot \nabla_\Gamma \psi\, \d S \in \mathbb{R}.
\end{equation*}
$\tilde{b}$ is well defined, since $\mathcal{O} \subset \mathbb{R}^d$. 
Furthermore, the map $\tilde{b}$ enjoys the following properties.
\begin{lemma}\label{Properties of trilinear-boundary-map}
Assume $\bu \in V$, $\varphi, \, \psi \in V_\Gamma$. The map $\tilde{b}$ is trilinear continuous and its satisfies
     \begin{equation}\label{First-Properties of trilinear-boundary-map}
        \vert \tilde{b}(\bu,\varphi,\psi) \rvert \leq C \|\bu\|_{V} \|\varphi\|_{V_\Gamma} \|\psi\|_{V_\Gamma}.
     \end{equation}
Moreover, there exists a bilinear form $\Tilde{B}: V \times V_\Gamma \to V_\Gamma^\prime$ such that
  \begin{equation*}
     \duality{\Tilde{B}(\bu,\varphi)}{\psi}{V_\Gamma}{V_\Gamma^\prime}= \Tilde{b}(\bu,\varphi,\psi), \; \bu \in V, \;\; \varphi,\;\psi \in V_\Gamma,
  \end{equation*}
and there exists a positive constant $C=C(\mathcal{O},\Gamma)$ such that
   \begin{equation}\label{Property-tilde B}
      \Vert \tilde{B}(\bu,\varphi) \Vert_{V_\Gamma^\prime} \leq C \Vert \bu \Vert_{V} \Vert \varphi \Vert_{V_\Gamma}, \; \bu \in V, \; \varphi \in V_\Gamma.
    \end{equation}
\end{lemma}
\begin{proof}
Let $\bu \in V$ and $\varphi,\; \psi \in V_\Gamma$ be arbitrary. Using the H\"older inequality together with the Trace embedding theorem, see, for instance, Adams \cite[Theorem 5.36, Section 5]{Adams_1975}, we infer 
\begin{align*}
\lvert \tilde{b}(\bu,\varphi,\psi) \rvert
\leq \Vert \bu \Vert_{\mathbb{L}^4(\Gamma)} \Vert \varphi \Vert_{L^4(\Gamma)} \Vert \nabla_\Gamma \psi \Vert_{\mathbb{L}^2(\Gamma)} 
\leq C(\Gamma,\mathcal{O}) \Vert \bu \Vert_{V} \Vert \varphi \Vert_{V_\Gamma} \Vert \psi \Vert_{V_\Gamma}.
\end{align*}
Hence, for an arbitrary but fixed $\psi \in V_\Gamma$, the mapping $\tilde{b}(\cdot,\cdot,\psi)$ defined on $V \times V_\Gamma$ with values in $\mathbb{R}$ is bilinear and continuous, which implies the existence of a continuous bilinear form $\tilde{B}: V \times V_\Gamma \to V_\Gamma^\prime$ satisfying the assertions in the lemma.
\end{proof}
To conclude this section, we introduce two continuous mappings that are also necessary for the analysis of the stochastic Brinkman system \eqref{eq1.1a}-\eqref{eq1.10a}. 
\newline
For fixed $\phi \in V_1,\; \varphi \in V_\Gamma$, we introduce the following linear maps $A_{\phi}: V_1 \to V_1^\prime$ and $\mathcal{A}_{\varphi}: V_\Gamma \to V_\Gamma^\prime$ whose actions are given by
  \begin{equation}\label{Definition of A_phi and A_varphi}
    \begin{aligned}
        \duality{A_{\phi}(\mu)}{v}{V_1}{V_1^\prime}
           &= \int_{\mathcal{O}} M_{\mathcal{O}}(\phi) \nabla \mu \cdot \nabla v\, \d x, \; \; \mu,\, v \in V_1=H^1(\mathcal{O}),
              \\
         \duality{\mathcal{A}_{\varphi}(\theta)}{\rho}{V_\Gamma}{V_\Gamma^\prime}
            &= \int_{\Gamma} M_\Gamma(\varphi) \nabla_\Gamma \theta \cdot \nabla_\Gamma \rho\, \d S, \; \; \rho, \, \theta \in  V_\Gamma=H^1(\Gamma),
    \end{aligned}
  \end{equation}
and by assumption \ref{item:H4} below, we have
\begin{equation}
   \begin{aligned}
     \duality{A_{\phi}(\mu)}{\mu}{V_1}{V_1^\prime}&\geq M_0 \lvert \nabla \mu \rvert^2, \quad \mu \in V_1, 
          \\
    \duality{\mathcal{A}_{\varphi}(\theta)}{\theta}{V_\Gamma}{V_\Gamma^\prime}&\geq N_0 \lvert \nabla_\Gamma \theta \rvert^2, \; \; \theta \in V_\Gamma.
   \end{aligned}
\end{equation}
Once more, thanks to the assumption \ref{item:H4} and H\"older inequality, we can easily check that $A_{\phi}: V_1 \to V_1^\prime$ and $\mathcal{A}_{\varphi}: V_\Gamma \to V_\Gamma^\prime$ are globally Lipschitz continuous, i.e. if $\mu,\;\tilde{\mu} \in V_1$ and $\theta,\;\tilde{\theta}\in V_\Gamma$, 
  \begin{equation}
    \begin{aligned}
       \Vert A_{\phi}(\mu) - A_{\phi}(\Tilde{\mu}) \Vert_{V_1^\prime}
         &\leq \bar{M}_0 \Vert \mu - \tilde{\mu} \Vert_{V_1}, 
         \\
       \Vert \mathcal{A}_{\varphi}(\theta) - \mathcal{A}_{\varphi}(\tilde{\theta}) \Vert_{V_\Gamma^\prime}
        &\leq \bar{N}_0 \Vert \theta - \tilde{\theta} \Vert_{V_\Gamma}.
    \end{aligned}
  \end{equation}
\section{The Stochastic Cahn-Hilliard-Brinkman Equations and main result}\label{sect3}
\subsection{Stochastic preliminaries}
We explicitly define the force injecting energy into the system we will analyze. Let $(\Omega, \mathcal{F}, \mathbb{F}= \{\mathcal{F}_t\}_{t\in[0,T]},\mathbb{P})$ be a filtered probability space,
with filtration $\mathbb{F}= \{\mathcal{F}_t\}_{t\in[0,T]}$ satisfying the so called usual assumptions, namely it is complete, right-continuous, and $\mathcal{F}_0$ contains all null sets, i.e. $\mathbb{P}$-negligible subset of $\mathcal{F}$ belongs to $\mathcal{F}_0$. \newline
Let $(U,(\cdot,\cdot)_U)$ and $(U_\Gamma, (\cdot,\cdot)_{U_\Gamma})$ be two separable Hilbert spaces.  For convenience, we fix once and for all two complete standard orthonormal systems $\{e_{1,k}\}_{k \in \mathbb{N}}$ on $U$ and $\{e_{2,k}\}_{k \in \mathbb{N}}$ on $U_\Gamma$, and we denote by $\mathbb{E}$ the mathematical expectation w.r.t. the probability measure $\mathbb{P}$. 
We assume that  $W \coloneq (W^k)_{k=1}^\infty$, resp. $W_\Gamma \coloneq (W_\Gamma^k)_{k=1}^\infty$, is a $U$-cylindrical Wiener process, resp. $U_\Gamma$-cylindrical Wiener process on $(\Omega, \mathcal{F}, \mathbb{F},\mathbb{P})$. 
This implies that there exists a sequence   $W^k$ (resp. $W_\Gamma^k$) of i.i.d. standard real Wiener processes such that 
\begin{equation}\label{wiener-process-representation}
         W(t,\cdot)= \sum_{k=1}^\infty W^k(t,\cdot) e_{1,k}(\cdot), \; \; W_\Gamma(t,\cdot)= \sum_{k=1}^\infty W_\Gamma^k(t,\cdot) e_{2,k}(\cdot),
     \end{equation}
where the series are convergent almost surely in appropriate Hilbert or Banach spaces. \newline
It is well known that the series in \eqref{wiener-process-representation} do not converge in $U$, resp. in  $U_\Gamma$,  unless $U$, resp. $U_\Gamma$,  is finite dimensional. Hence, $W$, resp. $W_\Gamma$, is not rigorously defined as a $U$-valued continuous process, resp. $U_\Gamma$-valued continuous process. Consequently, we will occasionally make use of a larger space $U_0 \supset U$, resp. $U_\Gamma^0 \supset U_\Gamma$, such that the embeddings of $U \embed U_0$, resp. $U_\Gamma \embed U_\Gamma^0$, are  Hilbert-Schmidt.
\begin{example}\label{ex-001}
An example of such space is the following 
\[
U_0 \coloneq \left\{\psi= \sum_{k=1}^\infty \alpha_k e_{1,k}: \; \Vert \psi \Vert_{U_0}^2= \sum_{k=1}^\infty \alpha_k^2/k^2< \infty \right\}, 
\]
where $\Vert u \Vert_{U_0} \coloneq \left(\sum_{k=1}^\infty \frac{1}{k^2} \vert (u,e_{1,k})_{U} \rvert^2 \right)^\frac{1}{2}$ for all $u \in U$. 
\end{example}
Moreover, by standard martingale arguments \dela{and the fact that each $\{W^k(t)\}_k$ and $\{W^k_\Gamma(t)\}_k$ are almost surely continuous, see, e.g., \cite{Prato,Liu+Rockner_2015}, we have that, for almost every $\omega \in \Omega$,}it follows that the processes $W(t)$, resp. $W_\Gamma(t)$ has a continuous $U_0$, resp. $U_0^\Gamma$-valued modification. 
\dela{It is important to point out that we can now identify $W$ and $W_\Gamma$ as a $Q$-Wiener process on $U_0$ and $Q_\Gamma$-Wiener process on $U_\Gamma$, respectively, with $Q=J \circ J^*$ and $Q_\Gamma= \Im \circ \Im^*$ being trace-class operators and such that $Q^\frac{1}{2} (U_0)= J(U)$ and $Q_\Gamma^\frac{1}{2} (U^0_\Gamma)= \Im(U_\Gamma)$, $J^*$ the adjoint of $J$, and $\Im^*$ the adjoint of $\Im$. In the following, we may implicitly assume these extensions by simply saying that $W$ and $W_\Gamma$ are cylindrical processes on $U$ and $U_\Gamma$, i.e., $W$ is a cylindrical Wiener process on $U$ (resp. $W_\Gamma$ is a cylindrical Wiener process on $U_\Gamma$) if and only if it is a $Q$-Wiener process on $U$ (resp. $Q_\Gamma$-Wiener process on $U_\Gamma$).} \newline
\dela{
For every $p,\, r \in [1,+\infty]$ and for any Banach space $X$, the symbol $L^p(\Omega;L^r(0,T;X))$ consists of all random functions $v: \Omega \times [0,T] \times \mathcal{O}\to L^r(0,T;X)$ such that $v$ is measurable w.r.t $(\omega,t)$ and for all $t$, $v$ is measurable w.r.t $\mathcal{F}_t$. We furthermore  endow this space with the norm
    \begin{equation*}
      \|v\|_{L^p(\Omega;L^r(0,T;X))} = \left(\mathbb{E} \|v\|_{L^r(0,T;X)}^p \right )^{1/p};
     \end{equation*}
when $r=\infty$, then the norm in the space $L^p(\Omega;L^\infty(0,T;X))$ is given by  
      \begin{equation*}
        \|v\|_{L^p(\Omega;L^\infty(0,T;X))} = \left(\mathbb{E} \|v\|_{L^\infty(0,T;X)}^p \right )^{1/p}.
      \end{equation*} 
}      
\begin{definition}\label{def-Ito integral}
    Let $\tilde{K}$ be a separable Hilbert space, $\tilde{W}$ be a $\tilde{K}$-cylindrical Wiener process on $(\Omega, \mathcal{F},\mathbb{F},\mathbb{P})$ and $\tilde{H}$ be a Hilbert space. We denote by $\mathcal{M}^2(\Omega \times [0,T]; \mathscr{T}_2(\tilde{K},\tilde{H}))$ the space of all equivalence classes of $\mathbb{F}$-progressively measurable processes $\phi:\Omega \times [0,T] \to \mathscr{T}_2(\tilde{K},\tilde{H})$ satisfying
  \begin{equation*}
   \mathbb{E} \int_0^T \Vert \phi(t) \Vert_{\mathscr{T}_2(\tilde{K},\tilde{H})}^2 \d t < \infty.
  \end{equation*}
For any $\phi \in \mathcal{M}^2(\Omega \times [0,T]; \mathscr{T}_2(\tilde{K},\tilde{H}))$, we have, by the theory of stochastic integration on infinite dimensional Hilbert space, cf. \cite[Chapter 5, Section 26]{Metivier_1982} or \cite[Chapter 4]{Prato},  that the process $M$ defined by 
\[
M(t)= \int_0^t \phi(s)\, \d \tilde{W}(s) 
, \; \; t \in [0,T],
\]
is a $\tilde{H}$-valued martingale.
\dela{where 
      \begin{align*}
         \bar{W}(t) \coloneq \sum_{k=1}^\infty W^k(t) J(e_{1,k}), \; \; t \in[0,T],
         \end{align*}
and  $\phi(\cdot) \circ J^{-1} \in \mathcal{M}^2(\Omega \times [0,T]; \mathscr{T}_2(Q^\frac{1}{2}(U_0),\tilde{H}))$.}
Moreover, the following It\^o isometry holds,
\begin{equation}\label{eq3.5}
\begin{aligned}
\mathbb{E} \left \Vert \int_0^t \phi(s) \dela{\circ J^{-1} }\d \tilde{W}(s)  \right \Vert_{\tilde{H}}^2
\dela{&= \mathbb{E} \int_0^t \Vert \phi(s) \circ J^{-1} \Vert_{\mathscr{T}_2(Q^\frac{1}{2}(U_0),\tilde{H})}^2 \d s }
= \mathbb{E} \int_0^t \Vert \phi(s) \Vert_{\mathscr{T}_2(\tilde{K},\tilde{H})}^2 \d s.
\end{aligned}
\end{equation}
\end{definition}
and by the Burkholder-Davis-Gundy (BDG) inequality,
       \begin{equation}\label{eqn-BDG Inequality}
            \mathbb{E} \sup_{s\in[0,t]} \left \Vert \int_0^s \phi(\tau) \dela{\circ J^{-1}} \d \tilde{W}(\tau)  \right \Vert_{\tilde{H}}^p \leq C(p) \mathbb{E} \left(\int_0^t \Vert \phi(s) \Vert_{\mathscr{T}_2(\tilde{K},\tilde{H})}^2 \d s \right)^{\frac p2}, \; p \in [1,\infty), \, t \in [0,T].
       \end{equation}
The above Definition \ref{def-Ito integral} can be applied to the case of $\Gamma$. 
Thus, the process defines by 
$$
M_\Gamma(t)
= \int_0^t \phi_\Gamma(s)\, \d W_\Gamma(s) \dela{\coloneq \int_0^t \phi_\Gamma(s) \circ \Im^{-1} \, \d \bar{W}_{\Gamma}(s)},\; \phi_\Gamma \in \mathcal{M}^2(\Omega \times [0,T]; \mathscr{T}_2(U_\Gamma,\tilde{H})),
$$ 
enjoys the same properties as for the stochastic process $M= (M(t): t \in [0,T])$. \newline
Below, we denote by $\mathcal{M}^{d\times d}(\mathbb{R})$ the vector space of all real $d\times d$-matrices and recall an important result on stochastic ODEs in finite dimension Hilbert spaces.
\begin{lemma}\label{stochastic-ordinary-theorem}
  Let the maps $\boldsymbol{\sigma}:[0,T] \times \mathbb{R}^d \times \Omega \to \mathcal{M}^{d\times d}(\mathbb{R})$ and $\bb:[0,T] \times \mathbb{R}^d \times \Omega \to \mathbb{R}^d$ be both continuous in $x \in \mathbb{R}^d$ for all $t\in[0,T]$ and $\omega\in\Omega$, and be progressively measurable on the probability space $(\Omega,\mathcal{F},\mathbb{P})$ with respect to the filtration $\{\mathcal{F}_t\}_{t\in[0,T]}$. Assume that the following conditions are satisfied for all $R>0$:
\begin{trivlist}
  \item[(i)] $\int\limits_0^T \sup\limits_{\lvert x\rvert\leq R} (\lvert \bb(t,x) \rvert + \lvert \boldsymbol{\sigma}(t,x) \rvert^2)\, \d t <\infty$; 
  
  \item[(ii)] $2 (x-y) \cdot (\bb(t,x) - \bb(t,y)) + \lvert \boldsymbol{\sigma}(t,x) - \boldsymbol{\sigma}(t,y) \rvert^2) \leq K_t(R),\;\; t \in[0,T],\,x,\,y\in \mathbb{R}^d,\; \lvert x \rvert,\,\lvert y \rvert\leq R$;
  
  \item[(iii)] $2x \cdot \bb(t,x) + \lvert \boldsymbol{\sigma}(t,x) \rvert^2) \leq K_1(t) (1 + \lvert x \rvert^2),\;\; t \in[0,T],\,x \in \mathbb{R}^d,\;\lvert x \rvert \leq R$;
\end{trivlist}
where $K_t(R)$ is an $\mathbb{R}^+ \text{-}$valued $(\mathcal{F}_t)\text{-}$adapted process satisfying
\[
\alpha_s(R)\coloneq \int_0^s K_t(R) \, \d t<\infty, \quad s\in(0,T]. 
\]
Then, for any $\mathcal{F}_0\text{-}$measurable map $\boldsymbol{X}_0: \Omega \to \mathbb{R}^d$, there exists a unique global  solution to the stochastic differential equation
\[
\d \boldsymbol{X}(t)= \bb(t,\boldsymbol{X}(t)) \d t + \boldsymbol{\sigma}(t,\boldsymbol{X}(t)) \, \d W(t).
\]
\end{lemma}
\begin{proof}
We refer to \cite[Theorem 3.1.1]{Prevot+Rockner_2007}.
\end{proof}
Before formulating our main assumptions, we recall the following definitions of strong and Borel measurability of an $\mathscr{T}_2(\mathbb{K},\mathbb{X})$-valued operator that we need in this work. Here, $\mathbb{K}$ and $\mathbb{X}$ are two given separable Hilbert spaces.
\newline
Let us fix a measurable space $(X,\mathscr{X})$ and $Y$ a Banach space. Let us denote the Borel $\sigma$-algebra of $Y$ by $\mathcal{B}(Y)$.
   \begin{definition}(\cite[Sec 1.1, p. 2]{Hytonen+Neerven+Veraar+Weis_2016})\label{Def-measurability}
       Assume that $\mathscr{Y}$ is a $\sigma$-field of subsets of $Y$. A function $f: X \to Y$ is $\mathscr{X}/\mathscr{Y}$ measurable iff the pre-image $f^{-1}(A) \in \mathscr{X}$ for any set $A \in \mathscr{Y}$.
  \end{definition}
In particular, a function $f: X \to Y$ is $\mathscr{X}$-Borel iff it is $\mathscr{X}/\mathcal{B}(Y)$ measurable.
\begin{definition}\label{Def-measurability-1}
A function $f: X \to \mathscr{T}_2(\mathbb{K},\mathbb{X})$ is called strongly $\mathscr{X}$-Borel measurable if for every $k \in \mathbb{K}$, the $\mathbb{X}$-valued functions 
$f(\cdot)k: X \ni \psi \mapsto f(\psi) k \in \mathbb{X}$ is $\mathscr{X}$-Borel according to Definition \ref{Def-measurability}.
\end{definition}
To prove the existence of weak solutions in the case of nondegenerate viscosity $\nu$, nondenegerate permeability $\lambda$, nondegenerate friction $\gamma$, and regular potentials $F$ and $G$, we make the following assumptions. 
\begin{assumption} \label{Hypo1}
\begin{enumerate}[label=(H\arabic*),ref=(H\arabic*)]
\item \label{item:H1} The set $\mathcal{O} \subset \mathbb{R}^d$, $d=2,3$, is bounded and of class $C^2$. 
\dela{\adda{We consider the "embedded" riemannian structure on the boundaryless $d-1$-dimensional manifold $\Gamma:=\partial \mathcal{O}$, i.e. 
 every tangent $d-1$-dimensional hyperplane $T_p\Gamma$, $p \in \Gamma$, in endowed with the inner product inherited from the ambient space $\mathbb{R}^d$ and assume that the resulting riemannian manifold is of $C^2$-class. This riemannian manifold has a unique Laplace-Beltrami operator $\Delta_\Gamma$. The domain of $\Delta_\Gamma$ is equal to the Sobolev space $H^2(\Gamma)$, $-\Delta_\Gamma$ is non-negative self-adjoint in the Hilbert space 
 $L^2(\Gamma)$, where is endowed with "volume measure" corresponding to the previously introduced riemannian structure. If fact, this voulme measure is equal to the surface measure from the isometric embedding $\Gamma \embed \mathbb{R}^d$. The integration with respect to this measure we denote by 
 $dS$. Finally, $I-\Delta_\Gamma$ has compact inverse and so there exists a sequence $(\lambda_i^\Gamma)_{i=1}^\infty$ on non-negative numbers and 
 an $H^2(\Gamma)$-valued  sequence $(\Lambda_i^\Gamma)_{i=1}^\infty$ which is an orthonormal basis of $L^2(\Gamma)$ and 
 \[
-\Delta_\Gamma \Lambda_i^\Gamma= \lambda_i^\Gamma \Lambda_i^\Gamma, \;\; i \in \mathbb{N}_\ast.
 \]
 }}
\item \label{item:H2} The maps 
\[
F_1: V_1 \to \mathscr{T}_2(U,V_1) \mbox{ and } F_2: V_\Gamma \to \mathscr{T}_2(U_\Gamma,V_\Gamma)
\]
are measurable in the sense of Definition \ref{Def-measurability-1}. 
\newline
Moreover, we assume that 
     \begin{align*}
        F_1: V_1 \to \mathscr{T}_2(U,L^2(\mathcal{O})) \mbox{ and } F_2: V_\Gamma \to \mathscr{T}_2(U_\Gamma,L^2(\Gamma))
     \end{align*}
are Lipschitz-continuous, while
\[
F_1: V_1 \to \mathscr{T}_2(U,V_1) \mbox{ and } F_2: V_\Gamma \to \mathscr{T}_2(U_\Gamma,V_\Gamma) 
\]
have linear growth. More precisely, there exist positive constants $C_1$ and $C_2$ such that
\begin{equation}\label{eq4.122-1}
\begin{aligned}
\Vert F_1(\phi_1) - F_1(\phi_2) \Vert_{\mathscr{T}_2(U,L^2(\mathcal{O})))}^2 &\leq C_2 \lvert \phi_1 - \phi_2 \rvert^2, \; \; &\phi_1,\, \phi_2 \in L^2(\mathcal{O}),  
\\
\Vert F_1(\phi) \Vert_{\mathscr{T}_2(U,V_1)}^2 &\leq C_1  (1 + \lvert \nabla \phi \rvert^2), \; \; &\phi \in V_1,
\end{aligned}
\end{equation}
and there exist positive constants $C_3$ and $C_4$ such that
\begin{equation}\label{eq4.122-2}
\begin{aligned}
\Vert F_2(\varphi_1) - F_2(\varphi_2) \Vert_{\mathscr{T}_2(U_\Gamma,L^2(\Gamma)))}^2 &\leq C_4 \lvert \varphi_1 - \varphi_2 \rvert_\Gamma^2, \; \; &\varphi_1,\, \varphi_2 \in L^2(\Gamma),  
\\
\Vert F_2(\varphi) \Vert_{\mathscr{T}_2(U_\Gamma,V_\Gamma)}^2 &\leq C_3  (1 + \lvert \nabla_\Gamma \varphi \rvert_\Gamma^2), \; \; &\varphi \in V_\Gamma.
\end{aligned}
\end{equation}
Furthermore, there exist  positive constants $\tilde{C}_1$ and $\tilde{C}_2$  such that 
   \begin{equation}\label{eqn-boundedness-of-F_1}
      \sum_{k=1}^\infty \Vert F_1(\phi) e_{1,k} \Vert_{L^\infty(Q_T)}^2
      \leq \tilde{C}_1, \mbox{ for every $\phi \in L^2(Q_T)$ }
     \end{equation}
    \begin{equation}\label{eqn-boundedness-of-F_2}
      \sum_{k=1}^\infty \Vert F_2(\varphi) e_{2,k} \Vert_{L^\infty(\Sigma_T)}^2\leq \tilde{C}_2, \;\; \mbox{ for every $\varphi \in L^2(\Sigma_T)$} .
    \end{equation}
\begin{example} 
Let $(\sigma_k)_{k\geq 1} \subset W^{1,\infty}(\mathbb{R})$ and $(\tilde{\sigma}_k)_{k\geq 1} \subset W^{1,\infty}(\mathbb{R})$ be a sequence of functions such that 
        \begin{equation*}
            \dela{C_1 \coloneq} \sum_{k=1}^\infty \Vert \sigma_k \Vert_{W^{1,\infty}(\mathbb{R})}^2<\infty \mbox{ and }  \dela{C_2\coloneq} \sum_k^\infty \|\tilde{\sigma}_k\|_{W^{1,\infty}(\mathbb{R})}^2< \infty.
        \end{equation*}
Define
      \begin{align*}
          F_1(\phi) e_{1,k}&= \sigma_k \circ \phi, \;\; k \in \{1,2,\ldots\}, \;\; \phi \in V_1
      \end{align*}
and
       \begin{equation*}
          F_2(\varphi) e_{2,k}= \tilde{\sigma}_k \circ\varphi, \;\; k \in \{1,2,\ldots\}, \;\; \varphi \in V_\Gamma.
       \end{equation*}
\end{example}
\item \label{item:H4} The maps $M_{\mathcal{O}}: \mathbb{R} \to \mathbb{R}$ and $M_\Gamma: \mathbb{R} \to \mathbb{R}$ are continuous, bounded and uniformly positive, i.e. there exist  positive constants 
$M_0, \, \bar{M}_0, \, N_0$ and $\bar{N}_0$ such that
\begin{equation*}
M_0 \leq M_{\mathcal{O}}(s) \leq \bar{M}_0, \quad N_0 \leq M_\Gamma(s) \leq \bar{N}_0, \quad  s \in \mathbb{R}.
\end{equation*}
\item \label{item:H5} Function $\nu \in W^{1,\infty}(\mathcal{O})$ is uniformly positive and bounded, i.e. there exist positive constants $\nu_0$ and $\bar{\nu}_0$ such that 
\[\nu_0\leq \nu(s) \leq \bar{\nu}_0, \;\; s \in \mathbb{R}.
\]

\item \label{item:H6} Function $\lambda \in W^{1,\infty}(\mathcal{O})$ is positive and bounded, i.e. there exists a positive constant $\lambda_0$ such that  
\[0<\lambda(s) \leq \lambda_0,  \;\; s \in \mathbb{R}. 
\]
\item \label{item:H7} Function $\gamma \in C(\mathcal{O})$ is uniformly positive and bounded, i.e. there exist positive constants $\gamma_0$ and $\bar{\gamma}_0$ such that 
\[ \gamma_0\leq \gamma(s) \leq \bar{\gamma}_0, \;\; s \in \mathbb{R}. 
\]
\item \label{item:H8} Functions $F \in C^2(\mathbb{R})$ and  $G \in C^2(\mathbb{R})$ are nonnegative,\dela{$F\geq 0$,} $F^\prime(0)=0$, \dela{$G\geq 0$,} $G^\prime(0)=0$, and there exist $C_F>0$, $\tilde{c}_F>0$, $C_G> 0$,  and $\tilde{c}_G>0$ such that
\begin{equation}\label{condition_ F'_and_F''}
\begin{aligned}
 \lvert F^\prime(r) \rvert &\leq C_F(1 + F(r)), \;\; & r \in \mathbb{R}, \\
 \lvert F''(r) \rvert &\leq C_F(1 + F(r)), \;\; & r \in \mathbb{R}, \\
 F''(r) &\geq - \tilde{c}_F, \;\; & r \in \mathbb{R}, 
\end{aligned}
\end{equation}
and
              \begin{equation}\label{condition_ G'_and_G''}
                \begin{aligned}
                 \lvert G^\prime(r) \rvert &\leq C_G(1 + G(r)), \;\; & r \in \mathbb{R}, \\
                 \lvert G''(r) \rvert &\leq C_G(1 + G(r)), \;\; & r \in \mathbb{R}, \\
                 G''(r) &\geq - \tilde{c}_G, \;\; & r \in \mathbb{R}.
                \end{aligned}
            \end{equation}
\end{enumerate}
\end{assumption} 
\dela{
\coma{
\todozb{Combine the previous two examples with the Remark below.}

\begin{remark}\label{Rk_1}
\begin{trivlist}
\item[(i)] 
Given functions $\sigma_k$ as in item \ref{item:H2}, the Hilbert-Schmidt norm of the operator $F_1(\phi)$ is equal to 
     \begin{align*}
         \Vert F_1(\phi) \Vert_{\mathscr{T}_2(U,L^2(\mathcal{O}))}^2 
          \coloneq \sum_{k=1}^\infty \lvert  [F_1(\phi)] (e_{1,k}) \rvert^2
           = \sum_{k=1}^\infty \lvert \sigma_k(\phi) \rvert^2 < \infty,
      \end{align*}
and obviously, it ensures that the operator $F_1: L^2(\mathcal{O}) \to \mathscr{T}_2(U, L^2(\mathcal{O}))$ is \adda{Lipschitz-continuous} and linearly bounded, and its restriction $F_1 \lvert_{V_1} \to \mathscr{T}_2(U, V_1)$ is linearly bounded, cf. \cite{Breit+Feireisl+Hofmanova_2018,Feireisl+Petcu_2021,Scarpa_2021}.
\item[(ii)] We use the symbol, with a slight abuse of notation, $\nabla F_1: V_1 \to \mathscr{T}_2(U,\mathbb{L}^2(\mathcal{O}))$ to denote the gradient process, which is well defined by
\[
\nabla [F_1(\phi) e_{1,k}] \coloneq \nabla \sigma_k(\phi)= \sigma_k^\prime(\phi) \nabla \phi, \; \; \phi \in V_1, \; \; k \in\{1,2,\ldots\}.
\]
\end{trivlist}
We notice that the same remark also hold for the operator $F_2$.
\end{remark}
}
}
\begin{remark}
$(i)$ As an example of regular potential satisfying conditions \eqref{condition_ F'_and_F''} or \eqref{condition_ G'_and_G''}, we can consider the following polynomial
   \begin{equation}\label{regular-potential}
     F_{\text{pol}}(s)= \frac{\alpha}{4} (s^2 - \beta^2)^2, \; \; s \in \mathbb{R},
   \end{equation}
where $\alpha>0$ and $\beta \in \mathbb{R}$.
\newline
$(ii)$ $F_{\text{pol}}$ can be seen as an approximation of the Flory-Huggins logarithmic potential introduced in \cite{Cahn+Hilliard_1958}:
   \begin{equation}\label{Eq-singular potential}
     F_{\log}(s)= \frac{\Theta}{2}[(1+s)\log(1+s) - (1-s)\log(1-s)] - \frac{\Theta_c}{2} s^2, \; \; s \in (-1,1),
   \end{equation}
where $\Theta$ and $\Theta_c$ are positive constants denoting, respectively, the absolute temperature and the critical temperature of the mixture. In particular, $\Theta$ and $\Theta_c$ satisfy the physical relation $0< \Theta < \Theta_c$, where the phase separation occurs. 
In general, $F_{\log}$ is referred to as a singular potential since its derivative $F_{\log}^\prime$ blows up at the pure phases $\pm 1$, and the parameter $\alpha$ in \eqref{regular-potential} is related to $\Theta$ and $\Theta_c$, whereas $\pm \beta$ are the two minima of $F_{\log}$.
\newline
$(iii)$ Problem \eqref{eq1.1a}-\eqref{eq1.10a} can also be analyzed in the case of singular potential, but due to the size of the paper, we have not consider this issue here and we postpone that question to a subsequent paper. Let us point out that the potential $F_{\log}$ ensures the existence of physical solutions, i.e. solutions such that $\phi,\,\varphi \in[-1,1]$, and this is most coherent with the physical derivation of the Cahn-Hilliard model itself in relation to its thermodynamical consistency.
\end{remark}
\subsection{Abstract formulation}
We set $\mathcal{U}= U \times U_\Gamma$.
\dela{
In order to define the concept of weak martingale solutions that we shall work with, we introduce the following function $[r]: [0,\infty) \to [0,\infty)$ defined by
\begin{equation}
[r]= \begin{cases}
       \frac1r \mbox{ if } r>0, \\
       0 \mbox{ if }  r=0,
      \end{cases}
\end{equation}
to deal with the cases $K>0$ and $K=0$ simultaneously.}
Then, given $\bu \in V$, define the following maps,
\begin{align*}
\mathbf{b}&: \mathbb{V} \ni \bX=(\phi,\varphi) \mapsto
\Bigl(- B_1(\bu,\phi) - A_{\phi}(\mu), - \mathcal{A}_{\varphi}(\theta) - \tilde{B}(\bu,\varphi) \Bigr) \in  \mathbb{V}^\prime,
\\
\boldsymbol{\sigma}&: \mathbb{V} \ni \bX=(\phi,\varphi) \mapsto
\mathrm{diag}(F_{1}(\phi),F_2(\varphi)) \in \mathscr{T}_2(\mathcal{U},\mathbb{H}),
\end{align*}
where, given $\phi \in H^2(\mathcal{O})$, $\mu$ is defined by 
   \begin{equation}\label{eqn-mu from pgi}
      \mu = -\eps \Delta \phi + \frac1\eps F^\prime(\phi)\in L^2(\mathcal{O}),
     \end{equation}
while, given $\varphi \in D(\Delta_\Gamma)= H^2(\Gamma)$,  $\theta$ is defined by 
   \begin{equation}\label{eqn-theta from pgi}
      \theta= -\eps_\Gamma \Delta_\Gamma \varphi + \frac{1}{\eps_\Gamma} G^\prime(\varphi) + \eps \partial_{\bn} \phi. 
     \end{equation}
Therefore, we rewrite the equations \eqref{eq1.3a} and \eqref{eq1.5a} in the following "compact" form
\begin{equation}\label{Compact-Modified-stochastic-CHNSEs-3}
\begin{cases}
\d \bX(t)= \boldsymbol{b}(\bX(t))\,\d t + \boldsymbol{\sigma}(\bX(t)) \, \d \mathcal{W}(t),
\\
\mu(t) = -\eps \Delta \phi(t) + \frac1\eps F^\prime(\phi(t)), \;\; t \in[0,T], \\
\theta(t)= -\eps_\Gamma \Delta_\Gamma \varphi(t) + \frac{1}{\eps_\Gamma} G^\prime(\varphi(t)) + \eps \partial_{\bn} \phi(t),  \;\; t \in [0,T].
\end{cases}
\end{equation}
Now we define what we mean by weak martingale solutions to the problem \eqref{eq1.1a}-\eqref{eq1.10a}.
\begin{definition}\label{def3.4}
Assume $K> 0$ and let $(\phi_0,\varphi_0): \mathcal{O} \to \mathbb{V}$ be a non-random initial data. By a weak martingale solution to \eqref{eq1.1a}-\eqref{eq1.10a}, we mean a system consisting of a complete probability space $(\tilde{\Omega}, \tilde{\mathcal{F}},\tilde{\mathbb{F}}, \tilde{\mathbb{P}})$, with the filtration $\tilde{\mathbb{F}}= (\tilde{\mathcal{F}}_t)_{t \in [0,T]}$ satisfying the usual conditions, and a process 
 \begin{equation}
     (\tilde{\bu}, \tilde{\phi}, \tilde{\varphi}, \tilde{\mu},\tilde{\theta}, \tilde{\mathcal{W}})=\left(\left(\tilde{\bu}(t), \tilde{\phi}(t), \tilde{\varphi}(t), \tilde{\mu}(t),\tilde{\theta}(t), \tilde{\mathcal{W}}(t)\right): t \in [0,T] \right)
     \end{equation}
     such that
\item[1.] $\tilde{\mathcal{W}}$ is a cylindrical Wiener process in a separable Hilbert space  $\mathcal{U}$ defined on $(\tilde{\Omega}, \tilde{\mathcal{F}},\tilde{\mathbb{F}})$.
 
\item[2.] $\tilde{\bu}: \tilde{\Omega} \times [0,T] \to V$ and $(\tilde{\mu},\tilde{\theta}): \tilde{\Omega} \times [0,T] \to \mathbb{V}$
are $\tilde{\mathbb{F}}$\dela{-adapted} progressively measurable processes with $\tilde{\mathbb{P}}$-a.e. paths  
  \begin{equation}\label{eq3.24}
  \begin{aligned}
   &\tilde{\bu} \in L^2(0,T; V), \quad \tilde{\bu}\lvert_\Gamma \,\in L^2(0,T; \mathbb{L}^2(\Gamma)), \\
   &(\tilde{\mu},\tilde{\theta}) \in L^2(0,T;\mathbb{V}),
   \end{aligned}
  \end{equation}
\item[3.] $\tilde{\bX} \coloneq (\tilde{\phi}, \tilde{\varphi}): \tilde{\Omega} \times [0,T] \to \mathbb{V}$ is a $\tilde{\mathbb{F}}$-progressively measurable processes with $\tilde{\mathbb{P}}$-a.e. paths 
     \begin{equation}\label{Eqn-3.11}
       \tilde{\bX} \in  C_w ([0,T]; \mathbb{V}) \cap L^\infty(0,T;\mathbb{V}), 
     \end{equation}
such that for any $\bv \in V$, $(\upsilon, \upsilon \lvert_\Gamma) \in \mathbb{V}$, and $(\psi, \psi \lvert_\Gamma) \in \mathbb{V}$, and every $t \in [0,T]$, the following equalities hold $\tilde{\mathbb{P}}$-a.s.,
\item[3.] 
\begin{equation}\label{eq3.22}
\begin{aligned}
  & 2 \int_{Q_t} \nu(\tilde{\phi}) D\tilde{\bu}: \nabla \bv\, \d x\, \d s + \int_{Q_t} \lambda(\tilde{\phi}) \tilde{\bu} \cdot \bv \, \d x\, \d s + \int_{\Sigma_t} \gamma(\tilde{\varphi}) \tilde{\bu} \cdot \bv \, \d S \, \d s 
 \\
  &\quad 
  =  - \int_{\Sigma_t} \tilde{\varphi} \nabla_\Gamma \tilde{\theta} \cdot \bv \, \d S \, \d s  - \int_{Q_t} \tilde{\phi} \nabla \tilde{\mu} \cdot \bv \, \d x\, \d s,
\end{aligned}
\end{equation}
             \begin{equation}\label{eq3.11}
              \begin{aligned}
              &\duality{\tilde{\bX}(t)}{(\upsilon, \upsilon \lvert_\Gamma)}{\mathbb{V}}{\mathbb{V}^\prime}
              - \int_0^t \duality{\boldsymbol{b}(\tilde{\bX}(s))}{(\upsilon,\upsilon \lvert_\Gamma)}{\mathbb{V}}{\mathbb{V}^\prime}\d s \\
              &\quad = \duality{(\phi_0, \varphi_0)}{(\upsilon, \upsilon \lvert_\Gamma)}{\mathbb{V}}{\mathbb{V}^\prime} + \left(\int_0^t \boldsymbol{\sigma}(\tilde{\bX}(s)) \, \d \tilde{\mathcal{W}}(s), (\upsilon, \upsilon \lvert_\Gamma) \right)_{\mathbb{H}},
              \end{aligned}
             \end{equation}

\begin{equation}\label{eq3.26}
 \begin{aligned}
  \int_{Q_t} \tilde{\mu} \psi \, \d x \, \d s + \int_{\Sigma_t} \tilde{\theta} \psi \lvert_\Gamma \, \d S\, \d s 
   &= \eps \int_{Q_t} \nabla \tilde{\phi} \cdot \nabla \psi \, \d x \, \d s + \int_{Q_t} \frac1\eps F'(\tilde{\phi}) \psi \, \d x \, \d s\\
   &\quad + \eps_\Gamma \int_{\Sigma_t} \nabla_\Gamma \tilde{\varphi} \cdot \nabla_\Gamma \psi \lvert_\Gamma \, \d S \, \d s +  \frac{1}{\eps_\Gamma} \int_{\Sigma_t} G'(\tilde{\varphi}) \psi \lvert_\Gamma \, \d S \, \d s \\
   &\quad  + (1/K) \int_{\Sigma_t} (\tilde{\varphi} - \tilde{\phi}) (\eps \psi \lvert_\Gamma -\eps \psi)\, \d S \, \d s.
  \end{aligned}
\end{equation}
\end{definition}
\begin{definition}\label{def-Lyapunov}
 We introduce the Lyapunov functional
\begin{align}\nonumber
\tilde{\mathcal{E}}_{tot}: \mathbb{V} \ni (\phi,\varphi) 
\mapsto & \int_{\mathcal{O}} \left(\frac{\eps}{2} \lvert \nabla \phi \rvert^2 + \frac1\eps F(\phi) + \frac12 \lvert \phi \rvert^2 \right) \d x
\\
\label{eqn-Lyapunov}
& + \int_{\Gamma} \left[\frac{\eps_\Gamma}{2} \lvert \nabla_\Gamma \varphi \rvert^2 + \frac{1}{\eps_\Gamma} G(\varphi) + \frac{\eps }{2 K} (\varphi - \phi)^2 \right] \, \d S \in [0,\infty),
\end{align}
where the integral $\int_{\Gamma} \cdots \, \d S  $ is the integral with respect to the "riemannian volume" measure.
\end{definition}
We are ready to state our main result.
\begin{theorem}\label{thm-first_main_theorem}
Assume the assumptions \ref{item:H1}-\ref{item:H8} hold. 
Moreover, suppose that there exist $c_{F^\prime}>0$ and $c_{G^\prime}>0$ such that for all \(r \in \mathbb{R}\),
        \begin{equation}\label{F'_and G'_additional_condition}
           \begin{aligned}
              \lvert F^\prime(r) \rvert \leq c_{F^\prime}(1 + \lvert r \rvert^{p-1}),
                \\
              \lvert G^\prime(r) \rvert \leq c_{G^\prime}(1 + \lvert r \rvert^{q-1}), 
           \end{aligned}
        \end{equation}
where 
     \begin{equation*}
        p \in 
             \begin{cases}
                   [2,\infty) &\mbox{ if }  d=2, \\
                   [2,6]  &\mbox{ if }  d=3,
              \end{cases}
                \mbox{ and }  q \in [2,\infty),
    \end{equation*}
and there exist $c_{F^\bis}>0$ and $c_{G^\bis}>0$ such that 
        \begin{equation}\label{F''_and G''_additional_condition}
           \begin{aligned}
              \lvert F^\bis(r) \rvert \leq c_{F^\bis}(1 + \lvert r \rvert^{p-2}), \; \; \forall r \in \mathbb{R},\\
              \lvert G^\bis(r) \rvert \leq c_{G^\bis}(1 + \lvert r \rvert^{q-2}), \; \; \forall r \in \mathbb{R},
           \end{aligned}
        \end{equation}
where 
     \begin{equation*}
        p \in 
             \begin{cases}
                   [2,\infty) &\mbox{ if }  d=2, \\
                   (0,4]  &\mbox{ if }  d=3,
              \end{cases}
                \mbox{ and }  q \in [2,\infty).
    \end{equation*}
Then, there exists a weak martingale solution
\[\left(\tilde{\Omega}, \tilde{\mathcal{F}},\tilde{\mathbb{F}}, \tilde{\mathbb{P}}, \tilde{\bu}, \tilde{\phi}, \tilde{\varphi}, \tilde{\mu}, \tilde{\theta}, \tilde{\mathcal{W}} \right)\]
to system \eqref{eq1.1a}-\eqref{eq1.10a} in the sense of Definition \ref{def3.4}. Moreover, the process  $(\tilde{\mu},\tilde{\theta})$ satisfies 
     \begin{equation}\label{Eqt-titlde-(mu,theta)-L4(0,T;H)-norm}
      (\tilde{\mu},\tilde{\theta}) \in L^4(0,T;\mathbb{H}), \quad \tilde{\mathbb{P}}\text{-a.s.,}
     \end{equation}
and we have
\begin{align}\label{eq3.29}
& \sup_{s \in [0,T]} \tilde{\mathbb{E}}[\tilde{\mathcal{E}}_{tot}(\tilde{\phi}(s), \tilde{\varphi}(s))] + \tilde{\mathbb{E}} \int_{Q_t} [2 \nu(\tilde{\phi}) \lvert D\tilde{\bu} \rvert^2 + \lambda(\tilde{\phi}) \lvert \tilde{\bu} \rvert^2] \,\d x \,\d s \notag
\\
& + \tilde{\mathbb{E}} \int_{\Sigma_t} \gamma(\tilde{\varphi}) \lvert \tilde{\bu} \rvert^2 \,\d S\,\d s + \frac12 \tilde{\mathbb{E}} \int_{Q_t} M_{\mathcal{O}}(\tilde{\phi}) \lvert \nabla \tilde{\mu} \rvert^2 \, \d x \,\d s + \tilde{\mathbb{E}} \int_{\Sigma_t} M_\Gamma(\tilde{\varphi}) \lvert \nabla_\Gamma \tilde{\theta} \rvert^2 \,\d S \,\d s  \notag
            \\
&\leq \tilde{\mathcal{E}}_{tot}(\phi_0,\varphi_0) + C \bigg[(1/K) + \tilde{\mathbb{E}} \int_0^{t} \Vert F_{1}(\tilde{\phi}) \Vert_{\mathscr{T}_2(U,L^2(\mathcal{O}))}^2\,\d s + \bar{\mathbb{E}} \int_0^{t} \lvert \nabla_\Gamma \tilde{\varphi} \rvert_{\Gamma}^2 \, \d s 
                 \\
&\quad + \left(1 + 1/K \right) \tilde{\mathbb{E}} \int_0^t  \lvert \nabla \tilde{\phi} \rvert^2 \, \d s + (1/K) \tilde{\mathbb{E}} \int_0^{t} \Vert F_2(\tilde{\varphi}) \Vert_{\mathscr{T}_2(U_\Gamma,L^2(\Gamma))}^2 \, \d s \notag
                      \\ 
&\quad + \tilde{\mathbb{E}} \int_0^{t} \sum_{k=1}^\infty \int_{\mathcal{O}} \lvert F^\bis(\tilde{\phi}) \rvert \lvert F_1(\tilde{\phi}) e_{1,k} \rvert^2 \, \d x \,\d s 
+ \tilde{\mathbb{E}} \int_0^{t}  \sum_{k=1}^\infty \int_{\Gamma} \lvert G^\bis(\tilde{\varphi}) \rvert \lvert F_2(\tilde{\varphi}) e_{2,k} \rvert^2\, \d S\,\d s \bigg], \; \; t \in [0,T]. \notag
\end{align}
where the constant $C$ depends on $\eps,\,\Gamma,\,\eps_\Gamma,\,C_1,\,C_2,\,C_{\Gamma,\mathcal{O}}$, and $\bar{M}_0$.
\newline
If $\mathcal{O}$ is of class $C^\iota$, $\iota \in \{2,3\}$, then 
\begin{align}
\label{Eqn-tilde-(phi,varphi)-L4(0,T;H2-norm}
 (\tilde{\phi},\tilde{\varphi}) \in L^4(0,T;\mathcal{H}^2), \;\; &\tilde{\mathbb{P}}\mbox{-a.s.},
\\
\label{Eqn-tilde-(F-prime(tilde{phi}),G-prime(tilde{varphi}))-L2(0,T;V-norm}
 (F^\prime(\tilde{\phi}),G^\prime(\tilde{\varphi})) \in L^2(0,T;\mathbb{V}), \;\; &\tilde{\mathbb{P}}\mbox{-a.s.},
\end{align}
and the process $(\tilde{\phi},\tilde{\varphi},\tilde{\mu},\tilde{\theta})$ satisfies the equations \eqref{eq1.4a} and \eqref{eq1.6a} in the strong sense, i.e.
\begin{align}
\label{Eqn-identification-tilde-mu}
\tilde{\mu}= -\eps \Delta \tilde{\phi } + \frac1\eps F^\prime(\tilde{\phi})  &\mbox{ a.e. in }  Q_T, \; \tilde{\mathbb{P}} \mbox{ -a.s., } 
\\
\label{Eqn-identification-tilde-theta}
\tilde{\theta}= -\eps_\Gamma \Delta_\Gamma \tilde{\varphi}+ \frac{1}{\eps_\Gamma} G'(\tilde{\varphi}) + \eps \partial_{\bn} \tilde{\phi} &\mbox{ a.e. on }   \Sigma_T, \; \tilde{\mathbb{P}} \mbox{ -a.s. } 
\end{align}
If $\mathcal{O}$ is of class $C^3$, then 
       \begin{equation}\label{Eqn-tilde-(phi,varphi)-L2(0,T;H3-norm}
          (\tilde{\phi},\tilde{\varphi}) \in L^2(0,T;\mathcal{H}^3), \;\; \tilde{\mathbb{P}}\mbox{-a.s.}
        \end{equation}
\end{theorem}
The complete proof of Theorem \ref{thm-first_main_theorem} is postponed in Section \ref{eqn-thm-first_main_theorem}.
\section{Existence of global weak solution\dela{: Case \mbox{$K>0$}}}\label{sect4}
This section is devoted to the proof of Theorem \ref{thm-first_main_theorem}. To be precise,  we will prove the  existence of a global weak martingale solution to the stochastic Cahn-Hilliard-Brinkman system, i.e. the problem \eqref{eq1.1a}-\eqref{eq1.10a}. To achieve this, we introduce an approximation scheme by combining the Yosida and Galerkin approximation schemes. First, we discuss the existence of the approximation scheme. Next, we prove a priori estimates that are useful for establishing the tightness of the family of Galerkin approximations. In a subsequent step, we construct our new stochastic basis by exploiting the Skorokhod-Jakubowski Theorem in the framework of non-metric spaces, see \cite{Jakubowski_1997}, and then pass to the limit in the approximating scheme. Finally, we derive additional a priori estimates that are useful for establishing the tightness of the family of Yosida approximation problems.
\subsection{Yosida approximation}
Define the following non-decreasing and continuous map
\[
\tilde{F}: \mathbb{R} \ni s \mapsto  \left[F(s) + \frac{\tilde{c}_F}{2} s^2 \right] \in [0,\infty).
\]
By assumption \ref{item:H8}, function $\tilde{F}$ takes finite values at every point in its domain $D(\tilde{F})= \mathbb{R}$. In particular, $\tilde{F}$ is a proper, convex, and lower semicontinuous function. Therefore, invoking the theory of maximal monotone operators, see, e.g., 
\cite[Chapter II]{Barbu_1976} or \cite{Brezis_1973,Showalter_1997}, the subgradient of  $\tilde{F}$ exists and we put
\[
\mathbb{A}= \partial \tilde{F}: D(\mathbb{A})= \mathbb{R} \to \mathbb{R}.
\]
Let us recall the following result, which establishes the action of the subgradient operator at the regularity points, see \cite[Chapter II, Example 3]{Barbu_1976}.
\begin{lemma}
Assume that the map $\psi: \mathbb{R} \to (-\infty,+\infty]$ is convex and differentiable at some $s \in \mathbb{R}$. Then, $\partial \psi(s)= \psi^\prime(s)$.
\end{lemma}
Thanks to the previous lemma and because $\tilde{F}$ is of $C^1$-class on $\mathbb{R}$, we infer that
    \begin{equation}
      \mathbb{A}(s)= \tilde{F}^\prime(s), \; \; s \in \mathbb{R},
   \end{equation}
where $\tilde{F}^\prime(s)$ refers to the standard derivative of $\tilde{F}$ at $s$. 
\newline
For every $\delta \in (0,1)$, we define a function $\tilde{F}_\delta$ by
\[
\tilde{F}_\delta(x)
= \inf_{y \in \mathbb{R}}\left\{ \frac{1}{2 \delta} \lvert x - y \rvert^2 + \tilde{F}(y) \right\}, \; \;  x \in \mathbb{R}
\]
and define
   \begin{align*}
      &J_\delta: \mathbb{R} \ni s \mapsto (I + \delta \mathbb{A})^{-1}(s) \in \mathbb{R},
      \\
      &\mathbb{A}_\delta: \mathbb{R} \ni s \mapsto \frac{1}{\delta} (s - J_\delta(s))  \in \mathbb{R},
   \end{align*}
the $1$-Lipschitz continuous resolvent operator and the Yosida approximation of $\mathbb{A}$, respectively.
\newline
By the elementary properties of the maximal monotone operators, see, for instance, \cite{Barbu_1976,Brezis_1973,Showalter_1997},
we approximate the potential $F$ by a sequence of nonnegative functions $(F_\delta)_{\delta>0}$, called the Yosida approximation of $F$:
    \begin{equation}
      F_\delta: \mathbb{R} \ni s \mapsto \left[\frac{\delta}{2} \lvert \mathbb{A}_\delta(s) \rvert^2 + \tilde{F}(J_\delta(s)) - \frac{\tilde{c}_F}{2} s^2 \right] \in [0,\infty).
    \end{equation}
Furthermore, $\tilde{F}_\delta$ and $F_\delta$ enjoy the following properties:
\begin{enumerate}[label=(P\arabic*),ref=(P\arabic*)]

\item \label{item:P1} $\tilde{F}_\delta(s)= \frac{\delta}{2} \lvert \mathbb{A}_\delta(s) \rvert^2 + \tilde{F}(J_\delta(s))$, for every $\delta>0$ and $s \in \mathbb{R}$,

\item \label{item:P2} $\lim_{\delta \to 0} F_\delta(s)= F(s)$, for all $s \in \mathbb{R}$,

\item \label{item:P3} $\tilde{F}(J_\delta(s)) \leq \tilde{F}_\delta(s) \leq \tilde{F}(s)$, for every $\delta> 0$ and $s \in \mathbb{R}$,

\item \label{item:P4} $F^\prime_\delta(r)= \mathbb{A}_\delta(r) - \tilde{c}_F r, \; r \in \mathbb{R}$, and $F^\prime_\delta$ is Lipschitz on $\mathbb{R}$ with constant $L_{F_\delta^\prime} \coloneq \frac{1}{\delta} + \tilde{c}_F$,

\item \label{item:P5} $\lvert F^\prime_\delta(s) \rvert \nearrow \lvert F^\prime(s) \rvert$ for all $s \in \mathbb{R}$ as $\delta \to 0$,

\item \label{item:P6} $F_\delta(0)= F(0)$ and $F^\prime_\delta(0)=0$, for every $\delta> 0$.
\end{enumerate}
\begin{remark}
Since $\mathbb{A}_\delta$ is differentiable on $\mathbb{R}$ and $\frac{1}{\delta}$-Lipschitz, by the properties \ref{item:P4} and \ref{item:P6}, we infer that $F_\delta$ satisfies the linear growth condition, i.e. there exists a positive constant $\tilde{c}_\delta$ depending on $\delta, \,\tilde{c}_F, \, F(0)$ such that 
     \begin{equation} \label{eq4.31}
       \lvert F_\delta(s) \rvert \leq \tilde{c}_\delta (1 + \lvert s \rvert^2), \; \; s \in \mathbb{R}.
     \end{equation}
\end{remark}
In a similar way, we define the Yosida approximation $G_\delta$ of $G$.
\newline
We fix $\delta \in (0,1)$ and we consider the following regularized stochastic Cahn-Hilliard-Brinkman problem
\begin{subequations}\label{Eqt1.1a}
\begin{align}
 -\diver(2 \nu(\phi_\delta) D\bu_\delta) + \lambda (\phi_\delta) \bu_\delta + \nabla p_\delta = \mu_\delta \nabla \phi_\delta \quad &\mbox{ in } Q_T,  \label{eq1.1aa} 
 \\
\diver \bu_\delta=0  &\mbox{ in } Q_T, \label{eq1.2aa}
    \\
\d \phi_\delta + [\diver(\phi_\delta \bu_\delta) - \diver(M_\mathcal{O}(\phi_\delta) \nabla \mu_\delta)]\d t= F_{1}(\phi_\delta) \d W &\mbox{ in } Q_T, \label{eq1.3aa}
	    \\
\mu_\delta= -\eps \Delta \phi_\delta + \frac1\eps F_\delta^\prime(\phi_\delta) &\mbox{ in } Q_T, \label{eq1.4aa} 
\end{align}
\end{subequations} 
\begin{subequations}\label{Eqt1.2a}
\begin{align}
\d \varphi_\delta + [\diver_\Gamma(\varphi_\delta \bu_\delta) - \diver_\Gamma (M_\Gamma(\varphi_\delta) \nabla_\Gamma \theta_\delta)]\d t= F_2(\varphi_\delta)\, \d W_\Gamma  &\mbox{ on }   \Sigma_T, \label{eq1.5aa}
\\
\theta_\delta= -\eps_\Gamma \Delta_\Gamma \varphi_\delta + \frac{1}{\eps_\Gamma} G_\delta^\prime(\varphi_\delta) + \eps \partial_{\bn} \phi_\delta  &\mbox{ on }   \Sigma_T, \label{eq1.6aa}
  \\
K \partial_{\bn} \phi_\delta= \varphi_\delta - \phi_\delta &\mbox{ on }   \Sigma_T, \label{eq1.7aa} 
\end{align}
\end{subequations}   
together with the following boundary and initial conditions
\begin{subequations}
\begin{align}
M_\mathcal{O} (\phi_\delta) \partial_{\bn} \mu_\delta= \bu_\delta \cdot \bn= 0 &\mbox{ on }  \Sigma_T, \label{eq1.8aa}
 \\
[2 \nu(\phi_\delta) (D \bu_\delta) \bn + \gamma(\varphi_\delta) \bu_\delta]_\tau= -[\varphi_\delta \nabla_\Gamma \theta_\delta]_\tau &\mbox{ on }   \Sigma_T, \label{eq1.9aa}                                           
\end{align}
\end{subequations}   
\begin{subequations}
\begin{align}
\phi_\delta(0)= \phi_0 \mbox{ in }  \mathcal{O} \mbox{ and }  \varphi_\delta(0)= \varphi_0 &\mbox{ on }  \Gamma. \label{eq1.10aa}
\end{align}
\end{subequations}   
Next, we proved the existence  of  a martingale solution to Problem \eqref{Eqt1.1a}-\eqref{eq1.10aa} in the sense of Definition \eqref{def3.4}, but with  the potential $F$ being replaced by  its Yosida approximation $F_\delta$.
\subsection{Galerkin approximation}\label{Sub-sect-Galerkin-approximation}
Since both $\mathcal{O}$ and $\Gamma$ are bounded, the classical spectral Theorem and the properties of our three linear operators imply the following assertions.

\begin{trivlist}
\item[(i)] There exists an increasing sequence $(\xi_j)_{j=1}^\infty$ with $\xi_1>0$ and $\xi_j \nearrow \infty$ as $j \to \infty$, the eigenvalues of the Stokes operator $\bA$, and a family $\{\bw_j\}_{j=1}^\infty \subset D(\bA)$ of eigen-functions of $\bA$ that is orthonormal in $H$ and such that $\bA \bw_j= \xi_j \bw_j$, $j=1,2,\ldots$.
\item[(ii)] The Laplace-Beltrami operator $\Delta_\Gamma$ admits a spectral decomposition, i.e. there exist an orthonormal
basis $\{\bar{\Lambda}_i\}_{i \in \mathbb{N}} \subset H^2(\Gamma)$ of $L^2(\Gamma)$, consisting of eigenvectors of $ \Delta_\Gamma$ and a set of eigenvalues $\{\lambda^\Gamma_i\}_{i \in \mathbb{N}} \subset \mathbb{N}$, where
\begin{equation}\label{eq4.29}
- \Delta_\Gamma \bar{\Lambda}_i= \lambda^\Gamma_i \bar{\Lambda}_i \mbox{ on } \Gamma \dela{, \quad (\bar{\Lambda}_i,\bar{\Lambda}_i)_\Gamma=1, \quad \forall i \in \mathbb{N}}.
\end{equation}
\item[(iii)] There exists  an orthonormal basis $\{\bar{\upsilon}_j \in H^2(\mathcal{O}): j=1,2,\ldots\}$ of $L^2(\mathcal{O})$, 
whose elements are the eigen-functions of the Poisson-Neumann eigenvalue problem
   \begin{equation}\label{eq4.28}
     \begin{cases}
      -\Delta \bar{\upsilon}_j &= \lambda^{\mathcal{O}}_j  \bar{\upsilon}_j \mbox{ in }  \mathcal{O}, 
       \\
     \frac{\partial \bar{\upsilon}_j}{\partial \bn}&=0 \hspace{0.5cm} \mbox{ on }  \Gamma,
      \end{cases}
    \end{equation}
where  $\lambda^{\mathcal{O}}_j$ are the corresponding eigenvalues.  
\end{trivlist}
We assume that  $(\Omega, \mathcal{F}, \mathbb{F}, \mathbb{P})$  is a filtered probability space  satisfying  the usual assumptions, cf. Section \ref{sect3}.  We also fix $n \in \mathbb{N}$ and  $T>0$. \newline
Define the following finite-dimensional linear subspaces,
\begin{align*}
V^n &\coloneq \linspan\{\bw_1,\ldots,\bw_n\} \subset V, \quad 
H_{n}^1 \coloneq \linspan\{\bar{\upsilon}_1,\ldots,\bar{\upsilon}_n\} \subset H^1(\mathcal{O}),
  \\
H_{\Gamma,n}^1  &\coloneq \linspan\{\bar{\Lambda}_1,\ldots,\bar{\Lambda}_n\} \subset H^1(\Gamma), \quad
\mathcal{V}_n  \coloneq \linspan\{(\bar{\upsilon}_1,\bar{\Lambda}_1), \ldots,(\bar{\upsilon}_n,\bar{\Lambda}_n)\} \subset \mathbb{V}.
\end{align*}
Next, we denote by $P_n$ the $H$-orthogonal projection onto $V^n$ and by $\mathcal{S}_n$ the orthogonal projection from $L^2(\mathcal{O})$ onto $H_{n}^1$. 
We also denote by $\mathcal{S}_{n,\Gamma}$ the orthogonal projection from $L^2(\Gamma)$ onto $H_{\Gamma,n}^1$. 
\newline
Now, we seek an approximation to a solution of problem \eqref{eq1.1aa}-\eqref{eq1.10aa} and the initial condition $(\phi_{0,n}, \varphi_{0,n})$ of  the following forms:
    \begin{equation}\label{approximation-eqts}
    \begin{aligned}
     \phi_{\delta,n}(t,\cdot,\omega)&= \sum_{k=1}^n a_k^n(t,\omega) \bar{\upsilon}_k(\cdot) \in H_{n}^1, \quad 
     \varphi_{\delta,n}(t,\cdot,\omega)= \sum_{k=1}^n b_k^n(t,\omega) \bar{\Lambda}_k(\cdot) \in H_{\Gamma,n}^1, \\
     \mu_{\delta,n}(t,\cdot,\omega)&= \sum_{k=1}^n c_k^n(t,\omega) \bar{\upsilon}_k(\cdot) \in H_{n}^1, \quad 
     \theta_{\delta,n}(t,\cdot,\omega)= \sum_{k=1}^n d_k^n(t,\omega) \bar{\Lambda}_k(\cdot) \in H_{\Gamma,n}^1, \\
     \bu_{\delta,n}(t,\cdot,\omega)&= \sum_{k=1}^n e_k^n(t,\omega) \bw_k(\cdot) \in V^n, \; \; (t,\omega) \in (0,T) \times \Omega,
    \end{aligned}
  \end{equation}
$$\phi_{\delta,n}(0)= \mathcal{S}_n \phi_\delta(0)= \mathcal{S}_n\phi_0 \coloneq\phi_{0,n},     ~ 
	\varphi_{\delta,n}(0)= \mathcal{S}_{n,\Gamma} \varphi_\delta(0)= \mathcal{S}_{n,\Gamma} \varphi_0\coloneq\varphi_{0,n}.
$$
\dela{
Obviously, as $n \to \infty$, we have 
    \begin{equation}
      \phi_{0,n} \to \phi_0 \mbox{ in }  V_1   \mbox{ and }  \varphi_{0,n} \to \varphi_0  \mbox{ in }  V_\Gamma.
     \end{equation}
Thus, from \eqref{convergence-initial-condition} we see that 
      \begin{equation}
        \|\phi_{0,n}\|_{V_1} \leq \|\phi_0\|_{V_1} \quad \text{and} \quad \|\varphi_{0,n}\|_{V_\Gamma} \leq \|\varphi_0\|_{V_\Gamma}.
      \end{equation}
}
Let us fix $t \in [0,T]$. We introduce the bilinear form $a_{n,t}$ defined by
\begin{align*}
 a_{n,t}:  V \times V \ni (\bu,\bv) \mapsto \int_{\mathcal{O}} \left[ 2 \nu(\phi_{\delta,n}(t)) D \bu: \nabla \bv + \lambda(\phi_{\delta,n}(t)) \bu \cdot \bv \, \d x\right] + \int_{\Gamma} \gamma(\varphi_{\delta,n}(t)) \bu \cdot \bv \, \d S \in \mathbb{R}, 
\end{align*}
which is part of the weak formulation of the Brinkman equation with the  Navier-slip boundary condition. Note that $a_{n,t}$ is $V$-continuous and $V$-coercive, i.e. there exist positive constants $C_0$ and $C_5$ such that for all $\bv \in V$,
    \begin{equation}\label{eq4.34}
      a_{n,t}(\bv,\bv) \geq C_0 \Vert \bv \Vert_{V}^2  \mbox{ and }  \lvert a_{n,t}(\bv,\bv) \rvert \leq C_5 \Vert \bv \Vert_{V}^2.
     \end{equation}
\dela{
Indeed, if $\bv \in V$, we infer from the definition of the space $V$ that there exists a sequence $\bv_m \in \mathcal{C}_{0,\sigma}^\infty(\mathcal{O})$ such that $\bv_m \to \bv$ in $V$ as $m\to \infty$. Moreover, by the trace theorem, there exist a bounded linear operator $\mathcal{T}_{\bn}: \mathbb{H}^1(\mathcal{O}) \to \mathbb{L}^2(\Gamma)$ satisfying $\mathcal{T}_{\bn} \bv_m= \bv_m\lvert_\Gamma$ and a positive constant $C$ depending on $\mathcal{O}$ such that for every $ m \in \mathbb{N}$,
$\Vert \mathcal{T}_{\bn} \bv_m \Vert_{\mathbb{L}^2(\Gamma)} \leq C \Vert \bv_m \Vert_{\mathbb{H}^1(\mathcal{O})}$. This, jointly with the assumptions \ref{item:H5}-\ref{item:H7}, implies that for every $m \in \mathbb{N}$,
\begin{align*}
&\lvert a_{n,t}(\bv_m,\bv_m) \rvert
=\left\lvert 2 \int_{\mathcal{O}} \nu(\phi_{\delta,n}(t)) \lvert D \bv_m \rvert^2 \, \d x + \int_{\mathcal{O}} \lambda(\phi_{\delta,n}(t)) \lvert \bv_m \rvert^2 \, \d x 
+ \int_{\Gamma} \gamma(\varphi_{\delta,n}(t)) \lvert \bv_m \rvert^2 \, \d S \right \rvert 
 \\
&\leq 2 \bar{\nu}_0 \int_{\mathcal{O}} D \bv_m : D \bv_m \, \d x + \lambda_0 \lvert \bv_m \rvert^2 + \bar{\gamma}_0 \Vert \mathcal{T}_{\bn} \bv_m \Vert_{\mathbb{L}^2(\Gamma)}^2 
\leq \bar{\nu}_0 \lvert \nabla \bv_m \rvert^2 + \lambda_0 \lvert \bv_m \rvert^2 + \bar{\gamma}_0 C \Vert \bv_m \Vert_{\mathbb{H}^1(\mathcal{O})}^2.
\end{align*}
Hence, by the LDCT, we deduce that as $m \to \infty$, 
   \begin{align*}
      \lvert a_{n,t}(\bv,\bv) \rvert 
      \leq \bar{\nu}_0 \lvert \nabla \bv \rvert^2 + \lambda_0 \lvert \bv \rvert^2 + \bar{\gamma}_0 C \Vert \bv \Vert_{\mathbb{H}^1(\mathcal{O})}^2
      \leq C(\bar{\nu}_0,\lambda_0,\bar{\gamma}_0,\mathcal{O}) \Vert \bv \Vert_{V}^2, \; \;  \bv \in V.
    \end{align*}
}
Indeed, if $\bv \in V$, then by the trace theorem, there exist a bounded linear operator $\mathcal{T}_{\bn}: \mathbb{H}^1(\mathcal{O}) \to \mathbb{L}^2(\Gamma)$ satisfying $\mathcal{T}_{\bn} \bv= \bv\lvert_\Gamma$ and a positive constant $C$ depending on $\mathcal{O}$ such that 
$\Vert \mathcal{T}_{\bn} \bv \Vert_{\mathbb{L}^2(\Gamma)} \leq C \Vert \bv \Vert_{\mathbb{H}^1(\mathcal{O})}$. 
This, jointly with the assumptions \ref{item:H5}-\ref{item:H7}, implies that,
\begin{align*}
&\lvert a_{n,t}(\bv,\bv) \rvert
=\left\lvert 2 \int_{\mathcal{O}} \nu(\phi_{\delta,n}(t)) \lvert D \bv \rvert^2 \, \d x + \int_{\mathcal{O}} \lambda(\phi_{\delta,n}(t)) \lvert \bv \rvert^2 \, \d x 
+ \int_{\Gamma} \gamma(\varphi_{\delta,n}(t)) \lvert \bv \rvert^2 \, \d S \right \rvert 
 \\
&\leq 2 \bar{\nu}_0 \int_{\mathcal{O}} D \bv : D \bv \, \d x + \lambda_0 \lvert \bv \rvert^2 + \bar{\gamma}_0 \Vert \mathcal{T}_{\bn} \bv \Vert_{\mathbb{L}^2(\Gamma)}^2 
\leq \bar{\nu}_0 \lvert \nabla \bv \rvert^2 + \lambda_0 \lvert \bv \rvert^2 + \bar{\gamma}_0 C \Vert \bv \Vert_{\mathbb{H}^1(\mathcal{O})}^2.
\end{align*}
From this latter inequality, we deduce that for every $t>0$, $n \in \mathbb{N}$, and $\delta \in (0,1)$,
   \begin{align*}
      \lvert a_{n,t}(\bv,\bv) \rvert 
      \leq \bar{\nu}_0 \lvert \nabla \bv \rvert^2 + \lambda_0 \lvert \bv \rvert^2 + \bar{\gamma}_0 C \Vert \bv \Vert_{\mathbb{H}^1(\mathcal{O})}^2
      \leq C(\bar{\nu}_0,\lambda_0,\bar{\gamma}_0,\mathcal{O}) \Vert \bv \Vert_{V}^2, \; \;  \bv \in V.
    \end{align*}
This proves the second part of \eqref{eq4.34}. \newline
Next, using \ref{item:H5}-\ref{item:H7} together with the Korn inequality \eqref{Korn-inequality}, we infer that
\begin{align*}
a_{n,t}(\bv,\bv) 
&\geq 2 \nu_0 \int_{\mathcal{O}} D \bv : D \bv \, \d x  + \gamma_0 \int_\Gamma \lvert \bv \rvert^2 \d S
\geq \min (2 \nu_0,\gamma_0) \left[\int_{\mathcal{O}} D \bv : D \bv \, \d x + \int_\Gamma \lvert \bv \rvert^2 \d S \right]
\\
&\geq C(\nu_0,\gamma_0,C_{\text{KN}}) \lvert \nabla \bv \rvert^2, \; \; \bv \in V.
\dela{
\\
&\geq \nu_0 \lvert \nabla \bv \rvert^2 + \gamma_0 \int_\Gamma \lvert \bv \rvert^2 \d S
\geq \min (\nu_0,\gamma_0) (\vert \nabla \bv \rvert^2 + \lvert \bv \rvert_\Gamma^2).}
\end{align*}
\dela{   
On the other hand, by Lemma \ref{Poincare-inequality}, we obtain
$$\Vert \bv \Vert_{V}^2 
\leq (1 + 2C_P^2) \lvert \nabla \bv \rvert^2 + 2C_P^2 \lvert \bv \rvert_\Gamma^2 
\leq (1 + 2C_P^2) (\lvert \nabla \bv \rvert^2 + \lvert \bv \rvert_\Gamma^2),
$$
from which we deduce that
  \begin{equation*}
    a_{n,t}(\bv,\bv) \geq  \min (\nu_0,\gamma_0) (1 + 2C_P^2)^{-1} \Vert \bv \Vert_V^2, \;\; \bv \in V.
    \end{equation*}
}
This proves the first part of \eqref{eq4.34}. 
From now on, since $a_{n,t}$ is $V$-continuous and $V$-coercive, using the Riesz lemma and Lax-Milgram theorem, see, for instance, \cite[Theorem 1.6]{Girault+Raviart_1979}, we deduce that there exists a unique $ V^n$-valued   process  $\bu_{\delta,n}(t)$ solving $\mathbb{P}$-a.s. the following equation,
\begin{equation}\label{Eqn-obtain-process-u}
   a_{n,t}(\bu_{\delta,n}(t),\bv)
    = - \int_{\Gamma} \varphi_{\delta,n}(t) \nabla_\Gamma \theta_{\delta,n}(t) \cdot \bv \, \d S  - \int_{\mathcal{O}} \phi_{\delta,n}(t) \nabla \mu_{\delta,n}(t) \cdot \bv \, \d x, \; \; \bv \in V^n \subset V.
\end{equation}
Next, we require that the sequence of processes $(\bu_{\delta,n},\phi_{\delta,n},\varphi_{\delta,n},\mu_{\delta,n},\theta_{\delta,n})_{n \in \mathbb{N}}$ solves the following finite dimensional problem for all $\bv \in V^n$, $\upsilon \in H_{n}^1$, $\upsilon \lvert_\Gamma \in H_{\Gamma,n}^1$, and $(\psi,\psi \lvert_\Gamma) \in \mathcal{V}_n$,
\begin{align}
\label{eq4.34a}
& 2 \int_{Q_t} \nu(\phi_{\delta,n}) D\bu_{\delta,n}: D\bv\, \d x\, \d s + \int_{Q_t} \lambda(\phi_{\delta,n}) \bu_{\delta,n} \cdot \bv \, \d x\, \d s + \int_{\Sigma_t} \gamma(\varphi_{\delta,n}) \bu_{\delta,n} \cdot \bv \, \d S \, \d s \notag \\
& 
=  - \int_{\Sigma_t} \varphi_{\delta,n} \nabla_\Gamma \theta_{\delta,n} \cdot \bv \, \d S \, \d s  - \int_{Q_t} \phi_{\delta,n} \nabla \mu_{\delta,n} \cdot \bv \, \d x\, \d s, 
\\
\label{eq4.34b}
&\duality{\phi_{\delta,n}(t)}{\upsilon}{V_1}{V_1^\prime} - \int_{Q_t} \phi_{\delta,n} \bu_{\delta,n} \cdot \nabla \upsilon \, \d x\, \d s + \int_{Q_t} M_\mathcal{O}(\phi_{\delta,n}) \nabla \mu_{\delta,n} \cdot \nabla \upsilon \, \d x\, \d s  \\ 
 &= \duality{\phi_{0,n}}{\upsilon}{V_1}{V_1^\prime} + \left(\int_0^t F_{1}(\phi_{\delta,n}) \,\d W, \upsilon \right), \notag
\end{align}
\begin{align}
 & \duality{\varphi_{\delta,n}(t)}{\upsilon \lvert_\Gamma }{V_\Gamma}{V_\Gamma^\prime} + \int_{\Sigma_t} (-\varphi_{\delta,n} \bu_{\delta,n} \cdot \nabla_\Gamma \upsilon \lvert_\Gamma + M_\Gamma(\varphi_{\delta,n}) \nabla_\Gamma \theta_{\delta,n} \cdot \nabla_\Gamma \upsilon \lvert_\Gamma) \, \d S \, \d s \label{eq4.34c} \\
 &= \duality{\varphi_{0,n}}{\upsilon \lvert_\Gamma}{V_\Gamma}{V_\Gamma^\prime} +  \left(\int_0^t F_2(\varphi_{\delta,n}) \,\d W_\Gamma, \upsilon \lvert_\Gamma \right)_\Gamma, \notag 
                   \\
&\int_{Q_t} \mu_{\delta,n} \psi \, \d x \, \d s + \int_{\Sigma_t} \theta_{\delta,n} \psi \lvert_\Gamma \, \d S\, \d s  
- \eps \int_{Q_t} \nabla \phi_{\delta,n} \cdot \nabla \psi \, \d x \, \d s - \int_{Q_t} \frac1\eps F_\delta^\prime(\phi_{\delta,n}) \psi \, \d x \, \d s \notag \\ 
&= \eps_\Gamma \int_{\Sigma_t} \nabla_\Gamma \varphi_{\delta,n} \cdot \nabla_\Gamma \psi \lvert_\Gamma \, \d S \, \d s +  \frac{1}{\eps_\Gamma} \int_{\Sigma_t} G_\delta^\prime(\varphi_{\delta,n}) \psi \lvert_\Gamma \, \d S \, \d s \label{eq4.34d} \\
&\qquad  + (1/K) \int_{\Sigma_t} (\varphi_{\delta,n} - \phi_{\delta,n}) (\eps \psi \lvert_\Gamma -\eps \psi)\, \d S \, \d s, \; \; t \in [0,T]. \notag
\end{align}
\dela{
Inserting \eqref{approximation-eqts} into \eqref{eq4.34d}, and then choosing as test functions $(\psi,\psi \lvert_\Gamma)= (\bar{\upsilon}_i,0)$ and $(\psi,\psi \lvert_\Gamma)= (0,\bar{\Lambda}_i)$ for each $i\in\{1,\ldots,n\}$, we see that the processes $a_i^n$, $b_i^n$, $c_i^n$, and $d_i^n$ satisfy, respectively, the following system of $2n$ algebraic equations
\begin{align}
c_i^n(t) 
=& \eps \lambda_{\mathcal{O}} a_i^n(t) + \frac{1}{\eps} \int_{\mathcal{O}} F_\delta^\prime\left(\sum_{k=1}^n a_k^n(t) \bar{\upsilon}_k\right) \bar{\upsilon}_i \,\d x - \eps[K] \sum_{k=1}^n \int_{\Gamma} b_k^n(t) \bar{\Lambda}_k \bar{\upsilon}_i \,\d S \notag \\
& + \eps [K] \sum_{k=1}^n \int_{\Gamma} a_k^n(t) \bar{\upsilon}_k \bar{\upsilon}_i \,\d S, \label{eq4.35}
   \\ 
 d_i^n(t)
=& (\eps_\Gamma \lambda^\Gamma_i + [K]) b_i^n(t) + \frac{1}{\eps_\Gamma} \int_{\Gamma} G_\delta^\prime\left(\sum_{k=1}^n b_k^n(t) \bar{\Lambda}_k\right) \bar{\Lambda}_i \, \d S \notag \\
&- [K] \sum_{k=1}^n a_k^n(t) \int_{\mathcal{O}} \bar{\upsilon}_k \bar{\Lambda}_i \, \d S. \label{eq4.36}
\end{align}
        
On the other hand, inserting \eqref{approximation-eqts} into \eqref{eq4.34b} and \eqref{eq4.34c}, choosing as test functions $\upsilon= \bar{\upsilon}_k$, $\upsilon \lvert_\Gamma= \bar{\Lambda}_k$ for each $k\in\{1,\ldots,n\}$, we deduce that the processes $a_k^n$, $b_k^n$, $c_k^n$, and $d_k^n$ satisfy the system of $2n$ ordinary stochastic differential equations 
\begin{align}\label{eqn-Galerkin-01}
&\d a_i^n(t) - \sum_{k=1}^n \int_{\mathcal{O}}  a_k^n(t) \bar{\upsilon}_k \bu_n(t) \cdot \nabla \bar{\upsilon}_i \, \d x + \sum_{k=1}^n c_k^n(t) \int_{\mathcal{O}} M_\mathcal{O} \left(\sum_{k=1}^n a_k^n(t) \bar{\upsilon}_k\right) \nabla \bar{\upsilon}_k \cdot \nabla \bar{\upsilon}_i \, \d x 
\notag \\
&\quad = \left(F_{1} \left(\sum_{k=1}^n a_k^n(t) \bar{\upsilon}_k\right) \,\d W(t), \bar{\upsilon}_i \right),
            \\ 
& \d b_i^n(t) - \sum_{k=1}^n b_k^n(t) \int_{\Gamma} \bar{\Lambda}_k \bu_n(t) \cdot \nabla_\Gamma \bar{\Lambda}_i \, \d S + \sum_{k=1}^n d_k^n(t) \int_{\Gamma} M_\Gamma \left(\sum_{k=1}^n b_k^n(t) \bar{\Lambda}_k\right) \nabla_\Gamma \bar{\Lambda}_k \cdot \nabla_\Gamma \bar{\Lambda}_i \, \d S \notag \\
&\quad=\left(F_{\Gamma} \left(\sum_{k=1}^n b_k^n(t) \bar{\Lambda}_k\right) \,\d W_\Gamma(t), \bar{\Lambda}_i \right)_\Gamma. \label{eqn-Galerkin-02}
\end{align}
Let us point out that, in order to derive the two previous equations, we used the fact that
    \begin{equation*}
       \int_{\mathcal{O}} \nabla \bar{\upsilon}_k \cdot \nabla \bar{\upsilon}_i \, \d x= \int_{\partial \mathcal{O}} \bar{\upsilon}_k \frac{\partial \bar{\upsilon}_i}{\partial \bn} \, \d S - \int_{\mathcal{O}} \bar{\upsilon}_k \Delta \bar{\upsilon}_i \, \d x= \lambda_{\mathcal{O}},
    \end{equation*}
due to \eqref{eq4.28}; and because of \eqref{eq4.29},
\begin{equation*}
 \int_{\Gamma} \nabla_\Gamma \bar{\Lambda}_k \cdot \nabla_\Gamma \bar{\Lambda}_i \, \d S 
 = \lambda_i\lvert_\Gamma.
\end{equation*}

\noindent
Besides, taking into account the identities \eqref{eq4.35} and \eqref{eq4.36}, we further obtain the following system of $2n$ ordinary stochastic differential equations:
\begin{subequations}\label{eq4.37}
\begin{align}
&\d a_i^n(t)  - \sum_{k=1}^n \int_{\mathcal{O}}  a_k^n(t) \bar{\upsilon}_k \bu_n(t) \cdot \nabla \bar{\upsilon}_i \, \d x  + \sum_{k=1}^n \left[ \eps \lambda_{\mathcal{O}} a_k^n(t)  + \frac{1}{\eps} \int_{\mathcal{O}} F_\delta^\prime\left(\sum_{k=1}^n a_k^n(t) \bar{\upsilon}_k\right) \bar{\upsilon}_k \,\d x  \right. \notag \\
&\left.- \eps[K] \sum_{j=1}^n \int_{\Gamma} b_j^n(t) \bar{\Lambda}_j \bar{\upsilon}_k \,\d S + \eps [K] \sum_{j=1}^n \int_{\Gamma} a_j^n(t) \bar{\upsilon}_j \bar{\upsilon}_k \,\d S\right] \int_{\mathcal{O}} M_\mathcal{O} \left(\sum_{k=1}^n a_k^n(t) \bar{\upsilon}_k\right) \nabla \bar{\upsilon}_k \cdot \nabla \bar{\upsilon}_i \, \d x \notag  
     \\
&\quad =  \left(F_{1} \left(\sum_{k=1}^n a_k^n(t) \bar{\upsilon}_k\right) \,\d W(t), \bar{\upsilon}_i \right), \label{eq4.37a}
          \\ \notag \\
&\d b_i^n(t) - \sum_{k=1}^n b_k^n(t) \int_{\Gamma} \bar{\Lambda}_k \bu_n(t) \cdot \nabla_\Gamma \bar{\Lambda}_i \, \d S + \sum_{k=1}^n \bigg[(\eps_\Gamma \lambda_k \lvert_\Gamma + [K]) b_k^n(t) + \frac{1}{\eps_\Gamma} \int_{\Gamma} G_\delta^\prime\left(\sum_{k=1}^n b_k^n(t) \bar{\Lambda}_k\right) \bar{\Lambda}_k \, \d S \notag \\
&\quad - [K] \sum_{j=1}^n a_j^n(t) \int_{\mathcal{O}} \bar{\upsilon}_j \bar{\Lambda}_k \, \d S \bigg] \int_{\Gamma} M_\Gamma \left(\sum_{k=1}^n b_k^n(t) \bar{\Lambda}_k\right) \nabla_\Gamma \bar{\Lambda}_k \cdot \nabla_\Gamma \bar{\Lambda}_i \, \d S \label{eq4.37b} \\
&\quad= \left(F_{\Gamma} \left(\sum_{k=1}^n b_k^n(t) \bar{\Lambda}_k\right) \,\d W_\Gamma(t), \bar{\upsilon}_i \right), \notag
                   \\ \notag \\
& a_i^n(0)= (\phi_0,\bar{\upsilon}_i), \quad b_i^n(0)= (\varphi_0,\bar{\Lambda}_i)_\Gamma. \label{eq4.37c}
\end{align}
 \end{subequations}
}
Throughout, we set 
\begin{equation}
\begin{aligned}
B_{1,n}(\bu,\phi) &\coloneq \mathcal{S}_n B_1(\bu,\phi), \quad \tilde{B}_n(\bu,\varphi) \coloneq \mathcal{S}_{n,\Gamma} \tilde{B}(\bu,\varphi), \\
F_{1,n}(\phi) &\coloneq  \left\{U \ni y \mapsto  \sum_{k=1}^\infty \mathcal{S}_n \sigma_k(\phi) y_k \in H_{n}^1
\right\}  \in \mathscr{T}_2(U,H_{n}^1), \\
F_{2,n}(\varphi) &\coloneq \left\{U_\Gamma \ni y \mapsto  \sum_{k=1}^\infty \mathcal{S}_{n,\Gamma} \tilde{\sigma}_k(\varphi) y_k \in H_{\Gamma,n}^1
\right\}  \in \mathscr{T}_2(U_\Gamma,H_{\Gamma,n}^1),
\end{aligned}
\end{equation}
and we note that
   \begin{align*}
     [F_{1,n}(\phi)](e_{1,k})= \mathcal{S}_n \sigma_k(\phi), \; \; \phi \in H_{n}^1, \mbox{ and } [F_{2,n}(\varphi)](e_{2,k})= \mathcal{S}_{n,\Gamma} \tilde{\sigma}_k(\varphi), \; \; \varphi \in H_{\Gamma,n}^1. 
   \end{align*}
From now on, \eqref{eq4.34a}-\eqref{eq4.34d}, can be seen as an ordinary stochastic differential equations in finite-dimensional space $ H_{n}^1 \times H_{\Gamma,n}^1$ and the solution is an $ H_{n}^1 \times H_{\Gamma,n}^1$-valued process $(\phi_{\delta,n}, \varphi_{\delta,n})$ such that
\begin{equation}\label{eqn-Galerkin-011}
\begin{cases}
\d \phi_{\delta,n}= [- B_{1,n}(\bu_{\delta,n},\phi_{\delta,n}) - \mathcal{S}_n A_{\phi_{\delta,n}}(\mu_{\delta,n})]\, \d t +  F_{1,n}(\phi_{\delta,n}) \,\d W(t),
\\
\d \varphi_{\delta,n}= [- \mathcal{S}_{n,\Gamma} \mathcal{A}_{\varphi_{\delta,n}}(\theta_{\delta,n}) - \tilde{B}_n(\bu_{\delta,n},\varphi_{\delta,n})]\,\d t + F_{2,n}(\varphi_{\delta,n})\, \d W_\Gamma,
\\
 \mu_{\delta,n}(t)= \mathcal{S}_n [-\eps \Delta \phi_{\delta,n}(t) + \eps^{-1} F_\delta^\prime(\phi_{\delta,n}(t))], \; \;  t \in [0,T], \\
\theta_{\delta,n}(t)= \mathcal{S}_{n,\Gamma} [-\eps_\Gamma \Delta_\Gamma \varphi_{\delta,n}(t) + \eps_\Gamma^{-1} G_\delta^\prime(\varphi_{\delta,n}(t)) + \eps \partial_{\bn} \phi_{\delta,n}(t)], \; \;  t \in [0,T],
\\
(\phi_{\delta,n}(0,\cdot),\varphi_{\delta,n}(0,\cdot))= (\phi_{0,n},\varphi_{0,n}).
\end{cases}
\end{equation}
The drift in that system is a function
\begin{equation}
\begin{aligned}
\mathbf{b}_n: &\mathcal{V}_n \ni \bX_{\delta,n}\coloneq (\phi_{\delta,n}, \varphi_{\delta,n}) \mapsto
\Bigl(- B_{1,n}(\bu_{\delta,n},\phi_{\delta,n}) - \mathcal{S}_n A_{\phi_{\delta,n}}(\mu_{\delta,n}), \\
& \hspace{5 truecm} - \mathcal{S}_{n,\Gamma} \mathcal{A}_{\varphi_{\delta,n}}(\theta_{\delta,n}) - \tilde{B}_n(\bu_{\delta,n},\varphi_{\delta,n}) \Bigr) \in \mathcal{V}_n,
\end{aligned}
\end{equation}
with $\bu_{\delta,n}: [0,T] \times \Omega \to V^n$ being the process obtained previously through an application of the Riesz lemma and Lax-Milgram theorem. Note that $\bu_{\delta,n}$ satisfies the equation \eqref{Eqn-obtain-process-u}.

\noindent
The diffusion coefficient is
  \begin{equation}
    \boldsymbol{\sigma}_n: \mathcal{V}_n \ni \bX_{\delta,n}= (\phi_{\delta,n}, \varphi_{\delta,n}) \mapsto \mathrm{diag} \left( F_{1,n}(\phi_{\delta,n}),F_{2,n}(\varphi_{\delta,n})\right) \in \mathscr{T}_2(\mathcal{U},\mathcal{V}_n).
  \end{equation}
Using the notation introduced above, the system \eqref{eqn-Galerkin-011} can be written in the following compact form:
   \begin{equation}\label{compact-stochastic-problem}
     \d \bX_{\delta,n}= \mathbf{b}_n(\bX_{\delta,n})\, \d t + \boldsymbol{\sigma}_n(\bX_{\delta,n})\,\d \mathcal{W}.
   \end{equation}
Before we continuous, let us think a bit about the above equation.
First of all, the space $\mathcal{V}_n$ is finite-dimensional and, therefore, all norms on it are equivalent. In the calculations below, we will use the norm which is the norm inherited from the bulk surface product space $\mathbb{V}$. Thus, we will use the It\^o Lemma to the free energy functional 
\begin{equation}\label{Eqn-Lyaponov-function}
\begin{aligned}
\mathcal{E}: \mathcal{V}_n \ni (\phi,\varphi) \mapsto \Bigl[\int_{\mathcal{O}} \left(\frac{\eps}{2} \lvert \nabla \phi \rvert^2 + \frac{1}{\eps} F_\delta(\phi) \right)\d x &+ \int_{\Gamma} \Bigl(\frac{\eps_\Gamma}{2} \lvert \nabla_\Gamma \varphi \rvert^2 + \frac{1}{\eps_\Gamma} G_\delta(\varphi) \\
& \hspace{1 truecm} + \frac{\eps}{2 K} (\varphi - \phi)^2 \Bigr)\d S\Bigr] \in \mathbb{R}.
\end{aligned}
\end{equation}
Let us point out that had if the maps $\boldsymbol{b}_n$ and $\boldsymbol{\sigma}_n$ were globally Lipschitz, one would be able to  apply the standard theory of an  stochastic differential equations in finite dimensional spaces. However, this is not the case. Fortunately,  by some clever calculations, we can prove that
the system \eqref{compact-stochastic-problem} satisfies the assumptions  of  Lemma \ref{stochastic-ordinary-theorem}. Hence, 
we deduce that there exists a global maximal solution, i.e. an adapted process $\bX_{\delta,n}=\{(\phi_{\delta,n}(t), \varphi_{\delta,n}(t)), t \in [0,T]\}$, in $\mathcal{V}_n$, unique solution of the stochastic differential equation \eqref{compact-stochastic-problem}. In particular, we infer that for every $ r \geq 2$, 
\[
\bX_{\delta,n} \in L^r(\Omega;\, C([0,T];\mathcal{V}_n)) \mbox{ and } (\mu_{\delta,n}, \theta_{\delta,n}) \in L^r(\Omega;\, C([0,T];\mathcal{V}_n)).
\]
\dela{
We now prove the following existence theorem.
\begin{theorem}\label{thm-Galerkin-01}
  Let the assumptions $(H1)\text{-}(H8)$ be satisfied. For every $n\in \mathbb{N}$, there exists a unique couple of $(\mathcal{F}_t)_t\text{-}$adapted processes $\boldsymbol{a}_n:\Omega \times [0,T] \to \mathbb{R}^n$ and $\boldsymbol{b}_n:\Omega \times [0,T] \to \mathbb{R}^n$ satisfying problem \eqref{eq4.37a}-\eqref{eq4.37c} or \eqref{compact-stochastic-problem}. Furthermore, for every $r\geq 2$, we have
\[
\boldsymbol{a}_n,\,\boldsymbol{b}_n \in L^r(\Omega;\, C([0,T];\mathbb{R}^n)), 
\]
implying
\[
\phi_{\delta,n},\; \mu_{\delta,n} \in L^r(\Omega;\, C([0,T];H_1^n)); \quad \varphi_{\delta,n},\; \theta_{\delta,n} \in L^r(\Omega;\, C([0,T];H_{\Gamma,n}^1)).
\]
\end{theorem}
\begin{proof}
Let us point out that if $\boldsymbol{b}(\cdot,\cdot)$ and $\boldsymbol{\sigma}(\cdot,\cdot)$ are globally Lipschitz-continuous, one can apply the standard theory of abstract stochastic evolution equations. However, our case does not fall in that case, so we need to verify the conditions of the vectorial version of Lemma \ref{stochastic-ordinary-theorem} as follows.

\noindent
We observe that
\begin{equation}\label{eq4.55}
\begin{aligned}
\langle \boldsymbol{b} (\cdot,\boldsymbol{\Upsilon}_n) - \boldsymbol{b}(\cdot,\Tilde{\Upsilon}_n),\boldsymbol{\Upsilon}_n - \Tilde{\Upsilon}_n\rangle_{\mathbb{V}^\prime,\mathbb{V}}
 = & \langle - B_1(\bu_{\delta,n},\phi_{\delta,n}) + B_1(\bu_{\delta,n},\tilde{\phi}_n^\delta),\phi_n - \tilde{\phi}_n^\delta \rangle_{V_1^\prime,V_1} \\
 &+ \langle - A_{\phi_{\delta,n}}(\mu_{\delta,n}) + A_{\tilde{\phi}_n^\delta}(\tilde{\mu}_n^\delta),\phi_{\delta,n} - \tilde{\phi}_n^\delta \rangle_{V_1^\prime,V_1}\\
 &+ \langle - \mathcal{A}_{\varphi_{\delta,n}}(\theta_{\delta,n}) +  \mathcal{A}_{\Tilde{\varphi}_n^\delta}(\Tilde{\theta}_n^\delta),\varphi_{\delta,n} - \Tilde{\varphi}_n^\delta \rangle_{V_\Gamma^\prime,V_\Gamma} \\
 &+ \langle - \tilde{B}(\bu_{\delta,n},\varphi_{\delta,n}) + \tilde{B}(\bu_{\delta,n},\Tilde{\varphi}_n^\delta),\varphi_{\delta,n} - \tilde{\varphi}_n^\delta \rangle_{V_\Gamma^\prime,V_\Gamma},
\end{aligned}
\end{equation}
             \begin{equation}\label{difference_chemical_potential_n}
             \begin{aligned}
                 \tilde{\mu}_n^\delta - \mu_{\delta,n}&= -\varepsilon \Delta (\tilde{\phi}_n^\delta - \phi_{\delta,n}) + \frac{1}{\varepsilon} (F_\delta^\prime(\tilde{\phi}_n^\delta) - F_\delta^\prime(\phi_{\delta,n})), 
                     \\
                 \nabla \tilde{\mu}_n^\delta&= -\varepsilon \nabla \Delta \tilde{\phi}_n^\delta + \frac{1}{\varepsilon} F_\delta''(\tilde{\phi}_n^\delta) \nabla \tilde{\phi}_n^\delta, 
                        \\
                 \nabla(\tilde{\mu}_n^\delta - \mu_{\delta,n})&= -\varepsilon \nabla \Delta (\tilde{\phi}_n^\delta - \phi_{\delta,n}) + \frac{1}{\varepsilon} F_\delta''(\tilde{\phi}_n^\delta) \nabla (\tilde{\phi}_n^\delta - \phi_{\delta,n}) + \frac{1}{\varepsilon} \left[F_\delta''(\tilde{\phi}_n^\delta) - F_\delta''(\phi_{\delta,n})\right] \nabla \phi_{\delta,n},
             \end{aligned}
             \end{equation}
and due to \eqref{eq1.7a} and the fact that $K>0$,
\begin{equation}
             \begin{aligned}\label{eq4.57} 
                 \tilde{\theta}_n^\delta - \theta_{\delta,n}&= -\varepsilon_\Gamma \Delta_\Gamma (\tilde{\varphi}_n^\delta - \varphi_{\delta,n}) + \frac{1}{\varepsilon_\Gamma} (G^\prime(\tilde{\varphi}_n^\delta) - G^\prime(\varphi_{\delta,n})) + \frac{\eps}{K} [\tilde{\varphi}_n^\delta - \varphi_{\delta,n} - (\tilde{\phi}_n^\delta - \phi_{\delta,n})], \\
                 &\nabla_\Gamma \tilde{\theta}_n^\delta= -\varepsilon_\Gamma \nabla_\Gamma \Delta_\Gamma \tilde{\varphi}_n^\delta + \frac{1}{\varepsilon_\Gamma} G''(\tilde{\varphi}_n^\delta) \nabla_\Gamma \tilde{\varphi}_n^\delta + \frac{\eps}{K} \nabla_\Gamma (\tilde{\varphi}_n^\delta - \tilde{\phi}_n^\delta).
             \end{aligned}
             \end{equation}
Let us proceed with estimating all the terms on the right-hand side of \eqref{eq4.55}. \newline
For the first term on the right-hand side of \eqref{eq4.55}, using \eqref{B1-Property}, we have
\begin{equation}\label{eq4.58}
\begin{aligned}
\langle - B_1(\bu_{\delta,n},\phi_{\delta,n}) + B_1(\bu_{\delta,n},\tilde{\phi}_n^\delta),\phi_{\delta,n} - 
 \tilde{\phi}_n^\delta \rangle_{V_1^\prime,V_1}
  &= \langle B_1(\bu_{\delta,n},\tilde{\phi}_n^\delta - \phi_{\delta,n}),\phi_{\delta,n} - \tilde{\phi}_n^\delta \rangle_{V_1^\prime,V_1} \\
  &\leq \|B_1(\bu_{\delta,n},\tilde{\phi}_n^\delta - \phi_{\delta,n})\|_{V_1^\prime} \|\phi_{\delta,n} - \tilde{\phi}_n^\delta\|_{V_1} \\
  &\leq C(\mathcal{O}) \|\bu_{\delta,n}\|_{V} \|\phi_{\delta,n} - \tilde{\phi}_n^\delta\|_{V_1}^2.
\end{aligned}
\end{equation}
For the fourth term, using \eqref{Property-tilde B}, we find
\begin{equation}\label{eq4.59}
 \begin{aligned}
 \langle - \tilde{B}(\bu_{\delta,n},\varphi_{\delta,n}) + \tilde{B}(\bu_{\delta,n},\tilde{\varphi}_n^\delta),\varphi_{\delta,n} - \tilde{\varphi}_n^\delta \rangle_{V_\Gamma^\prime,V_\Gamma}
  &= \langle \tilde{B}(\bu_{\delta,n}, \tilde{\varphi}_n^\delta - \varphi_{\delta,n}),\varphi_{\delta,n} - \tilde{\varphi}_n^\delta \rangle_{V_\Gamma^\prime,V_\Gamma} \\
  &\leq \|\tilde{B}(\bu_{\delta,n}, \tilde{\varphi}_n^\delta - \varphi_{\delta,n})\|_{V_\Gamma^\prime} \|\varphi_{\delta,n} - \tilde{\varphi}_n^\delta\|_{V_\Gamma} \\
  &\leq C(\Gamma,\mathcal{O}) \|\bu_{\delta,n}\|_{V} \|\varphi_{\delta,n} - \tilde{\varphi}_n^\delta\|_{V_\Gamma}^2.
\end{aligned}
\end{equation}
Before begin estimating the second term on the right-hand side of \eqref{eq4.55}, we observe that
\begin{align}\label{eq4.60a}
 &\langle - A_{\phi_{\delta,n}}(\mu_{\delta,n}) + A_{\tilde{\phi}_n^\delta}(\tilde{\mu}_n^\delta),\phi_{\delta,n} - \tilde{\phi}_n^\delta \rangle_{V_1^\prime,V_1} \notag \\
 &= - \int_{\mathcal{O}} M_\mathcal{O}(\phi_{\delta,n}) \nabla \mu_{\delta,n} \cdot \nabla (\phi_{\delta,n} - \tilde{\phi}_n^\delta)\, \d x + \int_{\mathcal{O}} M_\mathcal{O}(\tilde{\phi}_n^\delta) \nabla \tilde{\mu}_n^\delta \cdot \nabla (\phi_{\delta,n} - \tilde{\phi}_n^\delta)\, \d x \notag \\
 &= - \int_{\mathcal{O}} M_\mathcal{O}(\phi_{\delta,n}) \nabla \mu_{\delta,n} \cdot \nabla (\phi_{\delta,n} - \tilde{\phi}_n^\delta)\, \d x + \int_{\mathcal{O}} M_\mathcal{O}(\phi_{\delta,n}) \nabla \tilde{\mu}_n^\delta \cdot \nabla (\phi_{\delta,n} - \tilde{\phi}_n^\delta)\, \d x \\
 &\quad - \int_{\mathcal{O}} M_\mathcal{O}(\phi_{\delta,n}) \nabla \tilde{\mu}_n^\delta \cdot \nabla (\phi_{\delta,n} - \tilde{\phi}_n^\delta)\, \d x + \int_{\mathcal{O}} M_\mathcal{O}(\tilde{\phi}_n^\delta) \nabla \tilde{\mu}_n^\delta \cdot \nabla (\phi_{\delta,n} - \tilde{\phi}_n^\delta)\, \d x \notag \\
 &= \int_{\mathcal{O}} M_\mathcal{O}(\phi_{\delta,n}) \nabla (\tilde{\mu}_n^\delta - \mu_{\delta,n}) \cdot \nabla (\phi_{\delta,n} - \tilde{\phi}_n^\delta)\, \d x + \int_{\mathcal{O}} [M_\mathcal{O}(\tilde{\phi}_n^\delta) - M_\mathcal{O}(\phi_{\delta,n})] \nabla \tilde{\mu}_n^\delta \cdot \nabla (\phi_{\delta,n} - \tilde{\phi}_n^\delta)\, \d x. \notag
\end{align}
On the other hand, taking into account \eqref{difference_chemical_potential_n}, we find that
\begin{equation}\label{eq4.60}
\begin{aligned}
 & \langle - A_{\phi_{\delta,n}}(\mu_{\delta,n}) + A_{\tilde{\phi}_n^\delta}(\tilde{\mu}_n^\delta),\phi_{\delta,n} - \tilde{\phi}_n^\delta \rangle_{V_1^\prime,V_1} 
\\
 =&-\varepsilon \int_{\mathcal{O}} M_\mathcal{O}(\phi_{\delta,n}) \nabla \Delta (\tilde{\phi}_n^\delta - \phi_{\delta,n}) \cdot \nabla (\phi_{\delta,n} - \tilde{\phi}_n^\delta)\, \d x \\
 &+ \frac{1}{\varepsilon} \int_{\mathcal{O}} M_\mathcal{O}(\phi_n)F_\delta''(\tilde{\phi}_n^\delta) \nabla (\tilde{\phi}_n^\delta - \phi_{\delta,n}) \cdot \nabla (\phi_{\delta,n} - \tilde{\phi}_n^\delta)\, \d x 
       \\
 &+ \frac{1}{\varepsilon} \int_{\mathcal{O}} M_\mathcal{O}(\phi_{\delta,n}) \left[F_\delta''(\tilde{\phi}_n^\delta) - F_\delta''(\phi_{\delta,n})\right] \nabla \phi_{\delta,n} \cdot \nabla (\phi_{\delta,n} - \tilde{\phi}_n^\delta)\, \d x 
            \\
 & -\varepsilon \int_{\mathcal{O}} [M_\mathcal{O}(\tilde{\phi}_n^\delta) - M_\mathcal{O}(\phi_{\delta,n})]  \nabla \Delta \Tilde{\phi}_n^\delta\cdot \nabla (\phi_{\delta,n} - \tilde{\phi}_n^\delta)\, \d x 
                \\
 & + \frac{1}{\varepsilon} \int_{\mathcal{O}} [M_\mathcal{O}(\tilde{\phi}_n^\delta) - M_\mathcal{O}(\phi_{\delta,n})] F_\delta''(\tilde{\phi}_n^\delta) \nabla \tilde{\phi}_n^\delta \cdot \nabla (\phi_{\delta,n} - \tilde{\phi}_n^\delta)\, \d x.
\end{aligned}
\end{equation} 
Now thanks to H\"older's inequality, using $(H4)$ together with the fact that all norms are equivalent on finite dimensional space $H^1_n$, we obtain
\begin{equation}\label{eq4.61}
\begin{aligned}
-\varepsilon \int_{\mathcal{O}} M_\mathcal{O}(\phi_{\delta,n}) \nabla \Delta (\tilde{\phi}_n^\delta - \phi_{\delta,n}) \cdot \nabla (\phi_{\delta,n} - \tilde{\phi}_n^\delta)\, \d x 
&\leq \varepsilon M_\mathcal{O}^* \lvert \nabla \Delta (\tilde{\phi}_n^\delta - \phi_n) \rvert \lvert \nabla (\phi_n - \tilde{\phi}_n^\delta) \rvert \\
&\leq \varepsilon M_\mathcal{O}^* \|\tilde{\phi}_n^\delta - \phi_{\delta,n}\|_{H^3(\mathcal{O})} \|\tilde{\phi}_n^\delta - \phi_{\delta,n}\|_{V_1} \\
&\leq \varepsilon M_\mathcal{O}^* C(n) \|\tilde{\phi}_n^\delta - \phi_{\delta,n}\|_{V_1}^2.
\end{aligned}
\end{equation}
Since $\lvert F''_\delta \rvert \leq C_\delta$ for a certain constant $C_\delta>0$ independent of $n$, using $(H4)$, we find
\begin{equation}\label{eq4.62}
\begin{aligned}
 \frac{1}{\varepsilon} \int_{\mathcal{O}} M_\mathcal{O}(\phi_{\delta,n}) F''_\delta(\tilde{\phi}_n^\delta) \nabla (\tilde{\phi}_n^\delta - \phi_{\delta,n}) \cdot \nabla (\phi_{\delta,n} - \tilde{\phi}_n^\delta)\, \d x 
 &\leq M_\mathcal{O}^* \varepsilon^{-1} C_\delta \int_{\mathcal{O}} \vert \nabla (\phi_{\delta,n} - \tilde{\phi}_n^\delta) \rvert^2\, \d x \\
 &\leq M_\mathcal{O}^* \varepsilon^{-1} C_\delta \|\tilde{\phi}_n^\delta - \phi_{\delta,n}\|_{V_1}^2.
 \end{aligned}
\end{equation}
Arguing similarly as previously, we obtain
\begin{align*}
 \frac{1}{\varepsilon} \int_{\mathcal{O}} M_\mathcal{O}(\phi_{\delta,n}) \left[F_\delta''(\tilde{\phi}_n^\delta) - F_\delta''(\phi_{\delta,n})\right] \nabla \phi_{\delta,n} \cdot \nabla (\phi_{\delta,n} - \tilde{\phi}_n^\delta)\, \d x  
 &\leq 2 \varepsilon^{-1} M_\mathcal{O}^* C_\delta \int_{\mathcal{O}} \lvert \nabla \phi_{\delta,n} \rvert \lvert \nabla (\phi_{\delta,n} - \tilde{\phi}_n^\delta) \rvert\, \d x \\
 &\leq 2 \varepsilon^{-1}  M_\mathcal{O}^* C_\delta \|\phi_{\delta,n}\|_{V_1} \|\phi_{\delta,n} - \tilde{\phi}_n^\delta\|_{V_1}.
\end{align*}
If $\|(\phi_n,\varphi_n)\|_{\mathbb{V}}\leq R$, we further obtain
\begin{equation}\label{eq4.63}
\begin{aligned}
 &\frac{1}{\varepsilon} \int_{\mathcal{O}} M_\mathcal{O}(\phi_{\delta,n}) \left[F''_\delta(\tilde{\phi}_n^\delta) - F''_\delta(\phi_{\delta,n})\right] \nabla \phi_n \cdot \nabla (\phi_{\delta,n} - \tilde{\phi}_n^\delta)\, \d x \\
 &\leq 2 \varepsilon^{-1} M_\mathcal{O}^* C_\delta R \|\phi_{\delta,n} - \tilde{\phi}_n^\delta\|_{V_1} \\
 &\leq \varepsilon^{-1} M_\mathcal{O}^* C_\delta R (1 + \|\phi_{\delta,n} - \tilde{\phi}_n^\delta\|_{V_1}^2).
 \end{aligned}
 \end{equation}
Thanks to H\"older's inequality, using $(H4)$ together with the fact that all norms are equivalent on finite dimensional space $H^1_n$, we get
 \begin{equation}\label{eq4.64}
\begin{aligned}
-\varepsilon \int_{\mathcal{O}} [M_\mathcal{O}(\tilde{\phi}_n^\delta) - M_\mathcal{O}(\phi_{\delta,n})]  \nabla \Delta \tilde{\phi}_n^\delta \cdot \nabla (\phi_{\delta,n} - \tilde{\phi}_n^\delta)\, \d x 
&\leq 2 \varepsilon M_\mathcal{O}^* \lvert \nabla \Delta \tilde{\phi}_n^\delta \rvert \lvert \nabla (\phi_{\delta,n} - \tilde{\phi}_n^\delta) \rvert \\
&\leq 2 \varepsilon M_\mathcal{O}^* \|\tilde{\phi}_n^\delta\|_{H^3(\mathcal{O})} \|\tilde{\phi}_n^\delta - \phi_{\delta,n}\|_{V_1} \\
&\leq 2 \varepsilon M_\mathcal{O}^* C(n) \|\tilde{\phi}_n^\delta\|_{V_1} \|\tilde{\phi}_n^\delta - \phi_{\delta,n}\|_{V_1} \\
&\leq 2 \varepsilon M_\mathcal{O}^* C(n) R \|\tilde{\phi}_n^\delta - \phi_{\delta,n}\|_{V_1} \\
&\leq \varepsilon M_\mathcal{O}^* C(n) R (1 + \|\tilde{\phi}_n^\delta - \phi_{\delta,n}\|_{V_1}^2).
\end{aligned}
\end{equation}
Once more, owing to $(H4)$ and since $\lvert F''_\delta \rvert \leq C_\delta$, we obtain
\begin{align*}
\frac{1}{\varepsilon} \int_{\mathcal{O}} [M_\mathcal{O}(\tilde{\phi}_n^\delta) - M_\mathcal{O}(\phi_{\delta,n})] F''_\delta(\tilde{\phi}_n^\delta) \nabla \tilde{\phi}_n^\delta \cdot \nabla (\phi_{\delta,n} - \tilde{\phi}_n^\delta)\, \d x 
&\leq 2 \varepsilon^{-1} M_\mathcal{O}^* C_\delta \int_{\mathcal{O}} \lvert \nabla \tilde{\phi}_n^\delta \rvert \lvert \nabla (\phi_{\delta,n} - \tilde{\phi}_n^\delta) \rvert\, \d x \\
&\leq 2 \varepsilon^{-1} M_\mathcal{O}^* C_\delta  \|\tilde{\phi}_n^\delta\|_{V_1} \|\phi_{\delta,n} - \tilde{\phi}_n^\delta\|_{V_1}.
\end{align*}
Taking now into account the fact that $\|(\tilde{\phi}_n,\tilde{\varphi}_n)\|_{\mathbb{V}}\leq R$, we further obtain
\begin{equation}\label{eq4.65}
\begin{aligned}
&\frac{1}{\varepsilon} \int_{\mathcal{O}} [M_\mathcal{O}(\tilde{\phi}_n^\delta) - M_\mathcal{O}(\phi_{\delta,n})] F''_\delta(\tilde{\phi}_n^\delta) \nabla \tilde{\phi}_n^\delta \cdot \nabla (\phi_{\delta,n} - \tilde{\phi}_n^\delta)\, \d x \\
&\leq 2 \varepsilon^{-1}  M_\mathcal{O}^* C_\delta R \|\phi_{\delta,n} - \tilde{\phi}_n^\delta\|_{V_1} \\
&\leq \varepsilon^{-1} M_\mathcal{O}^* C_\delta R (1 + \|\phi_{\delta,n} - \tilde{\phi}_n^\delta\|_{V_1}^2).
\end{aligned}
\end{equation}
Plugging the estimates \eqref{eq4.61}-\eqref{eq4.65} on the right-hand side of \eqref{eq4.60}, we arrive at
  \begin{equation}\label{eq4.66}
   \begin{aligned}
     & \langle - A_{\phi_{\delta,n}}(\mu_{\delta,n}) + A_{\tilde{\phi}_n^\delta}(\tilde{\mu}_n^\delta),\phi_{\delta,n} - \tilde{\phi}_n^\delta \rangle_{V_1^\prime,V_1} \\
     &\leq M_\mathcal{O}^* [\varepsilon  C(n)  + \varepsilon^{-1}  C_\delta] \|\tilde{\phi}_n^\delta - \phi_{\delta,n}\|_{V_1}^2 \\
     & + M_\mathcal{O}^* [ 2 \varepsilon^{-1} C_\delta   + \varepsilon C(n)] R (1 + \|\phi_{\delta,n} - \tilde{\phi}_n^\delta\|_{V_1}^2).  
   \end{aligned}
  \end{equation}
For the third term on the right-hand side of \eqref{eq4.55}, as in \eqref{eq4.60a}, we have
\begin{align*}
 &\langle - \mathcal{A}_{\varphi_{\delta,n}}(\theta_{\delta,n}) +  \mathcal{A}_{\Tilde{\varphi}_n^\delta}(\tilde{\theta}_n^\delta),\varphi_{\delta,n} - \tilde{\varphi}_n^\delta \rangle_{V_\Gamma^\prime,V_\Gamma} \\
 &=\int_{\Gamma} M_\Gamma(\varphi_{\delta,n}) \nabla_\Gamma (\tilde{\theta}_n^\delta - \theta_{\delta,n}) \cdot \nabla_\Gamma (\varphi_{\delta,n} - \tilde{\varphi}_n^\delta)\, \d S + \int_{\Gamma} [M_\Gamma (\tilde{\varphi}_n^\delta) - M_\Gamma(\varphi_{\delta,n})] \nabla_\Gamma \tilde{\theta}_n^\delta \cdot \nabla_\Gamma (\varphi_{\delta,n} - \tilde{\varphi}_n^\delta)\, \d S.
\end{align*}
Besides, due to \eqref{eq4.57} and $(H4)$, we infer that
\begin{equation}\label{eq4.68}
\begin{aligned}
&\langle - \mathcal{A}_{\varphi_{\delta,n}}(\theta_{\delta,n}) +  \mathcal{A}_{\tilde{\varphi}_n^\delta}(\tilde{\theta}_n^\delta),\varphi_{\delta,n} - \tilde{\varphi}_n^\delta \rangle_{V_\Gamma^\prime,V_\Gamma} 
   \\
&= -\eps_\Gamma \int_{\Gamma} M_\Gamma(\varphi_{\delta,n}) \nabla_\Gamma \Delta_\Gamma (\tilde{\varphi}_n^\delta - \varphi_{\delta,n}) \cdot \nabla_\Gamma (\varphi_{\delta,n} - \tilde{\varphi}_n^\delta)\, \d S 
       \\
&\quad + \frac{1}{\eps_\Gamma} \int_{\Gamma} M_\Gamma(\varphi_{\delta,n}) \nabla_\Gamma (G_\delta^\prime(\tilde{\varphi}_n^\delta) - G_\delta^\prime(\varphi_{\delta,n})) \cdot \nabla_\Gamma (\varphi_{\delta,n} - \tilde{\varphi}_n^\delta)\, \d S  
             \\
&\quad - \underbrace{\frac{1}{K} \int_{\Gamma} M_\Gamma(\varphi_{\delta,n}) |\nabla_\Gamma (\varphi_{\delta,n} - \tilde{\varphi}_n^\delta)|^2\, \d S}_{\leq 0} 
                \\
&\quad + \frac{1}{K} \int_{\Gamma} M_\Gamma(\varphi_{\delta,n}) \nabla_\Gamma (\phi_{\delta,n} - \tilde{\phi}_n^\delta) \cdot \nabla_\Gamma (\varphi_{\delta,n} - \tilde{\varphi}_n^\delta)\, \d S
                      \\
&\quad - \eps_\Gamma \int_{\Gamma} [M_\Gamma (\tilde{\varphi}_n^\delta) - M_\Gamma(\varphi_{\delta,n})] \nabla_\Gamma \Delta_\Gamma \tilde{\varphi}_n^\delta \cdot \nabla_\Gamma (\varphi_{\delta,n} - \tilde{\varphi}_n^\delta)\, \d S  \\
&\quad + \frac{1}{\varepsilon_\Gamma} \int_{\Gamma} [M_\Gamma (\tilde{\varphi}_n^\delta) - M_\Gamma(\varphi_{\delta,n})] G''_\delta(\tilde{\varphi}_n^\delta) \nabla_\Gamma \tilde{\varphi}_n^\delta \cdot \nabla_\Gamma (\varphi_{\delta,n} - \tilde{\varphi}_n^\delta)\, \d S 
                               \\
&\quad + \frac{1}{K} \int_{\Gamma} [M_\Gamma (\tilde{\varphi}_n^\delta) - M_\Gamma(\varphi_{\delta,n})] \nabla_\Gamma (\tilde{\varphi}_n^\delta - \tilde{\phi}_n^\delta) \cdot \nabla_\Gamma (\varphi_{\delta,n} - \tilde{\varphi}_n^\delta)\, \d S.
\end{aligned}
\end{equation}
By H\"older's inequality, the assumption $(H4)$ and the fact that all norms are equivalent on finite dimensional space $H_{\Gamma,n}^1$, we get
\begin{align*}
 &-\eps_\Gamma \int_{\Gamma} M_\Gamma(\varphi_{\delta,n}) \nabla_\Gamma \Delta_\Gamma (\tilde{\varphi}_n^\delta - \varphi_{\delta,n}) \cdot \nabla_\Gamma (\varphi_{\delta,n} - \tilde{\varphi}_n^\delta)\, \d S \\
 &\leq \eps_\Gamma M_\Gamma^\ast |\nabla_\Gamma \Delta_\Gamma (\tilde{\varphi}_n^\delta - \varphi_{\delta,n})|_\Gamma |\nabla_\Gamma (\varphi_{\delta,n} - \tilde{\varphi}_n^\delta)|_\Gamma \\
 &\leq \eps_\Gamma M_\Gamma^\ast C(n) \|\varphi_{\delta,n} - \tilde{\varphi}_n^\delta\|_{V_\Gamma}^2,
\end{align*}
           \begin{align*}
            &-\eps_\Gamma \int_{\Gamma} [M_\Gamma (\tilde{\varphi}_n^\delta) - M_\Gamma(\varphi_{\delta,n})] \nabla_\Gamma \Delta_\Gamma \tilde{\varphi}_n^\delta \cdot \nabla_\Gamma (\varphi_{\delta,n} - \tilde{\varphi}_n^\delta)\, \d S \\
            &\leq 2 \eps_\Gamma M_\Gamma^\ast |\nabla_\Gamma \Delta_\Gamma \tilde{\varphi}_n^\delta|_\Gamma  |\nabla_\Gamma (\varphi_{\delta,n} - \tilde{\varphi}_n^\delta)|_\Gamma \\
            &\leq 2 \eps_\Gamma M_\Gamma^\ast C(n) \|\tilde{\varphi}_n^\delta\|_{V_\Gamma} \|\varphi_{\delta,n} - \tilde{\varphi}_n^\delta\|_{V_\Gamma} \\
            &\leq 2 \eps_\Gamma R M_\Gamma^\ast C(n)  \|\varphi_{\delta,n} - \tilde{\varphi}_n^\delta\|_{V_\Gamma} \\
            &\leq \eps_\Gamma R M_\Gamma^\ast C(n) [1 + \|\varphi_{\delta,n} - \tilde{\varphi}_n^\delta\|_{V_\Gamma}^2 ],
           \end{align*}
           
\begin{align*}
& \frac{1}{\eps_\Gamma} \int_{\Gamma} [M_\Gamma (\tilde{\varphi}_n^\delta) - M_\Gamma(\varphi_{\delta,n})] G''_\delta(\tilde{\varphi}_n^\delta) \nabla_\Gamma \tilde{\varphi}_n^\delta \cdot \nabla_\Gamma (\varphi_{\delta,n} - \tilde{\varphi}_n^\delta)\, \d S \\
&\leq 2 \eps_\Gamma^{-1} M_\Gamma^\ast C_\delta |\nabla_\Gamma \tilde{\varphi}_n^\delta| |\nabla_\Gamma (\varphi_{\delta,n} - \tilde{\varphi}_n^\delta)| \\
&\leq 2 \eps_\Gamma^{-1} M_\Gamma^\ast C_\delta \|\tilde{\varphi}_n^\delta\|_{V_\Gamma} \|\varphi_{\delta,n} - \tilde{\varphi}_n^\delta\|_{V_\Gamma} \\
&\leq  2 \eps_\Gamma^{-1} M_\Gamma^\ast C_\delta R \|\varphi_{\delta,n} - \tilde{\varphi}_n^\delta\|_{V_\Gamma} \\
&\leq \eps_\Gamma^{-1} M_\Gamma^\ast C_\delta R ( 1 + \|\varphi_{\delta,n} - \tilde{\varphi}_n^\delta\|_{V_\Gamma}^2),
\end{align*}
where we have also used the fact that $|G''_\delta|\leq C_\delta$ for a certain constant $C_\delta>0$ independent of $n$.

\noindent
By \eqref{condition_ G'_and_G''}, we have
            \begin{align*}
            &\frac{1}{\eps_\Gamma} \int_{\Gamma} M_\Gamma(\varphi_{\delta,n}) \nabla_\Gamma (G_\delta^\prime(\tilde{\varphi}_n^\delta) - G_\delta^\prime(\varphi_{\delta,n})) \cdot \nabla_\Gamma (\varphi_{\delta,n} - \tilde{\varphi}_n^\delta)\, \d S \\
            &= -\frac{1}{\eps_\Gamma} \int_{\Gamma} M_\Gamma(\varphi_{\delta,n}) [-G''_\delta(\tilde{\varphi}_n^\delta) \nabla_\Gamma \tilde{\varphi}_n^\delta + G''_\delta(\varphi_{\delta,n}) \nabla_\Gamma \varphi_{\delta,n}] \cdot \nabla_\Gamma (\varphi_{\delta,n} - \tilde{\varphi}_n^\delta)\, \d S \\
            &= -\frac{1}{\eps_\Gamma} \int_{\Gamma} M_\Gamma(\varphi_{\delta,n}) G''_\delta(\tilde{\varphi}_n^\delta) |\nabla_\Gamma (\varphi_{\delta,n} - \tilde{\varphi}_n^\delta)|^2\, \d S \\
            &\quad + \frac{1}{\eps_\Gamma} \int_{\Gamma} M_\Gamma(\varphi_{\delta,n}) [G''_\delta(\tilde{\varphi}_n^\delta) - G''_\delta(\varphi_{\delta,n})] \nabla_\Gamma \varphi_{\delta,n} \cdot \nabla_\Gamma (\varphi_{\delta,n} - \tilde{\varphi}_n^\delta)\, \d S \\
            &= -\underbrace{\frac{1}{\eps_\Gamma} \int_{\Gamma} M_\Gamma(\varphi_{\delta,n}) (G''_\delta(\tilde{\varphi}_n^\delta) + \tilde{c}_G) |\nabla_\Gamma (\varphi_{\delta,n} - \tilde{\varphi}_n^\delta)|^2\, \d S }_{\leq 0} + \frac{\tilde{c}_G}{\eps_\Gamma} \int_{\Gamma} M_\Gamma(\varphi_{\delta,n}) |\nabla_\Gamma (\varphi_{\delta,n} - \tilde{\varphi}_n^\delta)|^2\, \d S \\
            &\quad + \frac{1}{\eps_\Gamma} \int_{\Gamma} M_\Gamma(\varphi_{\delta,n}) [G''_\delta(\tilde{\varphi}_n^\delta) - G''_\delta(\varphi_{\delta,n})] \nabla_\Gamma \varphi_{\delta,n} \cdot \nabla_\Gamma (\varphi_{\delta,n} - \tilde{\varphi}_n^\delta)\, \d S
           \end{align*}
from which and the assumption $(H4)$ and the fact that $|G''_\delta|\leq C_\delta$ with $C_\delta$ independent of $n$, we deduce
\begin{align*}
&\frac{1}{\eps_\Gamma} \int_{\Gamma} M_\Gamma(\varphi_{\delta,n}) \nabla_\Gamma (G_\delta^\prime(\tilde{\varphi}_n^\delta) - G_\delta^\prime(\varphi_{\delta,n})) \cdot \nabla_\Gamma (\varphi_{\delta,n} - \tilde{\varphi}_n^\delta)\, \d S \\
&\leq \frac{\tilde{c}_G}{\eps_\Gamma} M_\Gamma^\ast |\nabla_\Gamma (\varphi_{\delta,n} - \tilde{\varphi}_n^\delta)|_\Gamma^2 + \frac{2 M_\Gamma^\ast}{\eps_\Gamma} |\nabla_\Gamma \varphi_{\delta,n}|_\Gamma  |\nabla_\Gamma (\varphi_{\delta,n} - \tilde{\varphi}_n^\delta)|_\Gamma \\
&\leq \frac{\tilde{c}_G}{\eps_\Gamma} M_\Gamma^\ast \|\varphi_{\delta,n} - \tilde{\varphi}_n^\delta\|_{V_\Gamma}^2 + \frac{2 M_\Gamma^\ast}{\eps_\Gamma} \|\varphi_{\delta,n}\|_{V_\Gamma} \|\varphi_{\delta,n} - \tilde{\varphi}_n^\delta\|_{V_\Gamma} \\
&\leq \frac{\tilde{c}_G}{\eps_\Gamma} M_\Gamma^\ast \|\varphi_{\delta,n} - \tilde{\varphi}_n^\delta\|_{V_\Gamma}^2 + \frac{2 M_\Gamma^\ast}{\eps_\Gamma} R \|\varphi_{\delta,n} - \tilde{\varphi}_n^\delta\|_{V_\Gamma} \\
&\leq \frac{\tilde{c}_G}{\eps_\Gamma} M_\Gamma^\ast \|\varphi_{\delta,n} - \tilde{\varphi}_n^\delta\|_{V_\Gamma}^2 + \frac{M_\Gamma^\ast}{\eps_\Gamma} R (1 + \|\varphi_{\delta,n} - \tilde{\varphi}_n^\delta\|_{V_\Gamma}^2).
\end{align*}
Consequently,
\begin{equation}\label{eq4.69}
\begin{aligned}
&-\eps_\Gamma \int_{\Gamma} M_\Gamma(\varphi_{\delta,n}) \nabla_\Gamma \Delta_\Gamma (\tilde{\varphi}_n^\delta - \varphi_{\delta,n}) \cdot \nabla_\Gamma (\varphi_{\delta,n} - \tilde{\varphi}_n^\delta)\, \d S 
       \\
& + \frac{1}{\eps_\Gamma} \int_{\Gamma} M_\Gamma(\varphi_{\delta,n}) \nabla_\Gamma (G_\delta^\prime(\tilde{\varphi}_n^\delta) - G_\delta^\prime(\varphi_{\delta,n})) \cdot \nabla_\Gamma (\varphi_{\delta,n} - \tilde{\varphi}_n^\delta)\, \d S  
             \\
&-\eps_\Gamma \int_{\Gamma} [M_\Gamma (\tilde{\varphi}_n^\delta) - M_\Gamma(\varphi_{\delta,n})] \nabla_\Gamma \Delta_\Gamma \tilde{\varphi}_n^\delta \cdot \nabla_\Gamma (\varphi_{\delta,n} - \tilde{\varphi}_n^\delta)\, \d S \\
&+ \frac{1}{\eps_\Gamma} \int_{\Gamma} [M_\Gamma (\tilde{\varphi}_n^\delta) - M_\Gamma(\varphi_{\delta,n})] G''_\delta(\tilde{\varphi}_n^\delta) \nabla_\Gamma \tilde{\varphi}_n^\delta \cdot \nabla_\Gamma (\varphi_{\delta,n} - \tilde{\varphi}_n^\delta)\, \d S
                      \\
&\leq \left[\eps_\Gamma  C(n) + \tilde{c}_G \eps_\Gamma^{-1} \right] M_\Gamma^\ast \|\varphi_{\delta,n} - \tilde{\varphi}_n^\delta\|_{V_\Gamma}^2 + \left[\eps_\Gamma^{-1} + \eps_\Gamma  C(n) + \eps_\Gamma^{-1}  C_\delta \right] M_\Gamma^\ast R (1 + \|\varphi_{\delta,n} - \tilde{\varphi}_n^\delta\|_{V_\Gamma}^2). 
\end{aligned}
\end{equation}
Let us observe that by definition of the tangential gradient (see, for e.g., \cite[Definition 2.3]{Dziuk+Elliott_2013}) and the Minkowski inequality, we have
    \begin{equation*}
      \|\nabla_\Gamma \phi_{\delta,n}\|_{\mathbb{L}^2(\Gamma)}
      = \|\nabla \phi_{\delta,n} - (\nabla \phi_{\delta,n} \cdot \bn)\bn\|_{\mathbb{L}^2(\Gamma)}
      \leq \|\nabla \phi_{\delta,n}\|_{\mathbb{L}^2(\Gamma)} + \|(\nabla \phi_{\delta,n} \cdot \bn)\bn\|_{\mathbb{L}^2(\Gamma)}.
    \end{equation*}
On the other hand, using the Cauchy-Schwarz inequality and the fact that $\bn$ is the unit outward normal vector to the boundary $\Gamma$, we get
\begin{align*}
\|(\nabla \phi_{\delta,n} \cdot \bn)\bn\|_{\mathbb{L}^2(\Gamma)}
&\leq \|\nabla \phi_{\delta,n} \cdot \bn\|_{L^2(\Gamma)} \\
&\leq  |\Gamma|^{1/2} \|\nabla \phi_{\delta,n}\|_{\mathbb{L}^2(\Gamma)},
\end{align*}
from which we infer that
  \begin{equation*}
      \|\nabla_\Gamma \phi_{\delta,n}\|_{\mathbb{L}^2(\Gamma)}
      \leq (1 + |\Gamma|^{1/2}) \|\nabla \phi_{\delta,n}\|_{\mathbb{L}^2(\Gamma)}.
    \end{equation*}
\adda{By the trace theorem (see, for instance, \cite[Theorem 5.36]{Adams_1975}), we have }
   \begin{equation*}
     \|\phi_{\delta,n}\|_{L^2(\Gamma)}\leq C(\mathcal{O},\Gamma) \|\phi_{\delta,n}\|_{H^1(\mathcal{O})}.
    \end{equation*}
This implies 
   \begin{equation}
     \|\nabla \phi_{\delta,n}\|_{\mathbb{L}^2(\Gamma)}
     \leq C(\mathcal{O},\Gamma) \|\phi_{\delta,n}\|_{H^2(\mathcal{O})},
    \end{equation}
and then 
\begin{equation}\label{eq4.71a}
  \|\nabla_\Gamma \phi_{\delta,n}\|_{\mathbb{L}^2(\Gamma)}
   \leq C(\mathcal{O},\Gamma) \|\phi_{\delta,n}\|_{H^2(\mathcal{O})}.
\end{equation}
From $(H4)$ and \eqref{eq4.71a}, we derive that
\begin{equation}\label{eq4.70}
\begin{aligned}
  &\frac{1}{K} \int_{\Gamma} M_\Gamma(\varphi_{\delta,n}) \nabla_\Gamma (\phi_{\delta,n} - \tilde{\phi}_n^\delta) \cdot \nabla_\Gamma (\varphi_n - \tilde{\varphi}_n^\delta)\, \d S \\
  &\leq M_\bar{\gamma}_0 K^{-1} \|\nabla_\Gamma (\phi_{\delta,n} - \tilde{\phi}_n^\delta)\|_{\mathbb{L}^2(\Gamma)} \|\nabla_\Gamma (\varphi_{\delta,n} - \tilde{\varphi}_n^\delta)\|_{\mathbb{L}^2(\Gamma)} \\
  &\leq M_\bar{\gamma}_0 K^{-1} C(\mathcal{O},\Gamma) \|\phi_{\delta,n} - \tilde{\phi}_n^\delta\|_{H^2(\mathcal{O})} \|\varphi_n - \tilde{\varphi}_n^\delta\|_{V_\Gamma} \\
  &\leq M_\bar{\gamma}_0 K^{-1} C(\mathcal{O},\Gamma,n) \|\phi_{\delta,n} - \tilde{\phi}_n^\delta\|_{V_1} \|\varphi_{\delta,n} - \tilde{\varphi}_n^\delta\|_{V_\Gamma} \\
  &\leq M_\bar{\gamma}_0 K^{-1} C(\mathcal{O},\Gamma,n) \|(\phi_{\delta,n} -\tilde{\phi}_n^\delta,\varphi_{\delta,n}-\tilde{\varphi}_n^\delta)\|_{\mathbb{V}}^2,
\end{aligned}
\end{equation}
where we have also used the fact that all norms are equivalent on finite dimensional space $H^1_n$.

\noindent
Once more, arguing as in \eqref{eq4.70}, we infer that
\begin{align*}
&\frac{1}{K} \int_{\Gamma} [M_\Gamma (\tilde{\varphi}_n^\delta) - M_\Gamma(\varphi_{\delta,n})] \nabla_\Gamma (\tilde{\varphi}_n^\delta - \tilde{\phi}_n^\delta) \cdot \nabla_\Gamma (\varphi_{\delta,n} - \tilde{\varphi}_n^\delta)\, \d S \\
&\leq 2 K^{-1} M_\bar{\gamma}_0 K^{-1} \|\nabla_\Gamma (\tilde{\varphi}_n^\delta - \tilde{\phi}_n^\delta)\|_{\mathbb{L}^2(\Gamma)} \|\nabla_\Gamma (\varphi_{\delta,n} - \tilde{\varphi}_n^\delta)\|_{\mathbb{L}^2(\Gamma)} \\
&\leq 2 K^{-1} M_\bar{\gamma}_0 K^{-1} (\|\nabla_\Gamma \tilde{\phi}_n^\delta\|_{\mathbb{L}^2(\Gamma)} + \|\nabla_\Gamma \tilde{\varphi}_n^\delta\|_{\mathbb{L}^2(\Gamma)}) \|\nabla_\Gamma (\varphi_{\delta,n} - \tilde{\varphi}_n^\delta)\|_{\mathbb{L}^2(\Gamma)} \\
&\leq 2 K^{-1} M_\bar{\gamma}_0 K^{-1} (C(\mathcal{O},\Gamma) \|\tilde{\phi}_n^\delta\|_{H^2(\mathcal{O})} + \|\tilde{\varphi}_n^\delta\|_{V_\Gamma}) \|\varphi_{\delta,n} - \tilde{\varphi}_n^\delta\|_{V_\Gamma} \\
&\leq 2 K^{-1} M_\bar{\gamma}_0 K^{-1} (C(\mathcal{O},\Gamma,n) \|\tilde{\phi}_n^\delta\|_{V_1} + \|\tilde{\varphi}_n^\delta\|_{V_\Gamma}) \|\varphi_{\delta,n} - \tilde{\varphi}_n^\delta\|_{V_\Gamma} \\
&\leq 2 R K^{-1} M_\bar{\gamma}_0 K^{-1} [C(\mathcal{O},\Gamma,n) + 1 ] \|\varphi_{\delta,n} - \tilde{\varphi}_n^\delta\|_{V_\Gamma},
\end{align*}
which, in turn, implies that
    \begin{equation}\label{eq4.71}
      \begin{split}
        &\frac{1}{K} \int_{\Gamma} [M_\Gamma (\tilde{\varphi}_n^\delta) - M_\Gamma(\varphi_{\delta,n})] \nabla_\Gamma (\tilde{\varphi}_n^\delta - \tilde{\phi}^\delta_n) \cdot \nabla_\Gamma (\varphi_{\delta,n} - \tilde{\varphi}_n^\delta)\, \d S \\
        &\leq R K^{-1} M_\bar{\gamma}_0 K^{-1} [C(\mathcal{O},\Gamma,n) + 1 ] (1 + \|\varphi_{\delta,n} - \tilde{\varphi}_n^\delta\|_{V_\Gamma}^2).
      \end{split}
    \end{equation}
Inserting \eqref{eq4.69}-\eqref{eq4.71} into the right-hand side of \eqref{eq4.68}, we obtain
\begin{equation}\label{eq4.72}
\begin{aligned}
&\langle - \mathcal{A}_{\varphi_{\delta,n}}(\theta_{\delta,n}) +  \mathcal{A}_{\tilde{\varphi}_n^\delta}(\tilde{\theta}_n^\delta),\varphi_{\delta,n} - \tilde{\varphi}_n^\delta \rangle_{V_\Gamma^\prime,V_\Gamma} \\
&\leq \left[\eps_\Gamma  C(n) + \tilde{c}_G \eps_\Gamma^{-1} \right] M_\Gamma^\ast \|\varphi_{\delta,n} - \tilde{\varphi}_n^\delta\|_{V_\Gamma}^2 + M_\bar{\gamma}_0 K^{-1} C(\mathcal{O},\Gamma,n) \|(\phi_{\delta,n} -\tilde{\phi}_n^\delta,\varphi_{\delta,n}-\tilde{\varphi}_n^\delta)\|_{\mathbb{V}}^2 \\
&\quad + \left[\eps_\Gamma^{-1} + \eps_\Gamma  C(n) + \eps_\Gamma^{-1}  C_\delta \right] M_\Gamma^\ast R (1 + \|\varphi_{\delta,n} - \tilde{\varphi}_n^\delta\|_{V_\Gamma}^2) \\
&\quad + R K^{-1} M_\bar{\gamma}_0 K^{-1} [C(\mathcal{O},\Gamma,n) + 1 ] (1 + \|\varphi_{\delta,n} - \tilde{\varphi}_n^\delta\|_{V_\Gamma}^2).
\end{aligned}
\end{equation}
Now, by plugging the estimates \eqref{eq4.58}, \eqref{eq4.59}, \eqref{eq4.66}, and \eqref{eq4.72} on the right-hand side of \eqref{eq4.55}, we deduce that there exists a positive constant $C>0$ depending on  $\delta,\,\mathcal{O},\,\Gamma,\,R,\,n,\,K,\,M_\bar{\gamma}_0,\,\bar{M}_0,\,\varepsilon,\,\varepsilon_\Gamma$ and $\tilde{c}_G$ such that
\begin{equation}\label{eq4.73}
\begin{aligned}
\langle \boldsymbol{b} (\cdot,\boldsymbol{\Upsilon}_n) - \boldsymbol{b}(\cdot,\Tilde{\Upsilon}_n),\boldsymbol{\Upsilon}_n - \tilde{\Upsilon}_n\rangle_{\mathbb{V}^\prime,\mathbb{V}}
\leq C (1 + \|\bu_{\delta,n}\|_{V_\sigma}) (1 + \|(\phi_{\delta,n} -\tilde{\phi}_n^\delta,\varphi_{\delta,n} -\tilde{\varphi}_n^\delta)\|_{\mathbb{V}}^2).
\end{aligned}
\end{equation}
From \eqref{eq4.73} together with the Lipschitz-continuity of the maps $F_{1}$ and $F_2$, we deduce that $\boldsymbol{b}(\cdot,\Upsilon_n)$ and $\boldsymbol{\sigma}(\cdot,\boldsymbol{\Upsilon}_n)$ satisfy the weak monotonicity property in the sense of Lemma \ref{stochastic-ordinary-theorem}$\text{-}(ii)$.

\noindent
Thanks to the assumptions $(H2)$ and $(H3)$, we see that
\begin{equation}\label{eq4.74}
\begin{aligned}
\|\boldsymbol{\sigma}(\cdot,\boldsymbol{\Upsilon}_n)\|_{\mathscr{T}_2(\mathcal{U},\mathbb{H})}^2
&= \|F_{1}(\phi_{\delta,n})\|_{\mathscr{T}_2(U,L^2(\mathcal{O}))}^2 + \|F_{\Gamma}(\varphi_{\delta,n})\|_{\mathscr{T}_2(U_\Gamma,L^2(\Gamma))}^2 \\
&= \sum_{k=1}^\infty \|F_1(\phi_{\delta,n})e_{1,k}\|_{L^2(\mathcal{O})}^2 + \sum_{k=1}^\infty \|F_2(\varphi_{\delta,n})e_{2,k}\|_{L^2(\Gamma)}^2\\
&= \sum_{k=1}^\infty \|\sigma_k(\phi_{\delta,n})\|_{L^2(\mathcal{O})}^2 + \sum_{k=1}^\infty \|\tilde{\sigma}_k(\varphi_{\delta,n})\|_{L^2(\Gamma)}^2\\
&\leq |\mathcal{O}| \sum_{k=1}^\infty \|\sigma_k\|_{L^{\infty}(\mathbb{R})} + |\Gamma| \sum_{k=1}^\infty \|\tilde{\sigma}_k\|_{L^{\infty}(\mathbb{R})} \\
&\leq |\mathcal{O}| C_1 + |\Gamma| C_2.
\end{aligned}
\end{equation}
Notice that
\begin{equation*}
\begin{aligned}
\langle \boldsymbol{b} (\cdot,\boldsymbol{\Upsilon}_n),\Upsilon_n\rangle_{\mathbb{V}^\prime,\mathbb{V}}
 =& \langle - A_{\phi_{\delta,n}}(\mu_{\delta,n}),\phi_{\delta,n} \rangle_{V_1^\prime,V_1} + \langle - \mathcal{A}_{\varphi_{\delta,n}}(\theta_{\delta,n}),\varphi_{\delta,n} \rangle_{V_\Gamma^\prime,V_\Gamma} \\
 &+ \langle - \tilde{B}(\bu_{\delta,n},\varphi_{\delta,n}),\varphi_{\delta,n} \rangle_{V_\Gamma^\prime,V_\Gamma}, 
\end{aligned}
\end{equation*}
where we used the fact that $\langle - B_1(\bu_{\delta,n},\phi_{\delta,n}),\phi_{\delta,n} \rangle_{V_1^\prime,V_1}=-b(\bu_{\delta,n},\phi_{\delta,n},\phi_{\delta,n})= 0$ due to \eqref{B1-first-Property}.

\noindent
From \eqref{Definition of A_phi and A_varphi} and \eqref{difference_chemical_potential_n}, we have
\begin{align*}
\langle - A_{\phi_n}(\mu_{\delta,n}),\phi_{\delta,n} \rangle_{V_1^\prime,V_1}
 &= - \int_{\mathcal{O}} M_{\mathcal{O}}(\phi_{\delta,n}) \nabla \mu_{\delta,n} \cdot \nabla \phi_{\delta,n} \, \d x \\
 &= \varepsilon \int_{\mathcal{O}} M_{\mathcal{O}}(\phi_{\delta,n}) \nabla \Delta \phi_{\delta,n} \cdot \nabla \phi_{\delta,n}\, \d x - \frac{1}{\varepsilon} \int_{\mathcal{O}} M_{\mathcal{O}}(\phi_{\delta,n}) F''_\delta(\phi_{\delta,n}) |\nabla \phi_{\delta,n}|^2\, \d x.
\end{align*}
Thanks to H\"older's inequality, using $(H4)$ together with the fact that all norms are equivalent on finite dimensional space $H^1_n$, we obtain
\begin{align*}
\eps \int_{\mathcal{O}} M_{\mathcal{O}}(\phi_{\delta,n}) \nabla \Delta \phi_{\delta,n} \cdot \nabla \phi_{\delta,n}\, \d x
\leq \eps \bar{M}_0 \lvert \nabla \Delta \phi_{\delta,n} \rvert \lvert  \nabla \phi_{\delta,n} \rvert 
 &\leq \eps \bar{M}_0 \|\phi_{\delta,n}\|_{H^3(\mathcal{O})} \|\phi_{\delta,n}\|_{V_1} \\
 &\leq \eps \bar{M}_0 C(n) \|\phi_{\delta,n}\|_{V_1}^2.
\end{align*}
Using the H\"older inequality and $(H4)$, we deduce that
  \begin{align*}
   &- \frac{1}{\varepsilon} \int_{\mathcal{O}} M_{\mathcal{O}}(\phi_{\delta,n}) F''_\delta(\phi_{\delta,n}) |\nabla \phi_{\delta,n}|^2\, \d x \\
   &= \underbrace{- \frac{1}{\varepsilon} \int_{\mathcal{O}} M_{\mathcal{O}}(\phi_{\delta,n}) (F''_\delta(\phi_{\delta,n}) + \tilde{c}_F) |\nabla \phi_{\delta,n}|^2\, \d x}_{\leq 0} +  \varepsilon^{-1} \tilde{c}_F \int_{\mathcal{O}} M_{\mathcal{O}}(\phi_{\delta,n}) |\nabla \phi_{\delta,n}|^2\, \d x \\
   &\leq \tilde{c}_F \varepsilon^{-1} \bar{M}_0 |\nabla \phi_{\delta,n}|^2 \\
   &\leq \tilde{c}_F \varepsilon^{-1} \bar{M}_0 \|\phi_{\delta,n}\|_{V_1}^2,
  \end{align*}
from which we get
    \begin{equation}\label{eq4.75}
      \begin{aligned}
       \langle - A_{\phi_{\delta,n}}(\mu_{\delta,n}),\phi_{\delta,n} \rangle_{V_1^\prime,V_1}
       &\leq (\varepsilon \bar{M}_0 C(n)  + \tilde{c}_F \varepsilon^{-1} \bar{M}_0) \|\phi_{\delta,n}\|_{V_1}^2 \\
      &\leq (\varepsilon \bar{M}_0 C(n)  + \tilde{c}_F \varepsilon^{-1} \bar{M}_0) \|(\phi_{\delta,n},\varphi_{\delta,n})\|_{\mathbb{V}}^2.
     \end{aligned}
   \end{equation}
In light of \eqref{Definition of A_phi and A_varphi} and \eqref{eq4.57}, we have
\begin{align*}
&\langle -\mathcal{A}_{\varphi_{\delta,n}}(\theta_{\delta,n}),\varphi_{\delta,n} \rangle_{V_\Gamma^\prime,V_\Gamma}
= -\int_{\Gamma} M_\Gamma(\varphi_{\delta,n}) \nabla_\Gamma \theta_{\delta,n} \cdot \nabla_\Gamma \varphi_{\delta,n}\, \d S \\
&=\eps_\Gamma \int_{\Gamma} M_\Gamma(\varphi_{\delta,n}) \nabla_\Gamma \Delta_\Gamma \varphi_{\delta,n} \cdot \nabla_\Gamma \varphi_{\delta,n}\, \d S - \eps_\Gamma^{-1} \int_{\Gamma} M_\Gamma(\varphi_{\delta,n}) G''_\delta(\varphi_{\delta,n}) |\nabla_\Gamma \varphi_{\delta,n}|^2 \, \d S \\
&\quad - K^{-1} \int_{\Gamma} M_\Gamma(\varphi_{\delta,n}) \nabla_\Gamma (\varphi_{\delta,n} - \phi_{\delta,n}) \cdot \nabla_\Gamma \varphi_{\delta,n}\, \d S \\
&=\eps_\Gamma \int_{\Gamma} M_\Gamma(\varphi_{\delta,n}) \nabla_\Gamma \Delta_\Gamma \varphi_{\delta,n} \cdot \nabla_\Gamma \varphi_{\delta,n}\, \d S
\underbrace{- \eps_\Gamma^{-1} \int_{\Gamma} M_\Gamma(\varphi_{\delta,n}) (G''_\delta(\varphi_{\delta,n}) + \tilde{c}_G) |\nabla_\Gamma \varphi_{\delta,n}|^2\, \d S}_{\leq 0} \\
&\quad + \eps_\Gamma^{-1} \tilde{c}_G \int_{\Gamma} M_\Gamma(\varphi_{\delta,n}) |\nabla_\Gamma \varphi_{\delta,n}|^2\, \d S 
\underbrace{- K^{-1} \int_{\Gamma} M_\Gamma(\varphi_{\delta,n}) |\nabla_\Gamma \varphi_{\delta,n}|^2\, \d S}_{\leq 0} 
+ K^{-1} \int_{\Gamma} M_\Gamma(\varphi_{\delta,n}) \nabla_\Gamma \phi_{\delta,n} \cdot \nabla_\Gamma \varphi_{\delta,n}\, \d S,
\end{align*}
from which we deduce that
\begin{equation}\label{eq4.76}
\begin{aligned}
&\langle -\mathcal{A}_{\varphi_n}(\theta_{\delta,n}),\varphi_{\delta,n} \rangle_{V_\Gamma^\prime,V_\Gamma} 
 \\
&\leq \varepsilon_\Gamma M_\bar{\gamma}_0 \|\nabla_\Gamma \Delta_\Gamma \varphi_{\delta,n}\|_{\mathbb{L}^2(\Gamma)} \|\nabla_\Gamma \varphi_{\delta,n}\|_{\mathbb{L}^2(\Gamma)} + \varepsilon_\Gamma^{-1} \tilde{c}_G  M_\bar{\gamma}_0 \|\nabla_\Gamma \varphi_{\delta,n}\|_{\mathbb{L}^2(\Gamma)}^2 
      \\
 &\quad + K^{-1} M_\bar{\gamma}_0 \|\nabla_\Gamma \phi_{\delta,n}\|_{\mathbb{L}^2(\Gamma)} \|\nabla_\Gamma \varphi_{\delta,n}\|_{\mathbb{L}^2(\Gamma)} 
          \\
 &\leq \varepsilon_\Gamma M_\bar{\gamma}_0 \|\varphi_{\delta,n}\|_{H^3(\Gamma)} \|\varphi_{\delta,n}\|_{V_\Gamma} + \varepsilon_\Gamma^{-1} \tilde{c}_G  M_\bar{\gamma}_0 \|\varphi_{\delta,n}\|_{V_\Gamma}^2 + K^{-1} M_\bar{\gamma}_0 C(\mathcal{O},\Gamma,n)\|\phi_{\delta,n}\|_{V_1} \|\varphi_{\delta,n}\|_{V_\Gamma} 
               \\
  &\leq \varepsilon_\Gamma M_\bar{\gamma}_0 C(n) \|\varphi_{\delta,n}\|_{V_\Gamma}^2 + \varepsilon_\Gamma^{-1} \tilde{c}_G  M_\bar{\gamma}_0 \|\varphi_{\delta,n}\|_{V_\Gamma}^2 + K^{-1} M_\bar{\gamma}_0 C(\mathcal{O},\Gamma,n) \|\phi_{\delta,n}\|_{V_1} \|\varphi_{\delta,n}\|_{V_\Gamma} 
                          \\
  &\leq [\varepsilon_\Gamma M_\bar{\gamma}_0 C(n) + \varepsilon_\Gamma^{-1} \tilde{c}_G  M_\bar{\gamma}_0 + K^{-1} M_\bar{\gamma}_0 C(\mathcal{O},\Gamma,n)] \|(\phi_{\delta,n},\varphi_{\delta,n})\|_{\mathbb{V}}^2,
\end{aligned}
\end{equation}
where we used \eqref{eq4.71a} together with the fact that all norms are equivalent on finite dimensional spaces $H^1_n$ and $H_{\Gamma,n}^1$, respectively.

\noindent
Note that
\begin{equation}\label{eq4.77}
\langle - \tilde{B}(\bu_{\delta,n},\varphi_{\delta,n}),\varphi_{\delta,n} \rangle_{V_\Gamma^\prime,V_\Gamma}
\leq \|\tilde{B}(\bu_{\delta,n},\varphi_{\delta,n})\|_{V_\Gamma^\prime} \|\varphi_{\delta,n}\|_{V_\Gamma}
\leq C(\Gamma,\mathcal{O}) \|\bu_{\delta,n}\|_{V} \|\varphi_{\delta,n}\|_{V_\Gamma}^2.
\end{equation}
By \eqref{eq4.74} and \eqref{eq4.75}-\eqref{eq4.77}, we deduce that there exists a positive constant $C=C(\tilde{c}_F,c_1,n,\varepsilon,\varepsilon_\Gamma,K,\tilde{c}_G,\tilde{c}_F,  M_\bar{\gamma}_0,M_\bar{\gamma}_0)$ such that
     \begin{equation*}
       \langle \boldsymbol{b} (\cdot,\boldsymbol{\Upsilon}_n),\Upsilon_n\rangle_{\mathbb{V}^\prime,\mathbb{V}} + \|\boldsymbol{\sigma}(\cdot,\boldsymbol{\Upsilon}_n)\|_{\mathscr{T}_2(\mathcal{U},\tilde{\mathbb{H}})}^2
      \leq C (1 + \|\bu_{\delta,n}\|_{V}) (1 + \|(\phi_{\delta,n},\varphi_{\delta,n})\|_{\mathbb{V}}^2).
     \end{equation*}
On the other hand, since the term $\int_0^T(1 + \|\bu_{\delta,n}(s)\|_{V})\, \d s$ is finite, $\mathbb{P}\text{-a.s.,}$ we derive the weak coercivity property in the sense of Lemma \ref{stochastic-ordinary-theorem}$\text{-}(iii)$. \newline
Hereafter, let $\by=(\phi,\varphi) \in \mathbb{V}$ be fixed and such that $\|\by\|_{\mathbb{V}}\leq 1$.

\noindent
From \eqref{B1-Property}, we have
     \begin{equation*}
       \lvert \langle - B_1(\bu_{\delta,n},\phi_{\delta,n}),\phi\rangle_{V_1^\prime,V_1}\rvert
       \leq \|B_1(\bu_{\delta,n},\phi_{\delta,n})\|_{V_1^\prime}
      \leq C(\mathcal{O}) \|\bu_{\delta,n}\|_{V} \|\phi_{\delta,n}\|_{V_1}.
     \end{equation*}
Using \eqref{Property-tilde B}, we get
     \begin{equation*}
      \lvert\langle - \tilde{B}(\bu_{\delta,n},\varphi_{\delta,n}),\varphi_{\delta,n} \rangle_{V_\Gamma^\prime,V_\Gamma}\rvert
      \leq \|\tilde{B}(\bu_{\delta,n},\varphi_{\delta,n})\|_{V_\Gamma^\prime}
      \leq C(\mathcal{O},\Gamma) \|\bu_{\delta,n}\|_{V} \|\varphi_{\delta,n}\|_{V_\Gamma}.
     \end{equation*}
Thanks to the assumption $(H4)$, we see that
\begin{align*}
|\langle - A_{\phi_{\delta,n}}(\mu_{\delta,n}),\phi\rangle_{V_1^\prime,V_1}| 
&\leq \varepsilon \int_{\mathcal{O}} M_{\mathcal{O}}(\phi_{\delta,n}) |\nabla \Delta \phi_{\delta,n}| |\nabla \phi|\,\d x + \varepsilon^{-1} \int_{\mathcal{O}} M_{\mathcal{O}}(\phi_{\delta,n}) |F''_\delta(\phi_{\delta,n})| |\nabla \phi_{\delta,n}| |\nabla \phi|\, \d x \\
&\leq \varepsilon \bar{M}_0 |\nabla \Delta \phi_{\delta,n}| |\nabla \phi| + \varepsilon^{-1} \bar{M}_0 C_\delta |\nabla \phi_{\delta,n}| |\nabla \phi| \\
&\leq \varepsilon \bar{M}_0 C(n) \|\phi_{\delta,n}\|_{V_1} + \varepsilon^{-1} \bar{M}_0 C_\delta \|\phi_{\delta,n}\|_{V_1}. 
\end{align*}
Next, observe that
\begin{align*}
\langle -\mathcal{A}_{\varphi_{\delta,n}}(\theta_{\delta,n}),\varphi\rangle_{V_\Gamma^\prime,V_\Gamma} 
&=\varepsilon_\Gamma \int_{\Gamma} M_\Gamma(\varphi_{\delta,n}) \nabla_\Gamma \Delta_\Gamma \varphi_{\delta,n} \cdot \nabla_\Gamma \varphi\, \d S - \varepsilon_\Gamma^{-1} \int_{\Gamma} M_\Gamma(\varphi_{\delta,n}) G''(\varphi_{\delta,n}) \nabla_\Gamma \varphi_{\delta,n} \cdot \nabla_\Gamma \varphi \, \d S \\
&\quad - K^{-1} \int_{\Gamma} M_\Gamma(\varphi_{\delta,n}) \nabla_\Gamma (\varphi_{\delta,n} - \phi_{\delta,n}) \cdot \nabla_\Gamma \varphi\, \d S.
\end{align*}
Now, from $(H4)$ and since all norms are equivalent on finite dimensional space $H_{\Gamma,n}^1$, we infer that
\begin{align*}
\left|\varepsilon_\Gamma \int_{\Gamma} M_\Gamma(\varphi_{\delta,n}) \nabla_\Gamma \Delta_\Gamma \varphi_{\delta,n} \cdot \nabla_\Gamma \varphi\, \d S\right| 
&\leq \varepsilon_\Gamma \int_{\Gamma} M_\Gamma(\varphi_{\delta,n}) |\nabla_\Gamma \Delta_\Gamma \varphi_{\delta,n}| |\nabla_\Gamma \varphi|\, \d S\\
&\leq \varepsilon_\Gamma M_\bar{\gamma}_0 \|\nabla_\Gamma \Delta_\Gamma \varphi_{\delta,n}\|_{\mathbb{L}^2(\Gamma)} \|\nabla_\Gamma \varphi\|_{\mathbb{L}^2(\Gamma)} \\
&\leq \varepsilon_\Gamma M_\bar{\gamma}_0 C(n) \|\varphi_{\delta,n}\|_{V_\Gamma}.
\end{align*}
Using $(H4)$ and the fact that $\lvert G''_\delta \rvert \leq C_\delta$ with $C_\delta>0$ independent of $n$, we obtain

\begin{align*}
\left|\eps_\Gamma^{-1} \int_{\Gamma} M_\Gamma(\varphi_{\delta,n}) G''_\delta(\varphi_{\delta,n}) \nabla_\Gamma \varphi_{\delta,n} \cdot \nabla_\Gamma \varphi \, \d S\right|
&\leq \eps_\Gamma^{-1} M_\Gamma^\ast C_\delta \|\nabla_\Gamma \varphi_{\delta,n}\|_{\mathbb{L}^2(\Gamma)} \|\nabla_\Gamma \varphi\|_{\mathbb{L}^2(\Gamma)} \\
&\leq \eps_\Gamma^{-1} M_\Gamma^\ast C_\delta \| \varphi_{\delta,n}\|_{V_\Gamma} \|\varphi\|_{V_\Gamma} \\
&\leq \eps_\Gamma^{-1} M_\Gamma^\ast C_\delta \| \varphi_{\delta,n}\|_{V_\Gamma}.
\end{align*}
We have
\begin{align*}
K^{-1} \int_{\Gamma} M_\Gamma(\varphi_{\delta,n}) |\nabla_\Gamma (\varphi_{\delta,n} - \phi_{\delta,n})| |\nabla_\Gamma \varphi|\, \d S 
&\leq K^{-1} M_\bar{\gamma}_0 (\|\nabla_\Gamma \varphi_{\delta,n}\|_{\mathbb{L}^2(\Gamma)} + \|\nabla_\Gamma \phi_{\delta,n}\|_{\mathbb{L}^2(\Gamma)}) \\
&\leq K^{-1} M_\bar{\gamma}_0 (\|\varphi_{\delta,n}\|_{V_\Gamma} + C(\mathcal{O},\Gamma) \|\phi_{\delta,n}\|_{H^2(\mathcal{O})}) 
\\
&\leq K^{-1} M_\bar{\gamma}_0 (\|\varphi_{\delta,n}\|_{V_\Gamma} + C(\mathcal{O},\Gamma,n) \|\phi_{\delta,n}\|_{V_1}),
\end{align*}
where we used the trace theorem in conjunction with the fact that all norms are equivalent on finite dimensional space $H_{n}^1$.

\noindent
Consequently,
\begin{align*}
\lvert\langle -\mathcal{A}_{\varphi_{\delta,n}}(\theta_{\delta,n}),\varphi\rangle_{V_\Gamma^\prime,V_\Gamma}\rvert
&\leq \eps_\Gamma M_\bar{\gamma}_0 C(n) \|\varphi_{\delta,n}\|_{V_\Gamma} + \eps_\Gamma^{-1} M_\Gamma^\ast C_\delta \| \varphi_{\delta,n}\|_{V_\Gamma} \\
&\quad + K^{-1} M_\bar{\gamma}_0 (\|\varphi_{\delta,n}\|_{V_\Gamma} + C(\mathcal{O},\Gamma,n) \|\phi_{\delta,n}\|_{V_1}).
\end{align*}
Collecting now the previous estimates, we deduce that
\begin{align}\label{eq4.78}
\begin{aligned}
\|\boldsymbol{b}(\cdot,\boldsymbol{\Upsilon}_n)\|_{\mathbb{V}^\prime}
&=\sup_{\by\in \mathbb{V},\,\|\by\|_{\mathbb{V}\leq 1}} \lvert \langle \boldsymbol{b} (\cdot,\boldsymbol{\Upsilon}_n),\by\rangle_{\mathbb{V}^\prime,\mathbb{V}} \rvert \\
&\leq C[\|\bu_{\delta,n}\|_{V} (\|\phi_{\delta,n}\|_{V_1} +  \|\varphi_{\delta,n}\|_{V_\Gamma}) +  \|\phi_{\delta,n}\|_{V_1}  +  \|\varphi_{\delta,n}\|_{V_\Gamma} +  \|\varphi_{\delta,n}\|_{V_\Gamma}^{q-1}],
\end{aligned}
\end{align}
where $C=C(\delta,\mathcal{O},\Gamma,n,\varepsilon, \bar{M}_0, \bar{M}_0,\varepsilon_\Gamma, M_\bar{\gamma}_0,K)$ is a positive large constant. It then follows from \eqref{eq4.78} and \eqref{eq4.74} that
\begin{align}
\begin{aligned}
&\|\boldsymbol{b}(\cdot,\boldsymbol{\Upsilon}_n)\|_{\mathbb{V}^\prime} + \|\boldsymbol{\sigma}(\cdot,\boldsymbol{\Upsilon}_n)\|_{\mathscr{T}_2(\mathcal{U},\mathbb{H})}^2 \\
&\leq  C[\|\bu_{\delta,n}\|_{V} (\|\phi_{\delta,n}\|_{V_1} +  \|\varphi_{\delta,n}\|_{V_\Gamma}) 1 +  \|\phi_{\delta,n}\|_{V_1}  +  \|\varphi_{\delta,n}\|_{V_\Gamma} +  \|\varphi_{\delta,n}\|_{V_\Gamma}^{2} + \|\phi_{\delta,n}\|_{V_1}^2]
\end{aligned}
\end{align}
for some $C=C(C_1,\mathcal{O},\Gamma,n,\varepsilon, \bar{M}_0, \delta, \bar{M}_0,\varepsilon_\Gamma, M_\bar{\gamma}_0,K,c_1)>0$. Hence,
\begin{align*}
&\int_0^T \sup_{\|\boldsymbol{\Upsilon}_n\|_{\mathbb{V}}\leq R} \left(\|\boldsymbol{b}(t,\boldsymbol{\Upsilon}_n)\|_{\mathbb{V}^\prime} + \|\boldsymbol{\sigma}(t,\boldsymbol{\Upsilon}_n)\|_{\mathscr{T}_2(\mathcal{U},\tilde{\mathbb{H}})}^2\right)\d t \\
&\leq C \int_0^T [ 1 + 2 R\|\bu_n(t)\|_{V} +  2 R  + 2R^2]\, \d t<\infty.
\end{align*}
This proves the item $(i)$ of Lemma \ref{stochastic-ordinary-theorem}.
\noindent
Finally, by Lemma \ref{stochastic-ordinary-theorem}, we obtain the existence of a unique strong solution $\boldsymbol{a}_n \in L^r(\Omega;\, C([0,T];\mathbb{R}^n))$ and $\boldsymbol{b}_n \in L^r(\Omega;\, C([0,T];\mathbb{R}^n))$ to the stochastic problem \eqref{compact-stochastic-problem}. The prove of Theorem \ref{thm-Galerkin-01} is now complete.
\end{proof}
}
\subsection{Uniform estimates} \label{subsection-4.3}
Here, we will establish the uniform energy estimate for the solution of the finite-dimensional stochastic differential equation \eqref{compact-stochastic-problem}. For this, we define a sequence of stopping times $\{\tau_\kappa\}_{\kappa \in\mathbb{N}}$,
   \begin{equation}
      \tau_\kappa \coloneq \inf\{t>0: (\Vert \phi_{\delta,n}(t) \Vert_{V_1}^2 + \Vert \varphi_{\delta,n}(t) - \phi_{\delta,n}(t) \Vert_{L^2(\Gamma)}^2 + \Vert \nabla_\Gamma \varphi_{\delta,n}(t) \Vert_{\mathbb{L}^2(\Gamma)}^2)^\frac{1}{2} \geq \kappa\} \wedge T
  \end{equation}
and the following functional
     \begin{align*}
       \mathcal{E}_{tot}: \mathcal{V}_n \ni (\phi,\varphi) \mapsto \left[\mathcal{E}(\phi,\varphi) + \frac12  \lvert \phi \rvert^2 \right] \in \mathbb{R}. 
      \end{align*}
\begin{lemma}\label{Lem-2}
Let $\delta\in(0,1)$ be fixed. Assume the assumptions of Theorem \ref{thm-first_main_theorem} hold. Then, the sequence $(\bu_{\delta,n},\phi_{\delta,n},\varphi_{\delta,n},\mu_{\delta,n},\theta_{\delta,n})_{n \in \mathbb{N}}$ satisfies for all $t\in[0,T]$ and $\mathbb{P}$-a.s.,
\begin{equation}\label{Main_Galerkin_Equality}
\begin{aligned}
&\mathcal{E}_{tot}(\phi_{\delta,n}(t\wedge \tau_\kappa), \varphi_{\delta,n}(t\wedge \tau_\kappa)) + \int_0^{t\wedge \tau_\kappa} \int_{\mathcal{O}} [2 \nu(\phi_{\delta,n}) \lvert D\bu_{\delta,n} \rvert^2 + \lambda(\phi_{\delta,n}) \lvert \bu_{\delta,n} \rvert^2] \d x \,\d s \\
& + \int_0^{t\wedge \tau_\kappa} \int_{\Gamma} \gamma(\varphi_{\delta,n}) \lvert \bu_{\delta,n} \rvert^2 \d S\,\d s + \int_0^{t\wedge \tau_\kappa} \int_{\mathcal{O}} M_{\mathcal{O}}(\phi_{\delta,n}) \lvert \nabla \mu_{\delta,n} \rvert^2 \d x \,\d s \\ 
& + \int_0^{t\wedge \tau_\kappa} \int_{\Gamma} M_\Gamma(\varphi_{\delta,n}) \lvert \nabla_\Gamma \theta_{\delta,n} \rvert^2 \,\d S \,\d s  
 \\
&= \mathcal{E}_{tot}(\phi_n(0),\varphi_n(0))  + \frac{\eps}{2 K} \int_0^{t\wedge \tau_\kappa} (\Vert F_{2,n}(\varphi_{\delta,n}) \Vert_{\mathscr{T}_2(U_\Gamma,L^2(\Gamma))}^2 + \Vert F_{1,n}(\phi_{\delta,n}) \Vert_{\mathscr{T}_2(U,L^2(\Gamma))}^2) \, \d s \\
& + \frac{\eps}{2} \int_0^{t\wedge \tau_\kappa}  \Vert \nabla F_{1,n}(\phi_{\delta,n}) \Vert_{L^2(U,\mathbb{L}^2(\mathcal{O}))}^2  \, \d s + \frac{\eps_\Gamma}{2} \int_0^{t\wedge \tau_\kappa} \Vert \nabla_\Gamma F_{2,n}(\varphi_{\delta,n}) \Vert_{\mathscr{T}_2(U_\Gamma,\mathbb{L}^2(\Gamma))}^2 \, \d s \\
& + \int_0^{t\wedge \tau_\kappa} (\mu_{\delta,n},F_1(\phi_{\delta,n})\, \d W) + \int_0^{t\wedge \tau_\kappa} (\theta_{\delta,n},(F_2(\varphi_{\delta,n}))\, \d W_\Gamma)_\Gamma  \\
& + \sum_{k=1}^\infty \int_0^{t\wedge \tau_\kappa} \frac{1}{2 \eps} \int_{\mathcal{O}} F^{\bis}_\delta(\phi_{\delta,n}) \lvert F_{1,n}(\phi_{\delta,n}) e_{1,k} \rvert^2 \, \d x \,\d s +  \frac{1}{2} \int_0^{t\wedge \tau_\kappa} \Vert F_{1,n}(\phi_{\delta,n})\|_{\mathscr{T}_2(U,L^2(\mathcal{O}))}^2\,\d s \\
& + \sum_{k=1}^\infty \int_0^{t\wedge \tau_\kappa} \frac{1}{2 \eps_\Gamma} \int_{\Gamma} G^{\bis}_\delta(\varphi_{\delta,n}) \lvert F_{2,n}(\varphi_{\delta,n})e_{2,k} \rvert^2\, \d S\,\d s \\
& - \eps (1/K) \sum_{k=1}^\infty  \int_0^{t\wedge \tau_\kappa} \int_\Gamma (F_{1,n}(\phi_{\delta,n}) e_{1,k}) (F_{2,n}(\varphi_{\delta,n}) e_{2,k})\, \d S \, \d s  \\
& - \int_0^{t\wedge \tau_\kappa} \int_{\mathcal{O}} M_{\mathcal{O}}(\phi_{\delta,n}) \nabla \mu_{\delta,n} \cdot \nabla \phi_{\delta,n} \, \d x \,\d s  + \int_0^{t\wedge \tau_\kappa}(\phi_{\delta,n}, F_{1}(\phi_{\delta,n})\,\d W).
\end{aligned}
\end{equation}
\end{lemma}
\begin{proof}
By testing \eqref{eq4.34a} by $\bu_{\delta,n}$ and \eqref{eq4.34c} by $\theta_{\delta,n}$, respectively, we obtain
\begin{equation}\label{eq.4.78}
\begin{aligned}
& \int_0^{t\wedge \tau_\kappa} \int_{\mathcal{O}} [2 \nu(\phi_{\delta,n}) \lvert D\bu_{\delta,n} \rvert^2 + \lambda(\phi_{\delta,n}) \lvert \bu_{\delta,n} \rvert^2] \,\d x \,\d s + \int_0^{t\wedge \tau_\kappa} \int_{\Gamma} \gamma(\varphi_{\delta,n}) \lvert \bu_{\delta,n} \rvert^2 \,\d S\,\d s \\
&\quad= - \int_0^{t\wedge \tau_\kappa} \int_{\Gamma} \varphi_{\delta,n} \nabla_\Gamma \theta_{\delta,n} \cdot \bu_{\delta,n} \, \d S\,\d s  - \int_0^{t\wedge \tau_\kappa} \int_{\mathcal{O}} \phi_{\delta,n} \nabla \mu_{\delta,n} \cdot \bu_{\delta,n} \,\d x \,\d s, \; t \in [0,T]. 
\end{aligned}
\end{equation}
Next, we apply the It\^o formula to the free energy functional $\mathcal{E}$ introduced in \eqref{Eqn-Lyaponov-function}. \newline
The first Fr\'echet derivative $D\mathcal{E}: \mathcal{V}_n \to \mathcal{V}_n^\prime$ of $\mathcal{E}$ is given by
\dela{
Since $F^\prime_\delta$ and $G_\delta^\prime$ are Lipschitz-continuous and, since all norms are equivalent on finite dimensional space $\mathcal{V}_n$, the map $\mathcal{E}$ is Fr\'echet differential with the first Fr\'echet derivative $D\mathcal{E}: \mathbb{V}_n \to \mathbb{V}_n^\prime$ given by
}
\begin{align*}
D\mathcal{E}(\phi,\varphi)[(v,v_1)]
&= \eps \int_{\mathcal{O}} \nabla \phi \cdot \nabla v \,\d x + \frac{1}{\eps} \int_{\mathcal{O}} F_\delta^\prime(\phi) v\, \d x + \eps_\Gamma \int_\Gamma \nabla_\Gamma \varphi \cdot \nabla_\Gamma v_1 \, \d S \\
&\quad + \frac{1}{\eps_\Gamma} \int_\Gamma G_\delta^\prime(\varphi) v_1 \, \d S + \eps (1/K) \int_\Gamma (\varphi - \phi) (v_1 - v) \, \d S, \; \; (\phi,\varphi), \, (v,v_1) \in \mathcal{V}_n.
\end{align*}
As a direct consequence, we have $D \mathcal{E}(\phi_{\delta,n},\varphi_{\delta,n})= (\mu_{\delta,n},\theta_{\delta,n})$.

\noindent
Now, we claim that the second Fr\'echet derivative $D^2 \mathcal{E}: \mathcal{V}_n \to \mathcal{L}(\mathcal{V}_n, \mathcal{V}_n^\prime)$ of $\mathcal{E}$ is given by
\begin{align*}
 D^2 \mathcal{E}(\phi,\varphi)[(v,h),(v_1,h_1)]
 &=  \eps \int_{\mathcal{O}} \nabla v \cdot \nabla v_1 \,\d x  + \eps_\Gamma \int_\Gamma \nabla_\Gamma h \cdot \nabla_\Gamma h_1 \, \d S + \frac{1}{\eps} \int_{\mathcal{O}} F^{\bis}_\delta(\phi) v v_1 \, \d x
  \\
 &\quad  + \eps (1/K) \int_\Gamma (h - v) (h_1 - v_1)\, \d S  + \frac{1}{\eps_\Gamma} \int_{\Gamma} G^{\bis}_\delta(\varphi) h h_1 \, \d S, 
\end{align*}
for all $(v,h),\, (v_1,h_1),\, (\phi,\varphi) \in \mathcal{V}_n$. In fact, let us fix $(v,h),\, (v_1,h_1),\, (\phi,\varphi) \in \mathcal{V}_n$. We have,
\begin{align*}
 & D\mathcal{E}((\phi,\varphi) + (v,h))[(v_1,h_1)] - D\mathcal{E}(\phi,\varphi)[(v_1,h_1)] \\
 &=  \eps \int_{\mathcal{O}} \nabla v \cdot \nabla v_1 \,\d x + \frac{1}{\eps} \int_{\mathcal{O}} [F_\delta^\prime(\phi + v) - F_\delta^\prime(\phi)] v_1\, \d x + \eps_\Gamma \int_\Gamma \nabla_\Gamma h \cdot \nabla_\Gamma h_1 \, \d S \\
 &\quad + \frac{1}{\eps_\Gamma} \int_\Gamma [G_\delta^\prime(\varphi + h) -G_\delta^\prime(\varphi)] h_1 \, \d S + \eps (1/K) \int_\Gamma (h - v) (h_1 - v_1) \, \d S.
\end{align*}
Notice that
\begin{align*}
&\int_{\mathcal{O}} [F_\delta^\prime(\phi + v) - F_\delta^\prime(\phi)] v_1\, \d x
= \int_{\mathcal{O}} v v_1 \int_0^1 F^{\bis}_\delta(\phi + \tau v) \, \d \tau\, \d x \\
&= \int_{\mathcal{O}} F^{\bis}_\delta(\phi) v v_1 \, \d x + \int_{\mathcal{O}} v v_1 \int_0^1 [F^{\bis}_\delta(\phi + \tau v) - F^{\bis}_\delta(\phi)] \, \d \tau\, \d x
\end{align*}
and 
\begin{equation*}
\int_{\Gamma} [G_\delta^\prime(\varphi + h) - G_\delta^\prime(\varphi)] h_1\, \d S  - \int_{\Gamma} G^{\bis}_\delta(\varphi) h h_1 \, \d S 
= \int_{\Gamma} h h_1 \int_0^1 [G^{\bis}_\delta(\varphi + \tau h) - G^{\bis}_\delta(\varphi)] \, \d \tau\, \d S.
\end{equation*}
Therefore, by the H\"older inequality and the fact that all norms are equivalent on the finite dimensional space $\mathcal{V}_n$, we infer that
\begin{align*}
&\left|\int_{\mathcal{O}} v v_1 \int_0^1 [F^{\bis}_\delta(\phi + \tau v) - F^{\bis}_\delta(\phi)] \, \d \tau\, \d x\right|
\leq \Vert v \Vert_{L^\infty(\mathcal{O})} \Vert v_1 \Vert_{L^\infty(\mathcal{O})} \int_0^1 \Vert F^{\bis}_\delta(\phi + \tau v) - F^{\bis}_\delta(\phi) \Vert_{L^1(\mathcal{O})} \, \d \tau \\
&\leq C(n) \Vert (v,h) \Vert_{\mathcal{V}_n} \Vert (v_1,h_1) \Vert_{\mathcal{V}_n} \int_0^1 \Vert F^{\bis}_\delta(\phi + \tau v) - F^{\bis}_\delta(\phi) \Vert_{L^1(\mathcal{O})} \, \d \tau.
\end{align*}
Since $F^{\bis}_\delta$ is continuous and bounded, an application of the DCT entails that the third factor on the RHS of the above inequality converges to zero as $(v,h) \to (0,0)$. Thus,
\begin{align*}
\sup_{\Vert (v_1,h_1) \Vert_{\mathcal{V}_n} \leq 1} \left|\int_{\mathcal{O}} v v_1 \int_0^1 [F^{\bis}_\delta(\phi + \tau v) - F^{\bis}_\delta(\phi)] \, \d \tau\, \d x\right| = o(\Vert (v,h) \Vert_{\mathcal{V}_n}).
\end{align*}
Consequently, 
   \begin{equation*}
    \int_{\mathcal{O}} [F_\delta^\prime(\phi + v) - F_\delta^\prime(\phi)] v_1\, \d x
    = \int_{\mathcal{O}} F^{\bis}_\delta(\phi) v v_1 \, \d x + o(\Vert (v,h) \Vert_{\mathcal{V}_n}).
\end{equation*}
Similarly, we prove that
   \begin{equation*}
    \int_{\Gamma} [G_\delta^\prime(\varphi + h) - G_\delta^\prime(\varphi)] h_1\, \d S 
    = \int_{\Gamma} G^{\bis}_\delta(\varphi) h h_1 \, \d S + o(\Vert (v,h) \Vert_{\mathcal{V}_n}).
\end{equation*}
Therefore,
\begin{align*}
 & D\mathcal{E}((\phi,\varphi) + (v,h))[(v_1,h_1)] - D\mathcal{E}(\phi,\varphi)[(v_1,h_1)] \\
 &=  \eps \int_{\mathcal{O}} \nabla v \cdot \nabla v_1 \,\d x  + \eps_\Gamma \int_\Gamma \nabla_\Gamma h \cdot \nabla_\Gamma h_1 \, \d S + \eps (1/K) \int_\Gamma (h - v) (h_1 - v_1) \, \d S \\
 &\quad + \frac{1}{\eps} \int_{\mathcal{O}} F^{\bis}_\delta(\phi) v v_1 \, \d x + \frac{1}{\eps_\Gamma} \int_{\Gamma} G^{\bis}_\delta(\varphi) h h_1 \, \d S + o(\Vert (v,h) \Vert_{\mathcal{V}_n}).
\end{align*}
This means that $D\mathcal{E}$ is also Fr\'echet differentiable with the Fr\'echet derivative $D^2\mathcal{E}$ defined as above. Besides, the first and second Fr\'echet derivatives $D \mathcal{E}$ and $D^2 \mathcal{E}$ are both continuous and bounded on bounded subsets of $\mathcal{V}_n$ due to the fact that $F_\delta^\prime$ and $G_\delta^\prime$ are Lipschitz-continuous, $F^{\bis}_\delta$ and $G^{\bis}_\delta$ are continuous and bounded.
\newline
\noindent
We can now apply the It\^o formula introduced in \cite[Theorem 4.32]{Prato}. 
Therefore, for every $t\in[0,T]$ and $\mathbb{P}$-a.s., we obtain
\begin{equation}\label{Eqn-4.22}
\begin{aligned}
&\mathcal{E}(\phi_{\delta,n}(t\wedge \tau_\kappa), \varphi_{\delta,n}(t\wedge \tau_\kappa)) 
- \int_0^{t\wedge \tau_\kappa} \duality{\boldsymbol{b}_n((\phi_{\delta,n}, \varphi_{\delta,n}))}{(\mu_{\delta,n},\theta_{\delta,n})}{\mathbb{V}}{\mathbb{V}^\prime}\, \d s \\
&= \mathcal{E}(\phi_n(0),\varphi_n(0)) +  \int_0^{t\wedge \tau_\kappa} ((\mu_{\delta,n},\theta_{\delta,n}),\boldsymbol{\sigma}_n((\phi_{\delta,n},\varphi_{\delta,n}))\, \d \mathcal{W})_{\mathbb{H}}  \\
& + \frac12 \bigg[ \int_0^{t\wedge \tau_\kappa} \bigg( \eps \Vert \nabla F_{1,n}(\phi_{\delta,n}) \Vert_{L^2(U,\mathbb{L}^2(\mathcal{O}))}^2 + \eps_\Gamma \Vert \nabla_\Gamma F_{2,n}(\varphi_{\delta,n}) \Vert_{\mathscr{T}_2(U_\Gamma,\mathbb{L}^2(\Gamma))}^2 \\
& + \frac{1}{\eps} \sum_{k=1}^\infty \int_{\mathcal{O}} F^{\bis}_\delta(\phi_{\delta,n}) \lvert F_{1,n}(\phi_{\delta,n}) e_{1,k} \rvert^2 \, \d x + \frac{1}{\eps_\Gamma} \sum_{k=1}^\infty \int_{\mathcal{O}} G^{\bis}_\delta(\varphi_{\delta,n}) \lvert F_2(\varphi_{\delta,n})e_{2,k} \rvert^2\, \d S \\
& + \eps [K] \Vert F_{2,n}(\varphi_{\delta,n}) \Vert_{\mathscr{T}_2(U_\Gamma,L^2(\Gamma))}^2  + \eps [K] \Vert F_{1,n}(\phi_{\delta,n}) \Vert_{\mathscr{T}_2(U,L^2(\Gamma))}^2  \\
& - 2  \eps [K] \sum_{k=1}^\infty  \int_\Gamma (F_{1,n}(\phi_{\delta,n}) e_{1,k}) (F_{2,n}(\varphi_{\delta,n}) e_{2,k})\, \d S \bigg) \d s
\bigg].
\end{aligned}
\end{equation}
Next, by applying the It\^o formula to the functional $\Phi(x)= \lvert x \rvert^2$, cf. \cite[Theorem 4.2.5]{Liu+Rockner_2015}, 
since $(B_{1,n}(\bu_{\delta,n},\phi_{\delta,n}),\phi_{\delta,n})= 0$, we infer that
\begin{equation}\label{eq4.81}
\begin{aligned}
 &\frac12 \lvert \phi_{\delta,n}(t\wedge \tau_\kappa) \rvert^2 + \int_0^{t\wedge \tau_\kappa} \int_{\mathcal{O}} M_{\mathcal{O}}(\phi_{\delta,n}) \nabla \mu_{\delta,n} \cdot \nabla \phi_{\delta,n} \, \d x \,\d s\\
 &= \frac12 \lvert \phi_n(0) \rvert^2 +  \frac12 \int_0^{t\wedge \tau_\kappa} \Vert F_{1,n}(\phi_{\delta,n}) \Vert_{\mathscr{T}_2(U,L^2(\mathcal{O}))}^2 \,\d s + \int_0^{t\wedge \tau_\kappa}(\phi_{\delta,n}, F_{1}(\phi_{\delta,n})\,\d W).
\end{aligned}
\end{equation}
Note that
\begin{align*}
&\duality{\boldsymbol{b}_n((\phi_{\delta,n}, \varphi_{\delta,n}))}{(\mu_{\delta,n},\theta_{\delta,n})}{\mathbb{V}}{\mathbb{V}^\prime}  
= \duality{- B_1(\bu_{\delta,n},\phi_{\delta,n}) - A_{\phi_{\delta,n}}(\mu_{\delta,n})}{\mu_{\delta,n}}{V_1}{V_1^\prime}  \\
&\quad \hspace{7 truecm} +  \duality{- \mathcal{A}_{\varphi_{\delta,n}}(\theta_{\delta,n}) - \tilde{B}(\bu_{\delta,n},\varphi_{\delta,n})}{\theta_{\delta,n}}{V_\Gamma}{V_\Gamma^\prime} \\
&= \int_{\mathcal{O}}\left[\bu_{\delta,n} \cdot \nabla \mu_{\delta,n} \phi_{\delta,n} - M_{\mathcal{O}}(\phi_{\delta,n}) \lvert \nabla \mu_{\delta,n} \rvert^2 \right] \d x  
- \int_{\Gamma} \left[M_\Gamma(\varphi_{\delta,n}) \lvert \nabla_\Gamma \theta_{\delta,n} \rvert^2 -  \varphi_{\delta,n} \bu_{\delta,n} \cdot \nabla_\Gamma \theta_{\delta,n} \right] \d S,
\end{align*}
\begin{align*}
\int_0^{t\wedge \tau_\kappa} ((\mu_{\delta,n},\theta_{\delta,n}),\boldsymbol{\sigma}_n((\phi_{\delta,n},\varphi_{\delta,n}))\, \d \mathcal{W})_{\mathbb{H}}
= \int_0^{t\wedge \tau_\kappa} (\mu_{\delta,n},F_1(\phi_{\delta,n})\, \d W) + \int_0^{t\wedge \tau_\kappa} (\theta_{\delta,n},F_2(\varphi_{\delta,n})\, \d W_\Gamma)_\Gamma.
\end{align*}
Plugging the two previous equalities into the RHS of \eqref{Eqn-4.22}, adding up side by side the resulting equality to \eqref{eq4.81}, we deduce \eqref{Main_Galerkin_Equality}. This completes the proof of Lemma \ref{Lem-2}.
\end{proof}
We now state and prove the following result.
\begin{proposition}\label{Proposition_first_priori_estimmate}
Suppose \ref{item:H1}-\ref{item:H8} are satisfied. Then, there exists a constant $C_{1,\delta}>0$ such that for every $n \in \mathbb{N}$, 
\begin{equation}\label{eq4.80}
\begin{aligned}
 &\mathbb{E} \sup_{s \in[0,T]} [\mathcal{E}_{tot}(\phi_{\delta,n}(s), \varphi_{\delta,n}(s))] + \mathbb{E} \int_0^{T} \int_{\mathcal{O}} [\lvert D \bu_{\delta,n}(s) \rvert^2 + \lambda(\phi_{\delta,n}(s)) \lvert \bu_{\delta,n}(s) \rvert^2] \,\d x \,\d s \\
 & + \mathbb{E} \int_0^{T} \lvert \bu_{\delta,n}(s) \rvert_\Gamma^2 \,\d s + \mathbb{E} \int_0^{T} [\lvert \nabla \mu_{\delta,n}(s) \rvert^2 + \lvert \nabla_\Gamma \theta_{\delta,n}(s) \rvert_{\Gamma}^2] \,\d s  
 \leq C_{1,\delta} [1 + \|(\phi_0,\varphi_0)\|_{\mathbb{V}}^2]. 
\end{aligned}
\end{equation}
Furthermore, there exists a positive constant $C_{2,\delta}>0$ such that for all $n \in \mathbb{N}$,
   \begin{equation}\label{eq4.81a}
     \mathbb{E} \sup_{s \in[0,T]} \lvert \varphi_{\delta,n}(s) \rvert_{\Gamma}^2
     \leq C_{2,\delta} [1 + \Vert (\phi_0,\varphi_0) \Vert_{\mathbb{V}}^2].
   \end{equation}
\end{proposition}
\begin{proof}[Proof of Proposition \ref{Proposition_first_priori_estimmate}]
Let us proceed with estimating all the terms on the RHS of \eqref{Main_Galerkin_Equality}.
\newline
Using Assumption \ref{item:H2} on $F_1$, i.e. \eqref{eqn-boundedness-of-F_1}, we infer that for every $n \in \mathbb{N}$ and every $\delta>0$,
\begin{equation}\label{eq4.83}
\begin{aligned}
&\Vert F_{1,n}(\phi_{\delta,n}) \Vert_{\mathscr{T}_2(U,L^2(\mathcal{O}))}^2 
\leq \Vert \mathcal{S}_n \Vert_{\mathcal{L}(L^2)}^2 \Vert F_{1}(\phi_{\delta,n}) \Vert_{\mathscr{T}_2(U,L^2(\mathcal{O}))}^2 
 \\
&\leq \Vert F_{1}(\phi_{\delta,n}) \Vert_{\mathscr{T}_2(U,L^2(\mathcal{O}))}^2 
\leq \tilde{C}_1 \lvert \mathcal{O} \rvert.
\end{aligned}
\end{equation}
\dela{
\begin{equation}\label{eq4.83}
\begin{aligned}
&\Vert F_{1,n}(\phi_{\delta,n}) \Vert_{\mathscr{T}_2(U,L^2(\mathcal{O}))}^2 
\leq \Vert \mathcal{S}_n \Vert_{\mathcal{L}(L^2)}^2 \Vert F_{1}(\phi_{\delta,n}) \Vert_{\mathscr{T}_2(U,L^2(\mathcal{O}))}^2 
\leq \Vert F_{1}(\phi_{\delta,n}) \Vert_{\mathscr{T}_2(U,L^2(\mathcal{O}))}^2 
 \\
&= \sum_{k=1}^\infty \Vert F_1(\phi_{\delta,n}) e_{1,k} \Vert_{L^2(\mathcal{O})}^2 
= \sum_{k=1}^\infty \Vert \sigma_k(\phi_{\delta,n}) \Vert_{L^2(\mathcal{O})}^2 
\leq \lvert \mathcal{O} \rvert \sum_{k=1}^\infty \Vert \sigma_k \Vert_{L^{\infty}(\mathbb{R})}^2 
\leq \lvert \mathcal{O} \rvert C_1.
\end{aligned}
\end{equation}
}
Once more, thanks to the assumption \ref{item:H2} on $F_1$\dela{and Remark \ref{Rk_1}}, we obtain for every $n \in \mathbb{N}$ and $\delta>0$,
\begin{equation}\label{eq4.82}
\begin{aligned}
&\Vert \nabla F_{1,n}(\phi_{\delta,n}) \Vert_{\mathscr{T}_2(U,\mathbb{L}^2(\mathcal{O}))}^2
\leq \Vert F_{1,n}(\phi_{\delta,n}) \Vert_{\mathscr{T}_2(U,V_1)}^2
\leq \Vert \mathcal{S}_n \Vert_{\mathcal{L}(V_1)}^2 \Vert F_{1}(\phi_{\delta,n}) \Vert_{\mathscr{T}_2(U,V_1)}^2
\\
&\leq C(\mathcal{O}) \Vert F_{1}(\phi_{\delta,n}) \Vert_{\mathscr{T}_2(U,V_1)}^2
\leq C(C_1,\tilde{C}_1,\mathcal{O}) (1 + \lvert \nabla \phi_{\delta,n} \rvert^2).
\end{aligned}
\end{equation}
\dela{
\begin{equation}\label{eq4.82}
\begin{aligned}
&\Vert \nabla F_{1,n}(\phi_{\delta,n}) \Vert_{\mathscr{T}_2(U,\mathbb{L}^2(\mathcal{O}))}^2
\leq \Vert F_{1,n}(\phi_{\delta,n}) \Vert_{\mathscr{T}_2(U,V_1)}^2
\leq \Vert \mathcal{S}_n \Vert_{\mathcal{L}(V_1)}^2 \Vert F_{1}(\phi_{\delta,n}) \Vert_{\mathscr{T}_2(U,V_1)}^2
\\
&\leq C \Vert F_{1}(\phi_{\delta,n}) \Vert_{\mathscr{T}_2(U,V_1)}^2
= C \Vert F_{1}(\phi_{\delta,n}) \Vert_{\mathscr{T}_2(U,L^2(\mathcal{O}))}^2 + C \sum_{k=1}^\infty \lvert \sigma_k^\prime(\phi_{\delta,n}) \nabla \phi_{\delta,n} \rvert^2 
\\
&\leq C(C_1,\mathcal{O}) + C \sum_{k=1}^\infty \Vert \sigma_k \Vert_{W^{1,\infty}(\mathbb{R})}^2 \lvert \nabla \phi_{\delta,n}\rvert^2 
\leq C(C_1,\mathcal{O}) + C(C_1) \lvert \nabla \phi_{\delta,n} \rvert^2
\\
&\leq C(C_1,\mathcal{O}) (1 + \lvert \nabla \phi_{\delta,n} \rvert^2).
\end{aligned}
\end{equation}
}
Next, by \eqref{eq4.82} and the trace theorem, see \cite[Theorem 5.36]{Adams_1975}, we infer     
\begin{equation}\label{eq4.82a}
  \begin{aligned}
   & \Vert F_{1,n}(\phi_{\delta,n}) \Vert_{\mathscr{T}_2(U,L^2(\Gamma))}^2
   = \sum_{k=1}^\infty \Vert F_{1,n}(\phi_{\delta,n}) e_{1,k} \Vert_{L^2(\Gamma)}^2 
   \leq  C(\mathcal{O},\Gamma) \sum_{k=1}^\infty \Vert F_{1,n}(\phi_{\delta,n}) e_{1,k} \Vert_{V_1}^2 \\
   &\leq C(\mathcal{O},\Gamma) \Vert F_{1,n}(\phi_{\delta,n}) \Vert_{\mathscr{T}_2(U,V_1)}^2
   \leq C(C_1,\mathcal{O},\Gamma) (1 + \lvert \nabla \phi_{\delta,n} \rvert^2).
  \end{aligned}
\end{equation}
\dela{
\begin{equation}
  \begin{aligned}
   \Vert F_1(\phi_{\delta,n}) \Vert_{\mathscr{T}_2(U,L^2(\Gamma))}^2
   = \sum_{k=1}^\infty \Vert F_1(\phi_{\delta,n}) e_{1,k} \Vert_{L^2(\Gamma)}^2 
   &\leq  C(\mathcal{O},\Gamma) \sum_{k=1}^\infty \Vert F_1(\phi_{\delta,n}) e_{1,k} \Vert_{V_1}^2 \\
   &\leq C(\mathcal{O},\Gamma)( C_1 + C_1\lvert \mathcal{O} \rvert) \\
   &\leq C(\mathcal{O},\Gamma,C_1).
  \end{aligned}
\end{equation}
}
Note that $J_\delta^\prime(s) \in (0,1)$, $s \in \mathbb{R}$. Thus, $\mathbb{A}_\delta^\prime(s) \in (0,\frac{1}{\delta})$. Therefore, by $\ref{item:P4}$, we deduce that $\lvert F^{\bis}_\delta(s) \rvert\leq \frac{1}{\delta} + \tilde{c}_F$, $s \in \mathbb{R}$. This, jointly with \eqref{eq4.83}, implies that
\begin{equation}\label{eq4.85}
\sum_{k=1}^\infty \int_{\mathcal{O}} \lvert F^{\bis}_\delta(\phi_{\delta,n}) \rvert \lvert F_{1,n}(\phi_{\delta,n})e_{1,k} \rvert^2 \,\d x 
\leq (\delta^{-1} + \tilde{c}_F) \Vert F_{1,n}(\phi_{\delta,n}) \Vert_{\mathscr{T}_2(U,L^2(\mathcal{O}))}^2 
\leq C_1 (\delta^{-1} + \tilde{c}_F). 
\end{equation}
For the first stochastically forced term in \eqref{Main_Galerkin_Equality}, we use the BDG inequality and \eqref{eq4.83}. Thus, there exists $C=C(C_1,\mathcal{O})$ such that for all $n \in \mathbb{N}$ and all $\delta>0$,
\begin{equation}\label{eq4.86}
\begin{aligned}
\mathbb{E} \sup_{\tau \in[0,t\wedge \tau_\kappa]} \left \lvert \int_0^{\tau} \left(\mu_{\delta,n},F_{1}(\phi_{\delta,n}) \,\d W\right) \right \rvert 
&\leq C \mathbb{E} \left(\int_0^{t\wedge \tau_\kappa} \lvert \mu_{\delta,n} \lvert^2 \Vert F_1(\phi_{\delta,n}) \Vert_{\mathscr{T}_2(U,L^2(\mathcal{O}))}^2 \,\d s \right)^\frac{1}{2} \\
&\leq C \mathbb{E} \left(\int_0^{t \wedge \tau_\kappa} \lvert \mu_{\delta,n} \rvert^2 \,\d s \right)^{1/2}.
\end{aligned}
\end{equation}
By the assumption \ref{item:H2} on $F_2$ and arguing similarly as in \eqref{eq4.83} and \eqref{eq4.82}, we deduce that for every $n \in \mathbb{N}$ and $\delta>0$,
\begin{align}
\label{eqt4.87} 
\Vert F_{2,n}(\varphi_{\delta,n}) \Vert_{\mathscr{T}_2(U_\Gamma,L^2(\Gamma))}^2 
&\leq C_3,
\\
\label{eq4.88}
\Vert \nabla_{\Gamma,n} F_{\Gamma}(\varphi_{\delta,n}) \Vert_{\mathscr{T}_2(U_\Gamma,\mathbb{L}^2(\Gamma))}^2
&\leq C(\Gamma,C_3) (1 +  \lvert \nabla_\Gamma \varphi_{\delta,n} \rvert_{\Gamma}^2).
\end{align}
Therefore, by the BDG inequality, we infer that for all $n \in \mathbb{N}$ and $\delta>0$,
\begin{align*}
\mathbb{E} \sup_{\tau \in[0,t\wedge \tau_\kappa]} \left \lvert \int_0^{\tau} (\theta_{\delta,n},F_2(\varphi_{\delta,n})\, \d W_\Gamma)_\Gamma\right \rvert
&\leq C \mathbb{E} \left(\int_0^{t\wedge \tau_\kappa} \lvert \theta_{\delta,n} \rvert_\Gamma^2 \Vert F_2(\varphi_{\delta,n}) \Vert_{\mathscr{T}_2(U_\Gamma,L^2(\Gamma))}^2\, \d s \right)^\frac{1}{2} \\
&\leq C(\Gamma,C_2) \mathbb{E} \left(\int_0^{t\wedge \tau_\kappa} \lvert \theta_{\delta,n} \rvert_\Gamma^2 \,\d s\right)^\frac{1}{2}.
\end{align*}
Now, we claim that for every $n \in \mathbb{N}$,
\begin{equation}\label{Eqn-4.36}
\begin{aligned}
\left(\int_0^{t\wedge \tau_\kappa} \Vert (\mu_{\delta,n},\theta_{\delta,n}) \Vert_{\mathbb{H}}^2 \, \d s \right)^{\frac12}
& \leq \frac{1}{4 C} \int_0^{t\wedge \tau_\kappa} \int_{\mathcal{O}} M_{\mathcal{O}}(\phi_{\delta,n}) \lvert \nabla \mu_{\delta,n} \rvert^2 \,\d x \,\d s  + \frac1C \int_0^{t\wedge \tau_\kappa} \Vert (\mu_{\delta,n},\theta_{\delta,n}) \Vert^2_{\mathbb{V}^\prime} \, \d s \\
&\qquad + \frac{1}{2 C} \int_0^{t\wedge \tau_\kappa} \int_{\Gamma} M_\Gamma(\varphi_{\delta,n}) \lvert \nabla_\Gamma \theta_{\delta,n} \rvert^2 \,\d S \,\d s + C,
\end{aligned}
\end{equation}
where $C$ may depend on $C_1,C_3,\,\,\mathcal{O},\, \Gamma,\, M_0$, and $N_0$. \newline
Let us fix $(v,v_\Gamma) \in \bigcup_n \mathcal{V}_n \subset \mathbb{V}$. Then from \eqref{eq4.34d}, we infer that
\begin{align*}
\duality{(\mu_{\delta,n},\theta_{\delta,n})}{(v,v_\Gamma)}{\mathbb{V}}{\mathbb{V}^\prime}
&= \eps \int_{\mathcal{O}} \nabla \phi_{\delta,n} \cdot \nabla v\, \d x + \int_{\mathcal{O}} \frac{1}{\eps} F_\delta^\prime(\phi_{\delta,n}) v\, \d x + \eps_\Gamma \int_\Gamma \nabla_\Gamma \varphi_{\delta,n} \cdot \nabla_\Gamma v_\Gamma \, \d S \\
&\quad  + \int_\Gamma \frac{1}{\eps_\Gamma} G_\delta^\prime(\varphi_{\delta,n}) v_\Gamma \, \d S + (1/K) \int_\Gamma (\varphi_{\delta,n} - \phi_{\delta,n}) (\eps v_\Gamma - \eps v) \, \d S.
\end{align*}
Moreover, since $\mathbb{A}_\delta^\prime(s) \in (0,\frac{1}{\delta})$ for all $s\in \mathbb{R}$ and $\mathbb{A}_\delta(0)=0$, we infer from \ref{item:P4} that $\lvert F_\delta^\prime(r) \rvert \leq \tilde{c}_{1,\delta} \lvert r \rvert$ for all $r \in \mathbb{R}$, with $\tilde{c}_{1,\delta}= \tilde{c}_F + \delta^{-1}$. Therefore,
    \begin{equation*}
      \lvert F_\delta^\prime(\phi_{\delta,n}) \rvert \coloneq \Vert F_\delta^\prime(\phi_{\delta,n}) \Vert_{L^2(\mathcal{O})}
      \leq \tilde{c}_{1,\delta} \lvert \phi_{\delta,n} \rvert.
    \end{equation*}
Analogously,
   \begin{equation*}
     \lvert G_\delta^\prime(\varphi_{\delta,n}) \rvert_\Gamma
     \leq \tilde{c}_{2,\delta} \lvert \phi_{\delta,n} \rvert_\Gamma, \mbox{ with } \tilde{c}_{2,\delta}= \tilde{c}_G + \delta^{-1}.
    \end{equation*}
Furthermore, by trace theorem, we have 
\[
\lvert v_\Gamma - v \rvert_\Gamma 
\leq (\lvert v_\Gamma \rvert_\Gamma + \lvert v \rvert_\Gamma) 
\leq \lvert v_\Gamma \rvert_\Gamma + C(\Gamma,\mathcal{O}) \Vert v \Vert_{V_1}.
\]
Now from the previous observation and H\"older inequality, we deduce that for every $n \in \mathbb{N}$,
\begin{equation}
\begin{aligned}
&\lvert \duality{(\mu_{\delta,n},\theta_{\delta,n})}{(v,v_\Gamma)}{\mathbb{V}}{\mathbb{V}^\prime} \rvert
\leq \eps \lvert \nabla \phi_{\delta,n} \rvert \lvert \nabla v \rvert + \eps^{-1} \lvert F_\delta^\prime(\phi_{\delta,n}) \rvert \lvert v \rvert + \eps_\Gamma^{-1} \lvert \nabla_\Gamma \varphi_{\delta,n} \rvert_{\Gamma} \lvert \nabla_\Gamma v_\Gamma \rvert_{\Gamma} \\
&\hspace{5.3cm} + \eps_\Gamma^{-1} \lvert G_\delta^\prime(\varphi_{\delta,n}) \rvert_\Gamma \lvert v_\Gamma \rvert_\Gamma + \eps [K] \lvert \varphi_{\delta,n} - \phi_{\delta,n} \rvert_\Gamma \lvert v_\Gamma - v \rvert_{\Gamma}
\\
&\leq C (\lvert \nabla \phi_{\delta,n} \rvert + \lvert \phi_{\delta,n} \rvert + \lvert \nabla_\Gamma \varphi_{\delta,n} \rvert_{\Gamma} +  \lvert \varphi_{\delta,n} \rvert_{\Gamma} + [K] \lvert \varphi_{\delta,n} - \phi_{\delta,n} \rvert_\Gamma) \Vert (v,v_\Gamma) \Vert_{\mathbb{V}},
\end{aligned}
\end{equation}
where $C$ may depend on $\tilde{c}_G,\,\tilde{c}_F,\,\mathcal{O},\, \Gamma,\, \eps,\,\eps_\Gamma$, and $\delta$. \newline
On the other hand, by the trace theorem, see e.g., \cite[Theorem 5.36]{Adams_1975}, there exists a constant $C=C(\mathcal{O},\Gamma)>0$ such that for every $n \in \mathbb{N}$,
   \begin{align*}
      \lvert \phi_{\delta,n} \rvert_\Gamma \coloneq \Vert \phi_{\delta,n} \Vert_{L^2(\Gamma)}
      \leq C \Vert \phi_{\delta,n} \Vert_{V_1}.
   \end{align*}
Consequently,
  \begin{equation}\label{Eqn-4.34}
    \Vert (\mu_{\delta,n},\theta_{\delta,n}) \Vert_{\mathbb{V}^\prime}
     \leq C (\Vert \phi_{\delta,n} \Vert_{V_1} + \lvert \nabla_\Gamma \varphi_{\delta,n} \rvert_{\Gamma} + (1/K) \lvert \varphi_{\delta,n} - \phi_{\delta,n} \rvert_\Gamma).
  \end{equation}
Subsequently, notice that
\begin{align*}
&\Vert (\mu_{\delta,n},\theta_{\delta,n}) \Vert_{\mathbb{H}}^2
= \duality{(\mu_{\delta,n},\theta_{\delta,n})}{(\mu_{\delta,n},\theta_{\delta,n})}{\mathbb{V}}{\mathbb{V}^\prime}
\leq \Vert (\mu_{\delta,n},\theta_{\delta,n}) \Vert_{\mathbb{V}^\prime} \Vert (\mu_{\delta,n},\theta_{\delta,n}) \Vert_{\mathbb{V}}
\\
&\leq \|(\mu_{\delta,n},\theta_{\delta,n})\|_{\mathbb{V}^\prime} (\|(\mu_{\delta,n},\theta_{\delta,n})\|_{\mathbb{H}} + \|(\nabla \mu_{\delta,n},\nabla_\Gamma \theta_{\delta,n})\|_{\mathbb{L}^2(\mathcal{O}) \times \mathbb{L}^2(\Gamma)}) \\
&\leq \frac{1}{2} \|(\mu_{\delta,n},\theta_{\delta,n})\|_{\mathbb{H}}^2 +  \frac{1}{2} \|(\mu_{\delta,n},\theta_{\delta,n})\|_{\mathbb{V}^\prime}^2 + \|(\mu_{\delta,n},\theta_{\delta,n})\|_{\mathbb{V}^\prime} \|(\nabla \mu_{\delta,n},\nabla_\Gamma \theta_{\delta,n})\|_{\mathbb{L}^2(\mathcal{O}) \times \mathbb{L}^2(\Gamma)},
\end{align*}
which, in turn, entails that
    \begin{equation}\label{Eqn-4.35}
      \Vert (\mu_{\delta,n},\theta_{\delta,n}) \Vert_{\mathbb{H}}^2
      \leq \Vert (\mu_{\delta,n},\theta_{\delta,n}) \Vert_{\mathbb{V}^\prime}^2 + 2 \Vert(\mu_{\delta,n},\theta_{\delta,n}) \Vert_{\mathbb{V}^\prime} (\lvert \nabla \mu_{\delta,n} \rvert + \lvert \nabla_\Gamma \theta_{\delta,n} \rvert_\Gamma).
    \end{equation}
It then follows from \eqref{Eqn-4.35} that
\begin{align*}
&\left(\int_0^{t\wedge \tau_\kappa} \Vert (\mu_{\delta,n},\theta_{\delta,n}) \Vert_{\mathbb{H}}^2 \, \d s \right)^{\frac12}
\leq \left( \int_0^{t\wedge \tau_\kappa} \Vert (\mu_{\delta,n},\theta_{\delta,n}) \Vert_{\mathbb{V}^\prime}^2 \, \d s \right)^{\frac12} 
+ \sqrt{2} \left(\int_0^{t\wedge \tau_\kappa} \Vert (\mu_{\delta,n},\theta_{\delta,n}) \Vert^2_{\mathbb{V}^\prime} \, \d s\right)^{\frac14} (\times)
\\
& (\times) \left(\int_0^{t\wedge \tau_\kappa} \lvert \nabla \mu_{\delta,n} \rvert^2 \, \d s\right)^{\frac14} +  \sqrt{2} \left(\int_0^{t\wedge \tau_\kappa} \Vert(\mu_{\delta,n},\theta_{\delta,n}) \Vert_{\mathbb{V}^\prime}^2\, \d s\right)^{\frac14} \left(\int_0^{t\wedge \tau_\kappa} \lvert \nabla_\Gamma \theta_{\delta,n} \rvert_\Gamma^2 \, \d s\right)^{\frac14}.
\end{align*}
This, jointly with the assumption \ref{item:H4} and the Young inequality yields \eqref{Eqn-4.36}. \newline
\dela{
\begin{align*}
& \left(\int_0^{t\wedge \tau_\kappa} \Vert (\mu_{\delta,n},\theta_{\delta,n}) \Vert^2_{\mathbb{V}^\prime} \, \d s\right)^{\frac14} \left(\int_0^{t\wedge \tau_\kappa} \lvert \nabla \mu_{\delta,n} \rvert^2 \, \d s\right)^{\frac14}
\\
&\leq \frac{ \sqrt{2} }{M_0^{\frac14}} \left(\int_0^{t\wedge \tau_\kappa} \Vert (\mu_{\delta,n},\theta_{\delta,n}) \Vert^2_{\mathbb{V}^\prime} \, \d s\right)^{\frac14} \left(\int_0^{t\wedge \tau_\kappa} \int_{\mathcal{O}} M_{\mathcal{O}}(\phi_{\delta,n}) \lvert \nabla \mu_{\delta,n} \rvert^2 \,\d x \,\d s \right)^{\frac14}
\\
&\leq \frac14 \int_0^{t\wedge \tau_\kappa} \int_{\mathcal{O}} M_{\mathcal{O}}(\phi_{\delta,n}) \lvert \nabla \mu_{\delta,n} \rvert^2 \,\d x \,\d s + \int_0^{t\wedge \tau_\kappa} \Vert (\mu_{\delta,n},\theta_{\delta,n}) \Vert^2_{\mathbb{V}^\prime} \, \d s + C,
\end{align*}
where $C$ may depend on $C_1,\,\mathcal{O},\, M_0$
\begin{align*}
&\left(\int_0^{t\wedge \tau_\kappa} \Vert(\mu_{\delta,n},\theta_{\delta,n}) \Vert_{\mathbb{V}^\prime}\, \d s\right)^{\frac14} \left(\int_0^{t\wedge \tau_\kappa} \lvert \nabla_\Gamma \theta_{\delta,n} \rvert_\Gamma^2 \, \d s\right)^{\frac14}
\\
&\leq \frac14 \int_0^{t\wedge \tau_\kappa} \int_{\Gamma} M_\Gamma(\varphi_{\delta,n}) \lvert \nabla_\Gamma \theta_{\delta,n} \rvert^2 \,\d S \,\d s + \int_0^{t\wedge \tau_\kappa} \Vert (\mu_{\delta,n},\theta_{\delta,n}) \Vert^2_{\mathbb{V}^\prime} \, \d s + C  
\end{align*}
where $C$ may depend on $C_2,\,\Gamma,\, N_0$.
}
Thanks to \eqref{Eqn-4.36} and \eqref{Eqn-4.34}, we infer that there exists a constant $C>0$, which may depend on $\tilde{c}_G,\,\tilde{c}_F,\,\mathcal{O},\, \Gamma,\, \eps,\,\eps_\Gamma$, $C_1,\,C_3,\,M_0,\,N_0$,
and $\delta$ such that for every $n \in \mathbb{N}$ and every $\kappa \in \mathbb{N}$,
\begin{equation}
\begin{aligned}
&\mathbb{E} \sup_{\tau \in[0,t\wedge \tau_\kappa]} \left \lvert \int_0^{\tau} \left(\mu_{\delta,n},F_{1}(\phi_{\delta,n}) \,\d W\right) \right \rvert + \mathbb{E} \sup_{\tau \in[0,t\wedge \tau_\kappa]} \left \lvert \int_0^{\tau} (\theta_{\delta,n},F_2(\varphi_{\delta,n})\, \d W_\Gamma)_\Gamma\right \rvert
\\
&\leq C \mathbb{E} \left(\int_0^{t\wedge \tau_\kappa} \Vert (\mu_{\delta,n},\theta_{\delta,n}) \Vert_{\mathbb{H}}^2 \, \d s \right)^{\frac12}
\leq \frac{1}{4} \mathbb{E} \int_0^{t\wedge \tau_\kappa} \int_{\mathcal{O}} M_{\mathcal{O}}(\phi_{\delta,n}) \lvert \nabla \mu_{\delta,n} \rvert^2 \,\d x \,\d s   \\
&\qquad +  \mathbb{E} \int_0^{t\wedge \tau_\kappa} \Vert (\mu_{\delta,n},\theta_{\delta,n}) \Vert^2_{\mathbb{V}^\prime} \, \d s + \frac{1}{4} \mathbb{E} \int_0^{t\wedge \tau_\kappa} \int_{\Gamma} M_\Gamma(\varphi_{\delta,n}) \lvert \nabla_\Gamma \theta_{\delta,n} \rvert^2 \,\d S \,\d s + C
\\
&\leq \frac{1}{4} \mathbb{E} \int_0^{t\wedge \tau_\kappa} \int_{\mathcal{O}} M_{\mathcal{O}}(\phi_{\delta,n}) \lvert \nabla \mu_{\delta,n} \rvert^2 \,\d x \,\d s + \frac{1}{2} \mathbb{E} \int_0^{t\wedge \tau_\kappa} \int_{\Gamma} M_\Gamma(\varphi_{\delta,n}) \lvert \nabla_\Gamma \theta_{\delta,n} \rvert^2 \,\d S \,\d s
\\
&\quad + C \mathbb{E} \int_0^{t\wedge \tau_\kappa} \mathcal{E}_{tot}(\phi_{\delta,n}, \varphi_{\delta,n}) \, \d s + C.
\end{aligned}
\end{equation}
Owing to \ref{item:H4}, the Cauchy-Schwarz and Young inequalities, we find that
\begin{align*}
& \left|\int_{\mathcal{O}} M_{\mathcal{O}}(\phi_{\delta,n}) \nabla \mu_{\delta,n} \cdot \nabla \phi_{\delta,n}\, \d x \right|  
\leq  \lvert \sqrt{M_{\mathcal{O}}(\phi_{\delta,n})} \nabla \mu_{\delta,n} \rvert \lvert \sqrt{M_{\mathcal{O}}(\phi_{\delta,n})} \nabla \phi_{\delta,n} \rvert \\
&\leq \frac14 \lvert \sqrt{M_{\mathcal{O}}(\phi_{\delta,n})} \nabla \mu_{\delta,n} \rvert^2  + \lvert \sqrt{M_{\mathcal{O}}(\phi_{\delta,n})} \nabla \phi_{\delta,n} \rvert^2 
\leq \frac14 \int_{\mathcal{O}} M_{\mathcal{O}}(\phi_{\delta,n}) \lvert \nabla \mu_{\delta,n} \rvert^2\,\d x + \bar{M}_0 \lvert \nabla \phi_{\delta,n} \rvert^2.
\end{align*}
\dela{
\coma{
On the other hand, using the Poincar\'e inequality, we infer that for any $n \in \mathbb{N}, \, \delta \in (0,1)$,
\begin{equation}
\begin{aligned}
\lvert \mu_{\delta,n} \rvert^2
\leq 2 (\lvert \mu_{\delta,n} - \langle \mu_{\delta,n}\rangle_{\mathcal{O}} \rvert^2 + \lvert \langle \mu_{\delta,n} \rangle_{\mathcal{O}} \rvert^2) 
\leq C(\mathcal{O}) (\lvert \nabla \mu_{\delta,n} \rvert^2 + \lvert \langle \mu_{\delta,n} \rangle_{\mathcal{O}} \rvert^2).
\end{aligned}
\end{equation}
On the one hand, \coma{taking $(\psi,\psi \lvert_\Gamma)\equiv (1,0)$ in \eqref{eq4.34d} ???}, we obtain
 \begin{equation*}
   \langle \mu_{\delta,n} \rangle_{\mathcal{O}}
   = - \frac{\eps [K]}{\lvert \mathcal{O} \rvert} \int_\Gamma (\varphi_{\delta,n} - \phi_{\delta,n})\, \d S + \frac{1}{\eps \lvert \mathcal{O} \rvert} \int_{\mathcal{O}} F_\delta^\prime (\phi_{\delta,n})\, \d x,
  \end{equation*}  
which in turn, using the H\"older inequality and  the fact that $F_\delta^\prime$ is  $L_{F_\delta^\prime}$-Lipschitz yields
\begin{align*}
&\lvert \langle \mu_{\delta,n} \rangle_{\mathcal{O}} \rvert
\leq \eps [K] \lvert \mathcal{O} \rvert^{-1} \lvert \Gamma \rvert^\frac{1}{2} \lvert \varphi_{\delta,n} - \phi_{\delta,n} \rvert_\Gamma + \eps^{-1} \lvert \mathcal{O} \rvert^{-1} \int_{\mathcal{O}} \lvert F_\delta^\prime (\phi_{\delta,n}) \rvert\, \d x \\
&\leq \eps [K] \lvert \mathcal{O} \rvert^{-1} \lvert \Gamma \rvert^{\frac12} \lvert \varphi_{\delta,n} - \phi_{\delta,n} \rvert_{\Gamma} + \eps^{-1} \lvert \mathcal{O} \rvert^{-1} L_{F_\delta^\prime} \Vert \phi_{\delta,n} \Vert_{L^1(\mathcal{O})} \\
&\leq \eps [K] \lvert \mathcal{O} \rvert^{-1} \lvert \Gamma \rvert^{\frac12} \lvert \varphi_{\delta,n} - \phi_{\delta,n} \rvert_{\Gamma} + \eps^{-1} \lvert \mathcal{O} \rvert^{-1/2} L_{F_\delta^\prime} \lvert \phi_{\delta,n} \rvert.
\end{align*}
Squaring the two sides of the previous inequality, we obtain
\begin{equation*}
\begin{aligned}
\|\langle \mu_{\delta,n} \rangle_{\mathcal{O}}\|_{L^2(\mathcal{O})}^2
 = |\mathcal{O}| |\langle \mu_{\delta,n}\rangle_{\mathcal{O}}|^2 
 \leq 2 |\mathcal{O}|^{-1} \eps^2 [K]^2 |\Gamma| |\varphi_{\delta,n} - \phi_{\delta,n}|_{\Gamma}^2 + 2 \eps^{-2} L_{F_\delta^\prime}^2 |\phi_{\delta,n}|^2.
\end{aligned}
\end{equation*}
So, there exists a constant $C=C(\mathcal{O},\eps,\Gamma,\delta,\tilde{c}_F)$ such that for every $n \in \mathbb{N}$,
   \begin{equation}\label{eq4.87a}
    \Vert \mu_{\delta,n} \Vert_{L^2(\mathcal{O})}^2
     \leq C (\lvert \nabla \mu_{\delta,n} \rvert^2 + [K]^2 \lvert \varphi_{\delta,n} - \phi_{\delta,n} \rvert_{\Gamma}^2 + \lvert \phi_{\delta,n} \rvert^2),
   \end{equation}
and then by the H\"older inequality, we deduce that for every $n \in \mathbb{N}$ and every $\kappa \geq 1$,
\begin{equation}\label{eq4.87}
\begin{aligned}
 \left(\int_0^{t\wedge \tau_\kappa} |\mu_{\delta,n}(s)|^2 \,\d s\right)^{\frac{1}{2}} 
 &\leq C \left[ \left(\int_0^{t\wedge \tau_\kappa} |\nabla \mu_{\delta,n}(s)|^2 \, \d s\right)^{\frac{1}{2}} + \left(\int_0^{t\wedge \tau_\kappa} |\phi_{\delta,n}(s)|^2 \,\d s \right)^{\frac{1}{2}} \right. \\
 &\quad \left. +  [K] \left(\int_0^{t\wedge \tau_\kappa} |\varphi_{\delta,n}(s) - \phi_{\delta,n}(s)|_{\Gamma}^2\, \d s\right)^{\frac{1}{2}} \right].
\end{aligned} 
\end{equation}
Plugging \eqref{eq4.87} into the RHS of \eqref{eq4.86}, using the H\"older and Young inequalities together with the assumption \ref{item:H4}, we deduce that there exists $C= C(\tilde{c}_F,C_1,\mathcal{O},\Gamma,\eps,K,\delta,M_0)$ such that for every $n \in \mathbb{N}$, $t \in [0,T]$, and every $\kappa \geq 1$,
\begin{align}
 &\mathbb{E} \sup_{\tau \in[0,t\wedge \tau_\kappa]} \left|\int_0^{\tau} \left(\mu_{\delta,n}(s),F_{1}(\phi_{\delta,n}(s)) \,\d W(s)\right) \right| 
 \leq C \left[1  + \frac{1}{4} \mathbb{E} \int_0^{t\wedge \tau_\kappa} \int_{\mathcal{O}} M_{\mathcal{O}}(\phi_{\delta,n}(s)) \lvert \nabla \mu_{\delta,n}(s) \rvert^2\,\d x \,\d s \right. \notag \\
 &\quad \left.  + \mathbb{E} \int_0^{t\wedge \tau_\kappa} \left(\frac{1}{6} |\phi_{\delta,n}(s)|^2  + \frac{\eps [K]}{4} \|\varphi_{\delta,n}(s) - \phi_{\delta,n}(s)\|_{L^2(\Gamma)}^2 \right)\, \d s \right].
\end{align}
By the Burkholder-Davis-Gundy, H\"older and Young inequalities jointly with \eqref{eq4.83}, we infer that for every $n \in \mathbb{N}$,
\begin{align*}
&\mathbb{E} \sup_{\tau \in[0,t\wedge \tau_\kappa]} \left|\int_0^{\tau} \left(\phi_{\delta,n}(s),F_{1}(\phi_{\delta,n}(s)) \,\d W(s)\right) \right|
\leq C \mathbb{E} \left(\int_0^{t\wedge \tau_\kappa} |\phi_{\delta,n}(s)|^2 \|F_{1}(\phi_{\delta,n}(s))\|_{\mathscr{T}_2(U,L^2(\mathcal{O}))}^2\,\d s \right)^\frac{1}{2} 
 \\
&\leq C \mathbb{E} \left(\int_0^{t\wedge \tau_\kappa} |\phi_{\delta,n}(s)|^2 \,\d s \right)^\frac{1}{2} 
\leq C \left(\mathbb{E} \int_0^{t\wedge \tau_\kappa} |\phi_{\delta,n}(s)|^2 \,\d s \right)^\frac{1}{2} 
\leq C +  \mathbb{E} \int_0^{t\wedge \tau_\kappa} \frac{1}{6} |\phi_{\delta,n}(s)|^2 \,\d s.
\end{align*}
Here $C= C(\mathcal{O},C_1)$. \newline
In  similar way, using the assumption \ref{item:H3} and arguing as in \eqref{eq4.83} and \eqref{eq4.82}, we obtain for every $n \in \mathbb{N}$,
\dela{
\begin{equation}
  \begin{aligned}
   \Vert F_{2,n}(\varphi_{\delta,n}) \Vert_{\mathscr{T}_2(U_\Gamma,L^2(\Gamma))}^2 
   = \sum_{k=1}^\infty \lvert F_{2,n}(\varphi_{\delta,n}(s))e_{2,k} \rvert_{\Gamma}^2
   &= \sum_{k=1}^\infty |\tilde{\sigma_k}(\varphi_{\delta,n}(s))|_{\Gamma}^2 \\
   &\leq |\Gamma| \sum_{k=1}^\infty \|\tilde{\sigma}_k\|_{L^{\infty}(\mathbb{R})}^2 \\
   &\leq |\Gamma| C_2
  \end{aligned}
\end{equation}
and
\begin{equation}
\begin{aligned}
\|\nabla_\Gamma F_{\Gamma}(\varphi_{\delta,n}(s))\|_{\mathscr{T}_2(U_\Gamma,\mathbb{L}^2(\Gamma))}^2
= \sum_{k=1}^\infty |\nabla_\Gamma F_{\Gamma}(\varphi_{\delta,n}(s)) e_{2,k}|_{\Gamma}^2
&= \sum_{k=1}^\infty |\tilde{\sigma}_k^\prime(\varphi_{\delta,n}(s)) \nabla_\Gamma \varphi_{\delta,n}(s)|_{\Gamma}^2 \\
&\leq \sum_{k=1}^\infty \|\tilde{\sigma}_k\|_{W^{1,\infty}(\mathbb{R})}^2 |\nabla_\Gamma \varphi_{\delta,n}(s)|_{\Gamma}^2 \\
&\leq C_2 |\nabla_\Gamma \varphi_{\delta,n}(s)|_{\Gamma}^2.
\end{aligned}
\end{equation}
}
\begin{align}
\label{eqt4.87} 
\Vert F_{2,n}(\varphi_{\delta,n}) \Vert_{\mathscr{T}_2(U_\Gamma,L^2(\Gamma))}^2 
&\leq \lvert \Gamma \rvert C_2,
\\
\label{eq4.88}
\Vert \nabla_{\Gamma,n} F_{\Gamma}(\varphi_{\delta,n}) \Vert_{\mathscr{T}_2(U_\Gamma,\mathbb{L}^2(\Gamma))}^2
&\leq C(\Gamma,C_2) (1 +  \lvert \nabla_\Gamma \varphi_{\delta,n} \rvert_{\Gamma}^2).
\end{align}
Using the Burkholder-Davis-Gundy inequality and \eqref{eqt4.87}, we infer for any $n \in \mathbb{N}$,
\begin{align*}
&\mathbb{E} \sup_{\tau \in[0,t\wedge \tau_\kappa]} \left|\int_0^{\tau} (\theta_{\delta,n}(s),(F_2(\varphi_{\delta,n}(s)))\, \d W_\Gamma(s))_\Gamma\right|
\\
& \leq C \mathbb{E} \left(\int_0^{t\wedge \tau_\kappa} \lvert \theta_{\delta,n}(s) \rvert_\Gamma^2 \Vert F_2(\varphi_{\delta,n}(s)) \Vert_{\mathscr{T}_2(U_\Gamma,L^2(\Gamma))}^2\, \d s \right)^\frac{1}{2} \\
&\leq C(\Gamma,C_2) \mathbb{E} \left(\int_0^{t\wedge \tau_\kappa} |\theta_{\delta,n}(s)|_\Gamma^2 \, \d s \right)^\frac{1}{2}
\leq C(\Gamma,C_2) \left(\mathbb{E} \int_0^{t\wedge \tau_\kappa} |\theta_{\delta,n}(s)|_\Gamma^2 \, \d s \right)^\frac{1}{2}.
\end{align*}
Accordingly, from \eqref{eq4.34d}, we learn that
\[
\langle \theta_{\delta,n} \rangle_\Gamma
= \frac{1}{\eps_\Gamma \lvert \Gamma \rvert} \int_{\Gamma} G_\delta^\prime(\varphi_{\delta,n}) \, \d S  + \frac{\eps [K]}{\lvert \Gamma \rvert} \int_\Gamma (\varphi_{\delta,n} - \phi_{\delta,n})\, \d S.
\]
So, using the H\"older inequality and since $G_\delta^\prime$ is $L_{G_\delta^\prime}$-Lipschitz, we deduce that 
\begin{align*}
&\lvert \langle \theta_{\delta,n} \rangle_{\Gamma} \rvert
\leq \eps [K] \lvert \Gamma \rvert^{-1/2} \lvert \varphi_{\delta,n} - \phi_{\delta,n} \rvert_{\Gamma} + \eps_\Gamma^{-1} \lvert \Gamma \rvert^{-1} L_{G_\delta^\prime} \Vert \varphi_{\delta,n} \Vert_{L^1(\Gamma)} 
 \\
&\leq \eps [K] \lvert \Gamma \rvert^{-1/2} \lvert \varphi_{\delta,n} - \phi_{\delta,n} \rvert_{\Gamma} + \eps_\Gamma^{-1} \lvert \Gamma \rvert^{-1/2} L_{G_\delta^\prime} \lvert \varphi_{\delta,n} \rvert_\Gamma 
   \\
&\leq \lvert \Gamma \rvert^{-1/2} (\eps [K] + \eps_\Gamma^{-1} L_{G_\delta^\prime})  \lvert \varphi_{\delta,n} - \phi_{\delta,n} \rvert_{\Gamma}  + \eps_\Gamma^{-1} \lvert \Gamma \rvert^{-1/2} L_{G_\delta^\prime} \lvert \phi_{\delta,n} \rvert_\Gamma.
\end{align*}
On the other hand, by the trace theorem, see, e.g., \cite[Theorem 5.36]{Adams_1975}, there exists a positif constant $C=C(\mathcal{O},\Gamma)$ such that for every $n \in \mathbb{N}$,
   \begin{align*}
      \lvert \phi_{\delta,n} \rvert_\Gamma \coloneq \Vert \phi_{\delta,n} \Vert_{L^2(\Gamma)}
      \leq C \Vert \phi_{\delta,n} \Vert_{V_1}.
   \end{align*}
Hence, there exists a constant $C=C(\mathcal{O},\Gamma,\delta,\eps,\eps_\Gamma,\tilde{c}_G,K)$ such that for every $n \in \mathbb{N}$,
  \begin{equation*}
   \lvert \langle \theta_{\delta,n} \rangle_{\Gamma} \rvert
   \leq  C \eps [K] \lvert \varphi_{\delta,n} - \phi_{\delta,n} \rvert_{\Gamma}  + C \Vert \phi_{\delta,n} \Vert_{V_1}.
  \end{equation*}
This jointly with the Poincar\'e inequality on $\Gamma$, cf. \cite[Theorem 2.12]{Dziuk+Elliott_2013}, implies that there exists $C=C(\mathcal{O},\Gamma,\delta,\eps,\eps_\Gamma,\tilde{c}_G,K)$ such that for every $n \in \mathbb{N}$, 
\begin{equation}\label{eqt4.92}
\begin{aligned}
&\lvert \theta_{\delta,n} \rvert_\Gamma^2
\leq 2 \lvert \theta_{\delta,n} - \langle \theta_{\delta,n}\rangle_{\Gamma} \rvert^2 + 2 \lvert \langle \theta_{\delta,n} \rangle_{\Gamma} \rvert^2 
\leq  C(\Gamma) \lvert \nabla_\Gamma \theta_{\delta,n} \rvert_{\Gamma}^2 + 2  \lvert \Gamma \rvert \lvert \langle \theta_{\delta,n} \rangle_{\Gamma} \rvert^2 
  \\
&\leq C \lvert \nabla_\Gamma \theta_{\delta,n} \rvert_{\Gamma}^2 + C \eps [K] \lvert \varphi_{\delta,n} - \phi_{\delta,n} \rvert_{\Gamma}^2 + C \Vert \phi_{\delta,n} \Vert_{V_1}^2.
\end{aligned}
\end{equation}
Hence, there exists a constant $C=C(\eps,\eps_\Gamma,\delta,\tilde{c}_G,\Gamma,\mathcal{O},C_2,K)$ such that
for any $n \in \mathbb{N}$,
\begin{equation}
\begin{aligned}
&\mathbb{E} \sup_{\tau \in[0,t\wedge \tau_\kappa]} \left|\int_0^{r} (\theta_{\delta,n}(s),(F_2(\varphi_{\delta,n}(s)))\, \d W_\Gamma(s))_\Gamma\right| 
\leq C + \frac{\eps_\Gamma}{2} \mathbb{E} \int_0^{t\wedge \tau_\kappa}  \lvert \nabla_\Gamma \theta_{\delta,n}(s) \rvert_{\Gamma}^2 \, \d s   \\
&\quad + \frac16  \mathbb{E} \int_0^{t\wedge \tau_\kappa} \Vert \phi_{\delta,n}(s) \Vert_{V_1}^2 \, \d s  + \frac{\eps [K]}{4}  \mathbb{E} \int_0^{t\wedge \tau_\kappa} \Vert \varphi_{\delta,n}(s) - \phi_{\delta,n}(s) \Vert_{L^2(\Gamma)}^2\, \d s.
\end{aligned}
\end{equation}
}
}
From \eqref{eqt4.87}, we infer that
\begin{equation}\label{eq4.92a}
\begin{aligned}
\sum_{k=1}^\infty \int_{\Gamma} \lvert G^{\bis}_\delta(\varphi_{\delta,n}) \rvert \lvert F_{2,n}(\varphi_{\delta,n})e_{2,k} \rvert^2 \, \d S 
\leq (\delta^{-1} + \tilde{c}_G)  \Vert F_{2,n}(\varphi_{\delta,n}) \Vert_{\mathscr{T}_2(U_\Gamma,L^2(\Gamma))}^2
\leq C_3 (\delta^{-1} + \tilde{c}_G). 
\end{aligned}
\end{equation}
Thanks to \eqref{eq4.82a} and \eqref{eqt4.87}, we infer that there exists $C= C(\mathcal{O},\Gamma,C_1,C_3)$ such that for every $n \in \mathbb{N}$,
\begin{equation}\label{eq4.93b}
\begin{aligned}
\sum_{k=1}^\infty \left \lvert \int_\Gamma (F_{1,n}(\phi_{\delta,n}) e_{1,k}) (F_{2,n}(\varphi_{\delta,n}) e_{2,k})\, \d S \right \rvert 
&\leq \Vert F_{1,n}(\phi_{\delta,n}) \Vert_{\mathscr{T}_2(U,L^2(\Gamma))} \Vert F_2(\varphi_{\delta,n}) \Vert_{\mathscr{T}_2(U_\Gamma,L^2(\Gamma))}  \\
&\leq C (1 + \lvert \nabla \phi_{\delta,n} \rvert^2).
\end{aligned}
\end{equation} 
Considering \eqref{eq4.31}, we infer that for every $n \in \mathbb{N}$,
\begin{equation*}
\Vert F_\delta (\phi_{0,n}) \Vert_{L^1(\mathcal{O})}
\leq \tilde{c}_\delta (\lvert \mathcal{O} \rvert + \lvert \phi_{0,n} \rvert^2)
\leq \tilde{c}_\delta (\lvert \mathcal{O} \rvert + \lvert \phi_0 \rvert^2)
\leq C(\delta,\mathcal{O}) (1 + \Vert \phi_{0} \Vert_{V_1}^2).
\end{equation*}
Analogously,
    \begin{equation*}
     \Vert G_\delta (\varphi_{0,n}) \Vert_{L^1(\Gamma)}
       \leq  \tilde{c}_\delta (\lvert \Gamma \rvert + \lvert \varphi_{0,n} \rvert_\Gamma^2)
       \leq C(\delta,\Gamma) (1 + \Vert \varphi_0 \Vert_{V_\Gamma}^2).
    \end{equation*}
Therefore, there exists a constant $C= C(\eps,\delta,\Gamma,\eps_\Gamma,\mathcal{O})$ such that for every $n \in \mathbb{N}$,
    \begin{equation*}
      \mathcal{E}_{tot}(\phi_n(0),\varphi_n(0))
      \leq C (1 + \Vert (\phi_0,\varphi_0) \Vert_{\mathbb{V}}^2).
    \end{equation*}
Plugging the previous estimates in \eqref{Main_Galerkin_Equality}, taking the supremum over $[0,t\wedge \tau_\kappa]$ and finally taking expectations with respect to $\mathbb{P}$, we deduce that for every $n \in \mathbb{N}$ and every $\kappa \geq 1$,
\dela{
\begin{equation}
\begin{aligned}
&\mathbb{E} \sup_{s \in[0,t\wedge \tau_\kappa]} [\mathcal{E}_{tot}(\phi_{\delta,n}(s), \varphi_{\delta,n}(s))] + \mathbb{E} \int_0^{t\wedge \tau_\kappa} \int_{\mathcal{O}} [2 \nu(\phi_{\delta,n}) \lvert D\bu_{\delta,n}(s) \rvert^2 + \lambda(\phi_{\delta,n}(s)) \lvert \bu_{\delta,n}(s) \rvert^2] \,\d x \,\d s \\
&\quad + \mathbb{E} \int_0^{t\wedge \tau_\kappa} \int_{\Gamma} \gamma(\varphi_{\delta,n}(s)) \lvert \bu_{\delta,n}(s) \rvert^2 \,\d S\,\d s + \frac{1}{2}\mathbb{E} \int_0^{t\wedge \tau_\kappa} \int_{\mathcal{O}} M_{\mathcal{O}}(\phi_{\delta,n}(s)) \lvert \nabla \mu_{\delta,n}(s) \rvert^2 \, \d x \,\d s \\ 
& \hspace{6.6cm}  + \mathbb{E} \int_0^{t\wedge \tau_\kappa} \int_{\Gamma} M_\Gamma(\varphi_{\delta,n}(s)) \lvert \nabla_\Gamma \theta_{\delta,n}(s) \rvert^2 \,\d S \,\d s  
 \\
&\leq C_\delta (1 + \|(\phi_0,\varphi_0)\|_{\mathbb{V}}^2) +  c_1 \mathbb{E} (t\wedge \tau_\kappa) + c_2 \mathbb{E} \int_0^{t\wedge \tau_\kappa} \sup_{0\leq s \leq \tau} [\mathcal{E}_{tot}(\phi_{\delta,n}(s), \varphi_{\delta,n}(s))] \, \d \tau,
\end{aligned}
\end{equation}
}
\begin{equation}\label{eq4.91}
\begin{aligned}
&\mathbb{E} \sup_{s \in[0,t\wedge \tau_\kappa]} [\mathcal{E}_{tot}(\phi_{\delta,n}(s), \varphi_{\delta,n}(s))] + \mathbb{E} \int_0^{t\wedge \tau_\kappa} \int_{\mathcal{O}} [2 \nu(\phi_{\delta,n}) \lvert D\bu_{\delta,n} \rvert^2 + \lambda(\phi_{\delta,n}) \lvert \bu_{\delta,n} \rvert^2] \,\d x \,\d s \\
& + \mathbb{E} \int_0^{t\wedge \tau_\kappa} \int_{\Gamma} \gamma(\varphi_{\delta,n}) \lvert \bu_{\delta,n} \rvert^2 \,\d S\,\d s + \frac{1}{2}\mathbb{E} \int_0^{t\wedge \tau_\kappa} \int_{\mathcal{O}} M_{\mathcal{O}}(\phi_{\delta,n}) \lvert \nabla \mu_{\delta,n} \rvert^2 \, \d x \,\d s \\ 
& + \frac12 \mathbb{E} \int_0^{t\wedge \tau_\kappa} \int_{\Gamma} M_\Gamma(\varphi_{\delta,n}) \lvert \nabla_\Gamma \theta_{\delta,n} \rvert^2 \,\d S \,\d s  
 \\
&\leq C \left[1 + \Vert (\phi_0,\varphi_0) \Vert_{\mathbb{V}}^2 +  \mathbb{E} (t \wedge \tau_\kappa) + \mathbb{E} \int_0^{t\wedge \tau_\kappa} \sup_{0 \leq s \leq \tau} [\mathcal{E}_{tot}(\phi_{\delta,n}(s), \varphi_{\delta,n}(s))] \, \d \tau \right],
\end{aligned}
\end{equation}
where the constant $C$ depends on $\eps,\,\eps_\Gamma,\,C_1,\,C_3,\,\tilde{c}_F,\,\tilde{c}_G,\,\mathcal{O},\,\Gamma,\, K,\,M_0$ and $\delta$.
Moreover, from the assumptions \ref{item:H4}-\ref{item:H7} and \eqref{eq4.91}, we deduce that for every $n \in \mathbb{N}$ and every $\kappa \geq 1$,
\dela{
\begin{equation}
\begin{aligned}
&\mathbb{E} \sup_{\tau \in[0,t\wedge \tau_\kappa]} [\mathcal{E}_{tot}(\phi_{\delta,n}(\tau), \varphi_{\delta,n}(\tau))] + \mathbb{E} \int_0^{t\wedge \tau_\kappa} \int_{\mathcal{O}} [2 \nu_\ast \lvert D\bu_{\delta,n}(s) \rvert^2 + \lambda(\phi_{\delta,n}(s)) \lvert \bu_{\delta,n}(s) \rvert^2] \,\d x \,\d s \\
&\quad +  \mathbb{E} \int_0^{t\wedge \tau_\kappa} \gamma_\ast |\bu_{\delta,n}(s)|_\Gamma^2 \,\d s + \mathbb{E} \int_0^{t\wedge \tau_\kappa} [0.5 M_0 |\nabla \mu_{\delta,n}(s)|^2 + N_0 |\nabla_\Gamma \theta_{\delta,n}(s)|_{\Gamma}^2] \,\d s  
 \\
&\leq C_\delta (1 + \|(\phi_0,\varphi_0)\|_{\mathbb{V}}^2) +  c_1 \mathbb{E} (t\wedge \tau_\kappa) + c_2 \mathbb{E} \int_0^{t\wedge \tau_\kappa} \sup_{0\leq s \leq \tau} [\mathcal{E}_{tot}(\phi_{\delta,n}(s), \varphi_{\delta,n}(s))] \, \d \tau.
\end{aligned}
\end{equation}
}
\begin{equation}\label{eq4.92}
\begin{aligned}
&\mathbb{E} \sup_{s \in[0,t\wedge \tau_\kappa]} [\mathcal{E}_{tot}(\phi_{\delta,n}(s), \varphi_{\delta,n}(s))] + \mathbb{E} \int_0^{t\wedge \tau_\kappa} \int_{\mathcal{O}} [2 \nu_\ast \lvert D\bu_{\delta,n} \rvert^2 + \lambda(\phi_{\delta,n}) \lvert \bu_{\delta,n} \rvert^2] \,\d x \,\d s \\
&\quad + \mathbb{E} \int_0^{t\wedge \tau_\kappa} \gamma_\ast \lvert \bu_{\delta,n} \rvert_\Gamma^2 \,\d s + \frac12 \mathbb{E} \int_0^{t\wedge \tau_\kappa} [M_0 \lvert \nabla \mu_{\delta,n} \rvert^2 + N_0 \lvert \nabla_\Gamma \theta_{\delta,n} \rvert_{\Gamma}^2] \,\d s  
 \\
&\leq C \left[1 + \Vert (\phi_0,\varphi_0) \Vert_{\mathbb{V}}^2 +  \mathbb{E} (t \wedge \tau_\kappa) + \mathbb{E} \int_0^{t\wedge \tau_\kappa} \sup_{0 \leq s \leq \tau} [\mathcal{E}_{tot}(\phi_{\delta,n}(s), \varphi_{\delta,n}(s))] \, \d \tau \right], \; t \in [0,T].
\end{aligned}
\end{equation}
\dela{
Now, we set 
\begin{equation*}
\mathcal{Z}(t)\coloneq \mathbb{E} \int_0^{t\wedge \tau_\kappa} \sup_{0\leq s \leq \tau} [\mathcal{E}_{tot}(\phi_{\delta,n}(s), \varphi_{\delta,n}(s))] \, \d \tau
\end{equation*}
and then from \eqref{eq4.92}, we deduce that
    \begin{equation*}
    \mathcal{Z}^\prime(t) 
    \leq C_\delta (1 + \|(\phi_0,\varphi_0)\|_{\mathbb{V}}^2) +  c_1 \mathbb{E} (t\wedge \tau_\kappa) + c_2 \mathcal{Z}(t).
   \end{equation*}
Now, the Gronwall lemma yields
\begin{equation*}
\mathcal{Z}(t)
\leq \frac{1}{c_2} \left[C_\delta (1 + \|(\phi_0,\varphi_0)\|_{\mathbb{V}}^2) + \frac{c_1 \mathbb{E} 1 \wedge \tau_\kappa}{c_2} + c_1 \right] e^{c_2t} - \frac{1}{c_2} [C_\delta (1 + \|(\phi_0,\varphi_0)\|_{\mathbb{V}}^2) +  c_1 \mathbb{E} (t\wedge \tau_\kappa)],
\end{equation*}
from which and \eqref{eq4.92}, we deduce that
\begin{align*}
&\mathbb{E} \sup_{\tau \in[0,t\wedge \tau_\kappa]} [\mathcal{E}_{tot}(\phi_{\delta,n}(\tau), \varphi_{\delta,n}(\tau))] + \mathbb{E} \int_0^{t\wedge \tau_\kappa} \int_{\mathcal{O}} [2 \nu_\ast \lvert D\bu_{\delta,n}(s) \rvert^2 + \lambda(\phi_{\delta,n}(s)) \lvert \bu_{\delta,n}(s) \rvert^2] \,\d x \,\d s \\
&\quad +  \mathbb{E} \int_0^{t\wedge \tau_\kappa} \gamma_\ast |\bu_{\delta,n}(s)|_\Gamma^2 \,\d s + \mathbb{E} \int_0^{t\wedge \tau_\kappa} [0.5 M_0 |\nabla \mu_{\delta,n}(s)|^2 + N_0 |\nabla_\Gamma \theta_{\delta,n}(s)|_{\Gamma}^2] \,\d s  
 \\
&\leq \left[C_\delta (1 + \|(\phi_0,\varphi_0)\|_{\mathbb{V}}^2) + \frac{c_1 \mathbb{E} 1 \wedge \tau_\kappa}{c_2} + c_1 \right] e^{c_2t}.
\end{align*}
}
Thus, by the Gronwall lemma, we infer that for every $n \in \mathbb{N}$ and every $\kappa \geq 1$,
\begin{align*}
&\mathbb{E} \sup_{s\in[0,t\wedge \tau_\kappa]} [\mathcal{E}_{tot}(\phi_{\delta,n}(s), \varphi_{\delta,n}(s))] + \mathbb{E} \int_0^{t\wedge \tau_\kappa} \int_{\mathcal{O}} [\lvert D\bu_{\delta,n}(s) \rvert^2 + \lambda(\phi_{\delta,n}(s)) \lvert \bu_{\delta,n}(s) \rvert^2] \,\d x \,\d s \\
&\quad +  \mathbb{E} \int_0^{t\wedge \tau_\kappa}  \lvert \bu_{\delta,n}(s) \rvert_\Gamma^2 \,\d s + \mathbb{E} \int_0^{t\wedge \tau_\kappa} [ \lvert \nabla \mu_{\delta,n}(s) \rvert^2 +|\nabla_\Gamma \theta_{\delta,n}(s)|_{\Gamma}^2] \,\d s  
\leq C [1 + \|(\phi_0,\varphi_0)\|_{\mathbb{V}}^2],
\end{align*}
where $C$ depends on $\nu_\ast,\,\eps,\,\eps_\Gamma,\,C_1,\,C_3,\,\tilde{c}_F,\,\tilde{c}_G,\,\mathcal{O},\,\Gamma,\, K,\,M_0,\, \delta,\,\gamma_\ast$, $M_0,\,N_0$, and $T$.
\dela{
On the other hand, since $\tau_\kappa \nearrow T$ $\mathbb{P}$-a.s. as $R \to \infty$, we conclude by passing to the limit in the previous inequality and by using the Fatou lemma that 
\begin{align*}
 &\mathbb{E} \sup_{\tau \in[0,T]} [\mathcal{E}_{tot}(\phi_{\delta,n}(\tau), \varphi_{\delta,n}(\tau))] + \mathbb{E} \int_0^{T} \int_{\mathcal{O}} [2 \nu_\ast \lvert D\bu_{\delta,n}(s) \rvert^2 + \lambda(\phi_{\delta,n}(s)) \lvert \bu_{\delta,n}(s) \rvert^2] \,\d x \,\d s \\
 &\quad +  \mathbb{E} \int_0^{T} \gamma_\ast |\bu_{\delta,n}(s)|_\Gamma^2 \,\d s + \mathbb{E} \int_0^{T} [0.5 M_0 |\nabla \mu_{\delta,n}(s)|^2 + N_0 |\nabla_\Gamma \theta_{\delta,n}(s)|_{\Gamma}^2] \,\d s  
 \\
 &\leq \left[C_\delta (1 + \|(\phi_0,\varphi_0)\|_{\mathbb{V}}^2) + \frac{c_1 T}{c_2} + c_1 \right] e^{c_2T},
\end{align*}
from which we easily derive \eqref{eq4.80}. 
}
Moreover, since $\tau_\kappa \nearrow T$ $\mathbb{P}$-a.s. as $k \to \infty$, we infer \eqref{eq4.80}, by passing to the limit in that previous inequality and by using the Fatou lemma. \newline
Subsequently, by the H\"older inequality and the trace theorem, we obtain for every $n \in \mathbb{N}$,
\begin{equation}\label{eq4.99}
\lvert \varphi_{\delta,n} \rvert_{\Gamma}^2 
\leq 2\lvert \varphi_{\delta,n} - \phi_{\delta,n} \rvert_{\Gamma}^2 + 2 \lvert \phi_{\delta,n} \rvert_{\Gamma}^2 \\
\leq 2 \lvert \varphi_{\delta,n} - \phi_{\delta,n} \rvert_{\Gamma}^2 + C(\mathcal{O},\Gamma) \Vert \phi_{\delta,n} \Vert_{V_1}^2.
\end{equation}
Finally, combining \eqref{eq4.99} and \eqref{eq4.80}, we deduce \eqref{eq4.81a}.
\dela{
  \begin{align*}
   \mathbb{E} \sup_{s \in[0,T]} |\varphi_{\delta,n}(s)|_{\Gamma}^2
    &\leq 2 \mathbb{E} \sup_{s \in[0,T]} |\varphi_{\delta,n}(s) - \phi_{\delta,n}(s)|_{\Gamma}^2 + C(\mathcal{O},\Gamma) \mathbb{E} \sup_{s \in[0,T]} \|\phi_{\delta,n}(s)\|_{V_1}^2 \\
    &\leq C(\mathcal{O},\Gamma,K,\eps) \mathbb{E} \sup_{s \in[0,T]} [\mathcal{E}_{tot}(\phi_{\delta,n}(s), \varphi_{\delta,n}(s))] \\
    &\leq C_{2,\delta} [1 + \|(\phi_0,\varphi_0)\|_{\mathbb{V}}^2],
  \end{align*}
where $C_{2,\delta}>0$ depends only on the constant parameter $C_{1,\delta}$. The proof of Proposition \ref{Proposition_first_priori_estimmate} is now complete.
}
\end{proof}
\begin{proposition}\label{Proposition_second_priori_estimmate}
Let Assumptions \ref{item:H1}-\ref{item:H8} hold. Assume $\delta \in (0,1)$ and $r \geq 2$. Then, there exists a constant $C_{r/2,\delta}>0$ such that for every $n \in \mathbb{N}$,
\begin{equation}\label{eq4.95}
\begin{aligned}
& \mathbb{E} \sup_{s \in[0,t]} [\mathcal{E}_{tot}(\phi_{\delta,n}(s), \varphi_{\delta,n}(s))]^{r/2} + \mathbb{E} \left(\int_0^t \lvert \nabla \bu_{\delta,n}(s) \rvert^2 \, \d s \right)^{r/2}
+ \mathbb{E} \left(\int_0^t \lvert \nabla\mu_{\delta,n}(s) \rvert^2 \d s \right)^{r/2}
\\
& + \mathbb{E} \left(\int_0^{t} \lvert \nabla_\Gamma \theta_{\delta,n}(s) \rvert^2 \d s\right)^{r/2} + \mathbb{E} \left(\int_{Q_t} \lambda(\phi_{\delta,n}(s,x)) \lvert \bu_{\delta,n}(s,x) \rvert^2 \,\d x \,\d s \right)^{r/2}
\\
&\leq C_{r/2,\delta} [1 + \Vert (\phi_0,\varphi_0) \Vert_{\mathbb{V}}^r], \; \; t \in [0,T].
\end{aligned}
\end{equation}
Furthermore, there exists a constant $C_{\delta,r}>0$ such that for every $n \in \mathbb{N}$,
    \begin{align}\label{eq4.101a}
       \mathbb{E} \sup_{\tau \in[0,t]} \lvert \varphi_{\delta,n}(\tau) \rvert_{\Gamma}^r
       \leq C_{\delta,r} [1 + \Vert (\phi_0,\varphi_0) \Vert_{\mathbb{V}}^r], \; \; t \in [0,T]. 
       \end{align}
    \end{proposition}
\begin{proof}
Let $\delta \in (0,1)$. Without loss of generality, let us fix a number $r \geq 4$. Once we have proved Proposition \ref{Proposition_first_priori_estimmate}, each term of \eqref{Main_Galerkin_Equality} is well defined, and therefore we do not need to use the stopping times anymore. Hence, by exploiting \eqref{Main_Galerkin_Equality}, raising the corresponding equality to the power $r/2$, taking the supremum in time over $[0,t]$, $t\in [0,T]$, and finally taking expectation with respect to $\mathbb{P}$, we infer that there exists a constant $c_r$ depending on $\eps,\,K,\,\eps_\Gamma$, and $r$ such that for every $n \in \mathbb{N}$,
\begin{equation}\label{eq4.97}
\begin{aligned}
& \mathbb{E} \sup_{\tau \in[0,t]} [\mathcal{E}_{tot}(\phi_{\delta,n}(\tau), \varphi_{\delta,n}(\tau))]^\frac{r}{2} + \mathbb{E} \left(\int_0^t \left[2 \int_{\mathcal{O}} \nu(\phi_{\delta,n}) \lvert D\bu_{\delta,n} \rvert^2 \, \d x + \int_{\Gamma} \gamma(\varphi_{\delta,n}) \lvert \bu_{\delta,n} \rvert^2 \d S \right] \d s \right)^\frac{r}{2}  
\\
& + \mathbb{E} \left(\int_{Q_t} \lambda(\phi_{\delta,n}) \lvert \bu_{\delta,n} \rvert^2 \,\d x \,\d s \right)^\frac{r}{2} + [M_0]^\frac{r}{2} \mathbb{E} \left(\int_0^{t} \lvert \nabla \mu_{\delta,n} \rvert^2 \,\d s \right)^\frac{r}{2} + [N_0]^\frac{r}{2} \mathbb{E} \left(\int_0^{t} \lvert \nabla_\Gamma \theta_{\delta,n} \rvert^2 \,\d s\right)^\frac{r}{2}  
 \\
&\leq c_r \bigg([\mathcal{E}_{tot}(\phi_n(0),\varphi_n(0))]^\frac{r}{2} + \mathbb{E} \left\lvert \int_0^{t} \Vert (F_{1,n}(\phi_{\delta,n}) \Vert_{\mathscr{T}_2(U,L^2(\mathcal{O}))}^2 + F_{2,n}(\varphi_{\delta,n}) \Vert_{\mathscr{T}_2(U_\Gamma,L^2(\Gamma))}^2)\,\d s  \right\rvert^\frac{r}{2} \\
&  + \mathbb{E} \left\lvert \int_0^t (\Vert \nabla F_{1,n}(\phi_{\delta,n}) \Vert_{L^2(U,\mathbb{L}^2(\mathcal{O}))}^2 + \Vert \nabla_\Gamma F_{2,n}(\varphi_{\delta,n})\Vert_{\mathscr{T}_2(U_\Gamma,\mathbb{L}^2(\Gamma))}^2)  \, \d s \right \rvert^\frac{r}{2} \\
& + \mathbb{E} \left\lvert \int_0^{t} \Vert F_{1,n}(\phi_{\delta,n}) \Vert_{\mathscr{T}_2(U,L^2(\Gamma))}^2 \, \d s \right \rvert^\frac{r}{2} + \mathbb{E} \sup_{\tau \in [0,t]} \left\lvert \int_0^{\tau} (\mu_{\delta,n},F_1(\phi_{\delta,n})\, \d W) \right \rvert^\frac{r}{2} \\
& + \mathbb{E} \sup_{\tau \in [0,t]} \left\lvert \int_0^{\tau} (\theta_{\delta,n},(F_2(\varphi_{\delta,n}))\, \d W_\Gamma)_\Gamma\right \rvert^\frac{r}{2} + \mathbb{E} \left(\sum_{k=1}^\infty \int_{Q_t} \lvert F^{\bis}_\delta(\phi_{\delta,n}) \rvert \lvert F_{1,n}(\phi_{\delta,n})e_{1,k} \rvert^2 \, \d x \,\d s \right)^\frac{r}{2} \\
& + \mathbb{E} \left(\sum_{k=1}^\infty \int_{\Sigma_t} \lvert G^{\bis}_\delta(\varphi_{\delta,n}) \rvert \lvert F_{2,n}(\varphi_{\delta,n}) e_{2,k} \rvert^2\, \d S\,\d s\right)^\frac{r}{2} + \mathbb{E} \left \lvert \sum_{k=1}^\infty  \int_{\Sigma_t} (F_{1,n}(\phi_{\delta,n}) e_{1,k}) (F_{2,n}(\varphi_{\delta,n}) e_{2,k})\, \d S \, \d s \right \rvert^\frac{r}{2} \\
& + \mathbb{E} \left\lvert \int_{Q_t} M_{\mathcal{O}}(\phi_{\delta,n}) \nabla \mu_{\delta,n} \cdot \nabla \phi_{\delta,n} \, \d x \,\d s  \right \rvert^\frac{r}{2} + \mathbb{E} \sup_{\tau \in [0,t]} \left\lvert \int_0^{\tau}(\phi_{\delta,n}, F_{1}(\phi_{\delta,n})\,\d W) \right \rvert^\frac{r}{2} \bigg).
\end{aligned}
\end{equation}
Let us control all the terms on the RHS of \eqref{eq4.97}. First, by \eqref{eq4.83} and \eqref{eqt4.87}, we infer that
   \begin{align*}
     \left \lvert \int_0^t (\Vert F_{1,n}(\phi_{\delta,n}(s)) \Vert_{\mathscr{T}_2(U,L^2(\mathcal{O}))}^2 + \Vert F_{2,n}(\varphi_{\delta,n}(s)) \Vert_{\mathscr{T}_2(U_\Gamma,L^2(\Gamma))}^2) \,\d s \right \rvert^{r/2} 
    \leq  (C_1 + C_3)^\frac{r}{2} t^\frac{r}{2}.
   \end{align*}
Next, by the H\"older inequality, \eqref{eq4.82a}, and \eqref{eq4.93b}, we deduce that there exists $C= C(\mathcal{O},\Gamma,C_1,C_3,r)$ such that for every $n \in \mathbb{N}$ and every $\delta>0$,
\begin{align*}
&\left \lvert \int_0^{t} \Vert F_{1,n}(\phi_{\delta,n}) \Vert_{\mathscr{T}_2(U,L^2(\Gamma))}^2 \, \d s \right \rvert^{r/2} + \left\lvert \sum_{k=1}^\infty \int_{\Sigma_t} (F_{1,n}(\phi_{\delta,n}) e_{1,k}) (F_{2,n}(\varphi_{\delta,n}) e_{2,k})\, \d S \, \d s \right \rvert^{r/2}
\\
&\leq C \left(\int_0^t(1 + \lvert \nabla \phi_{\delta,n} \rvert^2) \, \d s \right)^{r/2}
\leq  C \left[ \int_0^t (1 + \lvert \nabla \phi_{\delta,n} \rvert^{r})\,\d s \right] t^\frac{r - 2}{2}.
\end{align*}
Thanks to \eqref{eq4.82}, \eqref{eq4.88}, and the H\"older inequality, we deduce that there exists a constant $C$ depending on $C_1,\,C_3,\,\mathcal{O},\,\Gamma$, and $r$ such that for every $n \in \mathbb{N}$ and $\delta>0$,
\begin{align*}
& \left \lvert \int_0^t (\Vert \nabla F_{1,n}(\phi_{\delta,n}) \Vert_{L^2(U,\mathbb{L}^2(\mathcal{O}))}^2 + \Vert \nabla_\Gamma F_{2,n}(\varphi_{\delta,n}) \Vert_{\mathscr{T}_2(U_\Gamma,\mathbb{L}^2(\Gamma))}^2)\, \d s \right \rvert^{r/2}  \\
&\leq  C\left(\int_0^t(1 + \lvert \nabla \phi_{\delta,n} \rvert^2 + \lvert \nabla_\Gamma \varphi_{\delta,n} \rvert_{\Gamma}^2) \, \d s \right)^{r/2}  
\leq  C \left[\int_0^t (1 + \lvert \nabla \phi_{\delta,n} \rvert^{r} + \lvert \nabla_\Gamma \varphi_{\delta,n} \rvert_{\Gamma}^{r}) \,\d s \right] t^\frac{r - 2}{2}.
\end{align*}
By the H\"older and Young inequalities, together with the assumption \ref{item:H4}, we deduce there exists $C=C(\bar{M}_0,M_0,r)$ such that for all $n \in \mathbb{N}$ and every $\delta>0$,
\begin{align*}
\mathbb{E} \left\lvert \int_0^t (M_{\mathcal{O}}(\phi_{\delta,n}) \nabla \mu_{\delta,n}, \nabla \phi_{\delta,n}) \,\d s  \right \rvert^{r/2} 
&\leq C \left[ \mathbb{E} \left(\int_0^t \lvert \nabla\mu_{\delta,n} \rvert^2 \, \d s \right)^{r/2} \right]^\frac{1}{2} \left[ \mathbb{E} \int_0^t \lvert \nabla \phi_{\delta,n} \rvert^{r}\,\d s \right]^\frac{1}{2} t^\frac{r - 2}{4}
\\
&\leq \frac{[M_0]^{r/2}}{4} \mathbb{E} \left(\int_0^t \lvert \nabla\mu_{\delta,n} \rvert^2 \, \d s \right)^\frac{r}{2} + C t^\frac{r - 2}{2} \, \mathbb{E} \int_0^t \lvert \nabla \phi_{\delta,n} \rvert^{r}\,\d s.
\end{align*}
Using the BDG inequality, \eqref{eq4.83}, and \eqref{eqt4.87}, we infer that for every $n \in \mathbb{N}$ and every $\delta>0$,
\begin{align*}
&c_r \mathbb{E} \sup_{\tau \in [0,t]} \left \lvert \int_0^{\tau} (\mu_{\delta,n},F_1(\phi_{\delta,n})\, \d W) \right \rvert^{r/2} + c_r \mathbb{E} \sup_{\tau \in [0,t]} \left \lvert \int_0^{\tau} (\theta_{\delta,n},F_2(\varphi_{\delta,n})\, \d W_\Gamma)_\Gamma \right\rvert^{r/2}
\\
& \leq C \mathbb{E} \left(\int_0^t \lvert \mu_{\delta,n} \rvert^2 \Vert F_1(\phi_{\delta,n}) \Vert_{\mathscr{T}_2(U,L^2(\mathcal{O}))}^2 \, \d s\right)^{r/4} + C \mathbb{E} \left(\int_0^t \lvert \theta_{\delta,n} \rvert_\Gamma^2 \Vert F_2(\varphi_{\delta,n}) \Vert_{\mathscr{T}_2(U_\Gamma,L^2(\Gamma))}^2\, \d s\right)^{r/4}
\\
&\leq C \mathbb{E} \left(\int_0^t \Vert (\mu_{\delta,n},\theta_{\delta,n}) \Vert_{\mathbb{H}}^2 \, \d s \right)^{r/4}.
\end{align*}
Next, Note that as $(\phi_{\delta,n},\varphi_{\delta,n}) \in H_{n}^1 \times H_{\Gamma,n}^1$, then $\mu_{\delta,n} \in H_{n}^1$. From \eqref{eqn-Galerkin-011} and \eqref{condition_ F'_and_F''}, we deduce that for every $n \in \mathbb{N}$ and every $\delta>0$,
\begin{align*}
&\left \lvert \int_{\mathcal{O}} \mu_{\delta,n} \, \d x \right \rvert
= \lvert (\mathcal{S}_n [-\eps \Delta \phi_{\delta,n} + \eps^{-1} F_\delta^\prime(\phi_{\delta,n})],1) \rvert
= \lvert (-\eps \Delta \phi_{\delta,n} + \eps^{-1} F_\delta^\prime(\phi_{\delta,n}), \mathcal{S}_n 1) \rvert
\\
&= \lvert (-\eps \Delta \phi_{\delta,n} + \eps^{-1} F_\delta^\prime(\phi_{\delta,n}),1) \rvert
\leq \eps \lvert (\Delta \phi_{\delta,n},1) \rvert + \eps^{-1} \lvert (F_\delta^\prime(\phi_{\delta,n}),1) \rvert
\\
&\leq \frac{\eps}{K} \int_\Gamma \lvert \varphi_{\delta,n} - \phi_{\delta,n} \rvert \, \d S + \frac{1}{\eps} \int_{\mathcal{O}} \lvert F_\delta^\prime(\phi_{\delta,n}) \rvert \, \d x
\leq  C(1 + \lvert \varphi_{\delta,n} - \phi_{\delta,n} \rvert_\Gamma + \Vert F_\delta(\phi_{\delta,n}) \Vert_{L^1(\mathcal{O})}),
\end{align*}
where $C= C(K,\eps,\mathcal{O},\Gamma,C_F)$. This implies that 
     \begin{equation}\label{Eqn-mean-mu_n}
       \left \lvert \fint_{\mathcal{O}} \mu_{\delta,n} \, \d x \right \rvert
       \leq  C(1 + \lvert \varphi_{\delta,n} - \phi_{\delta,n} \rvert_\Gamma + \Vert F_\delta(\phi_{\delta,n}) \Vert_{L^1(\mathcal{O})}).
    \end{equation}
Similarly, from \eqref{condition_ G'_and_G''} and \eqref{eqn-Galerkin-011}, we deduce there exists a constant $C= C(K,\eps,\eps_\Gamma,\mathcal{O},\Gamma,C_G)$ such that for every $n \in \mathbb{N}$ and every $\delta>0$,
      \begin{equation}\label{Eqn-mean-theta_n}
       \left \lvert \fint_{\Gamma} \theta_{\delta,n} \, \d S \right \rvert
       \leq  C(1 + \lvert \varphi_{\delta,n} - \phi_{\delta,n} \rvert_\Gamma + \Vert G_\delta(\phi_{\delta,n}) \Vert_{L^1(\Gamma)}).
    \end{equation}
Subsequently, owing to \eqref{Eqn-mean-mu_n}, \eqref{Eqn-mean-theta_n},  the Poincar\'e inequality, and the bulk-surface Poincar\'e inequality \eqref{bulk-surface-Poincare-inequality}, we infer that for every $n \in \mathbb{N}$ and all $\delta>0$,
\begin{equation}\label{Eqn-L2-norm-mu-and-theta-delta-n}
\begin{aligned}
&\Vert (\mu_{\delta,n},\theta_{\delta,n}) \Vert_{\mathbb{H}}^2
\leq C \left(1 + \Vert (\nabla \mu_{\delta,n},\nabla_\Gamma \theta_{\delta,n}) \Vert_{\mathbb{H}}^2 + \left \lvert \fint_{\mathcal{O}} \mu_{\delta,n} \, \d x \right \rvert^2 + \left \lvert \fint_{\Gamma} \theta_{\delta,n} \, \d S \right \rvert^2 \right)
\\
&\leq C \left(1 + \Vert (\nabla \mu_{\delta,n},\nabla_\Gamma \theta_{\delta,n}) \Vert_{\mathbb{H}}^2 + \lvert \varphi_{\delta,n} - \phi_{\delta,n} \rvert_\Gamma^2 + \Vert F_\delta(\phi_{\delta,n}) \Vert_{L^1(\mathcal{O})}^2 + \Vert G_\delta(\phi_{\delta,n}) \Vert_{L^1(\Gamma)}^2 \right),
\end{aligned}
\end{equation}
from which we deduce that for every $n \in \mathbb{N}$ and every $\delta>0$,
\begin{align*}
&c_r \mathbb{E} \sup_{\tau \in [0,t]} \left \lvert \int_0^{\tau} (\mu_{\delta,n},F_1(\phi_{\delta,n})\, \d W) \right \rvert^{r/2} + c_r \mathbb{E} \sup_{\tau \in [0,t]} \left \lvert \int_0^{\tau} (\theta_{\delta,n},F_2(\varphi_{\delta,n})\, \d W_\Gamma)_\Gamma \right\rvert^{r/2}
\\
&\leq C \left[t^\frac{r}{4} + \mathbb{E} \left(\int_0^t \lvert \nabla \mu_{\delta,n} \rvert^2 \, \d s \right)^\frac{r}{4} + \mathbb{E} \left(\int_0^t \lvert \nabla_\Gamma \theta_{\delta,n}) \rvert_\Gamma^2 \, \d s \right)^\frac{r}{4} + \mathbb{E} \left(\int_0^t \lvert \varphi_{\delta,n} - \phi_{\delta,n} \rvert_\Gamma^r \, \d s \right)^{\frac12} t^\frac{r - 2}{4} \right. \\
&\left. \hspace{1 truecm} + t^\frac{r-4}{4} \mathbb{E}  \int_0^t \Vert F_\delta(\phi_{\delta,n}) \Vert_{L^1(\mathcal{O})}^{r/2} \, \d s  + t^\frac{r-4}{4} \mathbb{E} \int_0^t \Vert G_\delta(\phi_{\delta,n}) \Vert_{L^1(\Gamma)}^{r/2} \, \d s  \right].
\end{align*}
Furthermore, Young's inequality leads to
\begin{align*}
&\left(\int_0^t \lvert \nabla \mu_{\delta,n} \rvert^2 \, \d s \right)^{r/4} + \left(\int_0^t \lvert \nabla_\Gamma \theta_{\delta,n}) \rvert_\Gamma^2 \, \d s \right)^{r/4}
\leq \frac{[M_0]^{r/2}}{4 C}  \left(\int_0^t \lvert \nabla\mu_{\delta,n} \rvert^2 \d s \right)^{r/2}
\\
&\quad + \frac{[N_0]^{r/2}}{2C} \left(\int_0^{t} \lvert \nabla_\Gamma \theta_{\delta,n} \rvert^2 \d s\right)^{r/2} + C(K,\eps,\eps_\Gamma,\mathcal{O},\Gamma,C_G,M_0,N_0,r),
\end{align*}
with $C$ being independent of $n$ and $\delta$. Hence, there exists a positive constant $C$ independent of $n$ and $\delta$, but which may depend on $K,\,\eps,\,\eps_\Gamma,\,\mathcal{O},\,\Gamma,\,C_G,\, C_F,\,M_0,\,N_0$, and $r$ such that
\begin{equation}\label{Eqn-both-zeta-stochastic-integrals}
\begin{aligned}
&c_r \mathbb{E} \sup_{\tau \in [0,t]} \left \lvert \int_0^{\tau} (\mu_{\delta,n},F_1(\phi_{\delta,n})\, \d W) \right \rvert^{r/2} + c_r \mathbb{E} \sup_{\tau \in [0,t]} \left \lvert \int_0^{\tau} (\theta_{\delta,n},F_2(\varphi_{\delta,n})\, \d W_\Gamma)_\Gamma \right\rvert^{r/2}
\\
&\leq  \frac{[M_0]^{r/2}}{4} \mathbb{E} \left(\int_0^t \lvert \nabla\mu_{\delta,n} \rvert^2 \d s \right)^{r/2} + \frac{[N_0]^{r/2}}{2} \mathbb{E} \left(\int_0^{t} \lvert \nabla_\Gamma \theta_{\delta,n} \rvert^2 \d s\right)^{r/2} + C(t^\frac{r}{4} + t^\frac{r-2}{2})
\\
&\quad + \mathbb{E} \int_0^t \lvert \varphi_{\delta,n} - \phi_{\delta,n} \rvert_\Gamma^r \, \d s + t^\frac{r-4}{4} C \mathbb{E}  \int_0^t (\Vert F_\delta(\phi_{\delta,n}) \Vert_{L^1(\mathcal{O})}^{r/2} + \Vert G_\delta(\phi_{\delta,n}) \Vert_{L^1(\Gamma)}^{r/2}) \, \d s.
\end{aligned}
\end{equation}
For the last stochastic integral in \eqref{eq4.97}, we infer that for every $n \in \mathbb{N}$ and every $\delta>0$,
\begin{align*}
&\mathbb{E} \sup_{\tau \in [0,t]} \left \lvert \int_0^{\tau}(\phi_{\delta,n}, F_{1}(\phi_{\delta,n})\,\d W)\right\rvert^{r/2}
\leq C \mathbb{E} \left(\int_0^t \lvert \phi_{\delta,n} \rvert^2 \Vert F_1(\phi_{\delta,n})\Vert_{\mathscr{T}_2(U,L^2(\mathcal{O}))}^2 \, \d s \right)^{r/4} \\
&\leq C \mathbb{E} \left(\int_0^t \lvert \phi_{\delta,n} \rvert^2 \, \d s \right)^{r/4} 
\leq C \left[ \mathbb{E} \int_0^t \lvert \phi_{\delta,n} \rvert^{r} \, \d s \right]^\frac{1}{2} t^\frac{r - 2}{4} 
\leq C(r,\mathcal{O},C_1) + t^\frac{r - 2}{2} \, \mathbb{E} \int_0^t \lvert \phi_{\delta,n} \rvert^{r} \, \d s.
\end{align*}
Now, Plugging the previous estimates into the RHS of \eqref{eq4.97}, using \eqref{Korn-inequality} and the assumptions \ref{item:H5}-\ref{item:H7}, we deduce that for every $n \in \mathbb{N}$ and every $\delta>0$,
\begin{equation}\label{Eqn-E_tot-zeta-estimate}
\begin{aligned}
& \mathbb{E} \sup_{s \in[0,t]} [\mathcal{E}_{tot}(\phi_{\delta,n}(s), \varphi_{\delta,n}(s))]^{r/2} + \mathbb{E} \left(\int_0^t \lvert \nabla \bu_{\delta,n} \rvert^2 \, \d s \right)^\frac{r}{2}
+ \mathbb{E} \left(\int_0^t \lvert \nabla\mu_{\delta,n} \rvert^2 \d s \right)^{r/2}
\\
&  + \mathbb{E} \left(\int_0^{t} \lvert \nabla_\Gamma \theta_{\delta,n} \rvert^2 \d s\right)^{r/2} + \mathbb{E} \left(\int_{Q_t} \lambda(\phi_{\delta,n}) \lvert \bu_{\delta,n} \rvert^2 \,\d x \,\d s \right)^\frac{r}{2}
 \\
&\leq C\bigg[ 1 + [\mathcal{E}_{tot}(\phi_n(0),\varphi_n(0))]^{r/2} + \mathbb{E} \left(\sum_{k=1}^\infty \int_{Q_t} \lvert F^{\bis}_\delta(\phi_{\delta,n}) \rvert \lvert F_{1,n}(\phi_{\delta,n})e_{1,k} \rvert^2 \, \d x \,\d s \right)^\frac{r}{2} 
   \\
&\qquad + \mathbb{E} \left(\sum_{k=1}^\infty \int_{\Sigma_t} \lvert G^{\bis}_\delta(\varphi_{\delta,n}) \rvert \lvert F_{2,n}(\varphi_{\delta,n}) e_{2,k} \rvert^2\, \d S\,\d s\right)^\frac{r}{2} 
+  \mathbb{E} \int_0^t [\mathcal{E}_{tot}(\phi_{\delta,n}, \varphi_{\delta,n})]^{r/2} \, \d s \bigg],
\end{aligned}
\end{equation}
where $C$ may depend on $r,\,\nu_0,\,\lambda_0,\,K,\,\eps,\,\eps_\Gamma,\,\mathcal{O},\,\Gamma,\,C_G,\, C_F,\,M_0,\,N_0,\,  C_{\text{KN}}$, and $T$. \newline
On the other hand, from \eqref{eq4.85} and \eqref{eq4.92a}, we infer that there exists $C=C (\mathcal{O},\Gamma,C_1,C_3,\tilde{c}_F,\tilde{c}_G,\delta,r)$ such that for every $n \in \mathbb{N}$,
   \begin{align*}
     \left(\sum_{k=1}^\infty \int_{Q_t} F^{\bis}_\delta(\phi_{\delta,n}) \lvert F_{1,n}(\phi_{\delta,n})e_{1,k} \rvert^2 \d x \,\d s \right)^\frac{r}{2} + \left(\sum_{k=1}^\infty  \int_{\Sigma_t} G^{\bis}_\delta(\varphi_{\delta,n}) \lvert F_{2,n}(\varphi_{\delta,n})e_{2,k} \rvert^2 \d S \, \d s \right)^\frac{r}{2} 
    \leq C \, t^\frac{r}{2}.
  \end{align*}
On the one hand, note that for every $n \in \mathbb{N}$,
     \begin{equation*}
       [\mathcal{E}_{tot}(\phi_n(0),\varphi_n(0))]^{r/2}
      \leq C (1 + \Vert (\phi_0,\varphi_0) \Vert_{\mathbb{V}}^r).
    \end{equation*}
Consequently, there exists a positive constant $C$ independent of $n$ such that
\begin{align*}
& \mathbb{E} \sup_{s \in[0,t]} [\mathcal{E}_{tot}(\phi_{\delta,n}(s), \varphi_{\delta,n}(s))]^{r/2} + \mathbb{E} \left(\int_0^t \lvert \nabla \bu_{\delta,n} \rvert^2 \, \d s \right)^\frac{r}{2}
+ \mathbb{E} \left(\int_0^t \lvert \nabla\mu_{\delta,n} \rvert^2 \d s \right)^{r/2}
\\
& + \mathbb{E} \left(\int_0^{t} \lvert \nabla_\Gamma \theta_{\delta,n} \rvert^2 \d s\right)^{r/2} + \mathbb{E} \left(\int_{Q_t} \lambda(\phi_{\delta,n}) \lvert \bu_{\delta,n} \rvert^2 \,\d x \,\d s \right)^\frac{r}{2}
  \\
&\leq C\bigg[ 1 + \Vert (\phi_0,\varphi_0) \Vert_{\mathbb{V}}^r + \int_0^t \mathbb{E} \sup_{0 \leq s \leq \tau} [\mathcal{E}_{tot}(\phi_{\delta,n}(s), \varphi_{\delta,n}(s))]^{r/2} \d \tau \bigg], \; \; t \in [0,T].
\end{align*}
Hence, by applying the Gronwall inequality, we deduce \eqref{eq4.95}. Subsequently, considering \eqref{eq4.99} and \eqref{eq4.95}, we infer \eqref{eq4.101a}. This completes the proof of the proposition.
\end{proof}
\dela{
\begin{proposition}\label{Proposition_second_priori_estimmate} 
Let $r>2$ be arbitrary. Let the assumptions \ref{item:H1}-\ref{item:H8} be satisfied. Then, there exists a constant $C_{r/2,\delta}>0$ such that for every $n \in \mathbb{N}$, 
\begin{align}\label{eq4.95}
& \mathbb{E} \sup_{s \in[0,t]} [\mathcal{E}_{tot}(\phi_{\delta,n}(s), \varphi_{\delta,n}(s))]^\frac{r}{2} + \mathbb{E} \left(\int_{Q_t} \lvert D\bu_{\delta,n} \rvert^2 \,\d x \,\d s \right)^\frac{r}{2} + \mathbb{E} \left(\int_{Q_t} \lambda(\phi_{\delta,n}) \lvert \bu_{\delta,n} \rvert^2 \,\d x \,\d s \right)^\frac{r}{2} \notag
\\
&  + \mathbb{E} \left(\int_0^{t} \lvert \bu_{\delta,n}(s) \rvert_{\Gamma}^2 \,\d s\right)^\frac{r}{2} + \mathbb{E} \left(\int_0^{t} \lvert \nabla \mu_{\delta,n}(s) \rvert^2 \,\d s \right)^\frac{r}{2} +  \mathbb{E} \left(\int_0^{t} \lvert \nabla_\Gamma \theta_{\delta,n}(s) \rvert^2 \,\d s\right)^\frac{r}{2} \notag 
 \\
&\leq C_{r/2,\delta} [1 + \Vert (\phi_0,\varphi_0) \Vert_{\mathbb{V}}^r], \; \; t \in [0,T].
\end{align}
In addition, there exists a constant $C_{\delta,r}>0$ such that for every $n \in \mathbb{N}$,
  \begin{equation}\label{eq4.101a}
    \mathbb{E} \sup_{\tau \in[0,t]} \lvert \varphi_{\delta,n}(\tau) \rvert_{\Gamma}^r
    \leq C_{\delta,r} [1 + \Vert (\phi_0,\varphi_0) \Vert_{\mathbb{V}}^r], \; \; t \in [0,T]. 
  \end{equation}
\end{proposition}
\begin{proof}
Let $r>2$ be arbitrary but fixed. \dela{Once we have proved Proposition \ref{Proposition_first_priori_estimmate}, each term of Equation \eqref{Main_Galerkin_Equality} is well-defined, and therefore we do not need to use the stopping times anymore. Hence, by
} 
Exploiting \eqref{Main_Galerkin_Equality}, raising the corresponding equality to the power $r/2$, taking the supremum in time over $[0,t]$, $t\in [0,T]$, and finally taking expectation with respect to $\mathbb{P}$, we obtain for every $n \in \mathbb{N}$,
\begin{equation}\label{eq4.97}
\begin{aligned}
& \mathbb{E} \sup_{\tau \in[0,t]} [\mathcal{E}_{tot}(\phi_{\delta,n}(\tau), \varphi_{\delta,n}(\tau))]^\frac{r}{2} + \nu_0^\frac{r}{2} \mathbb{E} \left\lvert\int_{Q_t} 2 \lvert D\bu_{\delta,n} \rvert^2 \,\d x \,\d s \right \rvert^\frac{r}{2} + \mathbb{E} \left(\int_{Q_t} \lambda(\phi_{\delta,n}) \lvert \bu_{\delta,n} \rvert^2 \,\d x \,\d s \right)^\frac{r}{2} 
\\
& + \lambda_0^\frac{r}{2} \mathbb{E} \left(\int_0^{t} \lvert \bu_{\delta,n}(s) \rvert_{\Gamma}^2 \,\d s\right)^\frac{r}{2} + M_0^\frac{r}{2} \mathbb{E} \left(\int_0^{t} \lvert \nabla \mu_{\delta,n}(s) \rvert^2 \,\d s \right)^\frac{r}{2} + N_0^\frac{r}{2} \mathbb{E} \left(\int_0^{t} \lvert \nabla_\Gamma \theta_{\delta,n}(s) \rvert^2 \,\d s\right)^\frac{r}{2}  
 \\
&\leq c_r \bigg([\mathcal{E}_{tot}(\phi_n(0),\varphi_n(0))]^\frac{r}{2} + \mathbb{E} \left\lvert \coma{\int_0^{t} \Vert F_{1,n}(\phi_{\delta,n}(s)) \Vert_{\mathscr{T}_2(U,L^2(\mathcal{O}))}^2\,\d s } \right\rvert^\frac{r}{2} \\
& + \mathbb{E} \left\lvert \int_0^{t} \Vert F_{2,n}(\varphi_{\delta,n}) \Vert_{\mathscr{T}_2(U_\Gamma,L^2(\Gamma))}^2 \, \d s \right\rvert^\frac{r}{2} + \mathbb{E} \left\lvert \int_0^{t} \Vert F_{1,n}(\phi_{\delta,n}) \Vert_{\mathscr{T}_2(U,L^2(\Gamma))}^2 \, \d s \right \rvert^\frac{r}{2} \\
& + \mathbb{E} \left\lvert \int_0^{t} \Vert \nabla F_{1,n}(\phi_{\delta,n}(s)) \Vert_{L^2(U,\mathbb{L}^2(\mathcal{O}))}^2  \, \d s \right \rvert^\frac{r}{2} + \mathbb{E} \left \lvert\int_0^{t} \Vert \nabla_\Gamma F_{2,n}(\varphi_{\delta,n}(s))\Vert_{\mathscr{T}_2(U_\Gamma,\mathbb{L}^2(\Gamma))}^2 \, \d s \right \rvert^\frac{r}{2} \\
& + \mathbb{E} \sup_{\tau \in [0,t]} \left\lvert \int_0^{\tau} (\mu_{\delta,n}(s),F_1(\phi_{\delta,n}(s))\, \d W(s)) \right \rvert^\frac{r}{2} + \mathbb{E} \sup_{\tau \in [0,t]} \left\lvert \int_0^{\tau} (\theta_{\delta,n}(s),(F_2(\varphi_{\delta,n}(s)))\, \d W_\Gamma(s))_\Gamma\right \rvert^\frac{r}{2} \\
& + \mathbb{E} \left\lvert \sum_{k=1}^\infty \int_{Q_t} F^{\bis}_\delta(\phi_{\delta,n}) \lvert F_{1,n}(\phi_{\delta,n})e_{1,k} \rvert^2 \, \d x \,\d s \right \rvert^\frac{r}{2} + \mathbb{E} \left\lvert \sum_{k=1}^\infty \int_{\Sigma_t} G^{\bis}_\delta(\varphi_{\delta,n}) \lvert F_{2,n}(\varphi_{\delta,n}) e_{2,k} \rvert^2\, \d S\,\d s\right \rvert^\frac{r}{2} \\
& + \mathbb{E} \left\lvert \sum_{k=1}^\infty  \int_{\Sigma_t} (F_{1,n}(\phi_{\delta,n}) e_{1,k}) (F_{2,n}(\varphi_{\delta,n}) e_{2,k})\, \d S \, \d s \right \rvert^\frac{r}{2} + \mathbb{E} \left\lvert \coma{\int_{Q_t} M_{\mathcal{O}}(\phi_{\delta,n}) \nabla \mu_{\delta,n} \cdot \nabla \phi_{\delta,n} \, \d x \,\d s } \right \rvert^\frac{r}{2} \\
& + \mathbb{E} \sup_{\tau \in [0,t]} \left\lvert \coma{\int_0^{\tau}(\phi_{\delta,n}(s), F_{1}(\phi_{\delta,n}(s))\,\d W(s))} \right \rvert^\frac{r}{2} \bigg),
\end{aligned}
\end{equation}
where the constant $c_r$ may depend on $\eps,\,K,\,\eps_\Gamma$, and $r$. \newline
Using \eqref{eq4.83}, \eqref{eq4.82a}, and \eqref{eqt4.87}, we infer that for every $n \in \mathbb{N}$,
\begin{align*}
 & c_r \left(\mathbb{E} \left\lvert \coma{\int_0^{t} \Vert F_{1,n}(\phi_{\delta,n}(s)) \Vert_{\mathscr{T}_2(U,L^2(\mathcal{O}))}^2\,\d s} \right \rvert^\frac{r}{2} + \mathbb{E} \left \lvert\int_0^{t} \Vert F_{2,n}(\varphi_{\delta,n}) \Vert_{\mathscr{T}_2(U_\Gamma,L^2(\Gamma))}^2 \, \d s \right \rvert^\frac{r}{2} \right. \\
 &\left. + \mathbb{E} \left \lvert \int_0^{t} \Vert F_{1,n}(\phi_{\delta,n}) \Vert_{\mathscr{T}_2(U,L^2(\Gamma))}^2 \, \d s \right \rvert^\frac{r}{2} \right) 
\leq C(\mathcal{O},\Gamma,C_1,C_2,r) t^\frac{r}{2}, \; \; t \in [0,T].
\end{align*}
\coma{
Proceeding as in \eqref{eq4.93b}, we obtain for every $n \in \mathbb{N}$,
   \begin{align*}
     c_r \mathbb{E} \left\lvert \sum_{k=1}^\infty \int_{\Sigma_t} (F_{1,n}(\phi_{\delta,n}) e_{1,k}) (F_{2,n}(\varphi_{\delta,n}) e_{2,k})\, \d S \, \d s \right \rvert^\frac{r}{2}
     \leq C(\mathcal{O},\Gamma,C_1,C_2,r)\,  t^\frac{r}{2}, \; \; t \in [0,T].
  \end{align*}
  }
Next, arguing as in \eqref{eq4.85} and \eqref{eq4.92a}, we infer that for every $n \in \mathbb{N}$,
\begin{align*}
&\mathbb{E} \left|\sum_{k=1}^\infty \int_{Q_t} F^{\bis}_\delta(\phi_{\delta,n}) \lvert F_{1,n}(\phi_{\delta,n})e_{1,k} \rvert^2 \d x \d s \right \rvert^\frac{r}{2} + \mathbb{E} \left \lvert \sum_{k=1}^\infty  \int_{\Sigma_t} G^{\bis}_\delta(\varphi_{\delta,n}) \lvert F_{2,n}(\varphi_{\delta,n})e_{2,k} \rvert^2 \d S \d s \right \rvert^\frac{r}{2} 
 \\
&\leq C(\mathcal{O},\Gamma,C_1,C_2,\tilde{c}_F,\tilde{c}_G,\delta,r) \, t^\frac{r}{2}, \; \; t \in [0,T].
\end{align*}
Using \eqref{eq4.82} and \eqref{eq4.88}, and the H\"older inequality, we infer for every $n \in \mathbb{N}$,
\begin{align*}
& \mathbb{E} \left \lvert \int_0^{t} \Vert \nabla F_{1,n}(\phi_{\delta,n}) \Vert_{L^2(U,\mathbb{L}^2(\mathcal{O}))}^2\, \d s \right \rvert^\frac{r}{2} + \mathbb{E} \left \lvert\int_0^{t} \Vert \nabla_\Gamma F_{2,n}(\varphi_{\delta,n}) \Vert_{\mathscr{T}_2(U_\Gamma,\mathbb{L}^2(\Gamma))}^2 \, \d s \right\rvert^\frac{r}{2}  
  \\
&\leq \left[ C\mathbb{E} \left(\int_0^t(1 + \lvert \nabla \phi_{\delta,n} \rvert^2) \, \d s \right)^\frac{r}{2} + C \mathbb{E} \left(\int_0^t (1 +  \lvert \nabla_\Gamma \varphi_{\delta,n} \rvert_{\Gamma}^2) \, \d s \right)^\frac{r}{2} \right] \\
&\leq  \left[C \mathbb{E} \int_0^t (1 + \lvert \nabla \phi_{\delta,n} \rvert^{r})\,\d s + C \mathbb{E} \int_0^t (1 + \lvert \nabla_\Gamma \varphi_{\delta,n} \rvert_{\Gamma}^{r}) \,\d s \right] t^\frac{r - 2}{2},
\end{align*}
where $C$ may depend on $C_1,\,C_2,\,\mathcal{O},\,\Gamma$, and $r$. \newline
From the assumption \ref{item:H4}, H\"older and Young inequalities, we obtain for all $n \in \mathbb{N}$,
\begin{align*}
& \mathbb{E} \left\lvert \red{\int_0^t (M_{\mathcal{O}}(\phi_{\delta,n}) \nabla \mu_{\delta,n}, \nabla \phi_{\delta,n}) \,\d s } \right \rvert^\frac{r}{2} 
\leq C \left[ \mathbb{E} \left(\int_0^t \lvert \nabla\mu_{\delta,n} \rvert^2 \, \d s \right)^\frac{r}{2} \right]^\frac{1}{2} \left[ \mathbb{E} \int_0^t \lvert \nabla \phi_{\delta,n} \rvert^{r}\,\d s \right]^\frac{1}{2} t^\frac{r - 2}{4}
\\
&\leq \frac{[M_0]^\frac{r}{2}}{4} \mathbb{E} \left(\int_0^t \lvert \nabla\mu_{\delta,n} \rvert^2 \, \d s \right)^\frac{r}{2} + C(\bar{M}_0,M_0,r) t^\frac{r - 2}{2} \, \mathbb{E} \int_0^t \lvert \nabla \phi_{\delta,n}(s) \rvert^{r}\,\d s.
\end{align*}
Using the Burkholder-Davis-Gundy inequality together with \eqref{eq4.83}, \eqref{eqt4.87}, and \eqref{Eqn-4.35}, we infer that for every $n \in \mathbb{N}$,
\begin{align*}
&c_r \mathbb{E} \sup_{\tau \in [0,t]} \left \lvert \int_0^{\tau} (\mu_{\delta,n},F_1(\phi_{\delta,n})\, \d W) \right \rvert^\frac{r}{2} + c_r \mathbb{E} \sup_{\tau \in [0,t]} \left \lvert \int_0^{\tau} (\theta_{\delta,n},F_2(\varphi_{\delta,n})\, \d W_\Gamma)_\Gamma \right\rvert^\frac{r}{2}
\\
& \leq C \mathbb{E} \left(\int_0^t \lvert \mu_{\delta,n} \rvert^2 \Vert F_1(\phi_{\delta,n}) \Vert_{\mathscr{T}_2(U,L^2(\mathcal{O}))}^2 \, \d s\right)^\frac{r}{4} + C \mathbb{E} \left(\int_0^t \lvert \theta_{\delta,n} \rvert_\Gamma^2 \Vert F_2(\varphi_{\delta,n}) \Vert_{\mathscr{T}_2(U_\Gamma,L^2(\Gamma))}^2\, \d s\right)^\frac{r}{4}
\\
&\leq C \mathbb{E} \left(\int_0^t \Vert (\mu_{\delta,n},\theta_{\delta,n}) \Vert_{\mathbb{H}}^2 \, \d s \right)^\frac{r}{4}
\leq C \mathbb{E} \left(\int_0^t \Vert (\mu_{\delta,n},\theta_{\delta,n}) \Vert_{\mathbb{V}^\prime}^2 \, \d s \right)^\frac{r}{4}
\\
&\quad + C \mathbb{E} \left(\int_0^t \Vert(\mu_{\delta,n},\theta_{\delta,n}) \Vert_{\mathbb{V}^\prime} (\lvert \nabla \mu_{\delta,n} \rvert + \lvert \nabla_\Gamma \theta_{\delta,n} \rvert_\Gamma) \, \d s \right)^\frac{r}{4}.
\end{align*}
Next, using the H\"older inequality, we see that for all $t \in [0,T]$,
\begin{equation}\label{Eqn-4.48}
\begin{aligned}
&\mathbb{E} \left(\int_0^t \Vert (\mu_{\delta,n},\theta_{\delta,n}) \Vert_{\mathbb{V}^\prime}^2 \, \d s \right)^\frac{r}{4}
\leq  \left[\mathbb{E} \left(\int_0^t \Vert (\mu_{\delta,n},\theta_{\delta,n}) \Vert_{\mathbb{V}^\prime}^2 \, \d s \right)^\frac{r}{2} \right]^{\frac12}
\\
&\leq \left[\mathbb{E} \left( \left(\int_0^t \Vert (\mu_{\delta,n},\theta_{\delta,n}) \Vert_{\mathbb{V}^\prime}^{r} \, \d s \right)^\frac{2}{r} t^\frac{r - 2}{r} \right)^\frac{r}{2} \right]^\frac{1}{2}  
= \left[ \mathbb{E} \int_0^t \Vert (\mu_{\delta,n},\theta_{\delta,n}) \Vert_{\mathbb{V}^\prime}^{r} \, \d s \right]^\frac{1}{2} t^\frac{r - 2}{4}.
\end{aligned}
\end{equation}
Now, by applying the H\"older inequalities and \eqref{Eqn-4.48}, we infer that for all $n \in \mathbb{N}$,
\begin{equation}
\begin{aligned}
&\mathbb{E} \left(\int_0^t \Vert(\mu_{\delta,n},\theta_{\delta,n}) \Vert_{\mathbb{V}^\prime} (\lvert \nabla \mu_{\delta,n} \rvert + \lvert \nabla_\Gamma \theta_{\delta,n} \rvert_\Gamma) \, \d s \right)^\frac{r}{4}
\leq C \mathbb{E} \bigg[\left(\int_0^t \Vert (\mu_{\delta,n},\theta_{\delta,n}) \Vert^2_{\mathbb{V}^\prime} \, \d s\right)^\frac{r}{8} \cdot
\\
& \left(\int_0^t \lvert \nabla \mu_{\delta,n} \rvert^2 \, \d s\right)^\frac{r}{8} \bigg] + C \mathbb{E} \bigg[\left(\int_0^t \Vert(\mu_{\delta,n},\theta_{\delta,n}) \Vert_{\mathbb{V}^\prime}^2 \, \d s\right)^\frac{r}{8} \left(\int_0^t \lvert \nabla_\Gamma \theta_{\delta,n} \rvert_\Gamma^2 \, \d s\right)^\frac{r}{8} \bigg]
\\
&\leq C \left[ \mathbb{E} \int_0^t \Vert (\mu_{\delta,n},\theta_{\delta,n}) \Vert_{\mathbb{V}^\prime}^{r} \, \d s \right]^\frac{1}{4} \bigg[ \mathbb{E} \left(\int_0^t \lvert \nabla \mu_{\delta,n} \rvert^2 \, \d s\right)^\frac{r}{2} \bigg]^{\frac14} t^\frac{r - 2}{8}
\\
&\quad + C \left[ \mathbb{E} \int_0^t \Vert (\mu_{\delta,n},\theta_{\delta,n}) \Vert_{\mathbb{V}^\prime}^{r} \, \d s \right]^\frac{1}{4}  \bigg[ \mathbb{E} \left(\int_0^t \lvert \nabla_\Gamma \theta_{\delta,n} \rvert_\Gamma^2 \, \d s\right)^\frac{r}{2} \bigg]^{\frac14} t^\frac{r - 2}{8},
\end{aligned}
\end{equation}
which along with the Young inequalities and \ref{item:H4} implies that for all $n \in \mathbb{N}$ and $t \in [0,T]$,
\begin{align*}
&c_r \mathbb{E} \sup_{\tau \in [0,t]} \left \lvert \int_0^{\tau} (\mu_{\delta,n},F_1(\phi_{\delta,n})\, \d W) \right \rvert^\frac{r}{2} + c_r \mathbb{E} \sup_{\tau \in [0,t]} \left \lvert \int_0^{\tau} (\theta_{\delta,n},F_2(\varphi_{\delta,n})\, \d W_\Gamma)_\Gamma \right\rvert^\frac{r}{2}
\\
&\leq C  \left[ \mathbb{E} \int_0^t \Vert (\mu_{\delta,n},\theta_{\delta,n}) \Vert_{\mathbb{V}^\prime}^{r} \, \d s \right]^\frac{1}{2} t^\frac{r - 2}{4} + C \left[ \mathbb{E} \int_0^t \Vert (\mu_{\delta,n},\theta_{\delta,n}) \Vert_{\mathbb{V}^\prime}^{r} \, \d s \right]^\frac{1}{4} \bigg[ \mathbb{E} \left(\int_0^t \lvert \nabla \mu_{\delta,n} \rvert^2 \, \d s\right)^\frac{r}{2} \bigg]^{\frac14} t^\frac{r - 2}{8}
\\
&\quad + C \left[ \mathbb{E} \int_0^t \Vert (\mu_{\delta,n},\theta_{\delta,n}) \Vert_{\mathbb{V}^\prime}^{r} \, \d s \right]^\frac{1}{4}  \bigg[ \mathbb{E} \left(\int_0^t \lvert \nabla_\Gamma \theta_{\delta,n} \rvert_\Gamma^2 \, \d s\right)^\frac{r}{2} \bigg]^{\frac14} t^\frac{r - 2}{8}
\\
&\leq \frac{[M_0]^\frac{r}{2}}{4} \mathbb{E} \left(\int_0^t \lvert \nabla \mu_{\delta,n} \rvert^2 \, \d s\right)^\frac{r}{2} + \frac{[N_0]^\frac{r}{2}}{2} \mathbb{E} \left(\int_0^t \lvert \nabla_\Gamma \theta_{\delta,n} \rvert^2 \, \d s \right)^\frac{r}{2} \\
&\quad + t^\frac{r - 2}{2} \mathbb{E} \int_0^t \Vert (\mu_{\delta,n},\theta_{\delta,n}) \Vert_{\mathbb{V}^\prime}^{r} \, \d s + C,
\end{align*}
where $C$ may depend on $r,\,C_1,\, C_2,\,\tilde{c}_G,\,\tilde{c}_F,\,\mathcal{O},\, \Gamma,\, \eps,\,\eps_\Gamma,\, M_0,\, N_0$, and $\delta$. \newline
On the other hand, thanks to \eqref{Eqn-4.34}, we deduce that for all $n \in \mathbb{N}$ and $t \in [0,T]$,
\begin{equation}
\begin{aligned}
&c_r \mathbb{E} \sup_{\tau \in [0,t]} \left \lvert \int_0^{\tau} (\mu_{\delta,n},F_1(\phi_{\delta,n})\, \d W) \right \rvert^\frac{r}{2} + c_r \mathbb{E} \sup_{\tau \in [0,t]} \left \lvert \int_0^{\tau} (\theta_{\delta,n},F_2(\varphi_{\delta,n})\, \d W_\Gamma)_\Gamma \right\rvert^\frac{r}{2}
\\
&\leq \frac{[M_0]^\frac{r}{2}}{4} \mathbb{E} \left(\int_0^t \lvert \nabla \mu_{\delta,n} \rvert^2 \, \d s\right)^\frac{r}{2} + \frac{[N_0]^\frac{r}{2}}{2} \mathbb{E} \left(\int_0^t \lvert \nabla_\Gamma \theta_{\delta,n} \rvert^2 \, \d s \right)^\frac{r}{2} \\
&\quad + C \, t^\frac{r - 2}{2} \mathbb{E} \int_0^t [\mathcal{E}_{tot}(\phi_{\delta,n}(s), \varphi_{\delta,n}(s))]^\frac{r}{2} \, \d s + C.
\end{aligned}
\end{equation}
For the last stochastic integral in \eqref{eq4.97}, we infer that for every $n \in \mathbb{N}$,
\begin{align*}
&\mathbb{E} \sup_{\tau \in [0,t]} \left \lvert \red{\int_0^{\tau}(\phi_{\delta,n}, F_{1}(\phi_{\delta,n})\,\d W)}\right\rvert^\frac{r}{2}
\leq C \mathbb{E} \left(\int_0^t \lvert \phi_{\delta,n} \rvert^2 \Vert F_1(\phi_{\delta,n})\Vert_{\mathscr{T}_2(U,L^2(\mathcal{O}))}^2 \, \d s \right)^\frac{r}{4} \\
&\leq C \mathbb{E} \left(\int_0^t \lvert \phi_{\delta,n} \rvert^2 \, \d s \right)^\frac{r}{4} 
\leq C \left[ \mathbb{E} \int_0^t \lvert \phi_{\delta,n} \rvert^{r} \, \d s \right]^\frac{1}{2} t^\frac{r - 2}{4} 
\leq C(r,\mathcal{O},C_1) + t^\frac{r - 2}{2} \, \mathbb{E} \int_0^t \lvert \phi_{\delta,n} \rvert^{r} \, \d s.
\end{align*}
\dela{
\coma{
\begin{equation}
\begin{aligned}
c_r \mathbb{E} \sup_{\tau \in [0,t]} \left \lvert \int_0^{\tau} (\theta_{\delta,n},F_2(\varphi_{\delta,n})\, \d W_\Gamma) \right\rvert^\frac{r}{2}  
&\leq \frac{[N_0]^\frac{r}{2}}{2} \mathbb{E} \left(\int_0^t \lvert \nabla_\Gamma \theta_{\delta,n} \rvert^2 \, \d s \right)^\frac{r}{2} + C \\
& + t^\frac{r - 2}{2} \mathbb{E} \left[ \int_0^t \left(\frac13 \Vert \phi_{\delta,n} \Vert_{V_1}^{r} +  \frac12 |\varphi_{\delta,n} - \phi_{\delta,n} \rvert_{\Gamma}^{r} \right)\d s \right]. 
\end{aligned}
\end{equation}
Next, using the Burkholder-Davis-Gundy inequality and \eqref{eq4.83}, we infer that for every $n \in \mathbb{N}$,
\begin{equation*}
\begin{aligned}
c_r \mathbb{E} \sup_{\tau \in [0,t]} \left \lvert \int_0^{\tau} (\mu_{\delta,n},F_1(\phi_{\delta,n})\, \d W) \right \rvert^\frac{r}{2} 
& \leq C \mathbb{E} \left(\int_0^t \lvert \mu_{\delta,n} \rvert^2 \Vert F_1(\phi_{\delta,n}) \Vert_{\mathscr{T}_2(U,L^2(\mathcal{O}))}^2 \, \d s \right)^\frac{r}{4} \\
&\leq C \mathbb{E} \left(\int_0^t |\mu_{\delta,n}(s)|^2 \, \d s \right)^\frac{r}{4}.
\end{aligned}
\end{equation*}
 Besides, we recall that, cf. \eqref{eq4.87a},
\begin{equation*}
    \Vert \mu_{\delta,n} \Vert_{L^2(\mathcal{O})}^2
     \leq C(\mathcal{O},\eps,K,\Gamma,\tilde{c}_F,\delta) [\lvert \nabla \mu_{\delta,n} \rvert^2 + \lvert \varphi_{\delta,n} - \phi_{\delta,n} \rvert_{\Gamma}^2 + \lvert \phi_{\delta,n} \rvert^2].
   \end{equation*}
Thus, thanks to the H\"older and Young inequalities, we infer that there exists $C= C(C_1,\mathcal{O},\Gamma,\eps,\tilde{c}_F,\delta,M_0,r)$ such that
for every $n \in \mathbb{N}$,
\dela{
\noindent
Notice that
   \begin{equation*}
     \mathbb{E} \left( \int_0^t |\nabla \mu_{\delta,n}(s)|^2 \, \d s \right)^\frac{r}{4} 
     \leq \frac{1}{M_0^\frac{r}{4}} \left[M_0^\frac{r}{2} \mathbb{E} \left( \int_0^t |\nabla \mu_{\delta,n}(s)|^2 \, \d s \right)^\frac{r}{2} \right]^\frac{1}{2},
   \end{equation*}
\begin{align*}
 \mathbb{E} \left( \int_0^t |\phi_{\delta,n}(s)|^2 \, \d s \right)^\frac{r}{4} 
 &\leq \left[\mathbb{E} \left( \int_0^t |\phi_{\delta,n}(s)|^2 \, \d s \right)^\frac{r}{2} \right]^\frac{1}{2} \\
 &\leq \left[ \mathbb{E} \left( \left(\int_0^t |\phi_{\delta,n}(s)|^{r} \, \d s \right)^\frac{2}{r} t^\frac{r - 2}{r} \right)^\frac{r}{2} \right]^\frac{1}{2} \\
 &= \left[ \mathbb{E} \int_0^t |\phi_{\delta,n}(s)|^{r} \, \d s \right]^\frac{1}{2} t^\frac{r - 2}{4},
\end{align*}
and
\begin{equation*}
\mathbb{E} \left(\int_0^t |\varphi_{\delta,n}(s) - \phi_{\delta,n}(s)|_\Gamma^2 \, \d s \right)^\frac{r}{4} 
\leq \left[ \mathbb{E} \int_0^t |\varphi_{\delta,n}(s) - \phi_{\delta,n}(s)|_{\Gamma}^{r} \, \d s \right]^\frac{1}{2} t^\frac{r - 2}{4},
\end{equation*}
where we used H\"older's inequality. Hence, we have (after updating the constant $c_{1,r}$)
\begin{align*}
&\mathbb{E} \left( \int_0^t \|\mu_{\delta,n}(s)\|_{L^2(\mathcal{O})}^2\, \d s \right)^\frac{r}{4} \\
&\leq c_{1,r} \left[\mathbb{E} \left( \int_0^t |\nabla \mu_{\delta,n}(s)|^2 \, \d s \right)^\frac{r}{4} + \mathbb{E} \left( \int_0^t |\phi_{\delta,n}(s)|^2 \, \d s \right)^\frac{r}{4} + \mathbb{E} \left( \int_0^t |\varphi_{\delta,n}(s) - \phi_{\delta,n}(s)|_{\Gamma}^2 \, \d s \right)^\frac{r}{4} \right] \\
&\leq \frac{ c_{1,r}}{M_0^\frac{r}{4}} \left[M_0^\frac{r}{2} \mathbb{E} \left( \int_0^t |\nabla \mu_{\delta,n}(s)|^2 \, \d s \right)^\frac{r}{2} \right]^\frac{1}{2} + c_{1,r} \left[ \mathbb{E} \int_0^t |\phi_{\delta,n}(s)|^{r} \, \d s \right]^\frac{1}{2} t^\frac{r - 2}{4} \\
&\quad + c_{1,r} \left[ \mathbb{E} \int_0^t |\varphi_{\delta,n}(s) - \phi_{\delta,n}(s)|_{\Gamma}^{r} \, \d s \right]^\frac{1}{2} t^\frac{r - 2}{4},
\end{align*}
which, in turn, entails that
\begin{align*}
 & c_r \mathbb{E} \sup_{\tau \in [0,t]} \left\lvert \int_0^{\tau} (\mu_{\delta,n},F_1(\phi_{\delta,n})\, \d W) \right\rvert^\frac{r}{2} 
 \leq C \left[M_0^\frac{r}{2} \mathbb{E} \left( \int_0^t \lvert \nabla \mu_{\delta,n} \rvert^2 \, \d s \right)^\frac{r}{2} \right]^\frac{1}{2}  \\
&\quad + C \left[ \mathbb{E} \int_0^t \lvert \phi_{\delta,n} \rvert^{r} \, \d s \right]^\frac{1}{2} t^\frac{r - 2}{4} + C  \left[ \mathbb{E} \int_0^t \lvert \varphi_{\delta,n} - \phi_{\delta,n} \rvert_{\Gamma}^{r} \, \d s \right]^\frac{1}{2} t^\frac{r - 2}{4},
\end{align*}
where the constant $C= C(C_1,\mathcal{O},\Gamma,\eps,\tilde{c}_F,r,\delta,M_0)$ is independent of $n$. Finally, using Young's inequalities, we find
\begin{equation}
\begin{aligned}
& c_r \mathbb{E} \sup_{\tau \in [0,t]} \left|\int_0^{\tau} (\mu_{\delta,n}(s),F_1(\phi_{\delta,n}(s))\, \d W(s)) \right|^\frac{r}{2} 
\\
&\leq \frac{M_0^\frac{r}{2}}{4} \mathbb{E} \left(\int_0^t |\nabla\mu_{\delta,n}(s)|^2 \, \d s \right)^\frac{r}{2} + C_r + t^\frac{r - 2}{2} \mathbb{E}  \left[ \int_0^t (\frac{1}{3} |\phi_{\delta,n}(s)|^{r} + \frac{1}{2} |\varphi_{\delta,n}(s) - \phi_{\delta,n}(s)|_{\Gamma}^{r}) \, \d s \right],  
\end{aligned}
\end{equation}
}
\begin{equation}\label{eq4.101}
\begin{aligned}
& c_r \mathbb{E} \sup_{\tau \in [0,t]} \left\lvert \int_0^{\tau} (\mu_{\delta,n},F_1(\phi_{\delta,n})\, \d W) \right\rvert^\frac{r}{2} 
\leq C \left[M_0^\frac{r}{2} \mathbb{E} \left( \int_0^t \lvert \nabla \mu_{\delta,n} \rvert^2 \, \d s \right)^\frac{r}{2} \right]^\frac{1}{2}  \\
&\quad + C \left[ \mathbb{E} \int_0^t \lvert \phi_{\delta,n} \rvert^{r} \, \d s \right]^\frac{1}{2} t^\frac{r - 2}{4} + C  \left[ \mathbb{E} \int_0^t \lvert \varphi_{\delta,n} - \phi_{\delta,n} \rvert_{\Gamma}^{r} \, \d s \right]^\frac{1}{2} t^\frac{r - 2}{4}
\\
&\leq \frac{M_0^\frac{r}{2}}{4} \mathbb{E} \left(\int_0^t |\nabla\mu_{\delta,n}|^2 \, \d s \right)^\frac{r}{2} + C + t^\frac{r - 2}{2} \mathbb{E}  \left[\int_0^t (\frac{1}{3} \lvert \phi_{\delta,n} \rvert^{r} + \frac{1}{2} \lvert \varphi_{\delta,n} - \phi_{\delta,n} \rvert_{\Gamma}^{r}) \, \d s \right]. 
\end{aligned}
\end{equation}
Analogously, we prove that there exists $C>0$ depending on $\Gamma,\,K,\,\eps,\,\eps_\Gamma,\,\delta,\,\tilde{c}_G,\,\mathcal{O},\,N_0$, and $r$ such that for every $n \in \mathbb{N}$,
\dela{
\noindent
Let us move to the second stochastically forced term in \eqref{eq4.97}. On the one hand, using the the Burkholder-Davis-Gundy inequality in conjunction with \eqref{eqt4.87}, we have
\begin{align*}
 & c_r \mathbb{E} \sup_{\tau \in [0,t]} \left|\int_0^{\tau} (\theta_{\delta,n}(s),F_2(\varphi_{\delta,n}(s))\, \d W_\Gamma(s))_\Gamma \right|^\frac{r}{2} \\
 &\leq c_r \mathbb{E} \left(\int_0^t |\theta_{\delta,n}(s)|_{\Gamma}^2 \|F_2(\varphi_{\delta,n}(s))\|_{\mathscr{T}_2(U_\Gamma,L^2(\Gamma))}^2 \, \d s \right)^\frac{r}{4} \\
 &\leq c_r (|\Gamma| C_2)^\frac{r}{4} \mathbb{E} \left(\int_0^t |\theta_{\delta,n}(s)|_{\Gamma}^2 \, \d s \right)^\frac{r}{4}.
\end{align*}
On the other hand, we note that (see \eqref{eqt4.92})
   \begin{equation*}
     |\theta_{\delta,n}(s)|_\Gamma^2
     \leq c_{2,r} [|\nabla_\Gamma \theta_{\delta,n}(s)|_{\Gamma}^2 + |\varphi_{\delta,n}(s) - \phi_{\delta,n}(s)|_{\Gamma}^2  + \|\phi_{\delta,n}(s)\|_{V_1}^2],
  \end{equation*}
where the constant $c_{2,r}$ may depend on $\Gamma,\,K,\,\eps,\,\eps_\Gamma,\,\delta,\,\tilde{c}_G,\,\mathcal{O}$ and $r$. Then, arguing as in \eqref{eq4.101}, we get 
}
\begin{equation}
\begin{aligned}
c_r \mathbb{E} \sup_{\tau \in [0,t]} \left \lvert \int_0^{\tau} (\theta_{\delta,n},F_2(\varphi_{\delta,n})\, \d W_\Gamma) \right\rvert^\frac{r}{2}  
&\leq \frac{N_0^\frac{r}{2}}{2} \mathbb{E} \left(\int_0^t \lvert \nabla_\Gamma \theta_{\delta,n} \rvert^2 \, \d s \right)^\frac{r}{2} + C \\
& + t^\frac{r - 2}{2} \mathbb{E} \left[ \int_0^t \left(\frac13 \Vert \phi_{\delta,n} \Vert_{V_1}^{r} +  \frac12 |\varphi_{\delta,n} - \phi_{\delta,n} \rvert_{\Gamma}^{r} \right)\d s \right]. 
\end{aligned}
\end{equation}
For the last stochastic integral in \eqref{eq4.97}, we see that
\begin{align*}
  &\mathbb{E} \sup_{\tau \in [0,t]} \left \lvert \red{\int_0^{\tau}(\phi_{\delta,n}, F_{1}(\phi_{\delta,n})\,\d W)}\right\rvert^\frac{r}{2}
  \leq C \mathbb{E} \left(\int_0^t \lvert \phi_{\delta,n} \rvert^2 \Vert F_1(\phi_{\delta,n})\Vert_{\mathscr{T}_2(U,L^2(\mathcal{O}))}^2 \, \d s \right)^\frac{r}{4} \\
  &\leq C \mathbb{E} \left(\int_0^t \lvert \phi_{\delta,n} \rvert^2 \, \d s \right)^\frac{r}{4} 
  \leq C \left[ \mathbb{E} \int_0^t \lvert \phi_{\delta,n} \rvert^{r} \, \d s \right]^\frac{1}{2} t^\frac{r - 2}{4} 
  \leq C(r,\mathcal{O},C_1) + t^\frac{r - 2}{2} \, \mathbb{E} \int_0^t \frac{1}{3} |\phi_{\delta,n}|^{r} \, \d s.
\end{align*}
}
}
Collecting all the previous estimates, inserting all of them into the RHS of \eqref{eq4.97}, we infer after straightforward transformation, that there exists $C$ depending on $N_0 ,M_0,\,\bar{M}_0,\eps,\,\eps_\Gamma,\,\mathcal{O},\,\Gamma,\,C_1$, $\nu_0,\lambda_0,C_2,\,\tilde{c}_F,\,\tilde{c}_G,\,K,\,\delta$, and $r$ such that for any $n \in \mathbb{N}$ and any $t \in [0,T]$,
\dela{
\begin{equation}
\begin{aligned}
& \mathbb{E} \sup_{s \in[0,t]} [\mathcal{E}_{tot}(\phi_{\delta,n}(s), \varphi_{\delta,n}(s))]^\frac{r}{2} + \nu_0^\frac{r}{2} \mathbb{E} \left \lvert \int_{Q_t} 2 \lvert D\bu_{\delta,n} \rvert^2 \,\d x \,\d s \right \rvert^\frac{r}{2} + \mathbb{E} \left(\int_{Q_t} \lambda(\phi_{\delta,n}) \lvert \bu_{\delta,n} \rvert^2 \,\d x \,\d s \right)^\frac{r}{2} \\
&\quad  +  \mathbb{E} \left(\int_0^{t} \lvert \bu_{\delta,n} \rvert_{\Gamma}^2 \,\d s\right)^\frac{r}{2} + \frac{M_0^\frac{r}{2}}{2} \mathbb{E} \left(\int_0^{t} \lvert \nabla \mu_{\delta,n} \rvert^2 \,\d s \right)^\frac{r}{2} + \frac{N_0^\frac{r}{2}}{2} \mathbb{E} \left(\int_0^{t} \lvert \nabla_\Gamma \theta_{\delta,n} \rvert^2 \,\d s\right)^\frac{r}{2}  
 \\
&\leq c_{3,r} (1 + \|(\phi_0,\varphi_0)\|_{\mathbb{V}}^r + t^{r/2}) +  c_{4,r} \, t^\frac{r - 2}{2} \, \mathbb{E} \int_0^t \sup_{0\leq s \leq \tau} [\mathcal{E}_{tot}(\phi_{\delta,n}(s), \varphi_{\delta,n}(s))]^\frac{r}{2} \, \d \tau,
\end{aligned}
\end{equation}
}
\begin{equation}\label{eq4.104}
\begin{aligned}
& \mathbb{E} \sup_{s \in[0,t]} [\mathcal{E}_{tot}(\phi_{\delta,n}(s), \varphi_{\delta,n}(s))]^\frac{r}{2} +  \mathbb{E} \left \lvert \int_{Q_t} \lvert D\bu_{\delta,n} \rvert^2 \,\d x \,\d s \right \rvert^\frac{r}{2} + \mathbb{E} \left(\int_{Q_t} \lambda(\phi_{\delta,n}) \lvert \bu_{\delta,n} \rvert^2 \,\d x \,\d s \right)^\frac{r}{2} \\
&\quad  + \mathbb{E} \left(\int_0^{t} \lvert \bu_{\delta,n} \rvert_{\Gamma}^2 \,\d s\right)^\frac{r}{2} + \mathbb{E} \left(\int_0^{t} \lvert \nabla \mu_{\delta,n} \rvert^2 \,\d s \right)^\frac{r}{2} + \mathbb{E} \left(\int_0^{t} \lvert \nabla_\Gamma \theta_{\delta,n} \rvert^2 \,\d s\right)^\frac{r}{2}  
 \\
&\leq C (1 + \|(\phi_0,\varphi_0)\|_{\mathbb{V}}^r + t^{r/2}) +  C \, t^\frac{r - 2}{2} \, \mathbb{E} \int_0^t \sup_{0\leq s \leq \tau} [\mathcal{E}_{tot}(\phi_{\delta,n}(s), \varphi_{\delta,n}(s))]^\frac{r}{2} \, \d \tau.
\end{aligned}
\end{equation}
\dela{
so that the Gronwall lemma yields
\begin{equation}\label{eq4.105}
\mathbb{E} \int_0^t \sup_{0\leq s \leq \tau} [\mathcal{E}_{tot}(\phi_{\delta,n}(s), \varphi_{\delta,n}(s))]^\frac{r}{2} \, \d \tau
\leq c_{3,r} \int_0^t (1 + \|(\phi_0,\varphi_0)\|_{\mathbb{V}}^r + \tau^{r/2}) e^{c_{4,r} \int_\tau^t \, s^\frac{r - 2}{2}\, \d s} \d \tau
\end{equation}
for all $t\in[0,T]$, where $c_{3,r}>0$ may depend on various parameters, including $\eps,\,\eps_\Gamma,\,\mathcal{O},\,\Gamma,\,C_1,\,C_2,\,\tilde{c}_F,\,\tilde{c}_G,\,K,\,\delta,r$, and $c_{4,r}>0$ may depend on $\eps,\,\eps_\Gamma,\,K,\,C_1,\,C_2,\,M_0,\,\bar{M}_0$ and $r$. 
The estimate \eqref{eq4.95} then follows from \eqref{eq4.104} and \eqref{eq4.105}.
}
Hence, by applying the Gronwall, we deduce \eqref{eq4.95}. Subsequently, considering \eqref{eq4.99} and \eqref{eq4.95}, we infer \eqref{eq4.101a}.
\end{proof}
}
\begin{lemma}\label{Lema-Aldous-condition}
Let $\delta \in (0,1)$. For every $\epsilon>0$ and every $\eta>0$, there exists $\kappa>0$ such that for every sequence $(\tau_n)_{n \in \mathbb{N}}$ of $\mathbb{F}$-stopping times with $\tau_n \leq T$,
   \begin{equation}\label{Eqn-Aldous-condition-with-delta}
     \sup_{n \in \mathbb{N}} \sup_{0\leq \theta \leq \kappa} \mathbb{P} \left(\{\Vert \bX_{\delta,n}(\tau_n + \theta) - \bX_{\delta,n}(\tau_n) \Vert_{\mathbb{V}^\prime} \geq \eta \} \right)\leq \epsilon.
   \end{equation}
\end{lemma}
\begin{proof}
Let $\delta \in (0,1)$ and $\theta>0$. Let $(\tau_n)_{n \in \mathbb{N}}$ be a sequence of $\mathbb{F}$-stopping times such that $0\leq \tau_n \leq T$. Let $\eta,\, \epsilon>0$ be fixed and $\kappa>0$ to be chosen later. Then, from \eqref{compact-stochastic-problem}, we infer
\begin{align*}
& \mathbb{E} \Vert \bX_{\delta,n}(\tau_n + \theta) - \bX_{\delta,n}(\tau_n) \Vert_{\mathbb{V}^\prime}
\leq \mathbb{E} \int_{\tau_n}^{\tau_n + \theta} \Vert \mathbf{b}_n(\bX_{\delta,n}) \Vert_{\mathbb{V}^\prime}\, \d s + \mathbb{E}\left \Vert \int_{\tau_n}^{\tau_n + \theta} \boldsymbol{\sigma}_n(\bX_{\delta,n})\,\d \mathcal{W} \right \Vert_{\mathbb{V}^\prime}
\\
& \leq \mathbb{E} \int_{\tau_n}^{\tau_n + \theta} \Vert B_{1,n}(\bu_{\delta,n},\phi_{\delta,n}) \Vert_{V_1^\prime}\d s  + \mathbb{E} \int_{\tau_n}^{\tau_n + \theta} \Vert \mathcal{S}_n A_{\phi_{\delta,n}}(\mu_{\delta,n}) \Vert_{V_1^\prime} \d s + \mathbb{E} \int_{\tau_n}^{\tau_n + \theta} \Vert \mathcal{S}_{n,\Gamma} \mathcal{A}_{\varphi_{\delta,n}}(\theta_{\delta,n}) \Vert_{V_\Gamma^\prime} \d s
    \\
& + \mathbb{E} \int_{\tau_n}^{\tau_n + \theta} \Vert \tilde{B}_n(\bu_{\delta,n},\varphi_{\delta,n}) \Vert_{V_\Gamma^\prime}\,\d s
+ \mathbb{E} \left \Vert \int_{\tau_n}^{\tau_n + \theta} F_{1,n}(\phi_{\delta,n}) \,\d W \right \Vert_{V_1^\prime} + \left \Vert \int_{\tau_n}^{\tau_n + \theta} F_{2,n}(\varphi_{\delta,n})\, \d W_\Gamma \right \Vert_{V_\Gamma^\prime}.
\end{align*}
Next, by the H\"older inequality, \eqref{B1-Property}, \eqref{eq4.80}, and \eqref{eq4.95}, we deduce that for every $n \in \mathbb{N}$,
\begin{align*}
&\mathbb{E} \int_{\tau_n}^{\tau_n + \theta} \Vert B_{1,n}(\bu_{\delta,n},\phi_{\delta,n}) \Vert_{V_1^\prime}\,\d s
\leq \mathbb{E} \left[\theta^{\frac12} \left(\int_0^T \Vert B_{1,n}(\bu_{\delta,n},\phi_{\delta,n}) \Vert_{V_1^\prime}^2\d s \right)^{\frac12} \right] 
\\
&\leq \left(\mathbb{E} \int_0^T \Vert B_{1,n}(\bu_{\delta,n},\phi_{\delta,n}) \Vert_{V_1^\prime}^2\d s \right)^{\frac12} \cdot \theta^{\frac12}
\leq C \left[\mathbb{E} \left(\sup_{s \in[0,T]} \Vert \phi_{\delta,n}(s) \Vert_{V_1}^{2} \int_0^T \Vert \bu_{\delta,n} \Vert_{V_\sigma}^{2}\, \d s\right) \right]^{\frac12} \cdot \theta^{\frac12}
\\
&\leq C \left[\mathbb{E} \sup_{s \in[0,T]} \Vert \phi_{\delta,n}(s) \Vert_{V_1}^4 \right]^{\frac14}
\left[ \mathbb{E} \left(\int_0^T \Vert \bu_{\delta,n} \Vert_{V_\sigma}^{2}\, \d s\right)^2 \right]^{\frac14} \cdot \theta^{\frac12}
\leq C \theta^{\frac12}.
\end{align*}
By \eqref{Definition of A_phi and A_varphi}, the assumption \ref{item:H4}, \eqref{eq4.80}, and \eqref{eq4.95}, we infer that for every $n \in \mathbb{N}$,
\begin{align*}
&\mathbb{E} \int_{\tau_n}^{\tau_n + \theta} \Vert \mathcal{S}_n A_{\phi_{\delta,n}}(\mu_{\delta,n}) \Vert_{V_1^\prime}\,\d s
\leq \mathbb{E} \left[\left(\int_0^T \Vert \mathcal{S}_n A_{\phi_{\delta,n}}(\mu_{\delta,n}) \Vert_{V_1^\prime}^2\,\d s \right)^{\frac12} \cdot \theta^{\frac12} \right] 
\\
&\leq \left(\mathbb{E} \int_0^T \Vert \mathcal{S}_n A_{\phi_{\delta,n}}(\mu_{\delta,n}) \Vert_{V_1^\prime}^2\,\d s \right)^{\frac12} \cdot \theta^{\frac12} 
\leq C \left(\mathbb{E} \int_0^T \lvert \nabla \mu_{\delta,n}(s) \rvert^2\, \d s \right)^{\frac12} \cdot \theta^{\frac12}
\leq C \theta^{\frac12}.
\end{align*}
Analogously, by \eqref{Definition of A_phi and A_varphi} and \ref{item:H4} in conjunction with \eqref{eq4.80} and \eqref{eq4.95}, we deduce that for every $n \in \mathbb{N}$,
\begin{align*}
&\mathbb{E} \int_{\tau_n}^{\tau_n + \theta} \Vert \mathcal{S}_{n,\Gamma} \mathcal{A}_{\varphi_{\delta,n}}(\theta_{\delta,n}(s)) \Vert_{V_\Gamma^\prime}\,\d s
\leq C \theta^{\frac12}.
\end{align*}
Next, owing to \eqref{Property-tilde B}, \eqref{eq4.95}, and \eqref{eq4.101a}, we infer that for every $n \in \mathbb{N}$,
\begin{align*}
&\mathbb{E} \int_{\tau_n}^{\tau_n + \theta} \Vert \tilde{B}_n(\bu_{\delta,n},\varphi_{\delta,n}) \Vert_{V_\Gamma^\prime}\,\d s
\leq \mathbb{E} \left[\theta^{\frac12} \left(\int_0^T \Vert \tilde{B}_n(\bu_{\delta,n},\varphi_{\delta,n}) \Vert_{V_\Gamma^\prime}^2\d s \right)^{\frac12} \right] 
\\
&\leq \left(\mathbb{E} \int_0^T \Vert \tilde{B}_n(\bu_{\delta,n},\varphi_{\delta,n}) \Vert_{V_\Gamma^\prime}^2\d s \right)^{\frac12} \cdot \theta^{\frac12}
\leq C \left[\mathbb{E} \left(\sup_{s \in[0,T]} \Vert \varphi_{\delta,n}(s)\Vert_{V_\Gamma}^{2} \int_0^T \Vert \bu_{\delta,n} \Vert_{V_\sigma}^{2}\, \d s\right) \right]^{\frac12} \cdot \theta^{\frac12}
\\
&\leq C \left[\mathbb{E} \sup_{s \in[0,T]} \Vert \varphi_{\delta,n}(s)\Vert_{V_\Gamma}^4 \right]^{\frac14}
\left[ \mathbb{E} \left(\int_0^T \Vert \bu_{\delta,n} \Vert_{V_\sigma}^{2}\, \d s\right)^2 \right]^{\frac14} \cdot \theta^{\frac12}
\leq C \theta^{\frac12}.
\end{align*}
Using the It\^o isometry, \eqref{eq4.83} and \eqref{eqt4.87}, we deduce that for every $n \in \mathbb{N}$,

\begin{align*}
&\mathbb{E} \left \Vert \int_{\tau_n}^{\tau_n + \theta} F_{1,n}(\phi_{\delta,n}(s)) \,\d W(s) \right \Vert_{V_1^\prime} + \mathbb{E} \left \Vert \int_{\tau_n}^{\tau_n + \theta} F_{2,n}(\varphi_{\delta,n}(s))\, \d W_\Gamma(s) \right \Vert_{V_\Gamma^\prime}
\\
& \leq \left[\mathbb{E} \left \Vert \int_{\tau_n}^{\tau_n + \theta} F_{1,n}(\phi_{\delta,n}(s)) \,\d W(s) \right \Vert_{V_1^\prime}^2 \right]^{\frac12} + \left[\mathbb{E} \left \Vert \int_{\tau_n}^{\tau_n + \theta} F_{2,n}(\varphi_{\delta,n}(s))\, \d W_\Gamma(s) \right \Vert_{V_\Gamma^\prime}^2 \right]^{\frac12}
\\
&\leq \left[\mathbb{E} \int_{\tau_n}^{\tau_n + \theta} \Vert F_{1,n}(\phi_{\delta,n}(s)) \Vert_{\mathscr{T}_2(U,V_1^\prime)}^2\, \d s \right]^{\frac12} + \left[\mathbb{E} \int_{\tau_n}^{\tau_n + \theta} \Vert F_{2,n}(\varphi_{\delta,n}(s)) \Vert_{\mathscr{T}_2(U,V_\Gamma^\prime)}^2\, \d s \right]^{\frac12}
\\
& \leq \left[\mathbb{E} \int_{\tau_n}^{\tau_n + \theta} \Vert F_{1,n}(\phi_{\delta,n}(s)) \Vert_{\mathscr{T}_2(U,L^2(\mathcal{O}))}^2\, \d s \right]^{\frac12} + \left[\mathbb{E} \int_{\tau_n}^{\tau_n + \theta} \Vert F_{2,n}(\varphi_{\delta,n}(s)) \Vert_{\mathscr{T}_2(U,L^2(\Gamma))}^2\, \d s \right]^{\frac12}
\\
&\leq ([\lvert \mathcal{O} \rvert C_1]^{1/2} + [\lvert \Gamma \rvert C_2]^{1/2}) \theta^{1/2}.
\end{align*}
It then follows from the previous inequalities that there exists a constant $c_1>0$ independent of $n$ such that
   \begin{align*}
     \mathbb{E} \Vert \bX_{\delta,n}(\tau_n + \theta) - \bX_{\delta,n}(\tau_n) \Vert_{\mathbb{V}^\prime} \leq c_1 \theta^{1/2}.  
   \end{align*}
Taking $\kappa= \left(\frac{\eta \epsilon}{c_1} \right)^2$, then, by the Chebychev inequality we deduce \eqref{Eqn-Aldous-condition-with-delta}. 
\end{proof}
Here, given a Banach space $X$, the notation $X_w$ indicates that the space $X$ is endowed with the weak topology.
We now introduce the following topological space 
   \begin{equation}\label{Eqn-mathfrack{X}}
      \mathfrak{X}_T \coloneq \mathfrak{X}_{T,2} \times \mathfrak{X}_{T,1} \times \mathfrak{X}_{T,\Gamma},
    \end{equation}
where
   \begin{align*}
     \mathfrak{X}_{T,2}= C([0,T];\mathbb{H}) \cap C([0,T];\mathbb{V}_w) \cap L_w^4(0,T;\mathcal{H}^2) \cap L^2(0,T;\mathbb{V}),
     \\
     \mathfrak{X}_{T,1}= C([0,T];U_0) \mbox{ and } \mathfrak{X}_{T,\Gamma}= C([0,T];U_0^\Gamma).
   \end{align*}
We recall that we can find a set $\Omega_0 \in \mathcal{F}$ of measure zero such that $W(t,\cdot,\omega) \in C([0,T];U_0)$ and $W_\Gamma(t,\cdot,\omega) \in C([0,T];U_0^\Gamma)$ for any $\omega \in \Omega \backslash \Omega_0$. In the sequel, we use the notations $\mathcal{L}(W)$ and $\mathcal{L}(W_\Gamma)$ to denote the laws of $W=(W(t))_{t \in [0,T]}$ and $W_\Gamma=(W_\Gamma(t))_{t \in [0,T]}$ on $C([0,T];U_0)$ and on $C([0,T];U_0^\Gamma)$, respectively.
For every $n \in \mathbb{N}$, we consider the constant sequences of cylindrical Wiener processes 
$$
W_{1,n} \equiv W \quad \text{and} \quad W_{2,n} \equiv W_\Gamma,
$$
and we construct a family of probability laws on $\mathfrak{X}_{T,1}= C([0,T];U_0)$  and  on $\mathfrak{X}_{T,\Gamma}= C([0,T];U_0^\Gamma)$ by setting
\begin{align*}
\mathbb{P}_n(\cdot)&\coloneq \mathbb{P}(W_{1,n} \in \cdot) \in P_r(\mathfrak{X}_{T,1}), \\
\mathbb{P}_{\Gamma,n}(\cdot)&\coloneq \mathbb{P}(W_{2,n} \in \cdot) \in P_r(\mathfrak{X}_{T,\Gamma}),
\end{align*}
where $P_r(\mathfrak{X}_{T,1})$ and $P_r(\mathfrak{X}_{T,\Gamma})$ denote the set of all probability measures on the measurable spaces $(\mathfrak{X}_{T,1}, \mathcal{B}(\mathfrak{X}_{T,1}))$ and $(\mathfrak{X}_{T,\Gamma}, \mathcal{B}(\mathfrak{X}_{T,\Gamma}))$, respectively. 
\newline
We have the following important result.
\begin{proposition}\label{Eqn-tightness-delta}
Assume $\delta \in (0,1)$. Then, the sequence of laws of $((\phi_{\delta,n},\varphi_{\delta,n}),W_{1,n},W_{2,n})_{n \in \mathbb{N}}$ is tight on $\mathfrak{X}_T$.
\end{proposition}
\begin{proof}
First, the family of probability laws $\lbrace \mathbb{P}_n \rbrace_n$, $\lbrace \mathbb{P}_{\Gamma,n} \rbrace_n$ is tight in $P_r(\mathfrak{X}_{T,1})$, and in $P_r(\mathfrak{X}_{T,\Gamma})$, respectively.  Indeed, if we endow the space $\mathfrak{X}_{T,1}$ with the uniform convergence norm, it becomes a Polish space. By \cite[Theorem 6.8]{Billingsley_1999}, $P_r(\mathfrak{X}_{T,1})$ endowed with the Prohorov metric is a separable and complete metric space. Moreover, by construction, the family of probability laws $\lbrace \mathbb{P}_n \rbrace_n$ is reduced to one element which is the law $\mathcal{L}(W)$ of $W$ and belongs to $P_r(\mathfrak{X}_{T,1})$. Therefore, by \cite[Theorem 3.2, Chapter II]{Parthasarathy_1967} we deduce that the family $(\mathbb{P}_n)_n$ is tight on $P_r(\mathfrak{X}_{T,1})$. Similarly, we can prove that the family $(\mathbb{P}_{\Gamma,n})_n$ is tight on $P_r(\mathfrak{X}_{T,\Gamma})$.
\newline
Next, thanks to the propositions \ref{Proposition_first_priori_estimmate} and \ref{Proposition_second_priori_estimmate}, the lemmas \ref{H2_H3_priori_estimmates} and \ref{Lema-Aldous-condition}, we can apply \cite[Corollary 3.9]{Brzezniak+Moty_2013} and deduce that the sequence of laws of $\bX_{\delta,n}\coloneq (\phi_{\delta,n}, \varphi_{\delta,n})$ is tight. This completes the proof of Proposition \ref{Eqn-tightness-delta}.
\end{proof}

\dela{
Next, arguing as in \eqref{Eqn-4.57}, since $\mathcal{S}_n$ are uniformly bounded in the space $\mathscr{L}(V_1)$, see \cite{Temam_2001}, $H^2(\mathcal{O}) \embed H^{3/2}(\Gamma)$, and $H^2(\Gamma) \embed H^{3/2}(\Gamma)$, we infer that for every $n \in \mathbb{N}$,
\begin{equation}\label{Eqn-H3-phi-estimate}
\begin{aligned}
&\Vert \phi_{\delta,n} \Vert_{H^3(\mathcal{O})}^2
\leq  C \Vert \mu_{\delta,n} - \eps^{-1} \mathcal{S}_n  F_\delta^\prime(\phi_{\delta,n})\Vert_{V_1}^2 + C \Vert \phi_{\delta,n} \Vert_{H^2(\mathcal{O})}^2 + C \Vert \varphi_{\delta,n} - \phi_{\delta,n} \Vert_{H^{3/2}(\Gamma)}^2
\\
&\leq  C(\Vert \mu_{\delta,n} \Vert_{V_1}^2 + \Vert \mathcal{S}_n \Vert_{\mathcal{L}(V_1)}^2 \Vert F_\delta^\prime(\phi_{\delta,n}) \Vert_{V_1}^2 +  \Vert \phi_{\delta,n} \Vert_{H^2(\mathcal{O})}^2 + \Vert \varphi_{\delta,n} - \phi_{\delta,n} \Vert_{H^{3/2}(\Gamma)}^2)
\\
&\leq  C(K,\eps,\mathcal{O},\Gamma,\tilde{c}_F) (\Vert \mu_{\delta,n} \Vert_{V_1}^2 + \Vert(\phi_{\delta,n},\varphi_{\delta,n})\Vert_{\mathcal{H}^2}^2).
\end{aligned}
\end{equation}
Here we have also used the fact that $\lvert F_\delta^\prime(\phi_{\delta,n}) \rvert \leq \tilde{c}_{1,\delta} \lvert \phi_{\delta,n} \rvert$ and 
 $\lvert F^{\bis}_\delta \rvert \leq (\delta^{-1} + \tilde{c}_F)$. \newline
 Arguing as in \eqref{Eqn-4.59}, using the fact that $H^2(\mathcal{O}) \embed H^{\frac32}(\mathcal{O}) \embed H^1(\Gamma)$, and that the projections $\mathcal{S}_{n,\Gamma}$  are uniformly bounded in the space $\mathscr{L}(V_\Gamma)$, we infer that for every $n \in \mathbb{N}$,
\begin{equation}
\begin{aligned}
&\Vert \varphi_{\delta,n} \Vert_{H^3(\Gamma)}
\leq C (\Vert \theta_{\delta,n} - \eps_\Gamma^{-1} \mathcal{S}_{n,\Gamma} G_\delta^\prime(\varphi_{\delta,n}) - \eps(\varphi_{\delta,n} - \phi_{\delta,n}) \Vert_{V_\Gamma} + \Vert \varphi_{\delta,n} \Vert_{H^2(\Gamma)})
\\
&\leq C (\Vert \theta_{\delta,n} \Vert_{V_\Gamma} + \Vert \mathcal{S}_{n,\Gamma} G_\delta^\prime(\varphi_{\delta,n}) \Vert_{V_\Gamma} + \Vert \varphi_{\delta,n} \Vert_{V_\Gamma} + \Vert \phi_{\delta,n} \Vert_{V_\Gamma} + \Vert \varphi_{\delta,n} \Vert_{H^2(\Gamma)}) \\
&\leq C(\Vert \theta_{\delta,n} \Vert_{V_\Gamma} + \Vert \mathcal{S}_{n,\Gamma} \Vert_{\mathcal{L}(V_\Gamma)} \Vert G_\delta^\prime(\varphi_{\delta,n}) \Vert_{V_\Gamma} + \Vert \varphi_{\delta,n} \Vert_{H^2(\Gamma)} +  \Vert \phi_{\delta,n} \Vert_{H^2(\mathcal{O})}) \\
&\leq C(\Vert \theta_{\delta,n} \Vert_{V_\Gamma} + \Vert G_\delta^\prime(\varphi_{\delta,n}) \Vert_{V_\Gamma} +  \Vert(\phi_{\delta,n},\varphi_{\delta,n})\Vert_{\mathcal{H}^2}).
\end{aligned}
\end{equation}
Moreover, since the potential $G_\delta$ enjoy the following properties
    \begin{align*}
      \lvert G_\delta^\prime(\cdot) \rvert \leq \tilde{c}_{2,\delta} \lvert \cdot \rvert \mbox{ and } \lvert G^{\bis}_\delta(\cdot) \rvert \leq (\delta^{-1} + \tilde{c}_G),
   \end{align*}
we deduce that there exists $C=(\eps,\eps_\Gamma,\mathcal{O},\Gamma,K,\tilde{c}_G, \delta)$, $C$ being independent of $n$ such that
    \begin{equation}\label{Eqn-H3-varphi_Gamma-estimate}
      \Vert \varphi_{\delta,n} \Vert_{H^3(\Gamma)}^2
      \leq C(\Vert \theta_{\delta,n} \Vert_{V_\Gamma}^2  +  \Vert(\phi_{\delta,n},\varphi_{\delta,n})\Vert_{\mathcal{H}^2}^2).
    \end{equation}
Obviously, by \eqref{eq4.95}, \eqref{eq4.108}, \eqref{Eqn-H3-phi-estimate} and \eqref{Eqn-H3-varphi_Gamma-estimate}, we easily complete the proof of the second part of Lemma \ref{H2_H3_priori_estimmates}.
}
\subsection{Convergence to a weak solution}\label{subs-4.4}
In this section, we prove that there exists a martingale solution to the approximate system \eqref{eq1.1aa}-\eqref{eq1.10aa}, with $\delta$ fixed. 
For a $\mathfrak{X}_T$-valued r.v. $\bX$, we denote by $\mathcal{L}_{\mathfrak{X}_T}(\bX)$ its law in $\mathfrak{X}_T$. We state the following auxiliary result.
\begin{proposition}\label{Jakubowski-Skorokhod-representation theorem}
Assume $\delta \in (0,1)$. Then, there exists a subsequence $(n_k)_{k \in \mathbb{N}}$, a complete probability space $(\bar{\Omega},\bar{\mathcal{F}},\bar{\mathbb{P}})$, and on this space, $\mathfrak{X}_T$-valued random variables $((\bar{\phi}_{\delta,k},\bar{\varphi}_{\delta,k}),\bar{W}_{1,k},\bar{W}_{2,k})$ and $((\bar{\phi}_{\delta},\bar{\varphi}_{\delta}),\bar{W},\bar{W}_{\Gamma})$, $k \in \mathbb{N}$ such that
\begin{align}
\label{Eqn-equal-of-law-of-sequence}
&\mathcal{L}_{\mathfrak{X}_T}(((\bar{\phi}_{\delta,k},\bar{\varphi}_{\delta,k}),\bar{W}_{1,k},\bar{W}_{2,k}))= \mathcal{L}_{\mathfrak{X}_T}(((\phi_{\delta,n_k},\varphi_{\delta,n_k}),W_{1,n_k},W_{2,n_k})),
\\
\label{Eqn-almots-sure-convergenc-of-sequence}
&((\bar{\phi}_{\delta,k},\bar{\varphi}_{\delta,k}),\bar{W}_{1,k},\bar{W}_{2,k}) \to ((\bar{\phi}_{\delta},\bar{\varphi}_{\delta}),\bar{W},\bar{W}_{\Gamma}) \mbox{ in } \mathfrak{X}_T \mbox{ as } k \to \infty, \; \; \bar{\mathbb{P}}\mbox{-a.s.}
\end{align}
\end{proposition}
Before we embark on proving the above result, let us emphasize that \eqref{Eqn-almots-sure-convergenc-of-sequence} yields that 
   \begin{equation}\label{Eqn-convergence-of-Wiener-processes}
    (\bar{W}_{1,k},\bar{W}_{2,k}) \to (\bar{W},\bar{W}_{\Gamma})  \mbox{ in }  C([0,T];U_0 \times U_0^\Gamma) \mbox{ as } k \to \infty, \; \; \bar{\mathbb{P}}\mbox{-a.s.}.             
   \end{equation}
We will continue to denote the above sequences by $((\bar{\phi}_{\delta,n},\bar{\varphi}_{\delta,n}),\bar{W}_{1,n},\bar{W}_{2,n})_{n \in \mathbb{N}}$ and $((\phi_{\delta,n},\varphi_{\delta,n_k}),W_{1,n},W_{2,n})_{n \in \mathbb{N}}$, for the sake of brevity.
\begin{proof}
First, arguing as in \cite{Brzezniak+Moty_2013}, Proof of Corollary 3.12, we can check that the space $\mathfrak{X}_{T,2}$ satisfies Assumption $(10)$ in Jakubowski's paper \cite{Jakubowski_1997}. We point out that $\mathfrak{X}_{T,2}$ is in fact the product of countable spaces satisfying the assumption $(10)$ in \cite{Jakubowski_1997}. Secondly, since $\mathfrak{X}_{T,1} \times \mathfrak{X}_{T,\Gamma}$ is a Polish space, its satisfies the assumption $(10)$ in \cite{Jakubowski_1997}. Therefore, we can apply \cite[Theorem 2]{Jakubowski_1997} and thus deduce all the conclusions in the proposition. The proof of Proposition \ref{Jakubowski-Skorokhod-representation theorem} is now  complete.
\end{proof}
\begin{remark}
$(a)$ Proposition \ref{Jakubowski-Skorokhod-representation theorem} yields that there exists a new probability space $(\bar{\Omega},\bar{\mathcal{F}},\bar{\mathbb{P}})$ and a sequence of random variables $X_n: (\bar{\Omega},\bar{\mathcal{F}}) \to (\Omega,\mathcal{F})$ such that the law of $X_n$ is $\mathbb{P}$ for every $n \in \mathbb{N}$, i.e. $\mathbb{P}= \bar{\mathbb{P}} \circ X_n^{-1}$, so that composition with $X_n$ preserves laws, and the convergences \eqref{Eqn-almots-sure-convergenc-of-sequence} and \eqref{Eqn-convergence-of-Wiener-processes} hold.
\newline
$(b)$ Notice that from the proof of Skorokhod-Jakubowski Theorem, see \cite{Jakubowski_1997}, possibly enlarging the new probability space, we may assume that the new probability space $(\bar{\Omega},\bar{\mathcal{F}},\bar{\mathbb{P}})$ is independent of $\delta$.
\end{remark}
Let us consider a sequence of processes $(\bar{\bu}_{\delta,n})_n\,,(\bar{\mu}_{\delta,n})_n$, and $(\bar{\theta}_{\delta,n})_n$ defined by
   \begin{equation}
     \bar{\bu}_{\delta,n} \coloneq \bu_{\delta,n} \circ X_n, \; \;  \bar{\mu}_{\delta,n} \coloneq \mu_{\delta,n} \circ X_n, \; \; \bar{\theta}_{\delta,n} \coloneq \theta_{\delta,n} \circ X_n,
   \end{equation}
which enjoy the same properties as for the sequences $(\bu_{\delta,n})_n\,,(\mu_{\delta,n})_n$,  and $(\theta_{\delta,n})_n$, since the map $X_n$ preserves laws. \newline
In the following, let us choose and fix $r \geq 2$. By \eqref{Eqn-equal-of-law-of-sequence} and \eqref{Eqn-almots-sure-convergenc-of-sequence} in Proposition \ref{Jakubowski-Skorokhod-representation theorem}, recalling the estimates in Propositions \ref{Proposition_second_priori_estimmate} and \ref{Proposition_third_priori_estimmate}, and in Lemma \ref{H2_H3_priori_estimmates}, using also the Banach-Alaoglu Theorem and the Vitali convergence theorem, we have up to a subsequence still labeled by the same subscript,
\begin{equation}\label{eq4.122a}
\begin{aligned}
(\bar{\phi}_{\delta,n},\bar{\varphi}_{\delta,n}) \to (\bar{\phi}_{\delta},\bar{\varphi}_{\delta}) &\mbox{ in } L^\ell(\bar{\Omega};C([0,T];\mathbb{H})) \cap L^r(\bar{\Omega};L^2([0,T];\mathbb{V})), \; \;  \forall \ell< r,
\\
(\bar{\phi}_{\delta,n},\bar{\varphi}_{\delta,n}) \stackrel{*}{\rightharpoonup} (\bar{\phi}_{\delta},\bar{\varphi}_{\delta}) &\mbox{ in }  L^r(\bar{\Omega};L^\infty(0,T;\mathbb{V})), 
  \\
(\bar{\phi}_{\delta,n},\bar{\varphi}_{\delta,n}) \rightharpoonup (\bar{\phi}_{\delta},\bar{\varphi}_{\delta}) &\mbox{ in }  L^r(\bar{\Omega};L^4(0,T;\mathcal{H}^2)) \dela{\cap  L^r(\bar{\Omega};L^2(0,T;\mathcal{H}^3))},
    \\
(\bar{\mu}_{\delta,n},\bar{\theta}_{\delta,n}) \rightharpoonup (\bar{\mu}_{\delta},\bar{\theta}_{\delta})  &\mbox{ in }  L^r(\bar{\Omega};L^2(0,T;\mathbb{V})) \cap L^r(\bar{\Omega};L^4(0,T;\mathbb{H})), 
      \\
(\bar{W}_{1,n},\bar{W}_{2,n}) \to (\bar{W},\bar{W}_{\Gamma}) &\mbox{ in } L^\ell(\bar{\Omega};C([0,T];U_0 \times U_0^\Gamma)), \; \;  \forall \ell< r,
\\
\bar{\bu}_{\delta,n} \rightharpoonup \bar{\bu}_{\delta}  &\mbox{ in }  L^r(\bar{\Omega};L^2(0,T;V)), 
  \\
\bar{\bu}_{\delta,n} \lvert_\Gamma \rightharpoonup \bar{\bu}_{\delta} \lvert_\Gamma  &\mbox{ in }  L^r(\bar{\Omega};L^2(0,T;L^2(\Gamma))).
\end{aligned}
\end{equation}
Moreover, from the strong convergences of $\bar{\phi}_{\delta,n}$ and $\bar{\varphi}_{\delta,n}$ above, we have, modulo the extraction of a subsequence (still) denoted $(\bar{\phi}_{\delta,n})_{n \in \mathbb{N}}$ and $(\bar{\varphi}_{\delta,n})_{n \in \mathbb{N}}$:
      \begin{equation}\label{eq4.123a}
        (\bar{\phi}_{\delta,n},\bar{\varphi}_{\delta,n}) \to (\bar{\phi}_{\delta},\bar{\varphi}_{\delta})\mbox{-a.e. in } Q_T \times \tilde{\Omega}.
      \end{equation}
We have the following important auxiliary result.
\begin{proposition}\label{Prop:Proposition-4}
Let $\mathfrak{X}_\delta=(\bar{\phi}_{\delta}, \bar{\varphi}_{\delta})$ be the stochastic process as given by \eqref{eq4.122a}. Then, there exists a subsequence of $\mathfrak{X}_{\delta,n}\coloneq (\bar{\phi}_{\delta,n},\bar{\varphi}_{\delta,n})$ and $\bar{\mathcal{W}}_n\coloneq(\bar{W}_{1,n},\bar{W}_{2,n})^{tr}$ still labeled by the
same subscript, such that   as $n\to \infty$,
\begin{equation}
 \int_0^t \bG^n(\mathfrak{X}_{\delta,n}(s)) \, \d \bar{\mathcal{W}}^{n}(s) \to \int_0^t \bG(\mathfrak{X}_\delta(s)) \, \d \bar{\mathcal{W}}(s)  \mbox{ in probability in } L^2(0,T;\mathbb{H}),
\end{equation}
for all $t \in[0,T]$, with $\bar{\mathcal{W}} \coloneq (\bar{W},\bar{W}_{\Gamma})^{tr}$.
\end{proposition}
We report the proof of Proposition \ref{Prop:Proposition-4} in appendix \ref{eqn-Prop:Proposition-4}.

\noindent
To reach the limit in the weak formulation of the finite-dimensional problem \eqref{eq4.34a}-\eqref{eq4.34d}, in addition to the previous convergences, we will need the following convergences. These follow from the strong convergences of $(\bar{\phi}_{\delta,n})_n$ and $(\bar{\varphi}_{\delta,n})_n$, see \eqref{eq4.122a}, and the fact that $F_\delta^\prime$ and $G_\delta^\prime$ are Lipschitz-continuous. Thus, as $n \to \infty$ along a non-relabeled subsequence, we get
   \begin{equation}
     \begin{aligned}
        F_\delta^\prime(\bar{\phi}_{\delta,n}) \to F_\delta^\prime(\bar{\phi}_{\delta}) &\mbox{ in }  L^r(\bar{\Omega};L^2(0,T;L^2(\mathcal{O}))), 
          \\
        G_\delta^\prime(\bar{\varphi}_{\delta,n}) \to G_\delta^\prime(\bar{\varphi}_{\delta}) &\mbox{ in }  L^r(\bar{\Omega};L^2(0,T;L^2(\Gamma))).
     \end{aligned}
  \end{equation}
Moreover, by the pointwise convergences \eqref{eq4.123a} and the continuity of the maps $M_{\mathcal{O}},\,M_\Gamma,\,\nu,\,\lambda$, and $\gamma$, we infer that, as $n \to \infty$, 
\begin{equation}
  \begin{aligned}
    \nu(\bar{\phi}_{\delta,n}) &\to \nu(\bar{\phi}_{\delta}), \quad \lambda(\bar{\phi}_{\delta,n}) \to \lambda(\bar{\phi}_{\delta}), \quad M_{\mathcal{O}}(\bar{\phi}_{\delta,n}) \to M_{\mathcal{O}}(\bar{\phi}_{\delta}), \quad \d \mathbb{\bar{P}}\mbox{-a.e. in} \quad Q_T\\
    M_\Gamma(\bar{\varphi}_{\delta,n}) &\to M_\Gamma(\bar{\varphi}_{\delta}), \quad \gamma(\bar{\varphi}_{\delta,n}) \to \gamma(\bar{\varphi}_{\delta}), \quad \d \mathbb{\bar{P}}\mbox{-a.e. on} \quad  \Sigma_T.
  \end{aligned}
\end{equation}
Now, from the convergence results established in this section and the DCT, we can pass to the limit in the weak formulation of \eqref{eq1.1aa}-\eqref{eq1.10aa},\dela{cf. \eqref{eq4.34a}-\eqref{eq4.34d},} to deduce that the process $(\bar{\bu}_{\delta},\bar{\phi}_{\delta},\bar{\varphi}_{\delta},\bar{\mu}_{\delta},\bar{\theta}_{\delta})$ satisfies $\bar{\mathbb{P}}$-a.s. and for all test functions
\begin{align*}
 \bv \in \bigcup_{n\in \mathbb{N}} V^n \subset V, \;\; \upsilon \in \bigcup_{n\in \mathbb{N}} H_{n}^1 \subset V_1, \;\; \upsilon \lvert_\Gamma \in \bigcup_{n\in \mathbb{N}} H_{\Gamma,n}^1 \subset V_\Gamma, \;\; (\psi,\psi \lvert_\Gamma) \bigcup_{n\in \mathbb{N}} \mathcal{V}_n \subset \mathbb{V},
\end{align*}
\begin{align}
\label{eq4.130a}
& 2 \int_{Q_t} \nu(\bar{\phi}_{\delta}) D\bar{\bu}_{\delta}: D\bv\, \d x\, \d s + \int_{Q_t} \lambda(\bar{\phi}_{\delta}) \bar{\bu}_{\delta} \cdot \bv \, \d x\, \d s + \int_{\Sigma_t}\gamma(\bar{\varphi}_{\delta}) \bar{\bu}_{\delta} \cdot \bv \, \d S \, \d s, \notag \\
&\quad 
=  - \int_{\Sigma_t} \bar{\varphi}_{\delta} \nabla_\Gamma \bar{\theta}_{\delta} \cdot \bv \, \d S \, \d s  - \int_{Q_t} \bar{\phi}_{\delta} \nabla \bar{\mu}_{\delta} \cdot \bv \, \d x\, \d s, \; \; t \in [0,T],  
\\
\label{eq4.130b}
&\duality{\bar{\phi}_{\delta}(t)}{\upsilon}{V_1}{V_1^\prime} + \int_{Q_t} [- \bar{\phi}_{\delta} \bar{\bu}_{\delta} \cdot \nabla \upsilon  +  M_\mathcal{O}(\bar{\phi}_{\delta}) \nabla \bar{\mu}_{\delta} \cdot \nabla \upsilon] \d x\, \d s
 = (\phi_{0}, \upsilon) + \left(\int_0^t F_{1}(\bar{\phi}_{\delta}) \,\d \bar{W}, \upsilon \right), 
   \\
\label{eq4.130c}
&\duality{\bar{\varphi}_{\delta}(t)}{\upsilon \lvert_\Gamma }{V_\Gamma}{V_\Gamma^\prime} + \int_{\Sigma_t} [-\bar{\varphi}_{\delta} \bar{\bu}_{\delta} \cdot \nabla_\Gamma \upsilon \lvert_\Gamma + M_\Gamma(\bar{\varphi}_{\delta}) \nabla_\Gamma \bar{\theta}_{\delta} \cdot \nabla_\Gamma \upsilon \lvert_\Gamma] \d S \, \d s  \notag \\
 &\quad= (\varphi_{0}, \upsilon \lvert_\Gamma)_\Gamma +  \left(\int_0^t F_2(\bar{\varphi}_{\delta}) \,\d \bar{W}_{\Gamma}, \upsilon \lvert_\Gamma \right)_\Gamma, 
 \\
\label{eq4.130d}
&\int_{\mathcal{O}} \bar{\mu}_{\delta} \psi \, \d x  + \int_{\Gamma} \bar{\theta}_{\delta} \psi \lvert_\Gamma \, \d S  
- \eps \int_{\mathcal{O}} \nabla \bar{\phi}_{\delta} \cdot \nabla \psi \, \d x  - \int_{\mathcal{O}} \frac1\eps F_\delta^\prime(\bar{\phi}_{\delta}) \psi \, \d x \notag \\ 
&\quad= \eps_\Gamma \int_{\Gamma} \nabla_\Gamma \bar{\varphi}_{\delta} \cdot \nabla_\Gamma \psi \lvert_\Gamma \, \d S +  \frac{1}{\eps_\Gamma} \int_{\Gamma} G_\delta^\prime(\bar{\varphi}_{\delta}) \psi \lvert_\Gamma \, \d S  
+ (1/K) \int_{\Gamma} (\bar{\varphi}_{\delta} - \bar{\phi}_{\delta}) (\eps \psi \lvert_\Gamma -\eps \psi)\, \d S. 
\end{align}
Note that \eqref{eq4.130a}-\eqref{eq4.130d} also hold for all $\bv \in V,\, \upsilon \in V_1,\, \upsilon \lvert_\Gamma \in V_\Gamma$, and $(\psi,\psi \lvert_\Gamma) \in \mathbb{V}$ by a density argument. 
\section{Uniform estimates with respect to delta} \label{sect_Uniform_estimates_delta}
In this section, we derive additional uniform estimates that are independent of $\delta$. 
\dela{
Consequently, the constant $C>0$, whose specific dependencies are explicitly noted when necessary, remains independent of $\delta$ and can vary from one line to another. 
}
\newline
We start with the following inequalities. First, we have
   \begin{align*}
     \left \lvert (M_{\mathcal{O}}(\bar{\phi}_{\delta,n}) \nabla \bar{\mu}_{\delta,n}, \nabla \bar{\phi}_{\delta,n}) \right \rvert  
      \leq \frac{1}{2} \int_{\mathcal{O}} M_{\mathcal{O}}(\bar{\phi}_{\delta,n}) \lvert \nabla \bar{\mu}_{\delta,n} \rvert^2\,\d x + \frac{\bar{M}_0}{2} \lvert \nabla \bar{\phi}_{\delta,n} \rvert^2.
   \end{align*}
This, in conjunction with \eqref{Main_Galerkin_Equality}, \eqref{Eqn-E_tot-zeta-estimate}, and the fact that $\mathbb{P}= \bar{\mathbb{P}} \circ X_n^{-1}$ for every $n$, that is, $X_n$, preserves laws, yields that for every $t \in[0,T]$,
\begin{equation}\label{eq4.131}
\begin{aligned}
&\bar{\mathbb{E}}[\mathcal{E}_{tot}(\bar{\phi}_{\delta,n}(t), \bar{\varphi}_{\delta,n}(t))] + \bar{\mathbb{E}} \int_{Q_t} [2 \nu(\bar{\phi}_{\delta,n}) \lvert D\bar{\bu}_{\delta,n} \rvert^2 + \lambda(\bar{\phi}_{\delta,n}) \lvert \bar{\bu}_{\delta,n} \rvert^2] \d x \,\d s \\
& + \mathbb{\bar{E}} \int_{\Sigma_t} \gamma(\bar{\varphi}_{\delta,n}) \lvert \bar{\bu}_{\delta,n} \rvert^2 \,\d S\,\d s + \frac{1}{2} \mathbb{\bar{E}} \int_{Q_t} M_{\mathcal{O}}(\bar{\phi}_{\delta,n}) \lvert \nabla \bar{\mu}_{\delta,n} \rvert^2 \d x \,\d s + \mathbb{\bar{E}} \int_{\Sigma_t} M_\Gamma(\bar{\varphi}_{\delta,n}) \lvert \nabla_\Gamma \bar{\theta}_{\delta,n} \rvert^2 \d S \,\d s  
 \\
&\leq \mathcal{E}_{tot}(\phi_n(0),\varphi_n(0)) + \frac{1}{2} \mathbb{\bar{E}} \int_0^{t} \Vert F_{1,n}(\bar{\phi}_{\delta,n}) \Vert_{\mathscr{T}_2(U,L^2(\mathcal{O}))}^2\,\d s  + \frac{\eps_\Gamma C_2}{2} \mathbb{\bar{E}} \int_0^{t} \lvert \nabla_\Gamma \bar{\varphi}_{\delta,n} \rvert_{\Gamma}^2 \d s \\
& + \left(\frac{\bar{M}_0}{2} + \frac{\eps C_1}{2} \right) \mathbb{\bar{E}} \int_0^t  \lvert \nabla \bar{\phi}_{\delta,n} \rvert^2 \d s + \eps (1/K) \mathbb{\bar{E}} \int_0^{t} (\Vert F_{2,n}(\bar{\varphi}_{\delta,n})\Vert_{\mathscr{T}_2(U_\Gamma,L^2(\Gamma))}^2 + \Vert F_{1,n}(\bar{\phi}_{\delta,n}) \Vert_{\mathscr{T}_2(U,L^2(\Gamma))}^2) \d s  \\
& + \frac{1}{2 \eps} \mathbb{\bar{E}} \int_0^{t}  \sum_{k=1}^\infty \int_{\mathcal{O}} F^{\bis}_\delta(\bar{\phi}_{\delta,n}) \lvert F_{1,n}(\bar{\phi}_{\delta,n})e_{1,k} \rvert^2 \d x \,\d s 
+ \frac{1}{2 \eps_\Gamma} \mathbb{\bar{E}} \int_0^{t}  \sum_{k=1}^\infty \int_{\Gamma} G^{\bis}_\delta(\bar{\varphi}_{\delta,n}) \lvert F_{2,n}(\bar{\varphi}_{\delta,n}) e_{2,k}\rvert^2\d S\,\d s
\end{aligned}
\end{equation}
and for every $t \in[0,T]$, $n \in \mathbb{N}$, and every $\delta>0$,
\begin{equation}\label{Eqn-E_tot-zeta-estimate-1}
\begin{aligned}
& \bar{\mathbb{E}} \sup_{s \in[0,t]} [\mathcal{E}_{tot}(\bar{\phi}_{\delta,n}(s), \bar{\varphi}_{\delta,n}(s))]^\frac{r}{2} + \bar{\mathbb{E}} \left(\int_0^t \lvert \nabla \bar{\bu}_{\delta,n} \rvert^2 \d s \right)^\frac{r}{2}
+ \bar{\mathbb{E}} \left(\int_0^t \lvert \nabla\bar{\mu}_{\delta,n} \rvert^2 \d s \right)^\frac{r}{2}
\\
&  + \bar{\bar{\mathbb{E}}} \left(\int_0^{t} \lvert \nabla_\Gamma \bar{\theta}_{\delta,n} \rvert^2 \d s\right)^\frac{r}{2}
\leq C\bigg[ 1 + [\mathcal{E}_{tot}(\bar{\phi}_n(0),\bar{\varphi}_n(0))]^\frac{r}{2} +  \bar{\mathbb{E}} \int_0^t [\mathcal{E}_{tot}(\bar{\phi}_{\delta,n}(s), \bar{\varphi}_{\delta,n}(s))]^\frac{r}{2} \, \d s
\\
& + \bar{\mathbb{E}} \left(\sum_{k=1}^\infty \int_{Q_t} \lvert F^{\bis}_\delta(\bar{\phi}_{\delta,n}) \rvert \lvert F_{1,n}(\bar{\phi}_{\delta,n})e_{1,k} \rvert^2 \d x \,\d s \right)^\frac{r}{2} 
+ \bar{\mathbb{E}} \left(\sum_{k=1}^\infty \int_{\Sigma_t} \lvert G^{\bis}_\delta(\bar{\varphi}_{\delta,n}) \rvert \lvert F_{2,n}(\bar{\varphi}_{\delta,n}) e_{2,k} \rvert^2\d S\,\d s\right)^\frac{r}{2} 
 \bigg].
\end{aligned}
\end{equation}
The next step now is to pass to the limit by letting $n \to \infty$ in \eqref{eq4.131} and \eqref{Eqn-E_tot-zeta-estimate-1}. \newline
Let us consider \eqref{eq4.131}. By the Lipschitz continuity of the maps $F_{1,n}$, $F_{2,n}$, and similar reasoning as in  \eqref{eq4.123}, we infer that as $n \to \infty$,
    \begin{equation}\label{Eqn4.82}
      (F_{1,n}(\bar{\phi}_{\delta,n}), F_{2,n}(\bar{\varphi}_{\delta,n})) \to (F_1(\bar{\phi}_{\delta}),F_2(\bar{\varphi}_{\delta}))  \mbox{ in } L^\ell(\bar{\Omega};L^2(0,T;\mathscr{T}_2(\mathcal{U};\mathbb{H}))) \quad \forall \ell< r.
    \end{equation}
Next, since $F_\delta$ and $G_\delta$ are continuous and satisfy the linear growth condition, see \eqref{eq4.31}, and since 
     \begin{equation}
      \phi_{0,n} \to \phi_0 \mbox{ in }  V_1   \mbox{ and }  \varphi_{0,n} \to \varphi_0  \mbox{ in }  V_\Gamma,
     \end{equation}
we deduce that $F_\delta(\phi_{0,n}) \to F_\delta(\phi_{0})$ and $G_\delta(\varphi_{0,n}) \to G_\delta(\varphi_{0})$ as $n \to \infty$. 
\newline
Moreover, thanks to the embedding $H^1(\mathcal{O}) \embed L^2(\Gamma)$, we infer that for every $n \in \mathbb{N}$,
\begin{align*}
\frac12 \int_{\Gamma} (\varphi_{0,n} - \phi_{0,n})^2 \, \d S
\leq \lvert \varphi_{0,n} \rvert_\Gamma^2 + \lvert \phi_{0,n} \rvert_\Gamma^2
\leq \lvert \varphi_0 \rvert_\Gamma^2 + C(\mathcal{O})  \Vert \phi_0 \Vert_{V_1}^2.
\end{align*}
Consequently,
\begin{align*}
&\frac{\eps}{2} \lvert \nabla \phi_{0,n} \rvert^2 + \frac{\eps_\Gamma}{2} \lvert \nabla_\Gamma \varphi_{0,n} \rvert^2 + \frac{\eps}{2 K} \lvert \varphi_{0,n} - \phi_{0,n} \rvert_\Gamma^2
\leq C(\eps,\mathcal{O},K) \Vert \phi_0 \Vert_{V_1}^2 + (\eps_\Gamma + \eps (1/K)) \Vert \varphi_\Gamma \Vert_{V_\Gamma}^2,
\\
&\mathcal{E}_{tot}(\bar{\phi}_n(0),\bar{\varphi}_n(0))
\leq C(\eps,\mathcal{O},K) \Vert \phi_0 \Vert_{V_1}^2 + (\eps_\Gamma + \eps (1/K)) \Vert \varphi_\Gamma \Vert_{V_\Gamma}^2 + \frac{1}{\eps} \Vert F_\delta(\bar{\phi}_n(0)) \Vert_{L^1(\mathcal{O})} 
\\
&\quad \hspace{3.5 truecm} + \frac{1}{\eps_\Gamma} \Vert G_\delta(\bar{\varphi}_n(0)) \Vert_{L^1(\Gamma)}.
\end{align*}
In addition, to pass to the limit in \eqref{eq4.131}, we rely on the following weak convergences,
\begin{equation}\label{convergence-4.132}
\begin{aligned}
\sqrt{\nu(\bar{\phi}_{\delta,n})} D\bar{\bu}_{\delta,n} &\rightharpoonup \sqrt{\nu(\bar{\phi}_{\delta})} D\bar{\bu}_{\delta} &\mbox{ in } L^2(\bar{\Omega} \times Q_T;\mathbb{R}^{d \times d}), 
\\
\sqrt{\lambda(\bar{\phi}_{\delta,n})} \bar{\bu}_{\delta,n} &\rightharpoonup \sqrt{\lambda(\bar{\phi}_{\delta})} \bar{\bu}_{\delta} &\mbox{ in }  \mathbb{L}^2(\bar{\Omega} \times Q_T),
    \\
\sqrt{M_{\mathcal{O}}(\bar{\phi}_{\delta,n})} \nabla \bar{\mu}_{\delta,n} &\rightharpoonup \sqrt{M_{\mathcal{O}}(\bar{\phi}_{\delta})} \nabla \bar{\mu}_{\delta} &\mbox{ in }  \mathbb{L}^2(\bar{\Omega} \times Q_T),
          \\
\sqrt{\gamma(\bar{\varphi}_{\delta,n})} \bar{\bu}_{\delta,n} &\rightharpoonup \sqrt{\gamma(\bar{\varphi}_{\delta})} \bar{\bu}_{\delta}  &\mbox{ in }  \mathbb{L}^2(\bar{\Omega} \times  \Sigma_T),
               \\
\sqrt{M_{\Gamma}(\bar{\varphi}_{\delta,n})} \nabla_\Gamma \bar{\theta}_{\delta,n} &\rightharpoonup \sqrt{M_{\Gamma}(\bar{\varphi}_{\delta})} \nabla_\Gamma \bar{\theta}_{\delta} &\mbox{ in }  \mathbb{L}^2(\bar{\Omega} \times  \Sigma_T).
\end{aligned}
\end{equation}
Since $(\bar{\phi}_{\delta,n},\bar{\varphi}_{\delta,n}) \to (\bar{\phi}_{\delta},\bar{\varphi}_{\delta})$ in $C([0,T];\mathbb{H})$, $\bar{\mathbb{P}}$-a.s., cf. Proposition \ref{Jakubowski-Skorokhod-representation theorem}, by continuity of the map $\nu: \mathbb{R} \to \mathbb{R}$, we deduce that up to a subsequence, $\sqrt{\nu(\bar{\phi}_{\delta,n})} \to \sqrt{\nu(\bar{\phi}_{\delta})}$ $\d \mathbb{\bar{P}}$-a.e. in $Q_T$, as $n \to \infty$. Moreover, since the $\mu$-measure of $\bar{\Omega} \times Q_T$ is finite, the Severini–Egorov Theorem ensures that, for every $\epsilon>0$, there exists a measurable subset $\mathcal{B}_\epsilon$ of $\bar{\Omega} \times Q_T$ s.t. $\mu(\mathcal{B}_\epsilon^c)<\epsilon$, and $\sqrt{\nu(\bar{\phi}_{\delta,n})} \lvert_{1_{\mathcal{B}_\epsilon}} \underset{n \to \infty}{\to} \sqrt{\nu(\bar{\phi}_{\delta})} \lvert_{1_{\mathcal{B}_\epsilon}}$ in $L^\infty([0,T] \times \bar{\Omega})$.
\dela{and $(\sqrt{\nu(\bar{\phi}_{\delta,n})})_n$ converges to $\sqrt{\nu(\bar{\phi}_{\delta})}$ uniformly on $\mathcal{B}_\eps$}Here $\mathcal{B}_\epsilon^c= (\bar{\Omega} \times Q_T) \backslash \mathcal{B}_\epsilon$. \newline
Let $\bv \in \left \{A= (A_{ij})_{1\leq i, j \leq d}): \, A_{ij} \in  \bigcup_{n\in \mathbb{N}} V_\sigma^n \right \}$ and $\chi \in L^\infty([0,T] \times \bar{\Omega})$. 
Note that
\begin{align*}
&\int_{\bar{\Omega} \times Q_T} \sqrt{\nu(\bar{\phi}_{\delta,n}(t,x,\bar{\omega}))} D \bar{\bu}_{\delta,n}(t,x,\bar{\omega}): \chi(t,\bar{\omega}) \bv(x)\, \d \mu(t,x,\bar{\omega})
\\
&= \int_{\bar{\Omega} \times Q_T} \sqrt{\nu(\bar{\phi}_{\delta,n}(t,x,\bar{\omega}))} D \bar{\bu}_{\delta,n}(t,x,\bar{\omega}): \chi(t,\bar{\omega}) \bv(x) \lvert_{1_{\mathcal{B}_\epsilon}}\, \d \mu(t,x,\bar{\omega})
\\
&\qquad + \int_{\bar{\Omega} \times Q_T} \sqrt{\nu(\bar{\phi}_{\delta,n}(t,x,\bar{\omega}))} D \bar{\bu}_{\delta,n}(t,x,\bar{\omega}): \chi(t,\bar{\omega}) \bv(x) \lvert_{1_{\mathcal{B}^c_\epsilon}}\, \d \mu(t,x,\bar{\omega}).
\end{align*}
Now, since $\sqrt{\nu(\bar{\phi}_{\delta,n})} \lvert_{1_{\mathcal{B}_\eps}} \underset{n \to \infty}{\to} \sqrt{\nu(\bar{\phi}_{\delta})} \lvert_{1_{\mathcal{B}_\eps}}$ in $L^\infty(\bar{\Omega} \times Q_T)$, 
$\nabla \bar{\bu}_{\delta,n}  \underset{n \to \infty}{\rightharpoonup} \nabla \bar{\bu}_{\delta}$ in $L^2(\bar{\Omega} \times Q_T;\mathbb{R}^{d \times d})$, we infer that
\begin{align*}
&\int_{\bar{\Omega} \times Q_T} \sqrt{\nu(\bar{\phi}_{\delta,n}(t,x,\bar{\omega}))} D \bar{\bu}_{\delta,n}(t,x,\bar{\omega}): \chi(t,\bar{\omega}) \bv(x) \lvert_{1_{\mathcal{B}_\epsilon}}\, \d \mu(t,x,\bar{\omega})
\\
&\underset{n \to \infty}{\to} \int_{\bar{\Omega} \times Q_T} \sqrt{\nu(\bar{\phi}_{\delta}(t,x,\bar{\omega}))} D \bar{\bu}_{\delta}(t,x,\bar{\omega}):\chi(t,\bar{\omega}) \bv(x) \lvert_{1_{\mathcal{B}_\epsilon}}\, \d \mu(t,x,\bar{\omega}).
\end{align*}
Fix $\epsilon> 0$. By the H\"older inequality, the fact that $\nabla \bar{\bu}_{\delta,n}$ is uniformly bounded w.r.t. $n$ in $L^2(\bar{\Omega} \times Q_T;\mathbb{R}^{d \times d})$, we deduce that for every $n \in \mathbb{N}$,
\begin{align*}
&\left \lvert \int_{\bar{\Omega} \times Q_T} \sqrt{\nu(\bar{\phi}_{\delta,n}(t,x,\bar{\omega}))} D \bar{\bu}_{\delta,n}(t,x,\bar{\omega}): \chi(t,\bar{\omega}) \bv(x) \lvert_{1_{\mathcal{B}^c_\epsilon}}\, \d \mu(t,x,\bar{\omega}) \right \rvert
\\
&\leq [\bar{\nu}_0]^{\frac12} \Vert D\bar{\bu}_{\delta,n} \Vert_{L^2(\bar{\Omega} \times Q_T;\mathbb{R}^{d \times d})} \Vert \chi \Vert_{L^\infty([0,T] \times \bar{\Omega})} \Vert \bv \Vert_{L^\infty(\mathcal{O};\mathbb{R}^{d \times d})} [\mu(\mathcal{B}_\epsilon^c)]^{\frac12}
\leq C \eps^{1/2}.
\end{align*}
Subsequently, since $\epsilon$ is arbitrary, we infer that
   \begin{equation*}
     \lim_{n \to \infty} \left \lvert \int_{\bar{\Omega} \times Q_T} \sqrt{\nu(\bar{\phi}_{\delta,n}(t,x,\bar{\omega}))} D \bar{\bu}_{\delta,n}(t,x,\bar{\omega}):  \chi(t,\bar{\omega}) \bv(x) \lvert_{1_{\mathcal{B}^c_\epsilon}}\, \d \mu(t,x,\bar{\omega}) \right \rvert= 0.
  \end{equation*}
Consequently,
\begin{align*}
&\lim_{n \to \infty} \int_{\bar{\Omega} \times Q_T} \sqrt{\nu(\bar{\phi}_{\delta,n}(t,x,\bar{\omega}))} D \bar{\bu}_{\delta,n}(t,x,\bar{\omega}): \chi(t,\bar{\omega}) \bv(x)\, \d \mu(t,x,\bar{\omega})
\\
&= \int_{\bar{\Omega} \times Q_T} \sqrt{\nu(\bar{\phi}_{\delta}(t,x,\bar{\omega}))} D \bar{\bu}_{\delta}(t,x,\bar{\omega}): \chi(t,\bar{\omega}) \bv(x)\, \d \mu(t,x,\bar{\omega}).
\end{align*}
Due to density, the last equality holds for all test functions in the space $\mathbb{L}^2(\bar{\Omega} \times Q_T;\mathbb{R}^{d \times d})$.
\newline
This completes the proof of the first part of condition \eqref{convergence-4.132}. 
\newline
The proof of the other part can be proven similarly with straightforward modifications. \newline
Next, exploiting the previous convergences results and weak lower-semicontinuity of the norm, we deduce that for all $t \in [0,T]$ and for every $\delta>0$,
\begin{equation}\label{eq4.133}
\begin{aligned}
&\bar{\mathbb{E}}[\mathcal{E}_{tot}(\bar{\phi}_{\delta}(t), \bar{\varphi}_{\delta}(t))] + \bar{\mathbb{E}} \int_{Q_t} [2 \nu(\bar{\phi}_{\delta}) \lvert D\bar{\bu}_{\delta} \rvert^2 + \lambda(\bar{\phi}_{\delta}) \lvert \bar{\bu}_{\delta} \rvert^2] \,\d x \,\d s + \bar{\mathbb{E}} \int_{\Sigma_t} \gamma(\bar{\varphi}_{\delta}) \lvert \bar{\bu}_{\delta} \rvert^2 \,\d S\,\d s 
\\
& + \frac{1}{2} \bar{\mathbb{E}} \int_{Q_t} M_{\mathcal{O}}(\bar{\phi}_{\delta}) \lvert \nabla \bar{\mu}_{\delta} \rvert^2 \, \d x \,\d s + \bar{\mathbb{E}} \int_{\Sigma_t} M_\Gamma(\bar{\varphi}_{\delta}) \lvert \nabla_\Gamma \bar{\theta}_{\delta} \rvert^2 \,\d S \,\d s  
    \\
&\leq  C(\eps,\mathcal{O},K) \Vert \phi_0 \Vert_{V_1}^2 + (\eps_\Gamma + \eps (1/K)) \Vert \varphi_\Gamma \Vert_{V_\Gamma}^2 + \frac{1}{\eps} \Vert F_\delta(\phi_0) \Vert_{L^1(\mathcal{O})} + \frac{1}{\eps_\Gamma} \Vert G_\delta(\varphi_0) \Vert_{L^1(\Gamma)}
\\
& + \frac{1}{2} \bar{\mathbb{E}} \int_0^{t} \Vert F_{1}(\bar{\phi}_{\delta}) \Vert_{\mathscr{T}_2(U,L^2(\mathcal{O}))}^2\,\d s + 2 \eps (1/K) \bar{\mathbb{E}} \int_0^{t}  \sum_{k=1}^\infty \Vert F_1(\bar{\phi}_{\delta}) e_{1,k} \Vert_{L^2(\Gamma)}^2 \, \d s 
  \\
& + \eps (1/K) \bar{\mathbb{E}} \int_0^{t} \Vert F_2(\bar{\varphi}_{\delta})\Vert_{\mathscr{T}_2(U_\Gamma,L^2(\Gamma))}^2 \, \d s  + \frac{\eps_\Gamma C_2}{2} \bar{\mathbb{E}} \int_0^{t} \lvert \nabla_\Gamma \bar{\varphi}_{\delta} \lvert_{\Gamma}^2 \, \d s + \left(\frac{\bar{M}_0}{2} + \frac{\eps C_1}{2} \right) \bar{\mathbb{E}} \int_0^t  \lvert \nabla \bar{\phi}_{\delta} \rvert^2 \, \d s 
     \\
& + \frac{1}{2 \eps} \bar{\mathbb{E}} \int_0^{t}  \sum_{k=1}^\infty \int_{\mathcal{O}} F^{\bis}_\delta(\bar{\phi}_{\delta}) \lvert F_1(\bar{\phi}_{\delta})e_{1,k} \rvert^2 \, \d x \,\d s +  \frac{1}{2 \eps_\Gamma} \bar{\mathbb{E}} \int_0^{t}  \sum_{k=1}^\infty \int_{\Gamma} G^{\bis}_\delta(\bar{\varphi}_{\delta}) \lvert F_2(\bar{\varphi}_{\delta})e_{2,k} \rvert^2\, \d S\,\d s.
\end{aligned}
\end{equation}
Let us now consider \eqref{Eqn-E_tot-zeta-estimate-1}. We choose and fix $k \geq 1$. By \eqref{eq4.122a}-1) and \eqref{Eqn4.82}, we infer that
\begin{equation}
\begin{aligned}
\lvert F^{\bis}_\delta(\bar{\phi}_{\delta,n}) \rvert \lvert F_{1,n}(\bar{\phi}_{\delta,n}) e_{1,k} \rvert^2 
& \underset{n \to \infty}{\to} \lvert F^{\bis}_\delta(\bar{\phi}_{\delta}) \rvert \lvert F_1(\bar{\phi}_{\delta}) e_{1,k} \rvert^2 \mbox{ a.e. in } \bar{\Omega} \times Q_T,
  \\
\lvert G^{\bis}_\delta(\bar{\varphi}_{\delta,n}) \rvert \lvert F_{2,n}(\bar{\varphi}_{\delta,n}) e_{2,k} \rvert^2 & \underset{n \to \infty}{\to} \lvert G^{\bis}_\delta(\bar{\varphi}_{\delta}) \rvert \lvert F_2(\bar{\varphi}_{\delta}) e_{2,k} \rvert^2 \mbox{ a.e. in } \bar{\Omega} \times  \Sigma_T.
\end{aligned}
\end{equation}
Moreover, as in \eqref{eq4.85} and \eqref{eq4.92a}, we obtain
\begin{align*}
&\sum_{k=1}^\infty \int_{\mathcal{O}} \lvert F^{\bis}_\delta(\bar{\phi}_{\delta,n}) \rvert \lvert F_{1,n}(\bar{\phi}_{\delta,n})e_{1,k} \rvert^2 \, \d x +  \sum_{k=1}^\infty \int_{\Gamma} \lvert G^{\bis}_\delta(\bar{\varphi}_{\delta,n}) \rvert \lvert F_{2,n}(\bar{\varphi}_{\delta,n})e_{2,k} \rvert^2\, \d S
\\
&\leq (\delta^{-1} + \tilde{c}_F) C_1 + (\delta^{-1} + \tilde{c}_G) C_3.
\end{align*}
Hence, by the DCT, we deduce that as $n \to \infty$,
\begin{align*}
\bar{\mathbb{E}} \left(\sum_{k=1}^\infty \int_{Q_t} \lvert F^{\bis}_\delta(\bar{\phi}_{\delta,n}) \rvert \lvert F_{1,n}(\bar{\phi}_{\delta,n}) e_{1,k} \rvert^2 \d x \d s \right)^{r/2} 
&\to \bar{\mathbb{E}} \left(\sum_{k=1}^\infty \int_{Q_t} \lvert F^{\bis}_\delta(\bar{\phi}_{\delta}) \rvert \lvert F_1(\bar{\phi}_{\delta}) e_{1,k} \rvert^2 \d x \d s \right)^{r/2}
 \\
\bar{\mathbb{E}} \left(\sum_{k=1}^\infty \int_{\Sigma_t} \lvert G^{\bis}_\delta(\bar{\varphi}_{\delta,n}) \rvert \lvert F_{2,n}(\bar{\varphi}_{\delta,n}) e_{2,k} \rvert^2 \d S \d s\right)^{r/2} 
&\to \bar{\mathbb{E}} \left(\sum_{k=1}^\infty \int_{\Sigma_t} \lvert G^{\bis}_\delta(\bar{\varphi}_{\delta}) \rvert \lvert F_2(\bar{\varphi}_{\delta}) e_{2,k} \rvert^2 \d S \d s\right)^{r/2}. 
\end{align*}
We recall that $\bar{\phi}_{\delta,n} \to \bar{\phi}_{\delta}$ in $C([0,T];H_w^1(\mathcal{O}))$, $\bar{\mathbb{P}}$-a.s., see Proposition \ref{Jakubowski-Skorokhod-representation theorem}. So, using Corollary E.3 and Proposition E.1 from \cite{Brzezniak+Ferrario+Zanella_2024}, and the Fatou lemma, we deduce that
\begin{equation}\label{Eqn-4.83}
\begin{aligned}
\bar{\mathbb{E}} \sup_{s \in [0,t]} \lvert \nabla \bar{\phi}_{\delta}(s) \rvert^{r}
=\int_{\bar{\Omega}} \sup_{s \in [0,t]} \lvert \nabla \bar{\phi}_{\delta}(s,\bar{\omega}) \rvert^{r} \, \d \bar{\mathbb{P}}(\bar{\omega})
&\leq \int_{\bar{\Omega}} \liminf_{n \to \infty} \sup_{s \in [0,t]} \lvert \nabla \bar{\phi}_{\delta,n}(s,\bar{\omega}) \rvert^{r} \, \d \bar{\mathbb{P}}(\bar{\omega})
\\
&\leq \liminf_{n \to \infty} \int_{\bar{\Omega}} \sup_{s \in [0,t]} \lvert \nabla \bar{\phi}_{\delta,n}(s,\bar{\omega}) \rvert^{r} \, \d \bar{\mathbb{P}}(\bar{\omega}).
\end{aligned}
\end{equation}
In a similar way, we show that
\begin{align*}
\bar{\mathbb{E}} \sup_{s \in[0,t]} [\mathcal{E}_{tot}(\bar{\phi}_{\delta}(s), \bar{\varphi}_{\delta}(s))]^{r/2}
\leq \liminf_{n \to \infty} \int_{\bar{\Omega}} \sup_{s \in[0,t]} [\mathcal{E}_{tot}(\bar{\phi}_{\delta,n}(s,\bar{\omega}), \bar{\varphi}_{\delta,n}(s,\bar{\omega}))]^{r/2} \, \d \bar{\mathbb{P}}(\bar{\omega}).
\end{align*}
From now on, recalling \eqref{eq4.122a} and making use of the previous observations, we can then pass to the limit by letting $n \to \infty$ in \eqref{Eqn-E_tot-zeta-estimate-1} to deduce that for all $t \in [0,T]$,
\begin{equation}\label{Eqn-E_tot-zeta-estimate-2}
\begin{aligned}
& \bar{\mathbb{E}} \sup_{s \in[0,t]} [\mathcal{E}_{tot}(\bar{\phi}_{\delta}(s), \bar{\varphi}_{\delta}(s))]^{r/2} + \bar{\mathbb{E}} \left[ \left(\int_0^t \lvert \nabla \bar{\bu}_{\delta} \rvert^2 \d s \right)^{r/2}
+ \left(\int_0^t \lvert \nabla\bar{\mu}_{\delta} \rvert^2 \d s \right)^{r/2} + \left(\int_0^{t} \lvert \nabla_\Gamma \bar{\theta}_{\delta} \rvert^2 \d s\right)^{r/2} \right]
\\
&\leq C\bigg[ 1 + \Vert \phi_0 \Vert_{V_1}^r + \Vert \varphi_\Gamma \Vert_{V_\Gamma}^r + \Vert F_\delta(\phi_0) \Vert_{L^1(\mathcal{O})}^{r/2} + \Vert G_\delta(\varphi_0) \Vert_{L^1(\Gamma)}^{r/2}
+  \bar{\mathbb{E}} \int_0^t [\mathcal{E}_{tot}(\bar{\phi}_{\delta}(s), \bar{\varphi}_{\delta}(s))]^{r/2} \, \d s
\\
& + \bar{\mathbb{E}} \left(\sum_{k=1}^\infty \int_{Q_t} \lvert F^{\bis}_\delta(\bar{\phi}_{\delta}) \rvert \lvert F_1(\bar{\phi}_{\delta})e_{1,k} \rvert^2 \d x \,\d s \right)^{r/2} 
+ \bar{\mathbb{E}} \left(\sum_{k=1}^\infty \int_{\Sigma_t} \lvert G^{\bis}_\delta(\bar{\varphi}_{\delta}) \rvert \lvert F_{\Gamma}(\bar{\varphi}_{\delta}) e_{2,k} \rvert^2\d S\,\d s\right)^{p/2} 
 \bigg],
\end{aligned}
\end{equation}
where the constant $C$ is independent of $\delta$. 
\newline
Let us consider the terms $\Vert F_\delta(\phi_0) \Vert_{L^1(\mathcal{O})}$ and $\Vert G_\delta(\varphi_0) \Vert_{L^1(\Gamma)}$. From \ref{item:P1} and \ref{item:P3}, we infer that
\[
\Vert F_\delta(\phi_0) \Vert_{L^1(\mathcal{O})} \leq \Vert F(\phi_0) \Vert_{L^1(\mathcal{O})}.
\]
Analogously,
   \begin{equation*}
     \Vert G_\delta(\varphi_0) \Vert_{L^1(\Gamma)} \leq \Vert G(\varphi_0) \Vert_{L^1(\Gamma)}.
   \end{equation*}
Now, we consider the last two terms on the RHS of \eqref{Eqn-E_tot-zeta-estimate-2}. We observe that
\dela{
\begin{align*}
& \sum_{k=1}^\infty \int_{Q_t} \lvert F^{\bis}_\delta(\bar{\phi}_{\delta}(s,x)) \rvert \lvert F_1(\bar{\phi}_{\delta}(s,x))e_{1,k}(x) \rvert^2 \, \d x \, \d s
\\& \leq 
\left( \sum_{k=1}^\infty \sup_{s,x} \lvert F_1(\bar{\phi}_{\delta}(s,x))e_{1,k}(x) \rvert^2  \right) \int_{Q_t} \lvert F^{\bis}_\delta(\bar{\phi}_{\delta}(s,x)) \rvert \, \d x \, \d s
\end{align*}
}
\begin{align*}
& \sum_{k=1}^\infty \int_{Q_t} \lvert F^{\bis}_\delta(\bar{\phi}_{\delta}(s,x)) \rvert \lvert F_1(\bar{\phi}_{\delta}(s,x))e_{1,k}(x) \rvert^2 \, \d x \, \d s
= \sum_{k=1}^\infty \int_{Q_t} \lvert F^{\bis}_\delta(\bar{\phi}_{\delta}(s,x)) \rvert \lvert F_1(\bar{\phi}_{\delta}(s,x))e_{1,k}(x) \rvert^2 \, \d x\, \d s 
 \\
&\leq \int_0^t \Vert F^{\bis}_\delta(\bar{\phi}_{\delta}(s)) \Vert_{L^1(\mathcal{O})} \, \d s \cdot \sum_{k=1}^\infty \Vert F_1(\bar{\phi}_{\delta})e_{1,k} \Vert_{L^\infty(Q_t)}^2 
\leq \text{meas}(Q_t) \tilde{C}_1 \int_0^t \Vert F^{\bis}_\delta(\bar{\phi}_{\delta}(s)) \Vert_{L^1(\mathcal{O})} \, \d s.
\end{align*}
This, together with the assumption \eqref{condition_ F'_and_F''} on $F^{\bis}$ implies that
  \begin{equation*}
    \sum_{k=1}^\infty \int_{Q_t} \lvert F^{\bis}_\delta(\bar{\phi}_{\delta}(s,x)) \rvert \lvert F_1(\bar{\phi}_{\delta}(s,x)) e_{1,k}(x) \rvert^2 \, \d x \, \d s
    \leq C \left(1 + \int_0^t \Vert F_\delta(\bar{\phi}_{\delta}(s)) \Vert_{L^1(\mathcal{O})} \, \d s \right),
  \end{equation*}
for a certain constant $C$ which is independent of $\delta$. Next, H\"older's inequality leads to
\begin{align*}
\bar{\mathbb{E}} \left(\sum_{k=1}^\infty \int_{Q_t} \lvert F^{\bis}_\delta(\bar{\phi}_{\delta}) \rvert \lvert F_1(\bar{\phi}_{\delta})e_{1,k} \rvert^2 \, \d x \,\d s \right)^{r/2} 
\leq C \left[1 + t^\frac{r-2}{2} \bar{\mathbb{E}} \int_0^t \Vert F_\delta(\bar{\phi}_{\delta}) \Vert_{L^1(\mathcal{O})}^{r/2} \, \d s \right].
\end{align*}
Similarly, we deduce that there exists a constant $C$ independent of $\delta$ such that
   \begin{equation*}
     \bar{\mathbb{E}} \left(\sum_{k=1}^\infty \int_{\Sigma_t} \lvert G^{\bis}_\delta(\bar{\varphi}_{\delta}) \rvert \lvert 
       F_2(\bar{\varphi}_{\delta}) e_{2,k} \rvert^2 \, \d S \,\d s \right)^{r/2} 
     \leq C \left[1 + t^\frac{r-2}{2} \bar{\mathbb{E}} \int_0^t \Vert G_\delta(\bar{\phi}_{\delta}) \Vert_{L^1(\Gamma)}^{r/2} \, \d s \right], \; t \in [0,T].
   \end{equation*}
Plugging the above estimates into the RHS of \eqref{Eqn-E_tot-zeta-estimate-2}, we are led to 
\begin{equation}\label{Eqn-E_tot-zeta-estimate-3}
\begin{aligned}
& \bar{\mathbb{E}} \sup_{s \in[0,t]} [\mathcal{E}_{tot}(\bar{\phi}_{\delta}(s), \bar{\varphi}_{\delta}(s))]^{r/2} + \bar{\mathbb{E}} \bigg[ \left(\int_0^t \lvert \nabla \bar{\bu}_{\delta}(s) \rvert^2 \d s \right)^{r/2}
+ \left(\int_0^t \lvert \nabla\bar{\mu}_{\delta}(s) \rvert^2 \d s \right)^{r/2}
\\
& + \left(\int_0^{t} \lvert \nabla_\Gamma \bar{\theta}_{\delta}(s) \rvert^2 \d s\right)^{r/2} \bigg]
\leq C\bigg[ 1 + \Vert \phi_0 \Vert_{V_1}^r + \Vert \varphi_\Gamma \Vert_{V_\Gamma}^r + \Vert F(\phi_0) \Vert_{L^1(\mathcal{O})}^{r/2}
\\
& + \Vert G(\varphi_0) \Vert_{L^1(\Gamma)}^{r/2}
+  \bar{\mathbb{E}} \int_0^t \sup_{0\leq r \leq s} [\mathcal{E}_{tot}(\bar{\phi}_{\delta}(r), \bar{\varphi}_{\delta}(r))]^{r/2} \, \d s 
 \bigg], \; \; t \in [0,T].
\end{aligned}
\end{equation}
Let us point out that the constant $C$ in \eqref{Eqn-E_tot-zeta-estimate-3} does not depend on $\delta$. Therefore, by applying the Gronwall lemma, we deduce that there exists a positive constant $c_r$, with $c_r$ being independent of $\delta$ such that
\begin{equation}\label{eq4.143}
\begin{aligned}
\Vert \bar{\bu}_{\delta} \Vert_{L^r(\bar{\Omega};L^2(0,T;\mathbb{L}^2(\Gamma)))}  &\leq c_r, 
\\
\Vert \bar{\bu}_{\delta} \Vert_{L^r(\bar{\Omega};L^2(0,T;V))} &\leq c_r,
   \\
\Vert \bar{\phi}_{\delta} \Vert_{L^r(\bar{\Omega};L^\infty(0,T;V_1))} &\leq c_r, 
      \\
\Vert (F_\delta(\bar{\phi}_{\delta}),G_\delta(\bar{\varphi}_{\delta})) \Vert_{L^{r/2}(\bar{\Omega};L^\infty(0,T;L^1(\mathcal{O}) \times L^1(\Gamma)))}  &\leq c_r,
\\
\Vert \nabla_\Gamma \bar{\varphi}_{\delta} \Vert_{L^r(\bar{\Omega};L^\infty(0,T;\mathbb{L}^2(\mathcal{O})))} &\leq c_r,  \\
\Vert(\nabla \bar{\mu}_{\delta},\nabla_\Gamma \bar{\theta}_{\delta}) \Vert_{L^r(\bar{\Omega};L^2(0,T;\mathbb{H}))} &\leq c_r,
   \\
\Vert \bar{\varphi}_{\delta} - \bar{\phi}_{\delta} \Vert_{L^r(\bar{\Omega};L^\infty(0,T;L^2(\Gamma)))} &\leq c_r.
\end{aligned}
\end{equation}
Moreover, owing to \eqref{eq4.130d} and \eqref{eq4.143}, and arguing similarly as in \eqref{Eqn-L2-norm-mu-and-theta-delta-n}, we deduce there exists a positive constant $c_r$ independent of $\delta$ such that
    \begin{equation}\label{eq4.144}
      \Vert (\bar{\mu}_{\delta},\bar{\theta}_{\delta}) \Vert_{L^r(\bar{\Omega};L^2(0,T;\mathbb{H}))} \leq c_r.
     \end{equation}
Furthermore, from \eqref{eq4.143} and similar reasoning as in \eqref{eq4.99}, we deduce that 
    \begin{equation}\label{uniform_delta_estimate_varphi}
       \Vert \bar{\varphi}_{\delta} \Vert_{L^r(\bar{\Omega};L^\infty(0,T;L^2(\Gamma)))} \leq c_r.
    \end{equation}
On the other hand, from the assumption \ref{item:H2}, the estimates \eqref{eq4.143}, and \eqref{uniform_delta_estimate_varphi}, we get
    \begin{equation}
       \Vert (F_1(\bar{\phi}_{\delta}),F_2(\bar{\varphi}_{\delta})) \Vert_{L^\infty([0,T] \times \bar{\Omega};\mathscr{T}_2(\mathcal{U};\mathbb{H})) \cap L^r(\bar{\Omega};L^\infty(0,T;\mathscr{T}_2(\mathcal{U};\mathbb{V})))} \leq c_r. 
    \end{equation}
Let us now prove that $(\bar{\phi}_{\delta},\bar{\varphi}_{\delta})_\delta$ satisfies the following property: For every $\delta>0$,
    \begin{align}\label{Eqn-(phi_{n,delta},varphi_{n,delta})-H2-estimate-1}
       \bar{\mathbb{E}} \left(\int_0^T \Vert(\bar{\phi}_{\delta}(s),\bar{\varphi}_{\delta}(s)) \Vert_{\mathcal{H}^2}^2 \,\d s \right)^{r/2} \leq c_r.
    \end{align}
\dela{
\coma{
Firstly, taking $\psi= \mathbb{A}_\delta(\bar{\phi}_{\delta}) \coloneq F^\prime_\delta(\bar{\phi}_{\delta}) +  \tilde{c}_F \bar{\phi}_{\delta}$ and $\psi \lvert_\Gamma= \mathbb{A}_{\Gamma,\delta}(\bar{\varphi}_{\delta}) \coloneq G_\delta^\prime(\bar{\varphi}_{\delta}) + \tilde{c}_G \bar{\varphi}_{\delta}$ in the variational formulation \eqref{eq4.130d}, we obtain
\begin{equation}
\begin{aligned}
&\eps \int_{\mathcal{O}} \Psi^\prime_\delta(\bar{\phi}_{\delta}) \lvert \nabla \bar{\phi}_{\delta} \rvert^2 \, \d x + \eps_\Gamma \int_{\Gamma} \Psi^\prime_{\Gamma,\delta}(\bar{\varphi}_{\delta}) \lvert \nabla_\Gamma \bar{\varphi}_{\delta} \rvert^2 \, \d S + \frac1\eps \lvert F_\delta^\prime(\bar{\phi}_{\delta}) \rvert^2 + \frac{1}{\eps_\Gamma} \lvert G_\delta^\prime(\bar{\varphi}_{\delta}) \rvert_\Gamma^2
  \\
& = \int_{\mathcal{O}} \bar{\mu}_{\delta} F^\prime_\delta(\bar{\phi}_{\delta}) \, \d x + \tilde{c}_F \int_{\mathcal{O}} \bar{\mu}_{\delta}  \bar{\phi}_{\delta} \, \d x + \int_{\Gamma} \bar{\theta}_{\delta} G_\delta^\prime(\bar{\varphi}_{\delta}) \, \d S + \tilde{c}_G \int_{\Gamma} \bar{\theta}_{\delta}  \bar{\varphi}_{\delta} \, \d S - \frac{\tilde{c}_G}{\eps_\Gamma} \int_{\Gamma} G_\delta^\prime(\bar{\varphi}_{\delta}) \bar{\varphi}_{\delta} \, \d S
     \\
&\quad - \frac{\tilde{c}_F}{\eps} \int_{\mathcal{O}} F_\delta^\prime(\bar{\phi}_{\delta}) \bar{\phi}_{\delta} \, \d x + \eps [K] \int_{\Gamma} (\bar{\varphi}_{\delta} - \bar{\phi}_{\delta}) G_\delta^\prime(\bar{\varphi}_{\delta})  \, \d S  +  \eps \tilde{c}_G  [K]  \int_{\Gamma} (\bar{\varphi}_{\delta} - \bar{\phi}_{\delta}) \bar{\varphi}_{\delta} \, \d S 
\\
&\quad - \coma{ \eps [K] \int_{\Gamma} (\bar{\varphi}_{\delta} - \bar{\phi}_{\delta}) F^\prime_\delta(\bar{\phi}_{\delta}) \, \d S }   - \eps \tilde{c}_F [K]  \int_{\Gamma} (\bar{\varphi}_{\delta} - \bar{\phi}_{\delta}) \bar{\phi}_{\delta} \,\d S.  
\end{aligned}
\end{equation}
}
\todoan{
At this level, I want to deduce an estimate for $\lvert F_\delta^\prime(\bar{\phi}_{\delta}) \rvert^2 \coloneq \Vert F_\delta^\prime(\bar{\phi}_{\delta}) \Vert_{L^2}^2$ and $\lvert G_\delta^\prime(\bar{\varphi}_{\delta}) \rvert^2 \coloneq \Vert G_\delta^\prime(\bar{\varphi}_{\delta}) \Vert_{L^2}^2$. However, it seems difficult to control the term $\coma{ \eps [K] \int_{\Gamma} (\bar{\varphi}_{\delta} - \bar{\phi}_{\delta}) F^\prime_\delta(\bar{\phi}_{\delta}) \, \d S }$. So, to deal with this issue, we consider the additional condition \eqref{F''_and G''_additional_condition}.
}
}
The proof is given in four steps.

\noindent
\textbf{Step1:} Since the sequence $(\bar{\phi}_{\delta,n},\bar{\varphi}_{\delta,n})_n$ enjoys the same property as $(\phi_{\delta,n},\varphi_{\delta,n})_n$ for each $\delta>0$, then from \eqref{Eqn-4.57} and \eqref{Eqn-4.59}, we infer that there exists $C=C(K,\eps,\eps_\Gamma,\mathcal{O},\Gamma)$ such that for every $n \in \mathbb{N}$ and every $\delta>0$,
      \begin{equation}
        \Vert (\bar{\phi}_{\delta,n},\bar{\varphi}_{\delta,n}) \Vert_{\mathcal{H}^2}^2
        \leq  C (\Vert (\bar{\mu}_{\delta,n},\bar{\theta}_{\delta,n}) \Vert_{\mathbb{H}}^2 + \lvert F_\delta^\prime(\bar{\phi}_{\delta,n})\rvert^2 + \lvert G_\delta^\prime(\bar{\varphi}_{\delta,n}) \rvert_\Gamma^2 +  \Vert (\bar{\phi}_{\delta,n},\bar{\varphi}_{\delta,n}) \Vert_{\mathbb{V}}^2).
      \end{equation}
From the previous inequality, the proposition \ref{Jakubowski-Skorokhod-representation theorem} and the weak lower semicontinuity of the norm together with the assumption \ref{item:H8}, we deduce that for every $\delta>0$,
    \begin{equation}\label{Eqn-bar-delta-phi-varphi-H_2-norm}
        \Vert (\bar{\phi}_{\delta},\bar{\varphi}_{\delta}) \Vert_{\mathcal{H}^2}^2
        \leq  c_\ast (\Vert (\bar{\mu}_{\delta},\bar{\theta}_{\delta}) \Vert_{\mathbb{H}}^2 + \lvert F_\delta^\prime(\bar{\phi}_{\delta})\rvert^2 + \lvert G_\delta^\prime(\bar{\varphi}_{\delta}) \rvert_\Gamma^2 +  \Vert (\bar{\phi}_{\delta},\bar{\varphi}_{\delta}) \Vert_{\mathbb{V}}^2).
    \end{equation} 
\textbf{Step2:} Let $G_\delta^\prime(r) \coloneq \mathbb{A}_{\Gamma,\delta}(r) - \tilde{c}_G r, \, r \in \mathbb{R}$, where $\mathbb{A}_{\Gamma,\delta}$ is the regularized Yosida of $\mathbb{A}_{\Gamma}= G^\prime + \tilde{c}_G I$. 
Since $\lvert \mathbb{A}_{\Gamma,\delta}(\cdot) \rvert \leq \lvert \mathbb{A}_{\Gamma}(\cdot) \rvert$ for every $\delta>0$, cf. \cite[Chapter II, Proposition 1.1]{Barbu_1976}, we infer that for every $\delta \in(0,1)$,
\begin{equation}\label{Eqn-4.99}
\begin{aligned}
& \lvert G_\delta^\prime(\bar{\varphi}_{\delta}) \vert
\leq \lvert \mathbb{A}_{\Gamma,\delta}(\bar{\varphi}_{\delta}) \rvert + \tilde{c}_G \lvert \bar{\varphi}_{\delta} \rvert 
\leq \lvert \mathbb{A}_{\Gamma}(\bar{\varphi}_{\delta}) \rvert  + \tilde{c}_G \lvert \bar{\varphi}_{\delta} \rvert
  \\
&\leq \lvert G^\prime(\bar{\varphi}_{\delta}) \rvert + 2 \tilde{c}_G \lvert \bar{\varphi}_{\delta} \rvert 
\leq c_{G^\prime} + c_{G^\prime} \lvert \bar{\varphi}_{\delta} \rvert^{q-1} + 2 \tilde{c}_G \lvert \bar{\varphi}_{\delta} \rvert 
\leq C(q,c_{G^\prime},\tilde{c}_G) (1 + \lvert \bar{\varphi}_{\delta} \rvert^{q-1}).
\end{aligned}
\end{equation}
Notice that in \eqref{Eqn-4.99}, we have also used the assumption \eqref{F'_and G'_additional_condition}.
\newline
Similarly, we show that for every $\delta>0 $,
    \begin{equation}\label{Eqn-4.100}
      \lvert F_\delta^\prime(\bar{\phi}_{\delta}) \rvert 
       \leq C(p,c_{F^\prime},\tilde{c}_F) (1 + \lvert \bar{\phi}_{\delta} \rvert^{p-1}).
    \end{equation}
Now, from \eqref{Eqn-4.99} and the Sobolev embedding theorem, we infer that for every $\delta>0$,
     \begin{equation*}
       \lvert G_\delta^\prime(\bar{\varphi}_{\delta})\rvert_\Gamma
       \leq C (1 + \Vert \bar{\varphi}_{\delta} \Vert_{V_\Gamma}^{q-1}),
     \end{equation*}
This, together with \eqref{uniform_delta_estimate_varphi} yields that
   \begin{equation}\label{eq4.147}
      \Vert G_\delta^\prime(\bar{\varphi}_{\delta}) \Vert_{L^r(\bar{\Omega};L^\infty(0,T;L^2(\Gamma)))} \leq c_{r}.
   \end{equation}
Next, using \eqref{Eqn-4.100} and the embedding of $V_1 \embed L^\frac{6(p-1)}{5}(\mathcal{O})$, we deduce that for every $\delta>0$,
    \begin{equation*}
    \Vert F^\prime(\bar{\phi}_{\delta}) \Vert_{L^{6/5}(\mathcal{O})}^{6/5}
    \leq C + C \Vert \bar{\phi}_{\delta} \Vert_{L^\frac{6(p-1)}{5}(\mathcal{O})} 
     \leq C (1 + \Vert \bar{\phi}_{\delta} \Vert_{V_1}).
   \end{equation*}
Then, by the estimate \eqref{eq4.143}, we infer that there exists $c_r>0$ independent of $\delta$ such that
     \begin{equation}\label{eq4.159}
       \Vert F_\delta^\prime(\bar{\phi}_{\delta}) \Vert_{L^\frac{6 r}{5}(\bar{\Omega}; L^\infty(0,T;L^{\frac65}(\mathcal{O})))}
       \leq c_r.
     \end{equation}
\textbf{Step3:} \dela{We assume that $\mathcal{O}$ if of class $C^2$ and if $d=3$, we further assume $p<6$.} 
Since $H^1(\mathcal{O}) \embed L^{2(p-1)}(\mathcal{O})$
\begin{equation*}
\begin{aligned}
&\mbox{ if } p \geq 2 \; \mbox{ when } \;  d=2,
\\
&\mbox{ if } 2 \leq p \leq 4 \; \mbox{ when } \;  d=3,
\end{aligned}
\end{equation*}
we infer from \eqref{Eqn-4.100} in \textbf{Step2} that for every $\delta>0$,
\begin{equation}\label{eq4.148-1}
\begin{aligned}
 \lvert F_\delta^\prime(\bar{\phi}_{\delta}) \rvert
\leq C \left(1 + \Vert \bar{\phi}_{\delta} \Vert_{L^{2(p-1)}(\mathcal{O})}^{p-1} \right) 
\leq C \left(1 + \Vert \bar{\phi}_{\delta} \Vert_{V_1}^{p-1} \right) \quad &\mbox{ if } p \geq 2 \; \mbox{ when } \;  d=2,
\\
 \lvert F_\delta^\prime(\bar{\phi}_{\delta}) \rvert \leq C \left(1 + \Vert \bar{\phi}_{\delta} \Vert_{V_1}^{p-1} \right) \quad &\mbox{ if } 2 \leq p \leq 4 \; \mbox{ when } \;  d=3.
\end{aligned}
\end{equation}
Moreover, by the Sobolev inequalities, see for instance \cite[Theorem 10.1]{Friedman_1969}, there exists $C= C(\mathcal{O},p)$ such that for every $\delta>0$,
\begin{equation*}
\Vert \bar{\phi}_{\delta} \Vert_{L^{2(p-1)}(\mathcal{O})}
\leq C \Vert \bar{\phi}_{\delta} \Vert_{L^6(\mathcal{O})}^\frac{p+2}{2(p-1)} \Vert \bar{\phi}_{\delta} \Vert_{H^2(\mathcal{O})}^\frac{p-4}{2(p-1)} \quad \mbox{ if } 4< p \leq 6 \; \mbox{ when } \;  d=3.
\end{equation*}
This, jointly with  \eqref{Eqn-4.100} yields that for every $\delta>0$,
\begin{equation}\label{eq4.148-2}
\begin{aligned}
 \lvert F_\delta^\prime(\bar{\phi}_{\delta}) \rvert \leq C \left(1 +  \Vert \bar{\phi}_{\delta} \Vert_{L^6(\mathcal{O})}^\frac{p+2}{2} \Vert \bar{\phi}_{\delta} \Vert_{H^2(\mathcal{O})}^\frac{p-4}{2} \right) \quad \mbox{ if } 4< p \leq 6 \; \mbox{ when } \;  d=3.
\end{aligned}
\end{equation}
Let us choose and fix $\epsilon>0$. Without loss of generality, we assume $p\in(4,6)$. Thanks to \eqref{eq4.148-2}, the fact that $H^1(\mathcal{O}) \embed L^6(\mathcal{O})$ jointly with the Young inequality, we infer that there exists a constant $C=C(\mathcal{O},p,c_{F^\prime},\tilde{c}_F)$ independent of $\delta$ such that, 
\begin{equation}\label{eq4.150}
\begin{aligned}
\lvert F_\delta^\prime(\bar{\phi}_{\delta}) \rvert
\leq C \left(1 + \epsilon^{\frac{p-4}{2p-4}} + \epsilon^{-\frac{p-4}{6-p}} \Vert \bar{\phi}_{\delta} \Vert_{V_1}^\frac{p+2}{6-p} \right) + \epsilon \Vert \bar{\phi}_{\delta} \Vert_{H^2(\mathcal{O})} \mbox{ if } d=3.
\end{aligned}
\end{equation}
Next, plugging \eqref{eq4.150} into the RHS of \eqref{Eqn-bar-delta-phi-varphi-H_2-norm}, choosing $\epsilon=\frac{1}{\sqrt{2c_\ast}}$, we deduce in the case $d=3$ and when $p \in(4,6)$ that there exists a constant $C>0$ depending on $K,\,\eps,\,\eps_\Gamma,\,\mathcal{O},\,\Gamma,\,c_{F^\prime},\,\tilde{c}_F,\,c_{G^\prime},\,\tilde{c}_G,\,q$, and $p$ such 
that for any $\delta>0$,
\begin{equation}\label{Eqn-bar-phi-varphi-H_2-norm}
\frac12 \Vert (\bar{\phi}_{\delta},\bar{\varphi}_{\delta}) \Vert_{\mathcal{H}^2}^2
\leq C \left(1 + \Vert (\bar{\mu}_{\delta},\bar{\theta}_{\delta}) \Vert_{\mathbb{H}}^2 + \Vert (\bar{\phi}_{\delta},\bar{\varphi}_{\delta}) \Vert_{\mathbb{V}}^2 + \Vert \bar{\varphi}_{\delta} \Vert_{V_\Gamma}^{2(q-1)} + \Vert \bar{\phi}_{\delta} \Vert_{V_1}^\frac{2(p+2)}{6-p} \right).
\end{equation}
By inserting \eqref{eq4.148-1} into the RHS of \eqref{Eqn-bar-delta-phi-varphi-H_2-norm}, we infer that for every $\delta>0$,
\begin{equation}\label{Eqn-bar-phi-varphi-H_2-norm-1}
\Vert (\bar{\phi}_{\delta},\bar{\varphi}_{\delta}) \Vert_{\mathcal{H}^2}^2
\leq  C (1 + \Vert (\bar{\mu}_{\delta},\bar{\theta}_{\delta}) \Vert_{\mathbb{H}}^2 + \Vert (\bar{\phi}_{\delta},\bar{\varphi}_{\delta}) \Vert_{\mathbb{V}}^2 + \Vert \bar{\varphi}_{\delta} \Vert_{V_\Gamma}^{2(q-1)} + \Vert \bar{\phi}_{\delta} \Vert_{V_1}^{2(p-1)}),
\end{equation}
where $2 \leq p \leq 4$ if $d=3$ and $p\geq 2$ if $d=2$. 

\noindent
\textbf{Step4:} When $2 \leq p \leq 4$ if $d=3$ and $p\geq 2$ if $d=2$, the estimate \eqref{Eqn-(phi_{n,delta},varphi_{n,delta})-H2-estimate-1} is a direct consequence of \eqref{eq4.143}, \eqref{eq4.144}, \eqref{uniform_delta_estimate_varphi}, and \eqref{Eqn-bar-phi-varphi-H_2-norm-1}. \newline
When $4 < p \leq 6$ if $d=3$, \eqref{Eqn-(phi_{n,delta},varphi_{n,delta})-H2-estimate-1} is a direct consequence of \eqref{eq4.143}, \eqref{eq4.144}, \eqref{uniform_delta_estimate_varphi}, and \eqref{Eqn-bar-phi-varphi-H_2-norm}. This completes the proof of our claim, i.e. \eqref{Eqn-(phi_{n,delta},varphi_{n,delta})-H2-estimate-1}. \newline
From now on, by \eqref{eq4.148-1}, \eqref{eq4.148-2} together with \eqref{eq4.143}, \eqref{eq4.144}, \eqref{uniform_delta_estimate_varphi}, and \eqref{Eqn-(phi_{n,delta},varphi_{n,delta})-H2-estimate-1}, we deduce that for every $\delta>0$,
    \begin{equation}\label{eq4.148-3}
      \Vert F_\delta^\prime(\bar{\phi}_{\delta}) \Vert_{L^r (\bar{\Omega};L^2(0,T;L^2(\mathcal{O})))} \leq c_r.
    \end{equation}
\textbf{Further estimates} Fix $(v,v_\Gamma) \in \mathbb{V}$ such that $\Vert (v,v_\Gamma) \Vert_{\mathbb{V}} \leq 1$. Then from \eqref{eq4.130d}, we have
\begin{align*}
\duality{(\bar{\mu}_{\delta},\bar{\theta}_{\delta})}{(v,v_\Gamma)}{\mathbb{V}}{\mathbb{V}^\prime}
 &= \eps \int_{\mathcal{O}} \nabla \bar{\phi}_{\delta} \cdot \nabla v\, \d x + \int_{\mathcal{O}} \frac{1}{\eps} F_\delta^\prime(\bar{\phi}_{\delta}) v\, \d x + \eps_\Gamma \int_\Gamma \nabla_\Gamma \bar{\varphi}_{\delta} \cdot \nabla_\Gamma v_\Gamma \, \d S \\
 &\quad  + \int_\Gamma \frac{1}{\eps_\Gamma} G_\delta^\prime(\bar{\varphi}_{\delta}) v_\Gamma \, \d S + (1/K) \int_\Gamma (\bar{\varphi}_{\delta} - \bar{\phi}_{\delta}) (\eps v_\Gamma - \eps v) \, \d S.
\end{align*}
Therefore, for every $\delta>0$,
\begin{align*}
\lvert \duality{(\bar{\mu}_{\delta},\bar{\theta}_{\delta})}{(v,v_\Gamma)}{\mathbb{V}}{\mathbb{V}^\prime} \rvert
 &\leq \eps \lvert \nabla \bar{\phi}_{\delta} \rvert \lvert \nabla v \rvert + \frac{1}{\eps} \Vert F_\delta^\prime(\bar{\phi}_{\delta}) \Vert_{L^\frac{6}{5}(\mathcal{O})} \Vert v \Vert_{L^6(\mathcal{O})} + \frac{1}{\eps_\Gamma} \lvert \nabla_\Gamma \bar{\varphi}_{\delta} \rvert_{\Gamma} \lvert \nabla_\Gamma v_\Gamma \rvert_{\Gamma}
 \\
 &\quad  + \frac{1}{\eps_\Gamma} \lvert G_\delta^\prime(\bar{\varphi}_{\delta}) \rvert_\Gamma \lvert v_\Gamma \rvert_\Gamma + \eps (1/K) \lvert \bar{\varphi}_{\delta} - \bar{\phi}_{\delta} \rvert_\Gamma \lvert v_\Gamma - v \rvert_{\Gamma}.
\end{align*}
Next, since $V_1 \embed L^2(\Gamma)$, we infer there exists $C=C(\eps,\eps_\Gamma,K,\mathcal{O})$ such that for every $\delta>0$,
   \begin{equation}\label{eq4.150a}
       \Vert (\bar{\mu}_{\delta},\bar{\theta}_{\delta}) \Vert_{\mathbb{V}^\prime}
        \leq C (\lvert \bar{\varphi}_{\delta} - \bar{\phi}_{\delta} \rvert_\Gamma + \Vert (\bar{\phi}_{\delta},\bar{\varphi}_{\delta}) \Vert_{\mathbb{V}} + \Vert F_\delta^\prime(\bar{\phi}_{\delta}) \Vert_{L^\frac{6}{5}(\mathcal{O})} + \lvert G_\delta^\prime(\bar{\varphi}_{\delta}) \rvert_\Gamma).
   \end{equation}
Furthermore, arguing as in the proof of Proposition \ref{Proposition_third_priori_estimmate}, we have
    \begin{equation*}
      \|(\bar{\mu}_{\delta},\bar{\theta}_{\delta})\|_{\mathbb{H}}^2
      \leq \|(\bar{\mu}_{\delta},\bar{\theta}_{\delta})\|_{\mathbb{V}^\prime}^2 + 2 \|(\bar{\mu}_{\delta},\bar{\theta}_{\delta})\|_{\mathbb{V}^\prime} \Vert(\nabla \bar{\mu}_{\delta},\nabla_\Gamma \bar{\theta}_{\delta}) \Vert_{\mathbb{H}}.
    \end{equation*}
Therefore, by the H\"older inequality, we deduce that for every $\delta>0$,
\begin{align*}
& \bar{\mathbb{E}} \left(\int_0^T \Vert (\bar{\mu}_{\delta}(s),\bar{\theta}_{\delta}(s)) \Vert_{\mathbb{H}}^4\, \d s \right)^\frac{r}{4}
\leq C \bar{\mathbb{E}} \bigg[ \left(\int_0^T \Vert (\bar{\mu}_{\delta}(s),\bar{\theta}_{\delta}(s)) \Vert_{\mathbb{V}^\prime}^4\, \d s \right)^\frac{r}{4} \\
& + \left( \int_0^T \Vert (\bar{\mu}_{\delta}(s),\bar{\theta}_{\delta}(s)) \Vert_{\mathbb{V}^\prime}^2 \Vert(\nabla \bar{\mu}_{\delta}(s),\nabla_\Gamma \bar{\theta}_{\delta}(s)) \Vert_{\mathbb{H}}^2\, \d s\right)^\frac{r}{4}\bigg]
\leq C \bigg(\bar{\mathbb{E}} \sup_{s \in [0,T]} \Vert (\bar{\mu}_{\delta}(s),\bar{\theta}_{\delta}(s)) \Vert_{\mathbb{V}^\prime}^r
\\
& + C [\bar{\mathbb{E}} \sup_{s \in [0,T]} \Vert (\bar{\mu}_{\delta}(s),\bar{\theta}_{\delta}(s)) \Vert_{\mathbb{V}^\prime}^r]^{\frac12} \left[\bar{\mathbb{E}} \left(\int_0^T  \Vert(\nabla \bar{\mu}_{\delta}(s),\nabla_\Gamma \bar{\theta}_{\delta}(s)) \Vert_{\mathbb{H}}^2\, \d s\right)^\frac{r}{2} \right]^{\frac12} \bigg).
\end{align*}
From \eqref{eq4.143}, \eqref{uniform_delta_estimate_varphi}, \eqref{eq4.159}, \eqref{eq4.150a}, we infer there exists $c_r>0$ independent of $\delta$ such that
   \begin{equation}\label{Eqn-bar-mu-theta-delta-L_2--estimate}
    \Vert (\bar{\mu}_{\delta},\bar{\theta}_{\delta}) \Vert_{L^r(\bar{\Omega};L^4(0,T;\mathbb{H}))}\leq c_r.
   \end{equation}
Finally, from \eqref{Eqn-bar-mu-theta-delta-L_2--estimate} and similar reasoning as in \eqref{Eqn-(phi_{n,delta},varphi_{n,delta})-H2-estimate-1}, we find that for every $\delta>0$,
     \begin{equation}
       \bar{\mathbb{E}} \left(\int_0^T \Vert(\bar{\phi}_{\delta}(s),\bar{\varphi}_{\delta}(s)) \Vert_{\mathcal{H}^2}^4 \,\d s \right)^{r/4} \leq c_r.
    \end{equation}
\subsection{Tightness of the laws of the approximating sequences}
It follows from the above result that for every $T>0$, almost surely, the trajectories of the process $\bar{\bX}_\delta \coloneq (\bar{\phi}_{\delta},\bar{\varphi}_{\delta})$ belong to the space $\mathfrak{X}_{T,2}$, i.e. in 
   \begin{equation*}
     \mathfrak{X}_{T,2}= C([0,T];\mathbb{H}) \cap C([0,T];\mathbb{V}_w) \cap L_w^4(0,T;\mathcal{H}^2) \cap L^2(0,T;\mathbb{V}).
   \end{equation*}
Moreover, we are in a position to perform a stochastic compactness argument.
\begin{lemma}
Assume $\delta \in (0,1)$. Then, 
\item[(1)] the family of laws of $\lbrace (\bar{\phi}_{\delta},\bar{\varphi}_{\delta})\rbrace_{\delta>0}$ is tight in the space $\mathfrak{X}_{T,2}$.
\item[(2)] If we identify the Wiener process $\bar{W}$ with a constant sequence $(\bar{W}_{1,\delta})_{\delta>0}$ and the Wiener process $\bar{W}_{\Gamma}$ with a constant sequence $( \bar{W}_{2,\delta})_{\delta>0}$, then the corresponding family of laws is tight in the space $C([0,T];U_0)$ and in the space $C([0,T];U_0^\Gamma)$, respectively. 
\end{lemma}
\begin{proof}
The proof can be done by arguing exactly as in the proof of Proposition \ref{Eqn-tightness-delta}.
\end{proof}
We have the following auxiliary result.
\begin{proposition}\label{Jakubowski-Skorokhod-representation theorem-1}
Suppose $\delta \in (0,1)$. Then, there exists a subsequence $(\delta_k)_{k \in \mathbb{N}}$, a complete probability space $(\tilde{\Omega},\tilde{\mathcal{F}},\tilde{\mathbb{P}})$, and on this space, $\mathfrak{X}_T$-valued random variables $((\tilde{\phi}_{\delta_k},\tilde{\varphi}_{\delta_k}),\tilde{W}_{1,\delta_k},\tilde{W}_{2,\delta_k})$ and $((\tilde{\phi},\tilde{\varphi}),\tilde{W},\tilde{W}_{\Gamma})$, $k \in \mathbb{N}$ such that
\begin{align}
\label{Eqn-equal-of-law-of-sequence-1}
&\mathcal{L}_{\mathfrak{X}_T}(((\tilde{\phi}_{\delta_k},\tilde{\varphi}_{\delta_k}),\tilde{W}_{1,\delta_k},\tilde{W}_{2,\delta_k}))= \mathcal{L}_{\mathfrak{X}_T}(((\bar{\phi}_{\delta_k},\bar{\varphi}_{\delta_k}),\bar{W}_{1,\delta_k},\bar{W}_{2,\delta_k})),
\\
\label{Eqn-almots-sure-convergenc-of-sequence-2}
&((\tilde{\phi}_{\delta_k},\tilde{\varphi}_{\delta_k}),\tilde{W}_{1,\delta_k},\tilde{W}_{2,\delta_k}) \to ((\tilde{\phi},\tilde{\varphi}),\tilde{W},\tilde{W}_{\Gamma}) \mbox{ in } \mathfrak{X}_T \mbox{ as } k \to \infty, \; \; \tilde{\mathbb{P}}\mbox{-a.s.}
\end{align}
\end{proposition}
\begin{proof}
The proof can be done by arguing exactly as in Proposition \ref{Jakubowski-Skorokhod-representation theorem}.
\end{proof}
For the sake of brevity, we will denote the above sequences by $((\tilde{\phi}_{\delta},\tilde{\varphi}_{\delta}),\tilde{W}_{1,\delta},\tilde{W}_{2,\delta})_{\delta \in (0,1)}$ and $((\bar{\phi}_{\delta},\bar{\varphi}_{\delta}),\bar{W}_{1,\delta},\bar{W}_{2,\delta})_{\delta \in (0,1)}$.
\subsection{Passage to the limit as delta tends to zero} \label{Sect_passage to the limit_delta}
The goal of this section is to let $\delta \to 0$ and then derive the martingale solution of the original problem \eqref{eq1.1a}-\eqref{eq1.10a}. Note that the argument of the proof is similar to the one carry in Subsection \ref{subs-4.4}, thus we will omit some details for the sake of brevity.

\noindent
Fix $r \geq 2$. By the previous estimates, arguing as in Subsection \ref{subs-4.4},  we infer that up to a subsequence still labeled by the same subscript,
\begin{equation}\label{eq4.122b}
\begin{aligned}
(\tilde{\phi}_{\delta},\tilde{\varphi}_{\delta}) \to (\tilde{\phi},\tilde{\varphi}) &\mbox{ in } L^\ell(\bar{\Omega};C([0,T];\mathbb{H})) \cap L^r(\tilde{\Omega};L^2([0,T];\mathbb{V})), \; \;  \forall \ell< r,
\\
(\tilde{\phi}_{\delta},\tilde{\varphi}_{\delta}) \stackrel{*}{\rightharpoonup} (\tilde{\phi},\tilde{\varphi}) &\mbox{ in }  L^r(\tilde{\Omega};L^\infty(0,T;\mathbb{V})), 
  \\
(\tilde{\phi}_{\delta},\tilde{\varphi}_{\delta}) \rightharpoonup (\tilde{\phi},\tilde{\varphi}) &\mbox{ in }  L^r(\tilde{\Omega};L^4(0,T;\mathcal{H}^2)),
    \\
(\tilde{\mu}_{\delta},\tilde{\theta}_{\delta}) \rightharpoonup (\tilde{\mu},\tilde{\theta})  &\mbox{ in }  L^r(\tilde{\Omega};L^2(0,T;\mathbb{V})) \cap L^r(\tilde{\Omega};L^4(0,T;\mathbb{H})), 
      \\
(\tilde{W}_{1,\delta},\tilde{W}_{2,\delta}) \to (\tilde{W},\tilde{W}_{\Gamma}) &\mbox{ in } L^\ell(\tilde{\Omega};C([0,T];U_0 \times U_0^\Gamma)), \; \;  \forall \ell< r,
\\
\tilde{\bu}_{\delta} \rightharpoonup \tilde{\bu}  &\mbox{ in }  L^r(\tilde{\Omega};L^2(0,T;V)), 
  \\
\tilde{\bu}_{\delta} \lvert_\Gamma \rightharpoonup \tilde{\bu} \lvert_\Gamma  &\mbox{ in }  L^r(\tilde{\Omega};L^2(0,T;L^2(\Gamma))),
    \\
(F_1(\tilde{\phi}_\delta),F_2(\tilde{\varphi}_\delta)) \rightharpoonup (F_1(\tilde{\phi}),F_2(\tilde{\varphi}))  &\mbox{ in }  L^r(\tilde{\Omega};L^2(0,T;\mathbb{H})).
\end{aligned}
\end{equation}
Moreover, by the strong convergences of $\tilde{\phi}_{\delta}$ and $\tilde{\varphi}_{\delta}$ above, we have, modulo the extraction of a subsequence $(\tilde{\phi}_{\delta_n})_{\delta_n \in (0,1)}$ and $(\tilde{\varphi}_{\delta_n})_{\delta_n \in (0,1)}$,
      \begin{equation}\label{eq4.123b}
        (\tilde{\phi}_{\delta_n},\tilde{\varphi}_{\delta_n}) \to (\tilde{\phi},\tilde{\varphi}) \mbox{-a.e. in } Q_T \times \tilde{\Omega}.
      \end{equation}
Furthermore, we claim that
\begin{equation}\label{eq4.161}
\begin{aligned}
F_{\delta_n}^\prime(\tilde{\phi}_{\delta_n}) \rightharpoonup F^\prime(\tilde{\phi}) &\mbox{ in }  L^r(\tilde{\Omega};L^2(0,T;L^2(\mathcal{O}))),
       \\
G_{\delta_n}^\prime(\tilde{\varphi}_{\delta_n}) \stackrel{*}\rightharpoonup G^\prime(\tilde{\varphi}) &\mbox{ in }  L^r(\tilde{\Omega};L^\infty(0,T;L^2(\Gamma))).
\end{aligned}
\end{equation}
In fact, from the uniform bound \eqref{eq4.148-3} and the Banach-Alaoglu theorem, we have, up to a subsequence, 
   \begin{equation*}
      F_{\delta_n}^\prime(\tilde{\phi}_{\delta_n}) \rightharpoonup \mathcal{K}  \mbox{ in } L^r (\tilde{\Omega};L^2(0,T;L^2(\mathcal{O}))).
    \end{equation*}
On the other hand, we have $F_{\delta_n}^\prime= \mathbb{A}_{\delta_n} - \tilde{c}_F I= \mathbb{A}(J_{\delta_n}(\tilde{\phi}_{\delta_n})) - \tilde{c}_F I$ and 
\begin{align*}
\lvert J_{\delta_n}(\tilde{\phi}_{\delta_n}) - \tilde{\phi} \rvert
\leq \lvert J_\delta(\tilde{\phi}_{\delta_n}) - \tilde{\phi}_{\delta_n} \rvert + \lvert \tilde{\phi}_{\delta_n} - \tilde{\phi} \rvert
&= \delta \lvert \mathbb{A}_{\delta_n} (\tilde{\phi}_{\delta_n}) \rvert  + \lvert \tilde{\phi}_{\delta_n} - \tilde{\phi} \rvert
\\
&\leq {\delta_n} \lvert \mathbb{A} (\tilde{\phi}_{\delta_n}) \rvert  + \lvert \tilde{\phi}_{\delta_n} - \tilde{\phi} \rvert.
\end{align*}
Therefore, since $\tilde{\phi}_{\delta_n}\to \tilde{\phi}$-a.e. in $Q_T \times \tilde{\Omega}$, cf. \eqref{eq4.123b}, we infer $J_{\delta_n}(\tilde{\phi}_{\delta_n})\to \tilde{\phi}$-a.e. in $Q_T \times \tilde{\Omega}$.  Hence, $\mathbb{A}(J_{\delta_n}(\tilde{\phi}_{\delta_n})) \to \mathbb{A}(\tilde{\phi})$ a.e. in $\tilde{\Omega} \times Q_T$. As a direct consequence,
   \begin{equation*}
     \mathcal{K}= F^\prime(\tilde{\phi}) \mbox{ a.e. in }  Q_T \times \tilde{\Omega}.
   \end{equation*}
This proof the first part of \eqref{eq4.161}. Similarly, we deduce the second part of \eqref{eq4.161}. 
\newline
Note that $J_{\delta_n}(\tilde{\phi}_{\delta_n}) \to \tilde{\phi}$ a.e. in $Q_T \times \tilde{\Omega}$. Hence, by the weak lower semicontinuity, i.e. the lower semicontinuity of $L^1(\mathcal{O})$,
and the Fatou Lemma, we infer that
      \begin{equation*}
        \tilde{\mathbb{E}} \Vert F(\tilde{\phi}(t)) \Vert_{L^1(\mathcal{O})}
        \leq \liminf_{n \to \infty} \tilde{\mathbb{E}} \Vert F_{\delta_n}(\tilde{\phi}_{\delta_n}(t)) \Vert_{L^1(\mathcal{O})}, \; \; t \in (0,T).
      \end{equation*}
Similarly,
   \begin{equation*}
     \tilde{\mathbb{E}} \Vert G(\tilde{\varphi}(t)) \Vert_{L^1(\Gamma)}
     \leq \liminf_{n \to \infty} \tilde{\mathbb{E}} \Vert G_{\delta_n}(\tilde{\varphi}_{\delta_n}(t)) \Vert_{L^1(\Gamma)}, \; \; t \in (0,T).
   \end{equation*}
Observe that, under the assumption \eqref{F''_and G''_additional_condition}, the limit object in \eqref{eq4.161}${\text{-1)}}$ is, in fact, an element of the space $L^r(\tilde{\Omega};L^2(0,T;V_1))$, while the limit object in \eqref{eq4.161}${\text{-2)}}$ also belongs to $L^r(\tilde{\Omega};L^2(0,T;V_\Gamma))$. 
\newline
Let us prove the previous assertion in the case $d=3$, since the proof in case $d=2$ is similar. 
First, by the H\"older inequality and Assumption \eqref{F''_and G''_additional_condition}, we infer that
\begin{align*}
\lvert F^\bis(\tilde{\phi}) \nabla \tilde{\phi} \rvert^2
& = \int_{\mathcal{O}} \lvert  F^\bis(\tilde{\phi}) \rvert^2 \lvert \nabla \tilde{\phi} \rvert^2 \, \d x
\leq 2 c_{F^\bis}^2 \left[ \lvert \nabla \tilde{\phi} \rvert^2 + \Vert \tilde{\phi} \Vert_{L^{3(p-2)}(\mathcal{O})}^{2(p-2)} \Vert \nabla \tilde{\phi} \Vert_{\mathbb{L}^6(\mathcal{O})}^2 \right],
\\
\lvert G^\bis(\tilde{\varphi}) \nabla_\Gamma \tilde{\varphi} \rvert_\Gamma^2
&= \int_{\Gamma} \lvert G^\bis(\tilde{\varphi}) \rvert^2 \lvert \nabla_\Gamma \tilde{\varphi} \rvert^2 \, \d S
\leq 2 c_{G^\bis}^2 \left[ \lvert \nabla_\Gamma \tilde{\varphi} \rvert_\Gamma^2 + \Vert \tilde{\varphi} \Vert_{L^{3(q-2)}(\mathcal{O})}^{2(q-2)} \Vert \nabla_\Gamma \tilde{\phi} \Vert_{\mathbb{L}^6(\Gamma)}^2 \right].
\end{align*}
Next, using the Sobolev embeddings $H^1(\mathcal{O}) \embed L^{3(p-2)}(\mathcal{O})$ and $H^1(\mathcal{O}) \embed L^6(\mathcal{O})$, we deduce that there exists $C=C(c_{F^\bis},\mathcal{O},p)$ such that
   \begin{align}
    \lvert F^\bis(\tilde{\phi}) \nabla \tilde{\phi} \rvert^2
    \leq C \left[ \lvert \nabla \tilde{\phi} \rvert^2 + \Vert \tilde{\phi} \Vert_{V_1}^{2(p-2)} \Vert \tilde{\phi} \Vert_{H^2(\mathcal{O})}^2 \right],
  \end{align}
which in turn, yields that
\begin{equation}
\begin{aligned}
&\tilde{\mathbb{E}} \left(\int_0^T \lvert F^\bis(\tilde{\phi}(s)) \nabla \tilde{\phi}(s) \rvert^2 \, \d s \right)^{\frac{r}{2}}
\\
& \leq C \tilde{\mathbb{E}} \sup_{s \in [0,T]} \lvert \nabla \tilde{\phi}(s) \rvert^r + \left[ \tilde{\mathbb{E}} \sup_{s \in [0,T]} \Vert \tilde{\phi}(s) \Vert_{V_1}^{2r(p-2)} \right]^{\frac12} \left[ \tilde{\mathbb{E}} \left(\int_0^T \Vert \tilde{\phi}(s) \Vert_{H^2(\mathcal{O})}^2 \, \d s \right)^r \right]^{\frac12}.
\end{aligned}
\end{equation}
This, jointly with \eqref{eq4.161} and the properties of the limit object $\tilde{\phi}$ gives that 
    \begin{equation}\label{Eqn-tilde-F-prime(tilde{phi})-L2(0,T;V1-norm}
        F^\prime(\tilde{\phi}) \in L^r(\tilde{\Omega};L^2(0,T;V_1)).   
     \end{equation}
Similarly, we can show that there exists $C=C(c_{G^\bis},\Gamma,q)$ such that
   \begin{align}
    \lvert G^\bis(\tilde{\varphi}) \nabla_\Gamma \tilde{\varphi} \rvert_\Gamma^2
    \leq C \left[ \lvert \nabla_\Gamma \tilde{\varphi} \rvert_\Gamma^2 + \Vert \tilde{\varphi} \Vert_{V_\Gamma}^{2(q-2)} \Vert \tilde{\varphi} \Vert_{H^2(\Gamma)}^2 \right],
  \end{align}
and then from the properties of the process $\tilde{\varphi}$, we deduce that
     \begin{equation}\label{Eqn-tilde-G-prime(tilde{varphi})-L2(0,T;V-Gamma-norm}
       G^\prime(\tilde{\varphi}) \in L^r(\tilde{\Omega};L^2(0,T;V_\Gamma)). 
     \end{equation}
Arguing as in the proof of Proposition \ref{Prop:Proposition-4}, we obtain
  \begin{equation}\label{eqn-4.119}
   \int_0^t \bG^\delta(\tilde{\Upsilon}_\delta(s)) \, \d \tilde{\mathcal{W}}^{\delta}(s) \to \int_0^t \bG(\tilde{\Upsilon}(s)) \, \d \tilde{\mathcal{W}}(s) ~ \mbox{ in probability in } L^2(0,T;\mathbb{H}),
 \end{equation}
where 
$\tilde{\mathcal{W}} \coloneq (\tilde{W},\tilde{W}^\Gamma)^{tr},\, 
\tilde{\mathcal{W}}^\delta \coloneq (\tilde{W}_{1,\delta},\tilde{W}_{2,\delta})^{tr},\; 
\tilde{\Upsilon} \coloneq (\tilde{\phi}, \tilde{\varphi})$, 
$\tilde{\Upsilon}_\delta= (\tilde{\phi}_\delta,\tilde{\varphi}_\delta)$, 
\\
$\bG^\delta(\tilde{\Upsilon}_\delta) \coloneq \mathrm{diag}(F_1(\tilde{\phi}_\delta),F_2(\tilde{\varphi}_\delta))$, 
and
$\bG(\tilde{\Upsilon})\coloneq \mathrm{diag}(F_1(\tilde{\phi}),F_2(\tilde{\varphi}))$.

\noindent
Let us now complete the proof of the main result of this paper, i.e. Theorem \ref{thm-first_main_theorem}.
\section{Proof of Theorem \ref{thm-first_main_theorem}}\label{eqn-thm-first_main_theorem}
The proof of this theorem is given is three parts as follows.
\begin{proof}[Proof of Theorem \ref{thm-first_main_theorem}]

\begin{trivlist}
\item[Part 1.] 
Considering the previous convergences, using the continuity and boundedness of the maps $M_\mathcal{O},\,M_\Gamma,\,\nu,\,\lambda$ and $\Gamma$ together with the DCT, we see that, after letting $\delta \to 0^+$, the process $(\tilde{\bu}(t), \tilde{\phi}(t), \tilde{\varphi}(t), \tilde{\mu}(t),\tilde{\theta}(t), \tilde{\mathcal{W}}(t))_{t \in [0,T]}$ satisfies the weak variational formulations \eqref{eq3.24}-\eqref{eq3.26} in the probabilistic space $(\tilde{\Omega}, \tilde{\mathcal{F}},\tilde{\mathbb{F}}, \tilde{\mathbb{P}})$. $\tilde{\mathbb{F}}= \{\tilde{\mathcal{F}}_t\}_{t\in[0,T]}$ being the natural filtration.
\newline
Moreover, by arguing as in \eqref{Eqn-4.83}, we infer that 
    \begin{equation*}
      \tilde{\mathbb{E}} \sup_{s \in [0,T]} \Vert(\tilde{\phi}(s),\tilde{\varphi}(s)) \Vert_{\mathbb{V}}^2
       \leq \liminf_{\delta \to \infty} \int_{\tilde{\Omega}} \sup_{s \in [0,T]} \Vert(\tilde{\phi}_\delta(s,\tilde{\omega}),\tilde{\varphi}_\delta(s,\tilde{\omega})) \Vert_{\mathbb{V}}^2 \, \d \tilde{\mathbb{P}}(\tilde{\omega})< \infty.
    \end{equation*}
 Furthermore, by the lemma due to Strauss, see \cite{Strauss_1966,Vishik+Fursikov_1988}, that says:
\[
L^\infty(0,T;\mathbb{V}) \cap C([0,T];\mathbb{H}_w) \hookrightarrow  C([0,T];\mathbb{V}_w),
\]
and Proposition E.1 from \cite{Brzezniak+Ferrario+Zanella_2024}, we infer that $(\tilde{\phi},\tilde{\varphi}) \in C_w([0,T];\mathbb{V})$, $\tilde{\mathbb{P}}$-a.s. Therefore, we deduce that the process $(\tilde{\phi},\tilde{\varphi})$ satisfies \eqref{Eqn-3.11}. This completes the proof of the first part of Theorem \ref{thm-first_main_theorem}, i.e.  there exists a weak martingale solution
$(\tilde{\bu}(t), \tilde{\phi}(t), \tilde{\varphi}(t), \tilde{\mu}(t), \tilde{\theta}(t), \tilde{\mathcal{W}}(t))_{t \in [0,T]}$
to system \eqref{eq1.1a}-\eqref{eq1.10a} in the sense of Definition \ref{def3.4}. Moreover, the process $(\tilde{\mu},\tilde{\theta})$ satisfies: $(\tilde{\mu},\tilde{\theta}) \in L^4(0,T;\mathbb{H})$, $\tilde{\mathbb{P}}$-a.s., cf. \eqref{eq4.122b}. 
\item[Part 2.] 
The proof of \eqref{Eqn-tilde-(phi,varphi)-L4(0,T;H2-norm}${\text{-1)}}$ is a direct consequence of \eqref{eq4.122b} and the Banach-Alaoglu Theorem, while the proof of \eqref{Eqn-tilde-(phi,varphi)-L4(0,T;H2-norm}${\text{-2)}}$ is a direct consequence of \eqref{Eqn-tilde-F-prime(tilde{phi})-L2(0,T;V1-norm} and \eqref{Eqn-tilde-G-prime(tilde{varphi})-L2(0,T;V-Gamma-norm}.
\newline
The fact that the process $(\tilde{\phi},\tilde{\varphi},\tilde{\mu},\tilde{\theta})$ satisfies the equations \eqref{Eqn-identification-tilde-mu} and \eqref{Eqn-identification-tilde-theta} in the strong sense, follows from the properties of $(\tilde{\phi},\tilde{\varphi},\tilde{\mu},\tilde{\theta})$, the fact that $(F^\prime(\tilde{\phi}),G^\prime(\tilde{\varphi})) \in L^2(0,T;\mathbb{H})$, $\tilde{\mathbb{P}}\mbox{-a.s.}$, and the weak variational formulation \eqref{eq3.26}.
\newline
It remains to show that the process $(\tilde{\phi},\tilde{\varphi})$ enjoys the regularity property stated in \eqref{Eqn-tilde-(phi,varphi)-L2(0,T;H3-norm}, that is, $(\tilde{\phi},\tilde{\varphi}) \in L^2(0,T;\mathcal{H}^3), \;\; \tilde{\mathbb{P}}\mbox{-a.s.}$
Indeed, since $\mathcal{O}$ is assumed to be of class $C^3$ and since the process $(\tilde{\phi},\tilde{\varphi})$ satisfies \eqref{Eqn-identification-tilde-mu}-\eqref{Eqn-identification-tilde-theta}, appealing to theory of elliptic problems with bulk-surface coupling, see, for instance, \cite[Theorem 3.3]{Knopf+Liu_2021}), we infer there exists a constant $C>0$ such that
    \begin{equation*}
        \Vert (\tilde{\phi}(t),\tilde{\varphi}(t)) \Vert_{\mathcal{H}^3}^2 
        \leq C (\Vert \tilde{\mu}(t) \Vert_{V_1}^2 + \Vert F^\prime(\tilde{\phi}(t)) \Vert_{V_1}^2 + \Vert \tilde{\theta}(t) \Vert_{V_\Gamma}^2 + \Vert G^\prime(\tilde{\varphi}(t)) \Vert_{V_\Gamma}^2), \; \; t \in [0,T],
    \end{equation*}
from which we easily complete the proof of \eqref{Eqn-tilde-(phi,varphi)-L2(0,T;H3-norm}.
\item[Part 3.] 
Proof of \eqref{eq3.29}. First, note that $\mathcal{L}_{\mathfrak{X}_T}(((\tilde{\phi}_{\delta},\tilde{\varphi}_{\delta})))= \mathcal{L}_{\mathfrak{X}_T}(((\bar{\phi}_{\delta},\bar{\varphi}_{\delta})))$, see \eqref{Eqn-equal-of-law-of-sequence-1}, hence so $(\tilde{\phi}_{\delta},\tilde{\varphi}_{\delta})$ satisfies \eqref{eq4.133}. The next step now is to pass to the limit in the resulting inequality satisfied by $(\tilde{\phi}_{\delta},\tilde{\varphi}_{\delta})$.
From the embedding  $V_1 \embed L^2(\Gamma)$, we infer that for every $\delta>0$,
\begin{align*}
\sum_{k=1}^\infty \Vert F_1(\tilde{\phi}_\delta) e_{1,k} \Vert_{L^2(\Gamma)}^2 
\leq C(\Gamma,\mathcal{O}) \Vert F_1(\tilde{\phi}_\delta) \Vert_{\mathscr{T}_2(U,V_1)}^2  
\leq C(\Gamma,\mathcal{O},C_1)( 1 + \lvert \nabla \tilde{\phi}_\delta \rvert^2).
\end{align*}
Next, since $F_1$ and $F_2$ are Lipschitz, we infer that
\begin{align*}
\lim_{\delta \to 0} \tilde{\mathbb{E}} \int_0^t \Vert F_{1}(\tilde{\phi}_\delta(s)) \Vert_{\mathscr{T}_2(U,L^2(\mathcal{O}))}^2 \,\d s&= \tilde{\mathbb{E}} \int_0^t \Vert F_{1}(\tilde{\phi}(s)) \Vert_{\mathscr{T}_2(U,L^2(\mathcal{O}))}^2 \,\d s, 
\\
\lim_{\delta \to 0} \tilde{\mathbb{E}} \int_0^t \Vert F_2(\tilde{\varphi}_\delta(s)) \Vert_{\mathscr{T}_2(U_\Gamma,L^2(\Gamma))}^2 \, \d s&= \tilde{\mathbb{E}} \int_0^t \Vert F_2(\tilde{\varphi}(s)) \Vert_{\mathscr{T}_2(U_\Gamma,L^2(\Gamma))}^2 \, \d s, \; \; t \in [0,T].
\end{align*}
\noindent
Now, arguing as in the proof of \eqref{convergence-4.132}, we deduce that
\begin{align*}
\tilde{\mathbb{E}} \int_{Q_t} \nu(\tilde{\phi}) \lvert D\tilde{\bu} \rvert^2  \,\d x \,\d s  &\leq \liminf_{\delta \to 0} \tilde{\mathbb{E}} \int_{Q_t} \nu(\tilde{\phi}_\delta) \lvert D\tilde{\bu}_\delta \rvert^2  \,\d x \,\d s, 
\\
\tilde{\mathbb{E}} \int_{Q_t} \lambda(\tilde{\phi}) \lvert \tilde{\bu} \rvert^2 \,\d x \,\d s &\leq \liminf_{\delta \to 0} \tilde{\mathbb{E}} \int_{Q_t} \lambda(\tilde{\phi}_\delta) \lvert \tilde{\bu}_\delta \rvert^2 \,\d x \,\d s,
    \\
\tilde{\mathbb{E}} \int_{\Sigma_t} \gamma(\tilde{\varphi}) \lvert \tilde{\bu} \rvert^2 \,\d S\,\d s &\leq \liminf_{\delta \to 0} \tilde{\mathbb{E}} \int_{\Sigma_t} \gamma(\tilde{\varphi}_\delta) \lvert \tilde{\bu}_\delta \rvert^2 \,\d S\,\d s,
          \\
\tilde{\mathbb{E}} \int_{Q_t} M_{\mathcal{O}}(\tilde{\phi}) \lvert \nabla \tilde{\mu} \rvert^2 \, \d x \,\d s &\leq \liminf_{\delta \to 0} \tilde{\mathbb{E}} \int_{Q_t} M_{\mathcal{O}}(\tilde{\phi}_\delta) \lvert \nabla \tilde{\mu}_\delta \rvert^2 \, \d x \,\d s,
                \\
\tilde{\mathbb{E}} \int_{\Sigma_t} M_\Gamma(\tilde{\varphi}) \lvert \nabla_\Gamma \tilde{\theta} \rvert^2 \,\d S \,\d s &\leq \liminf_{\delta \to 0} \tilde{\mathbb{E}} \int_{\Sigma_t} M_\Gamma(\tilde{\varphi}_\delta) \lvert \nabla_\Gamma \tilde{\theta}_\delta \rvert^2 \,\d S \,\d s.
\end{align*}
Hereafter, we claim that
\begin{equation}\label{Eqn-4.117}
\begin{aligned}
& \lim_{\delta \to 0} \tilde{\mathbb{E}} \int_0^t \sum_{k=1}^\infty \int_{\mathcal{O}} F^\bis_\delta(\tilde{\phi}_\delta) \lvert F_1(\tilde{\phi}_{\delta})e_{1,k} \rvert^2 \, \d x \,\d s
= \tilde{\mathbb{E}} \int_0^t \sum_{k=1}^\infty \int_{\mathcal{O}} F^\bis(\tilde{\phi}) \lvert F_1(\tilde{\phi}) e_{1,k} \rvert^2 \, \d x \,\d s, 
\\
& \lim_{\delta \to 0} \tilde{\mathbb{E}} \int_0^t  \sum_{k=1}^\infty \int_{\Gamma} G^\bis_\delta(\tilde{\varphi}_\delta) \lvert F_2(\tilde{\varphi}_{\delta})e_{2,k} \rvert^2 \, \d S\,\d s
= \tilde{\mathbb{E}} \int_0^t  \sum_{k=1}^\infty \int_{\Gamma} G^\bis(\tilde{\varphi}) \lvert F_2(\tilde{\varphi}) e_{2,k} \rvert^2\, \d S\,\d s.
\end{aligned}
\end{equation}
We preliminary recall that $J_\delta$ is the inverse function of $f_\delta: \mathbb{R} \ni s \mapsto (I + \delta \mathbb{A})(s) \in \mathbb{R}$, which is differentiable with $f^\prime_\delta(s)= 1 + \delta (F^\bis(s) + \tilde{c}_F)>0$. This entails that $\mathbb{A}_\delta$ is differentiable in $\mathbb{R}$. Notice also that $\mathbb{A}_\delta(s)= \mathbb{A}(J_\delta(s))$. Then, by differentiation and \ref{item:P4},
\begin{align*}
F_\delta^\bis(s)
= \mathbb{A}^\prime_\delta(s) - \tilde{c}_F 
= J^\prime_\delta(r) F^\bis(J_\delta(s)) + \tilde{c}_F (J^\prime_\delta(s)-1), \; \; s \in \mathbb{R}.
\end{align*}
On the other hand, by \eqref{condition_ F'_and_F''} and \ref{item:P4} in conjunction with the fact that $J_\delta^\prime(s) \in (0,1)$, $s \in \mathbb{R}$, we infer that for all $s \in \mathbb{R}$,
\begin{align*}
&\lvert F^\bis_\delta(s) \rvert 
\leq \lvert J^\prime_\delta(r) F^\bis(J_\delta(s)) \rvert + \tilde{c}_F \lvert J^\prime_\delta(s) - 1 \rvert 
\leq \lvert F^\bis(J_\delta(s)) \rvert + 2 \tilde{c}_F
\leq  C_F F(J_\delta(s)) + C_F + 2 \tilde{c}_F
\\
&\leq  C_F \tilde{F}(J_\delta(s)) + C_F + 2 \tilde{c}_F
\leq C_F \tilde{F}(s) + C_F + 2 \tilde{c}_F
= C_F F(s) +  \frac{ C_F \tilde{c}_F}{2} s^2 + C_F + 2 \tilde{c}_F.
\end{align*}
This implies that for every $\delta>0$,
\begin{align*}
&\lvert F^\bis_\delta(\tilde{\phi}_\delta) \rvert \lvert F_1(\tilde{\phi}_{\delta})e_{1,k} \rvert^2
\leq \left[C_F F(\tilde{\phi}_\delta) +  \frac{ C_F \tilde{c}_F}{2} \lvert \tilde{\phi}_\delta \rvert^2 + C_F + 2 \tilde{c}_F \right] \lvert F_1(\tilde{\phi}_{\delta})e_{1,k} \rvert^2
\\
&\leq \tilde{C}_1 \left[C_F F(\tilde{\phi}_\delta) +  \frac{ C_F \tilde{c}_F}{2} \lvert \tilde{\phi}_\delta \rvert^2 + C_F + 2 \tilde{c}_F \right]
\leq C(\tilde{C}_1,C_F,\tilde{c}_F) \left[1 +  F(\tilde{\phi}_\delta) + \lvert \tilde{\phi}_\delta \rvert^2 \right].
\end{align*}
Besides, let $\mu$ be the measure on $\bar{\Omega} \times Q_T$. By \eqref{F'_and G'_additional_condition} and \eqref{eq4.122b}, we infer that
\[
\lim_{M \to \infty} \left(\sup_{\delta \in (0,1)} \left\{ \int_{\{(t,x,\tilde{\omega}) \in Q_T \times \tilde{\Omega}: \lvert 1 +  F(\tilde{\phi}_\delta(t,x,\tilde{\omega})) + \lvert \tilde{\phi}_\delta(t,x,\tilde{\omega}) \rvert^2 \rvert> M\}} \lvert 1 +  F(\tilde{\phi}_\delta) + \lvert \tilde{\phi}_\delta \rvert^2 \rvert \, \d \mu \right\} \right)= 0.
\]
Hence,
\[
\lim_{M \to \infty} \left(\sup_{\delta \in (0,1)} \left\{ \int_{\{(t,x,\tilde{\omega}) \in Q_T \times \tilde{\Omega}: \lvert F^\bis_\delta(\tilde{\phi}_\delta) \rvert \lvert F_1(\tilde{\phi}_{\delta})e_{1,k} \rvert^2> M\}} \lvert F^\bis_\delta(\tilde{\phi}_\delta) \rvert \lvert F_1(\tilde{\phi}_{\delta})e_{1,k} \rvert^2  \, \d \mu \right\} \right)= 0,
\]
i.e. the sequence $(\lvert F^\bis_\delta(\tilde{\phi}_\delta) \rvert \lvert F_1(\tilde{\phi}_{\delta})e_{1,k} \rvert^2)_\delta$ is uniformly integrable. Moreover, since 
\newline
$\lvert F^\bis_\delta(\tilde{\phi}_\delta) \rvert \lvert F_1(\tilde{\phi}_{\delta})e_{1,k} \rvert^2 \to \lvert F^\bis(\tilde{\phi}) \rvert \lvert F_1(\tilde{\phi})e_{1,k} \rvert^2$ a.e. in $Q_T \times \tilde{\Omega}$ 
\dela{and for each $\eta>0$,
\[
\lim_{\delta \to 0} \mu \{(t,x,\tilde{\omega}) \in Q_T \times \tilde{\Omega}: \lvert \tilde{\phi}_\delta(t,x,\tilde{\omega}) - \tilde{\phi}(t,x,\tilde{\omega}) \rvert> \eta\}=0,
\]
and the map $\mathbb{R} \ni s \mapsto \lvert F^\bis_\delta(s) \rvert \lvert \sigma_k(s) \rvert^2 \in [0,\infty)$ is continuous, we deduce that for each $\eps>0$,
\begin{align*}
\lim_{\delta \to 0} \mu \{(t,x,\tilde{\omega}) \in Q_T \times \tilde{\Omega}: &\lvert \lvert F^\bis_\delta(\tilde{\phi}_\delta(t,x,\tilde{\omega})) \rvert \lvert \sigma_k(\tilde{\phi}_\delta(t,x,\tilde{\omega})) \rvert^2 
\\
&- \lvert F^\bis(\tilde{\phi}(t,x,\tilde{\omega})) \rvert \lvert \sigma_k(\tilde{\phi}(t,x,\tilde{\omega})) \rvert^2 \rvert> \eps\}=0.
\end{align*}
Invoking now} we can apply the Vitali Convergence Theorem and deduce the first part of \eqref{Eqn-4.117}. Similarly we can prove the second part of \eqref{Eqn-4.117}.
\newline
Therefore, \eqref{eq3.29} follows from the previous convergences results and the weak lower semicontinuity of the norm. This completes the proof of Theorem \ref{thm-first_main_theorem}.
\end{trivlist}
\end{proof}

\appendix

\section{Auxilliary results}

We now state and prove the following result.
\begin{proposition}\label{Proposition_third_priori_estimmate}
Fix $r \geq 2$ and $\delta \in (0,1)$. Under the assumptions of Proposition \ref{Proposition_first_priori_estimmate}, we have the following result: There exists a constant $C_{3,\delta}>0$ such that for every $n \in \mathbb{N}$,
\begin{align}
\label{eq4.108}
\mathbb{E} \left( \int_0^t \Vert (\mu_{\delta,n}(s),\theta_{\delta,n}(s)) \Vert_{\mathbb{H}}^2 \, \d s \right)^{r/2} 
\leq C_{3,\delta} [1 + \Vert(\phi_0,\varphi_0)\Vert_{\mathbb{V}}^r],
  \\
\label{eq4.109a}
\mathbb{E} \left(\int_0^t \Vert(\mu_{\delta,n}(s),\theta_{\delta,n}(s))\Vert_{\mathbb{H}}^4 \, \d s \right)^{r/4}
\leq C_{3,\delta} [1 + \Vert(\phi_0,\varphi_0)\Vert_{\mathbb{V}}^r].
\end{align}
\end{proposition}
\begin{proof}
Let us choose and fix $r\geq 2$. By \eqref{Eqn-L2-norm-mu-and-theta-delta-n} and the estimate \eqref{eq4.95} in Proposition \ref{Proposition_second_priori_estimmate}, we easily derive \eqref{eq4.108}. \newline
To prove the other part, we recall that, cf. \eqref{Eqn-4.35},
\begin{equation*}
\Vert (\mu_{\delta,n},\theta_{\delta,n}) \Vert_{\mathbb{H}}^2
\leq \Vert (\mu_{\delta,n},\theta_{\delta,n}) \Vert_{\mathbb{V}^\prime}^2 + 2 \Vert(\mu_{\delta,n},\theta_{\delta,n}) \Vert_{\mathbb{V}^\prime} (\lvert \nabla \mu_{\delta,n} \rvert + \lvert \nabla_\Gamma \theta_{\delta,n} \rvert_\Gamma).
\end{equation*}
This, jointly with the H\"older inequality implies that for every $n \in \mathbb{N}$,
\begin{equation}
\begin{aligned}
&\mathbb{E} \left(\int_0^t \Vert (\mu_{\delta,n},\theta_{\delta,n}) \Vert_{\mathbb{H}}^4 \,\d s \right)^{r/4}
\leq C t^{r/4} \mathbb{E} [\sup_{s \in [0,t]} \Vert (\mu_{\delta,n}(s),\theta_{\delta,n}(s)) \Vert_{\mathbb{V}^\prime}^{r}]
\\
& + C \mathbb{E} \bigg[\sup_{s \in [0,t]}\Vert(\mu_{\delta,n}(s),\theta_{\delta,n}(s)) \Vert_{\mathbb{V}^\prime}^{r/2} \left(\int_0^t  (\lvert \nabla \mu_{\delta,n} \rvert^2 + \lvert \nabla_\Gamma \theta_{\delta,n} \rvert_\Gamma^2) \,\d s\right)^{r/4} \bigg]
\\
&\leq C [\mathbb{E} \sup_{s \in [0,t]} \Vert (\mu_{\delta,n}(s),\theta_{\delta,n}(s)) \Vert_{\mathbb{V}^\prime}^{r}]^{\frac12} \bigg[ t^{r/4} [\mathbb{E} \sup_{s \in [0,t]} \Vert (\mu_{\delta,n}(s),\theta_{\delta,n}(s)) \Vert_{\mathbb{V}^\prime}^{r}]^{\frac12} \\
&\hspace{6cm} + \bigg[ \mathbb{E} 
 \left(\int_0^t  (\lvert \nabla \mu_{\delta,n} \rvert^2 + \lvert \nabla_\Gamma \theta_{\delta,n} \rvert_\Gamma^2) \,\d s\right)^{r/2} \bigg]^{\frac12} \bigg].
\end{aligned}
\end{equation}
From the above inequality, \eqref{Eqn-4.34}, and \eqref{eq4.95}, we easily conclude the proof of \eqref{eq4.109a}.
\end{proof}
The estimates in Propositions \ref{Proposition_second_priori_estimmate} and \ref{Proposition_third_priori_estimmate} imply the following result.
\begin{lemma}\label{H2_H3_priori_estimmates}
Assume that all the hypotheses of Lemma \ref{Lem-2} hold. Then, for any $r \in [2,\infty)$, there exists a constant $C_{4,\delta}>0$ such that for every $n \in \mathbb{N}$,
    \begin{align}
    \label{Eqn-(phi_{n,delta},varphi_{n,delta})-H2-estimate}
       \mathbb{E} \left(\int_0^T \Vert \bX_{\delta,n}(s) \Vert_{\mathcal{H}^2}^4 \,\d s \right)^{r/4} &\leq C_{4,\delta} [1 + \Vert (\phi_0,\varphi_0) \Vert_{\mathbb{V}}^r].
       \dela{
          \\
         \label{Eqn-(phi_{n,delta},varphi_{n,delta})-H3-estimate}
       \mathbb{E} \left(\int_0^T \Vert(\phi_{\delta,n}(s),\varphi_{\delta,n}(s))\Vert_{\mathcal{H}^3}^2 \,\d s\right)^{\fracr 2} &\leq C_{4,\delta} [1 + \Vert (\phi_0,\varphi_0) \Vert_{\mathbb{V}}^r].}
    \end{align}
\end{lemma}
\begin{proof}
By the smoothness of $\mathcal{O}$ \dela{,\coma{i.e. $\mathcal{O}$ is assumed to be of class $C^3$}}and the regularity theory for Poissons's equation with Neumann boundary condition, see \cite[Section 5, Proposition 7.7]{Taylor_2011}), we infer there exists $C=C(K,\eps,\mathcal{O})$ such that for every $n \in \mathbb{N}$,
\begin{equation}\label{Eqn-4.57}
\begin{aligned}
&\Vert \phi_{\delta,n} \Vert_{H^2(\mathcal{O})}^2
\leq  C \lvert \mu_{\delta,n} - \eps^{-1} \mathcal{S}_n  F_\delta^\prime(\phi_{\delta,n})\rvert^2 + C \Vert \phi_{\delta,n} \Vert_{V_1}^2 + C \Vert \bX_{\delta,n} \Vert_{H^{1/2}(\Gamma)}^2
\\
&\leq  C \lvert \mu_{\delta,n} \rvert^2 + C \lvert F_\delta^\prime(\phi_{\delta,n})\rvert^2 + C \Vert \phi_{\delta,n} \Vert_{V_1}^2 + C \Vert \bX_{\delta,n} \Vert_{H^{1/2}(\Gamma)}^2.
\end{aligned}
\end{equation}
Moreover, since $H^1(\mathcal{O}) \embed H^{1/2}(\Gamma)$ and $H^1(\Gamma) \embed H^{1/2}(\Gamma)$, we obtain for all $n \in \mathbb{N}$,
  \begin{align*}
   \Vert \bX_{\delta,n} \Vert_{H^{1/2}(\Gamma)} 
     \leq C \Vert \varphi_{\delta,n}\Vert_{H^{1/2}(\Gamma)}  + C \Vert \phi_{\delta,n}\Vert_{H^{1/2}(\Gamma)} 
     \leq C (\Vert \varphi_{\delta,n} \Vert_{V_\Gamma}  + \Vert \phi_{\delta,n} \Vert_{V_1}).
 \end{align*}
Furthermore, as $\lvert F_\delta^\prime(\phi_{\delta,n}) \rvert \leq \tilde{c}_{1,\delta} \lvert \phi_{\delta,n} \rvert$, we deduce there exists $C=C(K,\eps,\mathcal{O},\Gamma,\delta)$ such that for every $n \in \mathbb{N}$,
   \begin{equation}\label{Eqn-4.58}
     \Vert \phi_{\delta,n} \Vert_{H^2(\mathcal{O})}^2
      \leq  C (\lvert \mu_{\delta,n} \rvert^2 + \Vert \bX_{\delta,n} \Vert_{\mathbb{V}}^2).
    \end{equation}
\dela{Next, since the domain $\mathcal{O}$ is assumed to be of class $C^3$,}
\dela{Next, since the domain $\mathcal{O}$ satisfies the uniform $C^2$-regularity condition, its boundary $\Gamma$ is indeed a compact submanifold of $\mathbb{R}^d$ without boundary.
}
Thus, by applying the regularity theory for elliptic problems on submanifolds, see, e.g., \cite[Section 5, Theorem 1.3]{Kardestuncer+Norrie+Brezzi_1987}, using the fact that $H^1(\mathcal{O}) \embed L^2(\Gamma)$, we infer that for any $n \in \mathbb{N}$,
\begin{equation}\label{Eqn-4.59}
\begin{aligned}
&\Vert \varphi_{\delta,n} \Vert_{H^2(\Gamma)}
\leq C (\lvert \theta_{\delta,n} - \eps_\Gamma^{-1} \mathcal{S}_{n,\Gamma} G_\delta^\prime(\varphi_{\delta,n}) - \eps(\varphi_{\delta,n} - \phi_{\delta,n}) \rvert_\Gamma + \Vert \varphi_{\delta,n} \Vert_{V_\Gamma})
\\
&\leq C (\lvert \theta_{\delta,n} \rvert_\Gamma + \lvert \mathcal{S}_{n,\Gamma} G_\delta^\prime(\varphi_{\delta,n}) \rvert_\Gamma + \lvert \varphi_{\delta,n} \rvert_\Gamma + \lvert \phi_{\delta,n} \rvert_\Gamma + \Vert \varphi_{\delta,n} \Vert_{V_\Gamma}) \\
&\leq C(\lvert \theta_{\delta,n} \rvert_\Gamma + \Vert \mathcal{S}_{n,\Gamma} \Vert_{\mathcal{L}(L^2(\Gamma))} \lvert G_\delta^\prime(\varphi_{\delta,n}) \rvert_\Gamma + \Vert \varphi_{\delta,n} \Vert_{V_\Gamma} + \Vert \phi_{\delta,n} \Vert_{V_1}) \\
&\leq C(\lvert \theta_{\delta,n} \rvert_\Gamma + \lvert G_\delta^\prime(\varphi_{\delta,n}) \rvert_\Gamma + \Vert (\phi_{\delta,n},\varphi_{\delta,n}) \Vert_{\mathbb{V}}),
\end{aligned}
\end{equation}
where $C$ may depend on $K,\,\eps,\,\Gamma$, and $\eps_\Gamma$. Moreover, since 
$\lvert G_\delta^\prime(\varphi_{\delta,n}) \rvert_\Gamma \leq \tilde{c}_{2,\delta} \lvert \phi_{\delta,n} \rvert_\Gamma$, we deduce that there exists $C=C(K,\eps,\eps_\Gamma,\mathcal{O},\Gamma,\delta)$ such that for every $n \in \mathbb{N}$,
    \begin{equation}\label{Eqn-4.60}
     \Vert \varphi_{\delta,n} \Vert_{H^2(\Gamma)}^2
      \leq  C (\lvert \theta_{\delta,n} \rvert_\Gamma^2 + \Vert \bX_{\delta,n} \Vert_{\mathbb{V}}^2).
    \end{equation}
Hence, from \eqref{Eqn-4.58} and \eqref{Eqn-4.60}, we find that there exists $C=C(K,\eps,\eps_\Gamma,\mathcal{O},\Gamma,\delta)$ such that for any $n \in \mathbb{N}$,    
  \begin{equation}\label{Eqn-4.61}
    \Vert \bX_{\delta,n} \Vert_{\mathcal{H}^2}^2
    \leq C (\Vert(\mu_{\delta,n},\theta_{\delta,n})\Vert_{\mathbb{H}}^2 + \Vert \bX_{\delta,n} \Vert_{\mathbb{V}}^2)
  \end{equation}
From now on, using \eqref{eq4.95}, \eqref{eq4.109a}, and \eqref{Eqn-4.61}, we can easily complete the proof \dela{of the first part of Lemma \ref{H2_H3_priori_estimmates}}of the lemma.
\end{proof}
\section{Proof of Proposition \ref{Prop:Proposition-4}}\label{eqn-Prop:Proposition-4}
\begin{proof}[Proof of Proposition \ref{Prop:Proposition-4}]
Let us choose and fix $\delta \in (0,1)$. Let
    \begin{equation*}
     \bG^n(\mathfrak{X}_{\delta,n})\coloneq \mathrm{diag}(F_{1,n}(\bar{\phi}_{\delta,n}),F_{2,n}(\bar{\varphi}_{\delta,n})).
   \end{equation*}
\dela{
Observe that from the assumptions \ref{item:H2} and \ref{item:H3}, and Remark \ref{Rk_1}, we have\todozb{These conditions below can me moved to be parts of assumptions \ref{item:H2} and \ref{item:H3}.}
\begin{equation}\label{eq4.122}
\begin{aligned}
\Vert F_1(\phi) \Vert_{\mathscr{T}_2(U,L^2(\mathcal{O})))}^2 &\leq C_1 \lvert \mathcal{O} \rvert, \; \; &\phi \in L^2(\mathcal{O}), 
\\
\Vert F_1(\phi_1) - F_1(\phi_2) \Vert_{\mathscr{T}_2(U,L^2(\mathcal{O})))}^2 &\leq C_1 \lvert \phi_1 - \phi_2 \rvert^2, \; \; &\phi_1,\, \phi_2 \in L^2(\mathcal{O}),  
   \\
\Vert F_2(\varphi)\Vert_{\mathscr{T}_2(U_\Gamma,L^2(\Gamma)))}^2 &\leq C_2 \lvert \Gamma \rvert, \;\; &\varphi \in L^2(\Gamma), 
\\
\Vert F_2(\varphi_1) - F_2(\varphi_2) \Vert_{\mathscr{T}_2(U_\Gamma,L^2(\Gamma)))}^2 &\leq C_2 \lvert \varphi_1 - \varphi_2 \rvert_\Gamma^2, \;\; &\varphi_1,\, \varphi_2 \in L^2(\Gamma).
\end{aligned}
\end{equation}
}
In view of \eqref{eq4.122-1} and \eqref{eq4.122-2}, we see that the components of $\bG^n$ are Lipschitz. Thus, by \eqref{eq4.122a}, we infer $\bG^n(\mathfrak{X}_{\delta,n}) \to \bG(\mathfrak{X}_\delta) \coloneq \mathrm{diag}(F_1(\bar{\phi}_{\delta}),F_2(\bar{\varphi}_{\delta}))$ in $L^2(0,T;\mathscr{T}_2(\mathcal{U},\mathbb{H}))$ $\bar{\mathbb{P}}$-a.s., as $n \to \infty$. Moreover, there exists a positive constant $C=C(C_1,C_3)$ independent of $n$ and $\delta$ such that
    \begin{equation*}
      \Vert \bG^n(\mathfrak{X}_{\delta,n}) \Vert_{\mathscr{T}_2(\mathcal{U},\mathbb{H})}^2 \leq C(\lvert \bar{\phi}_{\delta,n} \rvert^2 + \lvert \bar{\varphi}_{\delta,n} \rvert_\Gamma^2).
   \end{equation*}
This, altogether with Propositions \ref{Proposition_first_priori_estimmate} and \ref{Proposition_second_priori_estimmate} yields that $\lbrace \bG^n(\mathfrak{X}_{\delta,n}) \rbrace_n$ is uniformly bounded in $L^2(\bar{\Omega}; L^2(0,T;\mathscr{T}_2(\mathcal{U},\mathbb{H})))$. Therefore, by the Vitali-Convergence theorem, we infer that
    \begin{equation}\label{eq4.123}
       \bG^n(\mathfrak{X}_{\delta,n}) \to \bG(\mathfrak{X}_\delta)  \mbox{ in } L^2(\bar{\Omega};L^2(0,T;\mathscr{T}_2(\mathcal{U},\mathbb{H}))),  \mbox{ as } n \to \infty.
    \end{equation}
Hence, thanks to \eqref{eq4.122a} and \eqref{eq4.123}, we can now apply \cite[Lemma 2.1]{Debussche+Glatt+Temam_2011} to deduce that
   \begin{equation*}
     \int_0^t \bG^n(\mathfrak{X}_{\delta,n}(s)) \, \d \bar{\mathcal{W}}^{n}(s) \to \int_0^t \bG(\mathfrak{X}_\delta(s)) \, \d \bar{\mathcal{W}}(s) \mbox{  in probability in }  L^2(0,T;\mathbb{H}).
   \end{equation*}
This completes the proof of the proposition.
\end{proof}

\dela{

\newpage
\subsection{Uniform estimates with respect to delta} \label{sect_Uniform_estimates_delta}
In this section, we derive additional uniform estimates that are independent of the parameter $\delta$. Consequently, the positive constant $C>0$, whose specific dependencies are explicitly noted when necessary, remains independent of $\delta$ and can vary from one line to another. \newline
We start with the following energy identity.
\begin{proposition}
Let $(\bar{\bu}_{\delta}, \bar{\phi}_{\delta},\bar{\varphi}_{\delta},\bar{\mu}_{\delta},\bar{\theta}_{\delta})_{\delta \in (0,1)}$, $\bar{W}$, and $\bar{W}_{\Gamma}$ be the processes obtained in \eqref{eq4.122a} and such that the equalities \eqref{eq4.130a}-\eqref{eq4.130d} hold. Then for every $T>0$, we have
\end{proposition}
\begin{proof}
Fix $\delta \in (0,1)$. Define the following $\mathcal{C}^2$-mappings:
\begin{align*}
&\vartheta_1: V_1 \ni \phi \mapsto \frac{\eps}{2} \Vert \nabla \phi \Vert_{\mathbb{L}^2(\mathcal{O})}^2 \in [0,\infty), \qquad \qquad \;  \vartheta_2: V_\Gamma \ni \varphi \mapsto \frac{\eps_\Gamma}{2} \Vert \nabla_\Gamma \varphi \Vert_{\mathbb{L}^2(\Gamma)}^2 \in [0,\infty),
\\
&\vartheta_3: V_1 \ni \phi \mapsto \frac{1}{\eps} \int_{\mathcal{O}} F_\delta(\phi(x)) \, \d x \in [0,\infty), \quad \quad \; \vartheta_4: V_\Gamma \ni \varphi \mapsto \frac{1}{\eps_\Gamma} \int_{\Gamma} G_\delta(\varphi) \, \d S \in [0,\infty),
\\
&\vartheta_5: \mathbb{V} \ni (\phi,\varphi) \mapsto \frac{\eps [K]}{2} \Vert \varphi - \phi \Vert_{L^2(\Gamma)}^2 \in [0,\infty).
\end{align*}
Now, from the regularity of the processes $\bar{\phi}_{\delta}$ and $\bar{\varphi}_{\delta}$, cf. \eqref{eq4.122a}, we can apply the It\^o formula due to \cite{Gyongy+Krylov_1982} to the processes $\vartheta_1(\bar{\phi}_{\delta}(t))$, $\vartheta_2(\bar{\varphi}_{\delta}(t))$, and $\vartheta_5((\bar{\phi}_{\delta}(t),\bar{\varphi}_{\delta}(t)))$, respectively.\newline
Subsequently, observe that $F_\delta,\, G_\delta \in \mathcal{C}^2(\mathbb{R})$ and $F^{\bis}_\delta,\, G^{\bis}_\delta \in L^\infty(\mathbb{R})$, so that $F^{\bis}_\delta$ and $G^{\bis}_\delta$ are Lipschitz-continuous. In particular, the maps
   \begin{align*}
     \Psi_1: V_1 \ni \phi \mapsto F^\prime_{\delta}(\phi) \in V_1 \mbox{ and }  \Psi_2: V_\Gamma \ni \varphi \mapsto G^\prime_{\delta}(\varphi) \in V_\Gamma
    \end{align*}
are well defined and continuous. Thus, from the properties of the processes $\bar{\phi}_{\delta}$ and $\bar{\varphi}_{\delta}$, we can apply the It\^o formula due to \cite{Pardoux_1975} in the variational setting of Theorem 4.2 in that thesis to the processes $\vartheta_3(\bar{\phi}_{\delta}(t))$ and $\vartheta_4(\bar{\varphi}_{\delta}(t))$. \newline
In conclusion, in order to derive the energy identity \coma{???} we shall apply the It\^o formula to the processes $\vartheta_1(\bar{\phi}_{\delta}(t)),\,\vartheta_2(\bar{\varphi}_{\delta}(t)),\,\vartheta_3(\bar{\phi}_{\delta}(t)),\, \vartheta_4(\bar{\varphi}_{\delta}(t))$, and $\vartheta_5((\bar{\phi}_{\delta}(t),\bar{\varphi}_{\delta}(t)))$, then we add the resulting formulae. This procedure, which is possible because of the previous observations, the properties of the processes $\bar{\phi}_{\delta}(t)$ and $\bar{\varphi}_{\delta}(t)$, and the fact that
\begin{align*}
 K \partial_{\bn} \bar{\phi}_\delta= \bar{\varphi}_\delta - \bar{\phi}_\delta &\mbox{ on }   \Sigma_T, 
                                 \\
 M_\mathcal{O} (\bar{\phi}_\delta) \partial_{\bn} \bar{\mu}_\delta= 0 &\mbox{ on }  \Sigma_T,                                                           \\
\end{align*}

\todoan{
I wanted to apply the It\^o formula to $\vartheta_i$, $i=1,\ldots 5$. But the issue i have is how to show that the following integral is well defined using the properties of the processes $\bar{\phi}_\delta$ and $\bar{\varphi}_\delta$ (cf. \eqref{eq4.122a})
\begin{align*}
&\int_0^t (A_{\bar{\phi}_\delta}(\bar{\mu}_\delta), \bar{\varphi}_\delta - \bar{\phi}_\delta)_\Gamma \, \d S \; \; ???
\\
&\int_0^t (\mathcal{A}_{\bar{\varphi}_\delta}(\bar{\theta}_\delta), \Delta_\Gamma \bar{\varphi}_\delta)_\Gamma \, \d S \; \; ???
\end{align*}
I note that when we apply It\^o's formula and then add the processes $\vartheta_1(\bar{\phi}_{\delta}(t)),\,\vartheta_2(\bar{\varphi}_{\delta}(t)),\,\vartheta_3(\bar{\phi}_{\delta}(t)),\, \vartheta_4(\bar{\varphi}_{\delta}(t))$, and $\vartheta_5((\bar{\phi}_{\delta}(t),\bar{\varphi}_{\delta}(t)))$, these integrals vanish at the end. However, how to show that the aforementioned integrals are well-defined.
}

\todoan{
In the previous version, we took $(\psi,\psi \lvert_\Gamma)=(1,0)$ or $(\psi,\psi \lvert_\Gamma)=(0,1)$ as test function in  \eqref{eq4.34d} so as to deduce an estimate for the mean of the processes $\mu_{\delta,n}$ and $\theta_{\delta,n}$. However, it seems that $(1,0)$ is not an element of the subspace $\mathcal{V}_n$. Maybe I wrong??
}
\end{proof}

\newpage

\subsection{Uniform estimates with respect to delta} \label{sect_Uniform_estimates_delta}

\noindent
\textbf{Additional regularity for the phase-fields.}
We now assume that $\mathcal{O}$ is even of class $C^3$ and, if $d=3$, we further assume $p<6$.

\noindent
Since $\Psi(r)= F^\prime(r) + \tilde{c}_F r$ and $F_\delta^\prime(r)= \mathbb{A}_\delta(r) - \tilde{c}_F r= (\Psi \circ J_\delta) (r)- \tilde{c}_F r$ for all $r \in \mathbb{R}$, then $F''_\delta(r)= J^\prime_\delta(r) \Psi^\prime(J_\delta(r)) - \tilde{c}_F= J^\prime_\delta(r) F''(J_\delta(r)) + \tilde{c}_F (J^\prime_\delta(r)-1)$, from which, we obtain
\begin{align*}
|F''_\delta(r)| &\leq |J^\prime_\delta(r) F''(J_\delta(r))| + \tilde{c}_F | J^\prime_\delta(r) - 1| \\
&\leq |F''(J_\delta(r))| + 2 \tilde{c}_F \\
&\leq c_{F''}(1 + |J_\delta(r)|^{p-2})  + 2 \tilde{c}_F \\
&\leq c_{F''}(1 + |r|^{p-2})  + 2 \tilde{c}_F,
\end{align*}
where we have also used the fact that $J_\delta^\prime(r) \in (0,1)$ for every $r \in \mathbb{R}$ together with the assumption \eqref{F''_and G''_additional_condition}. Consequently,
\begin{align*}
\int_{\mathcal{O}} |F''_\delta(\bar{\phi}_{\delta,n}(t))|^2 |\nabla \bar{\phi}_{\delta,n}(t)|^2 \, \d x 
&\leq 2( c_{F''} +  2 \tilde{c}_F)^2 \int_{\mathcal{O}} |\nabla \bar{\phi}_{\delta,n}(t)|^2 \, \d x + 2 \int_{\mathcal{O}} |\bar{\phi}_{\delta,n}(t)|^{2(p-2)} |\nabla \bar{\phi}_{\delta,n}(t)|^2 \, \d x \\
&\leq 2( c_{F''} +  2 \tilde{c}_F)^2 \|\nabla \bar{\phi}_{\delta,n}(t)\|_{\mathbb{L}^2(\mathcal{O})}^2 + 2 \|\bar{\phi}_{\delta,n}(t)\|_{L^{2(p-2)}(\mathcal{O})}^{2(p-2)} \|\nabla \bar{\phi}_{\delta,n}(t)\|_{\mathbb{L}^\infty(\mathcal{O})}^2.
\end{align*}
\adda{Next, without loss of generality, we further assume $p \in[5,6)$. } From this previous inequality and the Gagliardo-Nirenberg inequality together with the Agmon inequality, we obtain
\begin{align*}
&\|F''_\delta(\bar{\phi}_{\delta,n}(t)) \nabla \bar{\phi}_{\delta,n}(t)\|_{\mathbb{L}^2(\mathcal{O})} \\
&\leq C \|\nabla \bar{\phi}_{\delta,n}(t)\|_{\mathbb{L}^2(\mathcal{O})} + \|\bar{\phi}_{\delta,n}(t)\|_{L^{2(p-2)}(\mathcal{O})}^{p-2} \|\nabla \bar{\phi}_{\delta,n}(t)\|_{\mathbb{L}^\infty(\mathcal{O})} \\
&\leq C \|\nabla \bar{\phi}_{\delta,n}(t)\|_{\mathbb{L}^2(\mathcal{O})} + C \|\bar{\phi}_{\delta,n}(t)\|_{L^6(\mathcal{O})}^\frac{p+1}{2} \|\bar{\phi}_{\delta,n}(t)\|_{H^2(\mathcal{O})}^{\frac{p-5}{2} + \frac{1}{2}} \|\bar{\phi}_{\delta,n}(t)\|_{H^3(\mathcal{O})}^\frac{1}{2} \\
&\leq C \|\bar{\phi}_{\delta,n}(t)\|_{V_1} + C \|\bar{\phi}_{\delta,n}(t)\|_{V_1}^\frac{p+1}{2} \|\bar{\phi}_{\delta,n}(t)\|_{H^2(\mathcal{O})}^\frac{p-4}{2} \|\bar{\phi}_{\delta,n}(t)\|_{H^3(\mathcal{O})}^\frac{1}{2} \\
&\leq C \|\bar{\phi}_{\delta,n}(t)\|_{V_1} + C \|\bar{\phi}_{\delta,n}(t)\|_{V_1}^\frac{p+1}{2} \|\bar{\phi}_{\delta,n}(t)\|_{V_1}^\frac{p-4}{4} \|\bar{\phi}_{\delta,n}(t)\|_{H^3(\mathcal{O})}^{\frac{1}{2} + \frac{p-4}{4}} \\
&\leq C \|\bar{\phi}_{\delta,n}(t)\|_{V_1} + C  \|\bar{\phi}_{\delta,n}(t)\|_{V_1}^\frac{3p-2}{4} \|\bar{\phi}_{\delta,n}(t)\|_{H^3(\mathcal{O})}^\frac{p-2}{4},
\end{align*} 
which, in turn, entails by application of the Young inequality that
\begin{align*}
\|F''_\delta(\bar{\phi}_{\delta,n}(t)) \nabla \bar{\phi}_{\delta,n}(t)\|_{\mathbb{L}^2(\mathcal{O})}
\leq C_\epsilon  + C_\epsilon \|\bar{\phi}_{\delta,n}(t)\|_{V_1}^\frac{3p-2}{6-p} + \frac{\epsilon}{2} \|\bar{\phi}_{\delta,n}(t)\|_{H^3(\mathcal{O})},
\end{align*}
for every $\epsilon>0$ and where the constant $C_\epsilon$ is independent of $n$ and $\delta$. Besides, from \eqref{F'_and G'_additional_condition} and the fact that $|F_\delta^\prime (\bar{\phi}_{\delta,n})| \leq   |F^\prime(\bar{\phi}_{\delta,n})| + 2 \tilde{c}_F |\bar{\phi}_{\delta,n}|$, we have
\begin{align*}
\|F_\delta^\prime (\bar{\phi}_{\delta,n}(t))\|_{L^2(\mathcal{O})}
&\leq C + C \|\bar{\phi}_{\delta,n}(t)\|_{L^{2(p-1)}(\mathcal{O})}^{p-1} \\
&\leq C + C  \|\bar{\phi}_{\delta,n}(t)\|_{L^6(\mathcal{O})}^\frac{p+2}{2} \|\bar{\phi}_{\delta,n}(t)\|_{H^2(\mathcal{O})}^\frac{p-4}{2} \\
&\leq C + C  \|\bar{\phi}_{\delta,n}(t)\|_{V_1}^\frac{p+2}{2} \|\bar{\phi}_{\delta,n}(t)\|_{H^2(\mathcal{O})}^\frac{p-4}{2} \\
&\leq C + C  \|\bar{\phi}_{\delta,n}(t)\|_{V_1}^\frac{p+2}{2} \|\bar{\phi}_{\delta,n}(t)\|_{V_1}^\frac{p-4}{4} \|\bar{\phi}_{\delta,n}(t)\|_{H^3(\mathcal{O})}^\frac{p-4}{4} \\
&\leq C + \|\bar{\phi}_{\delta,n}(t)\|_{V_1}^\frac{3p}{4} \|\bar{\phi}_{\delta,n}(t)\|_{H^3(\mathcal{O})}^\frac{p-4}{4} \\
&\leq C + C \|\bar{\phi}_{\delta,n}(t)\|_{V_1}^\frac{3p}{8-p} +  \frac{\epsilon}{2} \|\bar{\phi}_{\delta,n}(t)\|_{H^3(\mathcal{O})},
\end{align*}
where we have also use the Gagliardo-Nirenberg inequality and Young’s inequality in conjunction with the embedding of $H^1(\mathcal{O})$ in $L^6(\mathcal{O})$. Here, the constant $C$ may depends on $c_{F^\prime},\,p,\,\mathcal{O}$ and $\epsilon$.  By Young's inequality this leads to 
\begin{equation}\label{eq4.170}
\begin{aligned}
\|F_\delta^\prime (\bar{\phi}_{\delta,n}(t))\|_{V_1}
&\leq \|F_\delta^\prime(\bar{\phi}_{\delta,n}(t))\|_{L^2(\mathcal{O})} + \|F''_\delta(\bar{\phi}_{\delta,n}(t)) \nabla \bar{\phi}_{\delta,n}(t)\|_{\mathbb{L}^2(\mathcal{O})} \\
&\leq C  + C \|\bar{\phi}_{\delta,n}(t)\|_{V_1}^\frac{3p}{8-p} + C \|\bar{\phi}_{\delta,n}(t)\|_{V_1}^\frac{3p-2}{6-p} + \epsilon \|\bar{\phi}_{\delta,n}(t)\|_{H^3(\mathcal{O})} \\
&\leq C  + C \|\bar{\phi}_{\delta,n}(t)\|_{V_1}^\frac{3p-2}{6-p} + \epsilon \|\bar{\phi}_{\delta,n}(t)\|_{H^3(\mathcal{O})},
\end{aligned}
\end{equation}
for every $\epsilon>0$, with $C$ being independent of $n$ and $\delta$.

\noindent
We want now to estimate $\|G_\delta^\prime(\bar{\varphi}_{\delta,n})\|_{V_\Gamma}$. Since $|G''_\delta(r)| \leq c_{G''}(1 + |r|^{q-2})  + 2 \tilde{c}_G$ for all $r \in \mathbb{R}$, we have
\begin{align*}
\|G''_\delta(\bar{\varphi}_{\delta,n}(t)) \nabla_\Gamma \bar{\varphi}_{\delta,n}(t)\|_{\mathbb{L}^2(\mathcal{O})}
&\leq C \|\nabla_\Gamma \bar{\varphi}_{\delta,n}(t)\|_{\mathbb{L}^2(\mathcal{O})} + C \|\bar{\varphi}_{\delta,n}(t)\|_{L^{4(q-2)}(\Gamma)}^{2(q-2)} \|\nabla_\Gamma \bar{\varphi}_{\delta,n}(t)\|_{\mathbb{L}^4(\Gamma)}^2 \\
&\leq C \|\bar{\varphi}_{\delta,n}(t)\|_{V_\Gamma} + C \|\bar{\varphi}_{\delta,n}(t)\|_{V_\Gamma}^{2(q-2)} \|\nabla_\Gamma \bar{\varphi}_{\delta,n}(t)\|_{\mathbb{L}^4(\Gamma)}^2,
\end{align*}
where we used the following continuous embeddings: $H^1(\Gamma) \hookrightarrow H^{1/2}(\Gamma) \hookrightarrow L^{4(q-2)}(\Gamma)$ if $d=2$, and $H^1(\Gamma) \hookrightarrow L^{4(q-2)}(\Gamma)$ if $d=3$. Here $C$ is a positive constant independent of $n$ and $\delta$. Besides, by the Sobolev inequalities $\|\nabla_\Gamma \bar{\varphi}_{\delta,n}(t)\|_{\mathbb{L}^4(\Gamma)} \leq C(\Gamma) \|\bar{\varphi}_{\delta,n}(t)\|_{H^2(\Gamma)}^\frac{1}{2} \|\nabla_\Gamma \bar{\varphi}_{\delta,n}(t)\|_{\mathbb{L}^2(\Gamma)}^\frac{1}{2}$ and $\|\bar{\varphi}_{\delta,n}(t)\|_{H^2(\Gamma)}\leq C(\Gamma) \|\bar{\varphi}_{\delta,n}(t)\|_{H^1(\Gamma)}^\frac{1}{2} \|\bar{\varphi}_{\delta,n}(t)\|_{H^3(\Gamma)}^\frac{1}{2}$, we find

\begin{align*}
\|G''_\delta(\bar{\varphi}_{\delta,n}(t)) \nabla_\Gamma \bar{\varphi}_{\delta,n}(t)\|_{\mathbb{L}^2(\mathcal{O})}
&\leq C \|\bar{\varphi}_{\delta,n}(t)\|_{V_\Gamma} + C \|\bar{\varphi}_{\delta,n}(t)\|_{V_\Gamma}^{2(q-2)} \|\bar{\varphi}_{\delta,n}(t)\|_{H^2(\Gamma)} \|\nabla_\Gamma\bar{\varphi}_{\delta,n}(t)\|_{\mathbb{L}^2(\Gamma)} \\
&\leq C \|\bar{\varphi}_{\delta,n}(t)\|_{V_\Gamma} + C \|\bar{\varphi}_{\delta,n}(t)\|_{V_\Gamma}^{2(q-2) + 1} \|\bar{\varphi}_{\delta,n}(t)\|_{H^2(\Gamma)} \\
&\leq C \|\bar{\varphi}_{\delta,n}(t)\|_{V_\Gamma} + C \|\bar{\varphi}_{\delta,n}(t)\|_{V_\Gamma}^{2(q-2) + 1} \|\bar{\varphi}_{\delta,n}(t)\|_{H^1(\Gamma)}^\frac{1}{2} \|\bar{\varphi}_{\delta,n}(t)\|_{H^3(\Gamma)}^\frac{1}{2} \\
&\leq C \|\bar{\varphi}_{\delta,n}(t)\|_{V_\Gamma} + C \|\bar{\varphi}_{\delta,n}(t)\|_{V_\Gamma}^{2(q-2) + 3/2} \|\bar{\varphi}_{\delta,n}(t)\|_{H^3(\Gamma)}^\frac{1}{2} \\
&\leq C \|\bar{\varphi}_{\delta,n}(t)\|_{V_\Gamma} + C \|\bar{\varphi}_{\delta,n}(t)\|_{V_\Gamma}^{4q-5} + \epsilon \|\bar{\varphi}_{\delta,n}(t)\|_{H^3(\Gamma)} \\
&\leq C + C \|\bar{\varphi}_{\delta,n}(t)\|_{V_\Gamma}^{4q-5} + \epsilon \|\bar{\varphi}_{\delta,n}(t)\|_{H^3(\Gamma)}
\end{align*}
which holds for $d=3$ and for every $\epsilon>0$. Once more $C$ is independent of $n$ and $\delta$. In the case $d=2$, we have 
\begin{align*}
\|G''_\delta(\bar{\varphi}_{\delta,n}(t)) \nabla_\Gamma \bar{\varphi}_{\delta,n}(t)\|_{\mathbb{L}^2(\mathcal{O})}
&\leq C \|\bar{\varphi}_{\delta,n}(t)\|_{V_\Gamma} + C \|\bar{\varphi}_{\delta,n}(t)\|_{V_\Gamma}^{2(q-2)} \|\bar{\varphi}_{\delta,n}(t)\|_{H^2(\Gamma)}^\frac{1}{2} \|\nabla_\Gamma \bar{\varphi}_{\delta,n}(t)\|_{\mathbb{L}^2(\Gamma)}^\frac{3}{2} \\
&\leq C \|\bar{\varphi}_{\delta,n}(t)\|_{V_\Gamma} + C \|\bar{\varphi}_{\delta,n}(t)\|_{V_\Gamma}^{2(q-2)+3/2} \|\bar{\varphi}_{\delta,n}(t)\|_{H^2(\Gamma)}^\frac{1}{2} \\
&\leq C \|\bar{\varphi}_{\delta,n}(t)\|_{V_\Gamma} + C \|\bar{\varphi}_{\delta,n}(t)\|_{V_\Gamma}^{2(q-2)+3/2} \|\bar{\varphi}_{\delta,n}(t)\|_{H^1(\Gamma)}^\frac{1}{4} \|\bar{\varphi}_{\delta,n}(t)\|_{H^3(\Gamma)}^\frac{1}{4} \\
&\leq C \|\bar{\varphi}_{\delta,n}(t)\|_{V_\Gamma} + C \|\bar{\varphi}_{\delta,n}(t)\|_{V_\Gamma}^{2(q-2)+7/4} \|\bar{\varphi}_{\delta,n}(t)\|_{H^3(\Gamma)}^\frac{1}{4} \\
&\leq C \|\bar{\varphi}_{\delta,n}(t)\|_{V_\Gamma} + C \|\bar{\varphi}_{\delta,n}(t)\|_{V_\Gamma}^{\frac{8q-9}{3}} + \epsilon \|\bar{\varphi}_{\delta,n}(t)\|_{H^3(\Gamma)} \\
&\leq C + C \|\bar{\varphi}_{\delta,n}(t)\|_{V_\Gamma}^{4q-5} + \epsilon \|\bar{\varphi}_{\delta,n}(t)\|_{H^3(\Gamma)},
\end{align*}
where we used the Sobolev inequalities $\|\nabla_\Gamma \bar{\varphi}_{\delta,n}(t)\|_{\mathbb{L}^4(\Gamma)} \leq C(\Gamma) \|\bar{\varphi}_{\delta,n}(t)\|_{H^2(\Gamma)}^\frac{r-2}{2r} \|\nabla_\Gamma \bar{\varphi}_{\delta,n}(t)\|_{\mathbb{L}^2(\Gamma)}^{\frac{1}{2} + \frac{1}{r}}$ for all $r \in (2,\infty)$ together with the fact that
$\|\bar{\varphi}_{\delta,n}(t)\|_{H^2(\Gamma)}\leq C(\Gamma) \|\bar{\varphi}_{\delta,n}(t)\|_{H^1(\Gamma)}^\frac{1}{2} \|\bar{\varphi}_{\delta,n}(t)\|_{H^3(\Gamma)}^\frac{1}{2}$. In the case $d=2$, we note that $H^{1/2}(\Gamma) \not\subseteq L^\infty(\Gamma)$, so that this Sobolev inequality is not true for $r=2$. Hence, for both $d=2,3$, we have 
\begin{align*}
\|G''_\delta(\bar{\varphi}_{\delta,n}(t)) \nabla_\Gamma \bar{\varphi}_{\delta,n}(t)\|_{\mathbb{L}^2(\mathcal{O})}
\leq C + C \|\bar{\varphi}_{\delta,n}(t)\|_{V_\Gamma}^{4q-5} + \epsilon \|\bar{\varphi}_{\delta,n}(t)\|_{H^3(\Gamma)}
\end{align*}
for any $\epsilon>0$ with $C$ being independent of $n$ and $\delta$, and in turn, this implies that
\begin{equation}\label{eq4.171a}
\begin{aligned}
\|G_\delta^\prime (\bar{\varphi}_{\delta,n}(t))\|_{V_\Gamma}
&\leq \|G_\delta^\prime(\bar{\varphi}_{\delta,n}(t))\|_{L^2(\Gamma)} + \|G''_\delta(\bar{\varphi}_{\delta,n}(t)) \nabla_\Gamma \bar{\varphi}_{\delta,n}(t)\|_{\mathbb{L}^2(\Gamma)} \\
&\leq C + C \|\bar{\varphi}_{\delta,n}\|_{V_\Gamma}^{q-1} + C \|\bar{\varphi}_{\delta,n}(t)\|_{V_\Gamma}^{4q-5} + \epsilon \|\bar{\varphi}_{\delta,n}(t)\|_{H^3(\Gamma)} \\
&\leq C + C \|\bar{\varphi}_{\delta,n}(t)\|_{V_\Gamma}^{4q-5} + \epsilon \|\bar{\varphi}_{\delta,n}(t)\|_{H^3(\Gamma)},
\end{aligned}
\end{equation}
where we used Young's inequality. Moreover, by regularity theory for elliptic problems with bulk-surface coupling (see \cite[Theorem 3.3]{Knopf+Liu_2021}), we have
\begin{equation*}
\begin{aligned}
\|(\bar{\phi}_{\delta,n}(t),\bar{\varphi}_{\delta,n}(t))\|_{\mathcal{H}^3}^2 
&\leq C \|(f_{n,\delta}(t),g_{n,\delta}(t))\|_{\mathcal{H}^1}^2 \\
&\leq C (\|\bar{\mu}_{\delta,n}(t)\|_{V_1}^2 + \|F_\delta^\prime(\bar{\phi}_{\delta,n}(t))\|_{V_1}^2 + \|\bar{\theta}_{\delta,n}(t)\|_{V_\Gamma}^2 + \|G_\delta^\prime(\bar{\varphi}_{\delta,n}(t))\|_{V_\Gamma}^2),
\end{aligned}
\end{equation*}
with $C$ independent of $n$ and $\delta$. From this inequality together with \eqref{eq4.170} and \eqref{eq4.171a}, we obtain
\begin{equation}\label{eq4.140a}
\begin{aligned}
\|(\bar{\phi}_{\delta,n}(t),\bar{\varphi}_{\delta,n}(t))\|_{\mathcal{H}^3}^2
&\leq  C + C \|\bar{\mu}_{\delta,n}(t)\|_{V_1}^2 + C  \|\bar{\theta}_{\delta,n}(t)\|_{V_\Gamma}^2 + C \|\bar{\phi}_{\delta,n}(t)\|_{V_1}^\frac{6p-4}{6-p} + C \|\bar{\varphi}_{\delta,n}(t)\|_{V_\Gamma}^{8q-10} \\
&\quad + C \epsilon^2 \|(\bar{\phi}_{\delta,n}(t),\bar{\varphi}_{\delta,n}(t))\|_{\mathcal{H}^3}^2.
\end{aligned}
\end{equation}
Next, choosing $\epsilon$ sufficiently small and using the previous convergences, we find
\begin{align}\label{eq4.140b}
\|(\bar{\phi}_{\delta}(t),\bar{\varphi}_{\delta}(t))\|_{\mathcal{H}^3}^2
&\leq  C + C \|\bar{\mu}_{\delta}(t)\|_{V_1}^2 + C  \|\bar{\theta}_{\delta}(t)\|_{V_\Gamma}^2 + C \|\bar{\phi}_{\delta}(t)\|_{V_1}^\frac{6p-4}{6-p} + C \|\bar{\varphi}_{\delta}(t)\|_{V_\Gamma}^{8q-10},
\end{align}
which, in turn, entails
\begin{equation}\label{eq4.172}
\begin{aligned}
&\bar{\mathbb{E}} \left(\int_0^T \|(\bar{\phi}_{\delta}(t),\bar{\varphi}_{\delta}(t))\|_{\mathcal{H}^3}^2\, \d t \right)^\frac{\ell}{2} \\
&\leq  C + C \bar{\mathbb{E}} \left(\int_0^T \|\bar{\mu}_{\delta}(t)\|_{V_1}^2 \, \d t\right)^\frac{\ell}{2} + C \bar{\mathbb{E}} \left(\int_0^T  \|\bar{\theta}_{\delta}(t)\|_{V_\Gamma}^2 \, \d t \right)^\frac{\ell}{2} \\
&\quad + C \bar{\mathbb{E}} \sup_{t\in[0,T]} \|\bar{\phi}_{\delta}(t)\|_{V_1}^\frac{(3p-2)\ell}{6-p} + C \bar{\mathbb{E}} \sup_{t\in[0,T]} \|\bar{\varphi}_{\delta}(t)\|_{V_\Gamma}^{(4q-5)\ell}
\end{aligned}
\end{equation}
with $\ell\geq 2$ and $C$ independent of $\delta$. It then follows from \eqref{eq4.143}$\text{-}3)$, \eqref{uniform_delta_estimate_varphi}, \eqref{eq4.144}, and \eqref{eq4.172} that there exists a positive constant $ c_\ell$ independent of $\delta$ such that
\begin{equation}
\|(\bar{\phi}_{\delta},\bar{\varphi}_{\delta})\|_{L^\ell(\bar{\Omega};L^2(0,T;\mathcal{H}^3))} \leq c_{\ell}, \quad \forall \ell <\min \left \lbrace \frac{6-p}{3p-2},\frac{1}{4q-5} \right \rbrace r.
\end{equation}

\subsection{Passage to the limit as delta tends to zero} \label{Sect_passage to the limit_delta}

whereas, if now we assume that $\mathcal{O}$ is of class $C^3$ and, if $d=3$, we further assume $p<6$, we get
   \begin{equation}
      (\tilde{\phi}^\delta,\tilde{\varphi}^\delta)  \rightharpoonup (\tilde{\phi},\tilde{\varphi}) \quad \text{in} \quad  L^\ell(\tilde{\Omega};L^2(0,T;\mathcal{H}^3)) \quad \forall \ell <\min \left \lbrace \frac{6-p}{3p-2},\frac{1}{4q-5} \right \rbrace r.
   \end{equation}

\noindent
We will now prove that the inequality \eqref{eq3.29} holds true. First, we note that we can express the result \eqref{eq4.133} in the new variables since $X_\delta$ preserves laws, i.e., for every $\delta \in (0,1)$, we have 
\begin{equation}\label{eq4.171}
\begin{aligned}
&\tilde{\mathbb{E}}[ \mathcal{E}_{tot}(\tilde{\phi}^\delta(t), \tilde{\varphi}^\delta(t))] + \tilde{\mathbb{E}} \int_{Q_t} [2 \nu(\tilde{\phi}^\delta) \lvert D\tilde{\bu}^\delta \rvert^2 + \lambda(\tilde{\phi}^\delta) \lvert \tilde{\bu}^\delta \rvert^2] \,\d x \,\d s 
\\
& + \tilde{\mathbb{E}} \int_{\Sigma_t} \gamma(\tilde{\varphi}^\delta) \lvert \tilde{\bu}^\delta \rvert^2 \,\d S\,\d s + 0.5 \tilde{\mathbb{E}} \int_{Q_t} M_{\mathcal{O}}(\tilde{\phi}^\delta) \lvert \nabla \tilde{\mu}^\delta \rvert^2 \, \d x \,\d s + \tilde{\mathbb{E}} \int_{\Sigma_t} M_\Gamma(\tilde{\varphi}^\delta) \lvert \nabla_\Gamma \tilde{\theta}^\delta \rvert^2 \,\d S \,\d s  
            \\
&\leq \mathcal{E}_{tot}(\phi_0,\varphi_0) + \red{(\eps [K] C_{\Gamma,\mathcal{O}} + 0.5 ) \tilde{\mathbb{E}}\int_0^{t} \|F_{1}(\tilde{\phi}^\delta)\|_{\mathscr{T}_2(U,L^2(\mathcal{O}))}^2\,\d s }  
                 \\
&\quad  + 0.5 \eps_\Gamma C_2 \bar{\mathbb{E}} \int_0^{t} |\nabla_\Gamma \tilde{\varphi}^\delta|_{\Gamma}^2 \, \d s + \left(0.5 \bar{M}_0 + 0.5 \eps C_1 + \eps C_1 [K] C_{\Gamma,\mathcal{O}} \right) \tilde{\mathbb{E}} \int_0^t  \lvert \nabla \tilde{\phi}^\delta \rvert^2 \, \d s                    \\
&\quad + \eps [K] \tilde{\mathbb{E}} \int_0^{t} \|F_2(\tilde{\varphi}^\delta)\|_{\mathscr{T}_2(U_\Gamma,L^2(\Gamma))}^2 \, \d s +  0.5 \eps^{-1} \tilde{\mathbb{E}} \int_0^{t} \sum_{k=1}^\infty \int_{\mathcal{O}} F''_\delta(\tilde{\phi}^\delta) |\sigma_k(\tilde{\phi}^\delta)|^2 \, \d x \,\d s 
                                    \\
&\quad +  0.5 \eps_\Gamma^{-1} \tilde{\mathbb{E}} \int_0^{t}  \sum_{k=1}^\infty \int_{\Gamma} G''_\delta(\tilde{\varphi}^\delta) |\tilde{\sigma}_k(\tilde{\varphi}^\delta)|^2\, \d S\,\d s, 
\end{aligned}
\end{equation}
for every $t\in[0,T]$, 
The inequality \eqref{eq3.29} is then a direct consequence of the previous convergences results and the weak lower semicontinuity of the norm. This completes the proof of our first main result, i.e., Theorem \ref{thm-first_main_theorem}.
}
\dela{
\section{Existence of martingale solutions in the case K=0} \label{sect5}
In this section, we prove the existence of martingale solutions to the problem \eqref{Eqt1.1a}-\eqref{Eqt1.3a}-\eqref{eq1.10a} in the case $K=0$. The main result of this section is given as follows.

\begin{theorem}\label{second_main_theorem}
Let $d=2,3$, $T>0$ be a fixed positive time, and let the domain $\mathcal{O} \subset \mathbb{R}^d$ be of class $C^2$. We suppose that the assumptions $(H2)$-$(H8)$ hold. We further assume that the potentials $F$ and $G$ satisfy \eqref{F''_and G''_additional_condition} as in Theorem \ref{thm-first_main_theorem}. Then, there exists a weak martingale solution
$(\bu^0(t), \phi^0(t), \varphi^0(t), \mu^0(t), \mathcal{W}^0(t))_{t \in [0,T]}$ to system \eqref{Eqt1.1a}-\eqref{Eqt1.3a}-\eqref{eq1.10a} in the sense of Definition \ref{def3.4} with $K=0$. In addition, we have 
    \begin{equation}
      (\phi^0,\varphi^0) \in L^r(\Omega^0;L^2(0,T;\mathcal{H}^2)) \quad \text{and} \quad (F^\prime(\phi^0),G^\prime(\varphi^0))  \in L^r(\Omega^0; L^2(0,T;\mathbb{H})),
    \end{equation}
for every $r\geq 2$. 

\noindent
If $\mathcal{O}$ is of class $C^3$ and if in addition we  assume that there exist $c_{F''}>0$ and $c_{G''}>0$ such that
        \begin{align*}
          |F''(r)| \leq c_{F''}(1 + |r|^{p-2}), \\
          |G''(r)| \leq c_{G''}(1 + |r|^{q-2})
        \end{align*}
for all $r \in \mathbb{R}$, where 
     \begin{equation*}
        p \in 
             \begin{cases}
                   [2,\infty) \quad &\text{if} \quad d=2, \\
                   [2,6]  \quad &\text{if} \quad d=3,
              \end{cases}
                \quad \text{and} \quad q \in [2,\infty),
    \end{equation*}
then the following estimates hold:
    \begin{equation}\label{eq5.2}
     \begin{aligned}
       (\phi^0,\varphi^0) \in L^2(0,T;\mathcal{H}^3), \quad &\mathbb{P}^0\text{-a.s.},
       \\
      (F^\prime (\phi^0), G^\prime (\varphi^0)) \in L^2(0,T;\mathbb{V}), \quad & \mathbb{P}^0\text{-a.s.}
    \end{aligned}
\end{equation}
             
\end{theorem}

\begin{proof}
Fix $\delta \in (0,1)$. We recall that in Subsection \ref{Sub-sect-Galerkin-approximation}, we proved that there exists
$(\bu_{\delta,n},\phi_{\delta,n},\varphi_{\delta,n})_{n \in \mathbb{N}}$ solution of the problem \eqref{compact-stochastic-problem} both in the case $K=0$ and $K>0$. Furthermore, we notice that the sequence $(\bu_{\delta,n},\phi_{\delta,n},\varphi_{\delta,n})_{n \in \mathbb{N}}$ satisfies the equations \eqref{eq4.34a}-\eqref{eq4.34d}. In particular, in the case $K=0$, we have
\begin{equation}\label{eq4.34d-a}
\begin{aligned}
&\int_{\mathcal{O}} \mu_{\delta,n} \psi\, \d x  + \int_{\Gamma} \theta_{\delta,n} \psi \lvert_\Gamma\, \d S  
- \eps \int_{\mathcal{O}} \nabla \phi_{\delta,n} \cdot \nabla \psi\, \d x - \int_{\mathcal{O}} \frac1\eps F_\delta^\prime(\phi_{\delta,n}) \psi \, \d x \\ 
&= \eps_\Gamma \int_{\Gamma} \nabla_\Gamma \varphi_{\delta,n} \cdot \nabla_\Gamma \psi \lvert_\Gamma\, \d S  +  \frac{1}{\eps_\Gamma} \int_{\Gamma} G_\delta^\prime(\varphi_{\delta,n}) \psi \lvert_\Gamma \, \d S, \; \; (\psi,\psi \lvert_\Gamma) \in \mathbb{V}_0.
\end{aligned}
\end{equation}
Now, we begin by obtaining the counterpart of the results obtained in Subsection \ref{subsection-4.3}, when $K=0$.
\dela{
So, we consider once more the Faedo-Galerkin scheme \eqref{eq4.34a}-\eqref{eq4.34d}. Next, since the domain is assumed to be of class $C^2$, then by the regularity theory for Poisson's equation with inhomogeneous Neumann boundary condition (see \cite[Section 5, Proposition 7.7]{Taylor_2011}), we deduce that the eigenfunctions $\bar{\upsilon}_j$ enjoy the following regularity: $\bar{\upsilon}_j \in H^2(\mathcal{O})$ for all $j \in \mathbb{N}$. Moreover, by the regularity theory for elliptic problems on submanifolds (see \cite[Theorem 1.3]{Kardestuncer+Norrie+Brezzi_1987}), the eigenfunctions $\bar{\Lambda}_j$ enjoy the regularity: $\bar{\Lambda}_i \in H^2(\Gamma)$ for any $j \in \mathbb{N}$. Consequently, the process $(\phi_{\delta,n},\varphi_{\delta,n})$ whose existence has been proved in Theorem \ref{thm-Galerkin-01} satisfy $\mathbb{P}$-a.s. $\phi_{\delta,n} \in H_{n}^1 \cap H^2(\mathcal{O})$ and $\varphi_{\delta,n} \in H_{\Gamma,n}^1 \cap H^2(\Gamma)$, and in turn, in the case $K=0$, we infer from the weak variational formulation \eqref{eq4.34d} that $(\phi_{\delta,n},\varphi_{\delta,n})$ satisfies: 
\begin{subequations}
\begin{align}
- \eps \Delta \phi_{\delta,n}(t)&= f_{n,\delta}(t) \quad &\text{in} \quad \mathcal{O}, \quad \mathbb{P}\text{-a.s,}  \label{eq5.1a}
\\
- \eps_\Gamma \Delta_\Gamma \varphi_{\delta,n}(t) + \eps \partial_{\bn} \bar{\phi}_{\delta,n}(t)
&= g_{n,\delta}(t) \quad &\text{on} \quad \Gamma, \quad \mathbb{P}\text{-a.s,}  \label{eq5.1b}
    \\
\varphi_{\delta,n}(t) &= \phi_{\delta,n}(t) \quad &\text{on} \quad \Gamma, \quad \mathbb{P}\text{-a.s.,} \label{eq5.1c}
\end{align}
\end{subequations}
with
\begin{equation*}
f_{n,\delta}(t)= \mu_{\delta,n}(t) - \eps^{-1} \mathcal{S}_n [F_\delta^\prime(\phi_{\delta,n}(t))] \quad \text{and} \quad g_{n,\delta}(t)= \theta_{\delta,n}(t) - \eps_\Gamma^{-1} \mathcal{S}_{n,\Gamma}[G_\delta^\prime(\varphi_{\delta,n}(t))]
\end{equation*}
for almost $t \in[0,T]$. 
}
Indeed, arguing as in the proof of Lemma \ref{Lem-2}, we infer that for all $t\in[0,T]$ and $\mathbb{P}$-a.s.,
\begin{equation}\label{Main_Galerkin_Equality2}
\begin{aligned}
& E_{\delta}(\phi_{\delta,n}(t),\varphi_{\delta,n}(t\wedge \tau_\kappa)) + \int_0^{t\wedge \tau_\kappa} \int_{\Gamma} (\gamma(\varphi_{\delta,n}) \lvert \bu_{\delta,n} \rvert^2 + M_\Gamma(\varphi_{\delta,n}) \lvert \nabla_\Gamma \theta_{\delta,n} \rvert^2) \,\d S \,\d s  
\\
&\quad + \int_0^{t\wedge \tau_\kappa} \int_{\mathcal{O}} [2 \nu(\phi_{\delta,n}) \lvert D\bu_{\delta,n} \rvert^2 + \lambda(\phi_{\delta,n}) \lvert \bu_{\delta,n} \rvert^2 + M_{\mathcal{O}}(\phi_{\delta,n}) \lvert \nabla \mu_{\delta,n} \rvert^2]\,\d x \,\d s
 \\
&= E_\delta(\phi_n(0),\varphi_n(0)) +  \frac12 \int_0^{t\wedge \tau_\kappa} \Vert F_{1}(\phi_{\delta,n}) \Vert_{\mathscr{T}_2(U,L^2(\mathcal{O}))}^2\,\d s 
   \\
&\quad + \frac\eps2 \int_0^{t\wedge \tau_\kappa} \Vert \nabla F_1(\phi_{\delta,n}) \Vert_{L^2(U,\mathbb{L}^2(\mathcal{O}))}^2\, \d s + \frac{\eps_\Gamma}{2} \int_0^{t\wedge \tau_\kappa}  \Vert \nabla_\Gamma F_2(\varphi_{\delta,n}) \Vert_{\mathscr{T}_2(U_\Gamma,\mathbb{L}^2(\Gamma))}^2\, \d s 
\\
&\quad + \int_0^{t\wedge \tau_\kappa} (\mu_{\delta,n},F_1(\phi_{\delta,n})\,\d W(s)) + \int_0^{t\wedge \tau_\kappa} (\theta_{\delta,n},F_2(\varphi_{\delta,n})\,\d W_\Gamma(s))_\Gamma  
  \\
&\quad + \sum_{k=1}^\infty \int_0^{t\wedge \tau_\kappa} \left(\frac{1}{2 \eps} \int_{\mathcal{O}} F^{\bis}_\delta(\phi_{\delta,n}) \lvert F_1(\phi_{\delta,n})e_{1,k} \rvert^2\, \d x + \frac{1}{2 \eps_\Gamma} \int_{\Gamma} G^\bis_\delta(\varphi_{\delta,n}) \lvert F_2(\varphi_{\delta,n}) e_{2,k} \rvert^2\, \d S \right)\,\d s  
     \\
&\quad - \int_0^{t \wedge \tau_\kappa} \int_{\mathcal{O}} M_{\mathcal{O}}(\phi_{\delta,n}) \nabla \mu_{\delta,n} \cdot \nabla \phi_{\delta,n}\, \d x \,\d s  + \int_0^{t\wedge \tau_\kappa}(\phi_{\delta,n}, F_{1}(\phi_{\delta,n})\,\d W(s)),
\end{aligned}
\end{equation}
having setting
\[
E_\delta(\cdot,\tilde{\cdot})
= \int_{\mathcal{O}} \frac{\eps}{2} \lvert \nabla \cdot \rvert^2 + \frac{1}{\eps} F_\delta(\cdot)\, \d x + \int_{\Gamma} \frac{\eps_\Gamma}{2} \lvert \nabla_\Gamma \tilde{\cdot} \rvert^2 + \frac{1}{\eps_\Gamma} G_\delta(\tilde{\cdot})\, \d S  + \frac{1}{2} \int_{\mathcal{O}} \lvert \cdot \rvert^2\, \d x.
\]
We will now control the first and second stochastically forced terms on the right-hand side of \eqref{Main_Galerkin_Equality2} as we did in the proof of Proposition \ref{Proposition_first_priori_estimmate}, but with some small modifications. We note that the other terms can be control exactly as we did in that proof.

\noindent
We have (cf. Proposition \ref{Proposition_first_priori_estimmate})
\begin{equation}\label{eqt5.3a}
 \mathbb{E} \sup_{\tau \in[0,t\wedge \tau_\kappa]} \left|\int_0^{\tau} \left(\mu_{\delta,n}(s),F_{1}(\phi_{\delta,n}(s)) \,\d W(s)\right) \right| 
 \leq \mathbb{E} \left(\int_0^{t\wedge \tau_\kappa} |\mu_{\delta,n}(s)|^2 \,\d s \right)^\frac{1}{2}
\end{equation}
and 
               \begin{equation}\label{eqt5.3}
                |\mu_{\delta,n}|^2
                \leq C(|\nabla \mu_{\delta,n}|^2 + |\langle \mu_{\delta,n}\rangle_{\mathcal{O}}|^2),
                \end{equation}
where $C$ may depends only on $C_1$ and $\mathcal{O}$. Besides, in the case $K=0$, we have $[K]=0$ by definition. Thus, since $F_\delta^\prime$ is  $L_{F_\delta^\prime}$-Lipschitz continuous, we obtain
\begin{align*}
|\langle \mu_{\delta,n} \rangle_{\mathcal{O}}|
= \left| \frac{1}{|\mathcal{O}|} \int_{\mathcal{O}} \frac{1}{\eps} F_\delta^\prime (\phi_{\delta,n}) \, \d x \right| 
\leq  \eps^{-1} |\mathcal{O}|^{-1} \int_{\mathcal{O}} |F_\delta^\prime (\phi_{\delta,n})|\, \d x 
&\leq \eps^{-1} |\mathcal{O}|^{-1} L_{F_\delta^\prime} \|\phi_{\delta,n}\|_{L^1(\mathcal{O})} \\
&\leq \eps^{-1} |\mathcal{O}|^{-1/2} L_{F_\delta^\prime} |\phi_{\delta,n}|,
\end{align*}
from which we deduce that $|\mu_{\delta,n}|^2 
\leq C(|\nabla \mu_{\delta,n}|^2  + |\phi_{\delta,n}|^2)$, which $C$ depending on $C_1,\,\mathcal{O},\,\eps$ and $ L_{F_\delta^\prime}$. Consequently,
\begin{align*}
 \mathbb{E} \sup_{\tau \in[0,t\wedge \tau_\kappa]} \left|\int_0^{\tau} \left(\mu_{\delta,n},F_{1}(\phi_{\delta,n}) \,\d W(s)\right) \right| 
  \leq C +  \mathbb{E} \int_0^{t\wedge \tau_\kappa} \left(\frac{1}{4} \lvert \sqrt{M_{\mathcal{O}}(\phi_{\delta,n})} \nabla \mu_{\delta,n} \rvert^2\,\d x + \frac{1}{6} |\phi_{\delta,n}|^2 \right) \, \d s, 
\end{align*}
where $C$ depends on $C_1,\,\mathcal{O},\,\eps,\, L_{F_\delta^\prime}$ and $M_0$ and not on $n$.

\noindent
Once more, arguing as in the proof of Proposition \ref{Proposition_first_priori_estimmate}, we have
\begin{equation}\label{eqt5.5a}
  \mathbb{E} \sup_{\tau \in[0,t\wedge \tau_\kappa]} \left|\int_0^{\tau} (\theta_{\delta,n}(s),F_2(\varphi_{\delta,n}(s))\, \d W_\Gamma(s))_\Gamma\right| 
  \leq C(\Gamma,C_2) \left(\mathbb{E} \int_0^{t\wedge \tau_\kappa} |\theta_{\delta,n}(s)|_\Gamma^2 \, \d s \right)^\frac{1}{2},
\end{equation}
$C$ depends on $\Gamma$ and $C_2$, and because $\langle \theta_{\delta,n} \rangle_\Gamma= \eps_\Gamma^{-1} \langle G_\delta^\prime(\varphi_{\delta,n}) \rangle_\Gamma$ and $G_\delta^\prime$ is  $L_{G_\delta^\prime}$-Lipschitz continuous, one has \adda{
\begin{align*}
|\langle \theta_{\delta,n} \rangle_{\Gamma}|
\leq \eps_\Gamma^{-1} |\Gamma|^{-1} \int_{\Gamma} |G_\delta^\prime (\varphi_{\delta,n})|\, \d S 
\leq  \eps_\Gamma^{-1} |\Gamma|^{-1} L_{G_\delta^\prime} \|\varphi_{\delta,n}\|_{L^1(\Gamma)} 
&\leq  \eps_\Gamma^{-1} |\Gamma|^{-1/2} L_{G_\delta^\prime} |\varphi_{\delta,n}|_\Gamma \\
&\leq  \eps_\Gamma^{-1} |\Gamma|^{-1/2} L_{G_\delta^\prime} |\phi_{\delta,n}|_\Gamma \\
&\leq C \|\phi_{\delta,n}\|_{V_1},
\end{align*}
}
where we have also used  \eqref{eq5.1c} and the embedding of $H^1(\mathcal{O})$ in $L^2(\Gamma)$. Here $C$ depends on $\eps_\Gamma,\,\Gamma,\,\mathcal{O}$ and $L_{G_\delta^\prime}$. Hence, we have
\begin{align}\label{eqt5.4}
|\theta_{\delta,n}|_\Gamma^2
\leq 2 |\theta_{\delta,n} - \langle \theta_{\delta,n}\rangle_{\Gamma}|_\Gamma^2 + 2 |\langle \theta_{\delta,n} \rangle_{\Gamma}|_\Gamma^2 
\leq C(|\nabla_\Gamma \theta_{\delta,n}|_{\Gamma}^2 + \|\phi_{\delta,n}\|_{V_1}^2),
\end{align}
with $C=C(\Gamma,\eps_\Gamma,\mathcal{O},L_{G_\delta^\prime})$. Accordingly, we learn that
\begin{align*}
&\mathbb{E} \sup_{\tau \in[0,t\wedge \tau_\kappa]} \left|\int_0^{\tau} (\theta_{\delta,n},F_2(\varphi_{\delta,n})\, \d W_\Gamma(s))_\Gamma\right| \\
&\leq C +  \mathbb{E} \int_0^{t\wedge \tau_\kappa} \left(\frac{1}{2}  |\sqrt{M_{\Gamma}(\varphi_{\delta,n})} \nabla_\Gamma \theta_{\delta,n}|_\Gamma^2\,\d x + \frac{1}{6} \|\phi_{\delta,n}\|_{V_1}^2 \right) \, \d s, 
\end{align*}
where $C$ depends on $C_2,\,\Gamma,\,\mathcal{O},\,\eps_\Gamma,\, L_{G_\delta^\prime}$ and $N_0$ and not on $n$. With these previous estimates in hand, we can now proceed as we did in the proof of Proposition \ref{Proposition_first_priori_estimmate} and deduce that the process $(\bu_{\delta,n},\phi_{\delta,n},\varphi_{\delta,n},\mu_{\delta,n},\theta_{\delta,n})$ satisfies the following uniform bound with respect to $n$:
\begin{equation}\label{eqt5.5}
\begin{aligned}
 &\mathbb{E} \sup_{\tau \in[0,T]} [E_\delta(\phi_{\delta,n}(\tau), \varphi_{\delta,n}(\tau))] + \mathbb{E} \int_{Q_T} [\lvert D\bu_{\delta,n}(s) \rvert^2 + \lambda(\phi_{\delta,n}(s)) \lvert \bu_{\delta,n}(s) \rvert^2] \,\d x \,\d s \\
& + \mathbb{E} \int_0^{T} |\bu_{\delta,n}(s)|_\Gamma^2 \,\d s + \mathbb{E} \int_0^{T} [|\nabla \mu_{\delta,n}(s)|^2 + |\nabla_\Gamma \theta_{\delta,n}(s)|_{\Gamma}^2] \,\d s \leq C,
\end{aligned}
\end{equation}
where $C$ denote a generic positive constant depending only on $\tilde{c}_F,\,\tilde{c}_G,\,\mathcal{O},\,\Gamma,\,\delta,\,C_1,\,C_2$, $M_0,\,N_0$, the parameters of the system  \eqref{eq4.34a}-\eqref{eq4.34d} and the norms of the initial data. 

\noindent
Next, let $r>2$ be arbitrary but fixed. Coming back to \eqref{Main_Galerkin_Equality2}, using \eqref{eqt5.3}, \eqref{eqt5.4}, and arguing similarly as in the proof of Proposition \ref{Proposition_second_priori_estimmate}, we find
\begin{equation}\label{eqt5.6}
\begin{aligned}
& \mathbb{E} \sup_{\tau \in[0,t]} [E_\delta(\phi_{\delta,n}(\tau), \varphi_{\delta,n}(\tau))]^\frac{r}{2} + \mathbb{E} \left(\int_{Q_t} \lvert D\bu_{\delta,n} \rvert^2 \,\d x \,\d s \right)^\frac{r}{2}  + \mathbb{E} \left(\int_{Q_t} \lambda(\phi_{\delta,n}) \lvert \bu_{\delta,n} \rvert^2 \,\d x \,\d s \right)^\frac{r}{2} \\
&\quad + \mathbb{E} \left(\int_0^{t} |\bu_{\delta,n}|_{\Gamma}^2 \,\d s\right)^\frac{r}{2} + \mathbb{E} \left(\int_0^{t} \lvert \nabla \mu_{\delta,n} \rvert^2 \,\d s \right)^\frac{r}{2} + \mathbb{E} \left(\int_0^{t} \lvert \nabla_\Gamma \theta_{\delta,n} \rvert^2 \,\d s\right)^\frac{r}{2}  
\leq C \quad \forall t \in [0,T],
\end{aligned}
\end{equation}
where $C$ is a generic positive constant depending only on $T,\,\mathcal{O},\,\Gamma,\,\delta,\,C_1,\,C_2,,M_0,\,\bar{M}_0,\,\tilde{c}_F$, $\tilde{c}_G,\,\nu_0,\,N_0,\,\lambda_0,\,r$, the parameters of the system \eqref{eq4.34a}-\eqref{eq4.34d} and the norms of the initial data. Moreover, given $(v,v_\Gamma) \in \mathbb{V}$ such that $\|(v,v_\Gamma)\|_{\mathbb{V}} \leq 1$, we infer from \eqref{eq4.34d} that 
\begin{align*}
\langle (\mu_{\delta,n},\theta_{\delta,n}), (v,v_\Gamma) \rangle_{\mathbb{V}^\prime,\mathbb{V}}
 &= \eps \int_{\mathcal{O}} \nabla \phi_{\delta,n} \cdot \nabla v\, \d x + \int_{\mathcal{O}} \frac{1}{\eps} F_\delta^\prime(\phi_{\delta,n}) v\, \d x + \eps_\Gamma \int_\Gamma \nabla_\Gamma \varphi_{\delta,n} \cdot \nabla_\Gamma v_\Gamma \, \d S \\
 &\quad  + \int_\Gamma \frac{1}{\eps_\Gamma} G_\delta^\prime(\varphi_{\delta,n}) v_\Gamma \, \d S, 
\end{align*}
since $K=0$. Hence,
\begin{align*}
\|(\mu_{\delta,n},\theta_{\delta,n})\|_{\mathbb{V}^\prime}
\leq C(|\nabla \phi_{\delta,n}| + |\phi_{\delta,n}| + |\nabla_\Gamma \varphi_{\delta,n}|_{\Gamma} +  |\varphi_{\delta,n}|_\Gamma)
\leq C(\|\phi_{\delta,n}\|_{V_1} + |\nabla_\Gamma \varphi_{\delta,n}|_{\Gamma}),
\end{align*}
where we used the fact that $|F_\delta^\prime(\phi_{\delta,n})| \leq \tilde{c}_{1,\delta} |\phi_{\delta,n}|$, $ |G_\delta^\prime(\varphi_{\delta,n})|_\Gamma \leq \tilde{c}_{2,\delta} |\phi_{\delta,n}|_\Gamma$, and  \newline
$|\varphi_{\delta,n}|_\Gamma = |\phi_{\delta,n}|_\Gamma \leq C(\mathcal{O},\Gamma) \|\phi_{\delta,n}\|_{V_1}$ due to \eqref{eq5.1c} and the embedding of $V_1$ in $L^2(\Gamma)$. Here, the constant $C$ may depends on $\tilde{c}_F,\,\tilde{c}_G,\,\delta,\,\eps$ and $\eps_\Gamma$. It then follows 
from the estimates \eqref{eqt5.6} and similar reasoning as in the proof of the proposition \ref{Proposition_third_priori_estimmate} that
\begin{equation}\label{eqt5.7}
\begin{aligned}
\mathbb{E} \left(\int_0^t \|(\mu_{\delta,n}(s),\theta_{\delta,n}(s))\|_{\mathbb{H}}^4 \, \d s \right)^\frac{r}{4} &\leq C,
\\
\mathbb{E} \sup_{s\in [0,t]} \|F_{1}(\phi_{\delta,n}(s))\|_{\mathscr{T}_2(U,V_1)}^r &\leq C,
   \\
\mathbb{E} \sup_{s\in [0,t]} \|F_{\Gamma}(\varphi_{\delta,n}(s))\|_{\mathscr{T}_2(U_\Gamma,V_\Gamma)}^r &\leq C,   
\end{aligned}
\end{equation}
for all $t\in[0,T]$, with $C$ being independent of $n$. Furthermore, the fact that the laws of the approximating solutions are tights is a verbatim reproduction of similar results as in Subsection \ref{subs-4.4}. Hence, by Prokhorov's theorem (see \cite[Theorem 2.7]{Ikeda+Watanabe_1989}) and Skorokhod's representation theorem (see \cite[Theorem 1.10.4, Addendum 1.10.5]{van der Vaart+Wellner_1996}), recalling the estimates \eqref{eqt5.5}-\eqref{eqt5.7}, we can find a new probability space $(\bar{\Omega},\bar{\mathcal{F}},\bar{\mathbb{P}})$ and a sequence of random variables $X_n: (\bar{\Omega},\bar{\mathcal{F}}) \to (\Omega,\mathcal{F})$ such that the law of $X_n$ is $\mathbb{P}$ for every $n \in \mathbb{N}$, namely, $\mathbb{P}= \bar{\mathbb{P}} \circ X_n^{-1}$, so that composition with $X_n$ preserves laws, and the following convergences hold
\begin{equation}\label{eq5.8}
\begin{aligned}
\bar{\phi}_{\delta,n}&\coloneq \phi_{\delta,n} \circ X_n \to \bar{\phi}_{\delta} \quad &\text{in} \quad C([0,T];L^2(\mathcal{O})),
 \\
\bar{\varphi}_{\delta,n}&\coloneq \varphi_{\delta,n} \circ X_n \to \bar{\varphi}_{\delta} \quad &\text{in} \quad C([0,T];L^2(\Gamma)),  
     \\
\bar{\mu}_{\delta,n}&\coloneq \mu_{\delta,n} \circ X_n \rightharpoonup \bar{\mu}_{\delta} \quad &\text{in} \quad L^2(0,T;V_1) \cap L^4(0,T;L^2(\mathcal{O})), 
         \\
\bar{\theta}_{\delta,n}&\coloneq \theta_{\delta,n} \circ X_n \rightharpoonup \bar{\theta}_{\delta} \quad &\text{in} \quad L^2(0,T;V_\Gamma) \cap L^4(0,T;L^2(\Gamma)), 
              \\
\bar{W}_{1,n}&\coloneq W_{1,n} \circ X_n \to \bar{W} \quad &\text{in} \quad C([0,T];U_0),  
                  \\
\bar{W}_{2,n}&\coloneq W_{2,n} \circ X_n \to \bar{W}_{\Gamma} \quad &\text{in} \quad C([0,T];U_0^\Gamma),  
\end{aligned}
\end{equation}
$\bar{\mathbb{P}}$-a.s., for some limiting process $(\bar{\phi}_{\delta},\bar{\varphi}_{\delta},\bar{\mu}_{\delta},\bar{\theta}_{\delta},\bar{W},\bar{W}_{\Gamma})$ belonging to the specified spaces. Moreover, from \eqref{eq5.8} and the preservation of laws under $X_n$, recalling once more the estimates \eqref{eqt5.5}-\eqref{eqt5.7}, using also the Vitali convergence theorem and the Banach-Alaoglu theorem, we have for some subsequences still labeled by the same subscript
\begin{equation}\label{eq5.9}
\begin{aligned}
\bar{\phi}_{\delta,n} \to \bar{\phi}_{\delta} \quad &\text{in} \quad L^\ell(\bar{\Omega};C([0,T];L^2(\mathcal{O}))) \quad \forall \ell< r,
\\
(\bar{\phi}_{\delta,n},\bar{\varphi}_{\delta,n}) \stackrel{*}{\rightharpoonup} (\bar{\phi}_{\delta},\bar{\varphi}_{\delta}) \quad &\text{in} \quad L^r(\bar{\Omega};L^\infty(0,T;\mathbb{V}_0)) \cap L^\frac{r}{2}(\bar{\Omega}; W^{\alpha,r}(0,T;\mathbb{V}_0^\prime)),
    \\
\bar{\varphi}_{\delta,n} \to \bar{\varphi}_{\delta} \quad &\text{in} \quad L^\ell(\bar{\Omega};C([0,T];L^2(\Gamma))) \quad \forall \ell<r,
              \\
\bar{\mu}_{\delta,n} \rightharpoonup \bar{\mu}_{\delta} \quad &\text{in} \quad L^r(\bar{\Omega};L^2(0,T;V_1)) \cap L^r(\bar{\Omega};L^4(0,T;L^2(\mathcal{O}))), 
                  \\
\nabla \bar{\mu}_{\delta,n} \rightharpoonup \nabla \bar{\mu}_{\delta} \quad &\text{in} \quad L^r(\bar{\Omega};L^2(0,T;\mathbb{L}^2(\mathcal{O}))),
                       \\
\bar{\theta}_{\delta,n} \rightharpoonup \bar{\theta}_{\delta} \quad &\text{in} \quad L^r(\bar{\Omega};L^2(0,T;V_\Gamma)) \cap L^r(\bar{\Omega};L^4(0,T;L^2(\Gamma))), 
                            \\
\nabla_\Gamma \bar{\theta}_{\delta,n} \rightharpoonup \nabla_\Gamma \bar{\theta}_{\delta} \quad &\text{in} \quad L^r(\bar{\Omega};L^2(0,T;\mathbb{L}^2(\Gamma))), 
                                 \\
\bar{W}_{1,n} \to \bar{W} \quad &\text{in} \quad L^\ell(\bar{\Omega};C([0,T];U_0)) \quad \forall \ell< r,
                                     \\
\bar{W}_{2,n} \to \bar{W}_{\Gamma} \quad &\text{in} \quad L^\ell(\bar{\Omega};C([0,T];U_0^\Gamma)) \quad \forall \ell< r
\end{aligned}
\end{equation}
and
\begin{equation}
\begin{aligned}
\bar{\bu}_{\delta,n}\coloneq \bu_{\delta,n} \circ X_n \rightharpoonup \bar{\bu}_{\delta} \quad &\text{in} \quad L^r(\bar{\Omega};L^2(0,T;V)), 
   \\
\bar{\bu}_{\delta,n} \lvert_\Gamma \rightharpoonup \bar{\bu}_{\delta} \lvert_\Gamma \quad &\text{in} \quad L^r(\bar{\Omega};L^2(0,T;L^2(\Gamma))).
\end{aligned}
\end{equation}
On the other hand, arguing as in Subsection \ref{subs-4.4}, we learn that the limiting process $(\bar{\phi}_{\delta},\bar{\varphi}_{\delta},\bar{\mu}_{\delta},\bar{\theta}_{\delta},\bar{W},\bar{W}_{\Gamma})$ satisfies for any test functions $\bv \in V,\, \upsilon \in V_1,\, \upsilon \lvert_\Gamma \in V_\Gamma$ and $(\psi,\psi \lvert_\Gamma) \in \mathbb{V}_0$,
\begin{subequations}\label{eq5.11}
\begin{align}
  & 2 \int_{Q_t} \nu(\bar{\phi}_{\delta}(s)) D\bar{\bu}_{\delta}(s): D\bv\, \d x\, \d s + \int_{Q_t} \lambda(\bar{\phi}_{\delta}(s)) \bar{\bu}_{\delta}(s) \cdot \bv \, \d x\, \d s + \int_{\Sigma_t} \gamma(\bar{\varphi}_{\delta}(s)) \bar{\bu}_{\delta}(s) \cdot \bv \, \d S \, \d s \notag \\
  &\quad 
  =  - \int_{\Sigma_t} \bar{\varphi}_{\delta}(s) \nabla_\Gamma \bar{\theta}_{\delta}(s) \cdot \bv \, \d S \, \d s  - \int_{Q_t} \bar{\phi}_{\delta}(s) \nabla \bar{\mu}_{\delta}(s) \cdot \bv \, \d x\, \d s, \label{eq5.11a}
   \\ \notag \\
 &\langle \bar{\phi}_{\delta}(t),\upsilon\rangle_{V_1^\prime,V_1} - \int_{Q_t} \bar{\phi}_{\delta}(s) \bar{\bu}_{\delta}(s) \cdot \nabla \upsilon \, \d x\, \d s + \int_{Q_t} M_\mathcal{O}(\bar{\phi}_{\delta}(s)) \nabla \bar{\mu}_{\delta}(s) \cdot \nabla \upsilon \, \d x\, \d s \notag \\ 
 &\quad = (\phi_{0}, \upsilon) + \left(\int_0^t F_{1}(\bar{\phi}_{\delta}(s)) \,\d \bar{W}(s), \upsilon \right), \label{eq5.11b}
        \\
 &\langle \bar{\varphi}_{\delta}(t), \upsilon \lvert_\Gamma \rangle_{V_\Gamma^\prime,V_\Gamma} + \int_{\Sigma_t} (-\bar{\varphi}_{\delta} \bar{\bu}_{\delta} \cdot \nabla_\Gamma \upsilon \lvert_\Gamma + M_\Gamma(\bar{\varphi}_{\delta}) \nabla_\Gamma \bar{\theta}_{\delta} \cdot \nabla_\Gamma \upsilon \lvert_\Gamma) \, \d S \, \d s \label{eq5.11c} \\
 &\quad= (\varphi_{0}, \upsilon \lvert_\Gamma)_\Gamma +  \left(\int_0^t F_2(\bar{\varphi}_{\delta}(s)) \,\d \bar{W}_{\Gamma}(s), \upsilon \lvert_\Gamma \right)_\Gamma, \notag 
                   \\ \notag \\
&\int_{\mathcal{O}} \bar{\mu}_{\delta} \psi \, \d x  + \int_{\Gamma} \bar{\theta}_{\delta} \psi \lvert_\Gamma \, \d S  
- \eps \int_{\mathcal{O}} \nabla \bar{\phi}_{\delta} \cdot \nabla \psi \, \d x  - \int_{\mathcal{O}} \frac1\eps F_\delta^\prime(\bar{\phi}_{\delta}) \psi \, \d x \notag \\ 
&\quad= \eps_\Gamma \int_{\Gamma} \nabla_\Gamma \bar{\varphi}_{\delta} \cdot \nabla_\Gamma \psi \lvert_\Gamma \, \d S  +  \frac{1}{\eps_\Gamma} \int_{\Gamma} G_\delta^\prime(\bar{\varphi}_{\delta}) \psi \lvert_\Gamma \, \d S, \label{eq5.11d}
\end{align}
\end{subequations} 
for all $t \in[0,T]$ and $\bar{\mathbb{P}}$-a.s.
\subsection{Uniform estimates with respect to delta}
Here, we derive some uniform estimates with respect to $\delta$ that enable us to construct a solution to the original problem \eqref{eq1.1a}-\eqref{eq1.10a} (in the case $K=0$) through a stochastic compactness argument. Thus, the constant $C$ in this section is independent of $\delta$, but may depend on other parameters that are explicitly pointed out if necessary. Let us point that the arguments necessary to construct the solution to the problem \eqref{eq1.1a}-\eqref{eq1.10a} in the case $K=$ are very similar to the ones in Subsection \ref{sect_Uniform_estimates_delta}, and then, we will omit certain details.

\noindent
Since $X_n$ preserves laws, we can express \eqref{Main_Galerkin_Equality2} in the new variables, i.e., for any $r \geq 2$, we have
\begin{equation}\label{eq5.14}
\begin{aligned}
& \bar{\mathbb{E}} \sup_{\tau \in[0,t]} [E_\delta(\bar{\phi}_{\delta,n}(\tau), \bar{\varphi}_{\delta,n}(\tau))]^\frac{r}{2} + \bar{\mathbb{E}} \left(\int_{Q_t} \lambda(\bar{\phi}_{\delta,n}(s)) \lvert \bar{\bu}_{\delta,n}(s) \rvert^2 \,\d x \,\d s \right)^\frac{r}{2}  \\
&\quad +  \bar{\mathbb{E}} \left|\int_{0}^t |\nabla \bar{\bu}_{\delta,n}|^2  \,\d s \right|^\frac{r}{2}  +  \bar{\mathbb{E}} \left(\int_0^{t} |\bar{\bu}_{\delta,n}|_{\Gamma}^2 \,\d s\right)^\frac{r}{2} +  \bar{\mathbb{E}} \left(\int_0^{t} |\nabla \bar{\mu}_{\delta,n}|^2 \,\d s \right)^\frac{r}{2} +  \bar{\mathbb{E}} \left(\int_0^{t} |\nabla_\Gamma \bar{\theta}_{\delta,n}|^2 \,\d s\right)^\frac{r}{2}  
 \\
&\leq C \bigg[ 1 + [E_\delta(\phi_n(0),\varphi_n(0))]^\frac{r}{2} + \bar{\mathbb{E}} \left(\int_0^t |\nabla \bar{\phi}_{\delta,n}(s)|^2\, \d s \right)^\frac{r}{2} + \bar{\mathbb{E}} \left(\int_0^t |\nabla_\Gamma \bar{\theta}_{\delta,n}(s)|_\Gamma^2\, \d s \right)^\frac{r}{2} \\
&\quad + \bar{\mathbb{E}} \left(\int_0^t |\bar{\phi}_{\delta,n}(s)|^2 \, \d s \right)^\frac{r}{4} + \bar{\mathbb{E}} \left(\int_0^t |\bar{\mu}_{\delta,n}(s)|^2 \, \d s \right)^\frac{r}{4} + \bar{\mathbb{E}} \left(\int_0^t |\bar{\theta}_{\delta,n}(s)|_{\Gamma}^2 \, \d s \right)^\frac{r}{4} \\
&\quad + \bar{\mathbb{E}} \left|\sum_{k=1}^\infty \int_{Q_t} F''_\delta(\bar{\phi}_{\delta,n}(s)) |F_1(\bar{\phi}_{\delta,n}(s))e_{1,k}|^2 \, \d x \,\d s \right|^\frac{r}{2} \\
&\quad + \bar{\mathbb{E}} \left|\sum_{k=1}^\infty \int_{\Sigma_t} G''_\delta(\bar{\varphi}_{\delta,n}(s)) |F_2(\bar{\varphi}_{\delta,n}(s))e_{2,k}|^2\, \d S\,\d s\right|^\frac{r}{2}  \bigg],
\end{aligned}
\end{equation}
where we used the bounds on the diffusion coefficients \eqref{eq4.82}, \eqref{eq4.83}, \eqref{eq4.82a}, \eqref{eqt4.87} and \eqref{eq4.88}, and the bounds \eqref{eqt5.3a} and \eqref{eqt5.5a} for the stochastic forced terms. Here, the constant $C$ depends on $\eps,\,\eps_\Gamma,\,\nu_0,\,\lambda_0,\,M_0,\,N_0,\,\bar{M}_0,\,\mathcal{O},\,\Gamma,\,C_1$, $C_2,\,T$ and $r$. Next, recalling that $\varphi_{\delta,n}(t)= \phi_{\delta,n}(t)$ on  $\Gamma$ $\bar{\mathbb{P}}$-a.s. and $K=0$, following the steps of the proof of \eqref{eq4.141}
\begin{equation}
\begin{aligned}
& \bar{\mathbb{E}} \sup_{\tau \in[0,t]} [E_\delta(\bar{\phi}_{\delta}(\tau), \bar{\varphi}_{\delta}(\tau))]^\frac{r}{2} + \bar{\mathbb{E}} \left(\int_{Q_t} \lambda(\bar{\phi}_{\delta}(s)) \lvert \bar{\bu}_{\delta}(s) \rvert^2 \,\d x \,\d s \right)^\frac{r}{2}  \\
&\quad +  \bar{\mathbb{E}} \left|\int_{0}^t |\nabla \bar{\bu}_{\delta}(s)|^2  \,\d s \right|^\frac{r}{2}  +  \bar{\mathbb{E}} \left(\int_0^{t} |\bar{\bu}_{\delta}(s)|_{\Gamma}^2 \,\d s\right)^\frac{r}{2} \\
&\quad +  \bar{\mathbb{E}} \left(\int_0^{t} |\nabla \bar{\mu}_{\delta}(s)|^2 \,\d s \right)^\frac{r}{2} +  \bar{\mathbb{E}} \left(\int_0^{t} |\nabla_\Gamma \bar{\theta}_{\delta}(s)|^2 \,\d s\right)^\frac{r}{2}  
 \\
&\leq C \left[ 1 + [\tilde{E}(\phi_0,\varphi_0)]^\frac{r}{2} + \bar{\mathbb{E}} \int_0^t \sup_{0\leq s \leq \tau} [E_\delta(\bar{\phi}_{\delta}(s), \bar{\varphi}_{\delta}(s))]^\frac{r}{2} \, \d \tau \right]
\end{aligned}
\end{equation}
for all $r \geq 4$ and for a certain certain constant $C$ independent of $\delta$. Here 
$$\tilde{E}(\phi_0,\varphi_0)
= \frac{\eps}{2} |\nabla \phi_0|^2 + \frac{1}{\eps} \|F(\phi_0)\|_{L^1(\mathcal{O})} + \frac{\eps_\Gamma}{2} |\nabla_\Gamma \varphi_0|_{\Gamma}^2 + \frac{1}{\eps_\Gamma} \|G(\varphi_0)\|_{L^1(\Gamma)} + \frac{1}{2} |\phi_0|^2.$$ 
Hence, the Gronwall entails that there exists $c_r>0$ independent of $\delta$ such that
\begin{equation}\label{eq5.16}
\begin{aligned}
\|\bar{\bu}_{\delta}\|_{L^r(\bar{\Omega};L^2(0,T;\mathbb{L}^2(\Gamma)))}  &\leq c_r, 
\\
\|\bar{\bu}_{\delta}\|_{L^r(\bar{\Omega};L^2(0,T;V))} &\leq c_r,
   \\
\|(\bar{\phi}_{\delta},\bar{\varphi}_{\delta})\|_{L^r(\bar{\Omega};L^\infty(0,T;\mathbb{V}_0))} &\leq c_r, 
      \\
\|F_\delta(\bar{\phi}_{\delta})\|_{L^\frac{r}{2}(\bar{\Omega};L^\infty(0,T;L^1(\mathcal{O})))}  &\leq c_r,
         \\
\|G_\delta(\bar{\varphi}_{\delta})\|_{L^\frac{r}{2}(\bar{\Omega};L^\infty(0,T;L^1(\Gamma)))} &\leq c_r, \\
\|\nabla_\Gamma \bar{\theta}_{\delta}\|_{L^r(\bar{\Omega};L^\infty(0,T;\mathbb{L}^2(\Gamma)))} &\leq c_r,  \\
\|\nabla_\Gamma \bar{\theta}_{\delta}\|_{L^r(\bar{\Omega};L^2(0,T;\mathbb{L}^2(\Gamma)))} &\leq c_r,       \\
\|\nabla \bar{\mu}_{\delta} \|_{L^r(\bar{\Omega};L^2(0,T;\mathbb{L}^2(\mathcal{O})))} &\leq c_r,
                         \\
\|\nabla_\Gamma \bar{\varphi}_{\delta}\|_{L^r(\bar{\Omega};L^\infty(0,T;\mathbb{L}^2(\Gamma)))} &\leq c_r. 
\end{aligned}
\end{equation}
Moreover, we observe that 
$|\bar{\mu}_{\delta,n}|^2 
 \leq C( 1 + |\nabla \bar{\mu}_{\delta,n}|^2 + \|F_\delta(\bar{\phi}_{\delta,n})\|_{L^1(\mathcal{O})}^2 + |\bar{\phi}_{\delta,n}|^2)
 $,
with $C$ depending only on $\eps,\,\mathcal{O},\,\Gamma,\,C_F,\,\tilde{c}_F$, and
$|\bar{\theta}_{\delta,n}|_\Gamma^2
\leq C (1 + |\nabla_\Gamma \bar{\theta}_{\delta,n}|_\Gamma^2 +   \|G_\delta(\bar{\varphi}_{\delta,n})\|_{L^1(\Gamma)}^2 + \|\bar{\phi}_{\delta,n}\|_{V_1}^2)
$ for some $C=C(C_G,\mathcal{O},\Gamma,\tilde{c}_G,\eps,\eps_\Gamma)$. This, together with \eqref{eq5.9}, the pointwise convergence for $(\bar{\phi}_{\delta,n},\bar{\varphi}_{\delta,n})_n$, the continuity of the maps $F_\delta$ and $G_\delta$, and the weak lower semicontinuity of the norm imply that 
\begin{equation}\label{eq5.17}
\begin{aligned}
\|\bar{\mu}_{\delta}\|_{L^r(\bar{\Omega};L^2(0,T;L^2(\mathcal{O})))} \leq c_r, \\
\|\bar{\theta}_{\delta}\|_{L^r(\bar{\Omega};L^2(0,T;L^2(\Gamma)))} \leq c_r,
\end{aligned}
\end{equation}
with $c_r$ independent of $\delta$. Furthermore, from \eqref{eq5.16} and \eqref{eq5.17}, we obtain
\begin{equation}\label{eq5.18}
\begin{aligned}
\|\bar{\theta}_{\delta}\|_{L^r(\bar{\Omega};L^2(0,T;V_\Gamma))} &\leq c_r, \\
\|\bar{\mu}_{\delta}\|_{L^r(\bar{\Omega};L^2(0,T;V_1))} &\leq c_r,
\end{aligned}
\end{equation}
with $c_r$ being independent of $\delta$.

\noindent
We will now derive an uniform estimate w.r.t $\delta$ for $F_\delta^\prime(\bar{\phi}_{\delta})$ and $G_\delta^\prime(\bar{\varphi}_{\delta})$. In comparison to the the proof of \eqref{eq4.147}, or \eqref{eq4.159}, or \eqref{eq4.158}, we will not use the assumption \eqref{F''_and G''_additional_condition}, which implies \eqref{F'_and G'_additional_condition}. Indeed, recalling that $F_\delta^\prime(\bar{\phi}_{\delta})= \mathbb{A}_\delta(\bar{\phi}_{\delta}) - \tilde{c}_F \bar{\phi}_{\delta}$, then taking $(\psi,\psi \lvert_\Gamma)=(\mathbb{A}_\delta(\bar{\phi}_{\delta}),0)$ as test function in \eqref{eq5.11d}, we have
\begin{align*}
 &\eps \int_{\mathcal{O}} \Psi^\prime_\delta(\bar{\phi}_{\delta}) |\nabla \bar{\phi}_{\delta}|^2 \, \d x + \int_{\mathcal{O}} \frac1\eps |F_\delta^\prime(\bar{\phi}_{\delta})|^2 \, \d x \\ 
 &=  \int_{\mathcal{O}}  \bar{\mu}_{\delta} F_\delta^\prime(\bar{\phi}_{\delta}) \, \d x + \tilde{c}_F \int_{\mathcal{O}} \bar{\mu}_{\delta} \bar{\phi}_{\delta} \, \d x - \frac{\tilde{c}_F}{\eps} \int_{\mathcal{O}} F_\delta^\prime(\bar{\phi}_{\delta}) \bar{\phi}_{\delta} \, \d x \\
 &\leq \frac{1}{2\eps} \int_{\mathcal{O}} |F_\delta^\prime(\bar{\phi}_{\delta})|^2 \, \d x + \frac{3 \eps}{2} \int_{\mathcal{O}} |\bar{\mu}_{\delta}|^2\, \d x  + \frac{3 \tilde{c}_F^2}{2\eps} \int_{\mathcal{O}} |\bar{\phi}_{\delta}|^2 \, \d x
\end{align*}
from which and the fact that $\Psi^\prime_\delta(r)\geq 0$ for all $r\in \mathbb{R}$, we obtain
    \begin{equation}\label{eq5.19}
      |F_\delta^\prime(\bar{\phi}_{\delta})|^2
      \leq 3 \eps^2 |\bar{\mu}_{\delta}|^2  + 3 \tilde{c}_F^2 |\bar{\phi}_{\delta}|^2.
    \end{equation}
Moreover, recalling that $G_\delta^\prime(\bar{\varphi}_{\delta})= \Psi^\Gamma_\delta(\bar{\varphi}_{\delta}) - \tilde{c}_G \bar{\varphi}_{\delta}$, then taking $(\psi,\psi \lvert_\Gamma)=(0,\Psi^\Gamma_\delta(\bar{\varphi}_{\delta}))$ as test function in \eqref{eq5.11d}, we have
\begin{align*}
&\eps_\Gamma \int_{\Gamma} (\mathbb{A}_\delta^\Gamma)^\prime(\bar{\varphi}_{\delta}) |\nabla_\Gamma \bar{\varphi}_{\delta}|^2 \, \d S +  \frac{1}{\eps_\Gamma} \int_{\Gamma} |G_\delta^\prime(\bar{\varphi}_{\delta})|^2 \, \d S \\
&= \int_{\Gamma} \bar{\theta}_{\delta} G_\delta^\prime(\bar{\varphi}_{\delta}) \, \d S + \tilde{c}_G \int_{\Gamma} \bar{\theta}_{\delta} \bar{\varphi}_{\delta} \, \d S - \frac{\tilde{c}_G}{\eps_\Gamma} \int_{\Gamma} G_\delta^\prime(\bar{\varphi}_{\delta}) \bar{\varphi}_{\delta} \, \d S
\end{align*}
from which and the monotonicity of $\Psi^\Gamma_\delta$, we get
    \begin{equation}\label{eq5.20}
        |G_\delta^\prime(\bar{\varphi}_{\delta})|_\Gamma^2
          \leq 3 \eps_\Gamma^2 |\bar{\theta}_{\delta}|_\Gamma^2  + 3 \tilde{c}_G^2 |\bar{\varphi}_{\delta}|_\Gamma^2.
    \end{equation}
Hence, from \eqref{eq5.16}$\text{-}3)$, \eqref{eq5.17}, \eqref{eq5.19}, and \eqref{eq5.20}, we infer 
\begin{equation}\label{eq5.21}
\begin{aligned}
\|F_\delta^\prime(\bar{\phi}_{\delta})\|_{L^r(\bar{\Omega};L^2(0,T;L^2(\mathcal{O})))} \leq c_r, \\
\|G_\delta^\prime(\bar{\varphi}_{\delta})\|_{L^r(\bar{\Omega};L^2(0,T;L^2(\Gamma)))} \leq c_r,
\end{aligned}
\end{equation}
where the constant $c_r$ is independent of $\delta$. On the one hand, since the sequence $(\bar{\phi}_{\delta,n},\bar{\varphi}_{\delta,n},\bar{\mu}_{\delta,n},\bar{\theta}_{\delta,n})$ satisfies the system \eqref{eq5.1a}-\eqref{eq5.1c}, $\bar{\mathbb{P}}$-a.s., by using the regularity theory for elliptic problems with bulk-surface coupling (see \cite[Theorem 3.3]{Knopf+Liu_2021}), we have
\begin{align*}
\|(\bar{\phi}_{\delta,n}(t),\bar{\varphi}_{\delta,n}(t))\|_{\mathcal{H}^2}^2
&\leq C \|(f_{n,\delta}(t),g_{n,\delta}(t))\|_{\mathbb{H}}^2 \\
&\leq C (|\bar{\mu}_{\delta,n}(t)|^2 + | F_\delta^\prime(\bar{\phi}_{\delta,n}(t))|^2 + |\bar{\theta}_{\delta,n}(t)|_{\Gamma}^2 + |G_\delta^\prime(\bar{\varphi}_{\delta,n}(t))|_\Gamma^2)
\end{align*}
$C$ may depend on $\eps,\,\eps_\Gamma$ and not of $n$ and $\delta$. On the other hand, from the convergence \eqref{eq5.9}, the pointwise convergence for $(\bar{\phi}_{\delta,n},\bar{\varphi}_{\delta,n})$, the Lipschitz-continuity of $F_\delta^\prime$ and $G_\delta^\prime$ and the lower-semicontinuity, we find
   \begin{equation}\label{eq5.22}
     \|(\bar{\phi}_{\delta}(t),\bar{\varphi}_{\delta}(t))\|_{\mathcal{H}^2}^2
     \leq C (|\bar{\mu}_{\delta}(t)|^2 + | F_\delta^\prime(\bar{\phi}_{\delta}(t))|^2 + |\bar{\theta}_{\delta}(t)|_{\Gamma}^2 + |G_\delta^\prime(\bar{\varphi}_{\delta}(t))|_\Gamma^2), 
   \end{equation}
for all $t\in[0,T]$ and $\bar{\mathbb{P}}$-a.s. It then follows from \eqref{eq5.17}, \eqref{eq5.21} and \eqref{eq5.22} that 
\begin{equation}
\|(\bar{\phi}_{\delta},\bar{\varphi}_{\delta})\|_{L^r(\bar{\Omega};L^2(0,T;\mathcal{H}^2))} \leq c_r,
\end{equation}
with $c_r$ being independent of $\delta$.

\noindent
\textbf{Passage to the limit as $\delta \to 0^+$.}
We are now in a position to pass to the limit as $\delta\to 0^+$ along a suitable subsequence. However, the argument is very similar  to the one of Subsection \ref{Sect_passage to the limit_delta}. Thus, we will omit some details for the sake of brevity. As in Subsection \ref{subs-4.4}, we can check that the family of laws of $(\bar{\phi}_{\delta},\bar{\varphi}_{\delta},\bar{W}_{1,\delta}\coloneq\bar{W},\bar{W}_{2,\delta}\coloneq\bar{W}_{\Gamma})_\delta$ is tight in the product space 
$
\Lambda\coloneq \mathcal{Z}_\phi \times \mathcal{Z}_\varphi \times C([0,T];U_0) \times C([0,T];U_0^\Gamma).
$ 
Once more, Owing to the Prokhorov and Skorokhod's theorems (see \cite[Theorem 2.7]{Ikeda+Watanabe_1989} and \cite[Theorem 1.10.4, Addendum 1.10.5]{van der Vaart+Wellner_1996}), there exists a probability space $(\Omega^0,\mathcal{F}^0,\mathbb{P}^0)$ and a sequence of random variables $X_\delta: (\Omega^0,\mathcal{F}^0) \to (\bar{\Omega},\bar{\mathcal{F}})$ such that the law of $X_\delta$ is $\bar{\mathbb{P}}$ for every $n \in \mathbb{N}$, namely, $\bar{\mathbb{P}}= \mathbb{P}^0 \circ X_\delta^{-1}$, so that composition with $X_\delta$ preserves laws, and the following convergences hold as $\delta \to 0^+$
\begin{equation}\label{eq5.24}
\begin{aligned}
\tilde{\bu}^\delta\coloneq\bar{\bu}_{\delta} \circ X_\delta \rightharpoonup \bu^0 \quad &\text{in} \quad L^r(\Omega^0;L^2(0,T;V)), 
\\
\tilde{\bu}^\delta \lvert_\Gamma \rightharpoonup \bu^0 \lvert_\Gamma \quad &\text{in} \quad L^r(\Omega^0;L^2(0,T;L^2(\Gamma))),
    \\
\tilde{\phi}^\delta\coloneq\bar{\phi}_{\delta} \circ X_\delta \to \phi^0 \quad &\text{in} \quad L^\ell(\Omega^0;C([0,T];L^2(\mathcal{O}))) \quad \forall \ell< r,
    \\
\tilde{\varphi}^\delta\coloneq\bar{\varphi}_{\delta} \circ X_\delta \to \varphi^0 \quad &\text{in} \quad L^\ell(\Omega^0;C([0,T];L^2(\Gamma))) \quad \forall \ell<r,
         \\
(\tilde{\phi}^\delta,\tilde{\varphi}^\delta) \stackrel{*}\rightharpoonup (\phi^0,\varphi^0) \quad &\text{in} \quad L^r(\Omega^0;L^\infty(0,T;\mathbb{V}_0)) \cap L^\frac{r}{2}(\Omega^0; W^{\alpha,r}(0,T;\mathbb{V}_0^\prime)),
\\
(\tilde{\phi}^\delta,\tilde{\varphi}^\delta) \rightharpoonup (\phi^0,\varphi^0) \quad & \text{in} \quad L^r(\Omega^0;L^2(0,T;\mathcal{H}^2)),
   \\
\tilde{\mu}^\delta\coloneq\bar{\mu}_{\delta} \circ X_\delta \rightharpoonup \mu^0 \quad &\text{in} \quad L^r(\Omega^0;L^2(0,T;V_1)), 
   \\
\tilde{\theta}^\delta\coloneq\bar{\theta}_{\delta} \circ X_\delta \rightharpoonup \theta^0 \quad &\text{in} \quad L^r(\Omega^0;L^2(0,T;V_\Gamma)),  
      \\
\tilde{W}_{1,\delta} \to W^0 \quad &\text{in} \quad L^\ell(\Omega^0;C([0,T];U_0)) \quad \forall \ell< r,
         \\
\tilde{W}_{2,\delta} \to W^{\Gamma,0} \quad &\text{in} \quad L^\ell(\Omega^0;C([0,T];U_0^\Gamma)) \quad \forall \ell< r,
\\               
F_\delta^\prime(\tilde{\phi}^\delta) \rightharpoonup F^\prime(\phi^0) \quad &\text{in} \quad L^r(\Omega^0; L^2(0,T;L^2(\mathcal{O}))),
   \\
G_\delta^\prime(\tilde{\varphi}^\delta) \rightharpoonup G^\prime(\varphi^0) \quad &\text{in} \quad L^r(\Omega^0;L^2(0,T;L^2(\Gamma))), 
\end{aligned}
\end{equation}
for some limiting processes $(\bu^0,\phi^0,\varphi^0,\mu^0,\theta^0,W^0,W^{\Gamma,0})$ belonging to the specified spaces.

\noindent
Next, arguing as in the proof of Proposition \ref{Prop:Proposition-4}, we have
\begin{align*}
\int_0^t F_{1}(\bar{\phi}_{\delta}(s)) \,\d \bar{W}_{1,\delta}(s) &\to \int_0^t F_{1}(\phi^0(s)) \,\d W^0(s) ~ &\text{in probability in} ~ L^2(0,T;L^2(\mathcal{O})), \\
\int_0^t F_2(\bar{\varphi}_{\delta}(s)) \,\d \bar{W}_{2,\delta}(s) &\to \int_0^t F_2(\varphi^0(s)) \,\d W^{\Gamma,0}(s) ~ &\text{in probability in} ~  L^2(0,T;L^2(\Gamma)).
\end{align*}
Hence, from the previous convergences and similar reasoning as in Subsection \ref{Sect_passage to the limit_delta}, we can let $\delta \to 0^+$ in the variational formulations and obtain that the limit processes form a martingale solution of the original problem \eqref{eq1.1a}-\eqref{eq1.10a} in the sense of Definition \ref{def3.4} with $K=0$. Moreover, we have
\begin{align*}
\mu^0= -\eps \Delta \phi^0 + \frac1\eps F'(\phi^0) \quad  &\text{a.e. in} \quad Q_T, \quad \mathbb{P}^0\text{-a.s.,} 
\\
\theta^0= -\eps_\Gamma \Delta_\Gamma \varphi^0 + \frac{1}{\eps_\Gamma} G'(\varphi^0) + \eps \partial_{\bn} \phi^0 \quad &\text{a.e. on} \quad  \Sigma_T, \quad \mathbb{P}^0\text{-a.s.} 
\end{align*}
Finally, to complete the proof of Theorem \ref{second_main_theorem}, we further assume that $\mathcal{O}$ if of class $C^3$ and that \eqref{F''_and G''_additional_condition} holds true with $p<6$. Then, arguing as in \eqref{eq4.170}, 
\eqref{eq4.171a}, and \eqref{eq4.140a}, we have
\begin{equation}\label{eq5.26}
\|F_\delta^\prime (\bar{\phi}_{\delta,n}(t))\|_{V_1}
\leq C  + C \|\bar{\phi}_{\delta,n}(t)\|_{V_1}^\frac{3p-2}{6-p} + \epsilon \|\bar{\phi}_{\delta,n}(t)\|_{H^3(\mathcal{O})},
\end{equation}
             \begin{equation}\label{eq5.27}
               \|G_\delta^\prime (\bar{\varphi}_{\delta,n}(t))\|_{V_\Gamma}
                \leq C + C \|\bar{\varphi}_{\delta,n}(t)\|_{V_\Gamma}^{4q-5} + \epsilon \|\bar{\varphi}_{\delta,n}(t)\|_{H^3(\Gamma)},
            \end{equation}
and
\begin{equation}\label{eq5.28}
\begin{aligned}
\|(\bar{\phi}_{\delta,n}(t),\bar{\varphi}_{\delta,n}(t))\|_{\mathcal{H}^3}^2
&\leq  C + C \|\bar{\mu}_{\delta,n}(t)\|_{V_1}^2 + C  \|\bar{\theta}_{\delta,n}(t)\|_{V_\Gamma}^2 + C \|\bar{\phi}_{\delta,n}(t)\|_{V_1}^\frac{6p-4}{6-p} + C \|\bar{\varphi}_{\delta,n}(t)\|_{V_\Gamma}^{8q-10} \\
&\quad + C \epsilon^2 \|(\bar{\phi}_{\delta,n}(t),\bar{\varphi}_{\delta,n}(t))\|_{\mathcal{H}^3}^2,
\end{aligned}
\end{equation}
with $C$ independent of $n$ and $\delta$ and for every $\epsilon(0,1)$. By choosing $\epsilon$ sufficiently small in \eqref{eq5.28}, we deduce that 
\begin{equation}\label{eq5.29}
\begin{aligned}
\|(\bar{\phi}_{\delta,n}(t),\bar{\varphi}_{\delta,n}(t))\|_{\mathcal{H}^3}^2
&\leq  C + C \|\bar{\mu}_{\delta,n}(t)\|_{V_1}^2 + C  \|\bar{\theta}_{\delta,n}(t)\|_{V_\Gamma}^2 + C \|\bar{\phi}_{\delta,n}(t)\|_{V_1}^\frac{6p-4}{6-p} + C \|\bar{\varphi}_{\delta,n}(t)\|_{V_\Gamma}^{8q-10}.
\end{aligned}
\end{equation}
As a direct consequence of \eqref{eq5.29}, \eqref{eq5.26} and \eqref{eq5.27}, we have
\begin{equation*}
\begin{aligned}
\|F_\delta^\prime (\bar{\phi}_{\delta,n}(t))\|_{V_1}^2
&\leq C  +  C \|\bar{\mu}_{\delta,n}(t)\|_{V_1}^2 + C  \|\bar{\theta}_{\delta,n}(t)\|_{V_\Gamma}^2 + C \|\bar{\phi}_{\delta,n}(t)\|_{V_1}^\frac{6p-4}{6-p} + C \|\bar{\varphi}_{\delta,n}(t)\|_{V_\Gamma}^{8q-10} 
   \\
\|G_\delta^\prime (\bar{\varphi}_{\delta,n}(t))\|_{V_\Gamma}^2
 &\leq C  +  C \|\bar{\mu}_{\delta,n}(t)\|_{V_1}^2 + C  \|\bar{\theta}_{\delta,n}(t)\|_{V_\Gamma}^2 + C \|\bar{\phi}_{\delta,n}(t)\|_{V_1}^\frac{6p-4}{6-p} + C \|\bar{\varphi}_{\delta,n}(t)\|_{V_\Gamma}^{8q-10},                
\end{aligned}
\end{equation*}
where $C$ is independent of $n$ and $\delta$. Besides, using the convergence result \eqref{eq5.9} and the weak lower semicontinuity, we get
\begin{equation}\label{eq5.30}
\|(\bar{\phi}_{\delta}(t),\bar{\varphi}_{\delta}(t))\|_{\mathcal{H}^3}^2
\leq  C + C \|\bar{\mu}_{\delta}(t)\|_{V_1}^2 + C  \|\bar{\theta}_{\delta}(t)\|_{V_\Gamma}^2 + C \|\bar{\phi}_{\delta}(t)\|_{V_1}^\frac{6p-4}{6-p} + C \|\bar{\varphi}_{\delta}(t)\|_{V_\Gamma}^{8q-10}
\end{equation}
and 
\begin{equation}\label{eq5.32a}
\begin{aligned}
\|F_\delta^\prime (\bar{\phi}_{\delta}(t))\|_{V_1}^2
&\leq C  +  C \|\bar{\mu}_{\delta}(t)\|_{V_1}^2 + C  \|\bar{\theta}_{\delta}(t)\|_{V_\Gamma}^2 + C \|\bar{\phi}_{\delta}(t)\|_{V_1}^\frac{6p-4}{6-p} + C \|\bar{\varphi}_{\delta,n}(t)\|_{V_\Gamma}^{8q-10} 
   \\
\|G_\delta^\prime (\bar{\varphi}_{\delta}(t))\|_{V_\Gamma}^2
 &\leq C  +  C \|\bar{\mu}_{\delta}(t)\|_{V_1}^2 + C  \|\bar{\theta}_{\delta}(t)\|_{V_\Gamma}^2 + C \|\bar{\phi}_{\delta}(t)\|_{V_1}^\frac{6p-4}{6-p} + C \|\bar{\varphi}_{\delta}(t)\|_{V_\Gamma}^{8q-10},                
\end{aligned}
\end{equation}
$\bar{\mathbb{P}}$-a.s. with $C$ being independent of $\delta$. Moreover, from \eqref{eq5.30} and H\"older's inequality, we see that
\begin{equation}\label{eq5.32}
\begin{aligned}
\bar{\mathbb{E}} \left(\int_0^T \|(\bar{\phi}_{\delta}(t),\bar{\varphi}_{\delta}(t))\|_{\mathcal{H}^3}^2\, \d t \right)^\frac{\ell}{2} 
&\leq  C + C \bar{\mathbb{E}} \left(\int_0^T \|\bar{\mu}_{\delta}\|_{V_1}^2 \, \d t\right)^\frac{\ell}{2} + C \bar{\mathbb{E}} \left(\int_0^T  \|\bar{\theta}_{\delta}\|_{V_\Gamma}^2 \, \d t \right)^\frac{\ell}{2} \\
&\quad + C \bar{\mathbb{E}} \sup_{t\in[0,T]} \|\bar{\phi}_{\delta}(t)\|_{V_1}^\frac{(3p-2)\ell}{6-p} + C \bar{\mathbb{E}} \sup_{t\in[0,T]} \|\bar{\varphi}_{\delta}(t)\|_{V_\Gamma}^{(4q-5)\ell}
\end{aligned}
\end{equation}
with $\ell\geq 2$ and $C$ being independent of $\delta$. Using \eqref{eq5.16}$\text{-}3)$ and  \eqref{eq5.18}, the right-hand side of \eqref{eq5.32} is uniformly bounded in $L^\ell(\bar{\Omega};L^2(0,T;\mathcal{H}^3))$ for all $\ell <\min \left \lbrace \frac{6-p}{3p-2},\frac{1}{4q-5} \right \rbrace r$, so
that by weak lower semicontinuity  and \eqref{eq5.24}, we have $(\phi^0,\varphi^0) \in L^\ell(\bar{\Omega};L^2(0,T;\mathcal{H}^3))$ for all $\ell <\min \left \lbrace \frac{6-p}{3p-2},\frac{1}{4q-5} \right \rbrace r$. This proves \eqref{eq5.2}$\text{-}1)$. Once more, using \eqref{eq5.16}$\text{-}3)$ and  \eqref{eq5.18}, the right-hand side of \eqref{eq5.32a}$\text{-}1)$ and \eqref{eq5.32a}$\text{-}2)$ are uniformly bounded in $L^\ell(\bar{\Omega};L^2(0,T;V_1))$  and $L^\ell(\bar{\Omega};L^2(0,T;V_\Gamma))$ for all $\ell <\min \left \lbrace \frac{6-p}{3p-2},\frac{1}{4q-5} \right \rbrace r$, respectively; so that due to
the weak lower semicontinuity of the norm and convergence \eqref{eq5.24}, we infer that $(F^\prime (\phi^0), G^\prime (\varphi^0)) \in L^\ell(\bar{\Omega};L^2(0,T;\mathbb{V}))$ for all $\ell <\min \left \lbrace \frac{6-p}{3p-2},\frac{1}{4q-5} \right \rbrace r$. This proves \eqref{eq5.2}$\text{-}2)$ and then completes the proof of Theorem \ref{second_main_theorem}.
\end{proof}
}

\dela{
\newpage 
\begin{proof}
Let $\lbrace K_n \rbrace_{n\in \mathbb{N}}$ be a sequence of positive real numbers such that $K_n \to 0$ as $n \to \infty$.
Thanks to Theorem \ref{thm-first_main_theorem}, there exists a complete probability space $(\tilde{\Omega},\tilde{\mathcal{F}},\tilde{\mathbb{P}})$ and a sequence of random variables $(\tilde{\bu}^{K_n}, \tilde{\phi}^{K_n}, \tilde{\varphi}^{K_n}, \tilde{\mu}^{K_n}, \tilde{\theta}^{K_n},\tilde{\mathcal{W}} \equiv \tilde{\mathcal{W}}^{K_n})_{n\in \mathbb{N}}$ that is a probabilistic weak solution to the problem \eqref{Eqt1.1a}-\eqref{Eqt1.3a}-\eqref{eq1.10a} in the sense of Definition \ref{def3.4}. Hereafter, for the sake of brevity, we will denote this sequence by $(\tilde{\bu}^n, \tilde{\phi}^n, \tilde{\varphi}^n, \tilde{\mu}^n, \tilde{\theta}^n, \tilde{\mathcal{W}} \equiv \tilde{\mathcal{W}}^n)_{n\in \mathbb{N}}$. From \eqref{eq3.11} and similar reasoning as in the proof of Lemma \ref{Lem-2}, we see that the process $(\tilde{\bu}^n, \tilde{\phi}^n, \tilde{\varphi}^n, \tilde{\mu}^n, \tilde{\theta}^n,\tilde{\mathcal{W}} \equiv \tilde{\mathcal{W}}^n)_{n\in \mathbb{N}}$ satisfies
\begin{align*}
&\mathcal{E}_{n}(\tilde{\phi}^n(t), \tilde{\varphi}^n(t)) + \int_{Q_t} [2 \nu(\tilde{\phi}^n) |D\tilde{\bu}^n|^2 + \lambda(\tilde{\phi}^n) |\tilde{\bu}^n|^2] \,\d x \,\d s \\
&\quad + \int_{\Sigma_t} \gamma(\tilde{\varphi}^n) |\tilde{\bu}^n|^2 \,\d S\,\d s + \int_{Q_t} M_{\mathcal{O}}(\tilde{\phi}^n) |\nabla \tilde{\mu}^n|^2 \, \d x \,\d s + \int_{\Sigma_t} M_\Gamma(\tilde{\varphi}^n) |\nabla_\Gamma \tilde{\theta}^n|^2 \,\d S \,\d s  
 \\
&\leq \mathcal{E}_{n}(\phi_0,\varphi_0) +  \frac{1}{2} \int_0^{t} \|F_{1}(\tilde{\phi}^n(s))\|_{\mathscr{T}_2(U,L^2(\mathcal{O}))}^2\,\d s  \\
&\quad + \frac{\eps [K_n]}{2} \int_0^{t} (\|F_2(\tilde{\varphi}^n)\|_{\mathscr{T}_2(U_\Gamma,L^2(\Gamma))}^2 + \|F_1(\tilde{\phi}^n)\|_{\mathscr{T}_2(U,L^2(\Gamma))}^2) \, \d s \\
&\quad + \frac{\eps}{2} \int_0^{t}  \|\nabla F_1(\tilde{\phi}^n)\|_{L^2(U,\mathbb{L}^2(\mathcal{O}))}^2  \, \d s + \frac{\eps_\Gamma}{2} \int_0^{t}   \|\nabla_\Gamma F_2(\tilde{\varphi}^n)\|_{\mathscr{T}_2(U_\Gamma,\mathbb{L}^2(\Gamma))}^2 \, \d s \\
&\quad + \int_0^{t} (\tilde{\mu}^n,F_1(\tilde{\phi}^n)\, \d \tilde{W}(s)) + \int_0^{t} (\tilde{\theta}^n,(F_2(\tilde{\varphi}^n))\, \d \tilde{W}^\Gamma(s))_\Gamma  \\
&\quad + \sum_{k=1}^\infty \int_0^{t} \frac{1}{2 \eps} \int_{\mathcal{O}} F''(\tilde{\phi}^n) |F_1(\tilde{\phi}^n)e_{1,k}|^2 \, \d x \,\d s + \sum_{k=1}^\infty \int_0^{t} \frac{1}{2 \eps_\Gamma} \int_{\Gamma} G''(\tilde{\varphi}^n) |F_2(\tilde{\varphi}^n)e_{2,k}|^2\, \d S\,\d s \\
&\quad - \eps [K_n] \sum_{k=1}^\infty  \int_0^{t} \int_\Gamma (F_1(\tilde{\phi}^n) e_{1,k}) (F_2(\tilde{\varphi}^n) e_{2,k})\, \d S \, \d s \\
&\quad - \int_0^{t} \int_{\mathcal{O}} M_{\mathcal{O}}(\tilde{\phi}^n) \nabla \tilde{\mu}^n \cdot \nabla \tilde{\phi}^n \, \d x \,\d s  + \int_0^{t}(\tilde{\phi}^n, F_{1}(\tilde{\phi}^n)\,\d \tilde{W}(s)),
\end{align*}
for all $t\in[0,T]$, $n \in \mathbb{N}$, and $\tilde{\mathbb{P}}$-a.s. Here $\mathcal{E}_n(\cdot,\tilde{\cdot})
= \int_{\mathcal{O}} \frac{\eps}{2} \lvert \nabla \cdot \rvert^2 + \frac{1}{\eps} F(\cdot) \, \d x + \int_{\Gamma} \frac{\eps_\Gamma}{2} \lvert \nabla_\Gamma \tilde{\cdot} \rvert^2 + \frac{1}{\eps_\Gamma} G(\tilde{\cdot}) \, \d S + \frac{\eps [K_n]}{2} \int_\Gamma (\tilde{\cdot} - \cdot)^2 \, \d S + \frac{1}{2} \int_{\mathcal{O}} |\cdot|^2\, \d x.
$
From this inequality, we obtain for every $r \geq 2$:
\begin{equation}\label{eq5.1}
\begin{aligned}
&\tilde{\mathbb{E}} \sup_{s \in [0,t]} [\mathcal{E}_{n}(\tilde{\phi}^n(s), \tilde{\varphi}^n(s))]^\frac{r}{2} + \tilde{\mathbb{E}} \left(\int_{Q_t} \nu(\tilde{\phi}^n) |D\tilde{\bu}^n|^2  \,\d x \,\d s \right)^\frac{r}{2} \\
&\quad + \tilde{\mathbb{E}} \left(\int_{Q_t} \lambda(\tilde{\phi}^n) |\tilde{\bu}^n|^2 \,\d x \,\d s \right)^\frac{r}{2} + \tilde{\mathbb{E}} \left(\int_{\Sigma_t} \gamma(\tilde{\varphi}^n) |\tilde{\bu}^n|^2 \,\d S\,\d s\right)^\frac{r}{2} \\
&\quad  + \tilde{\mathbb{E}} \left(\int_{Q_t} M_{\mathcal{O}}(\tilde{\phi}^n) |\nabla \tilde{\mu}^n|^2 \, \d x \,\d s \right)^\frac{r}{2} + \tilde{\mathbb{E}} \left(\int_{\Sigma_t} M_\Gamma(\tilde{\varphi}^n) |\nabla_\Gamma \tilde{\theta}^n|^2 \,\d S \,\d s \right)^\frac{r}{2}  
 \\
&\leq C_r \bigg([\mathcal{E}_{n}(\phi_0,\varphi_0)]^\frac{r}{2} + \tilde{\mathbb{E}} \left(\int_0^{t} \|F_{1}(\tilde{\phi}^n(s))\|_{\mathscr{T}_2(U,L^2(\mathcal{O}))}^2\,\d s\right)^\frac{r}{2}  \\
&\quad + \tilde{\mathbb{E}} \left([K_n] \int_0^{t} \|F_2(\tilde{\varphi}^n)\|_{\mathscr{T}_2(U_\Gamma,L^2(\Gamma))}^2 \, \d s \right)^\frac{r}{2} + \tilde{\mathbb{E}} \left([K_n] \int_0^{t} \|F_1(\tilde{\phi}^n)\|_{\mathscr{T}_2(U,L^2(\Gamma))}^2 \, \d s \right)^\frac{r}{2} \\
&\quad + \tilde{\mathbb{E}} \left(\int_0^{t} \|\nabla F_1(\tilde{\phi}^n)\|_{L^2(U,\mathbb{L}^2(\mathcal{O}))}^2  \, \d s \right)^\frac{r}{2} + \tilde{\mathbb{E}} \left(\int_0^{t} \|\nabla_\Gamma F_2(\tilde{\varphi}^n)\|_{\mathscr{T}_2(U_\Gamma,\mathbb{L}^2(\Gamma))}^2 \, \d s \right)^\frac{r}{2} \\
&\quad + \tilde{\mathbb{E}} \sup_{\tau \in [0,t]} \left|\int_0^{\tau} (\tilde{\mu}^n,F_1(\tilde{\phi}^n)\, \d \tilde{W}(s))\right|^\frac{r}{2} + \tilde{\mathbb{E}} \sup_{\tau \in [0,t]} \left|\int_0^{\tau} (\tilde{\theta}^n,(F_2(\tilde{\varphi}^n))\, \d \tilde{W}^\Gamma(s))_\Gamma \right|^\frac{r}{2}  \\
&\quad + \tilde{\mathbb{E}} \left(\sum_{k=1}^\infty \int_0^{t} \int_{\mathcal{O}} |F''(\tilde{\phi}^n)| |F_1(\tilde{\phi}^n)e_{1,k}|^2 \, \d x \,\d s \right)^\frac{r}{2} \\
&\quad + \tilde{\mathbb{E}} \left(\sum_{k=1}^\infty \int_0^{t} \int_{\Gamma} |G''(\tilde{\varphi}^n)| |F_2(\tilde{\varphi}^n)e_{2,k}|^2\, \d S\,\d s\right)^\frac{r}{2} \\
&\quad + \tilde{\mathbb{E}} \left| [K_n] \sum_{k=1}^\infty  \int_0^{t} \int_\Gamma (F_1(\tilde{\phi}^n) e_{1,k}) (F_2(\tilde{\varphi}^n) e_{2,k})\, \d S \, \d s \right|^\frac{r}{2} \\
&\quad + \tilde{\mathbb{E}} \left(\int_0^{t} \left|\int_{\mathcal{O}} M_{\mathcal{O}}(\tilde{\phi}^n) \nabla \tilde{\mu}^n \cdot \nabla \tilde{\phi}^n \, \d x \right| \,\d s \right)^\frac{r}{2}  + \tilde{\mathbb{E}} \sup_{\tau \in [0,t]} \left|\int_0^{\tau}(\tilde{\phi}^n, F_{1}(\tilde{\phi}^n)\,\d \tilde{W}(s))\right|^\frac{\tau}{2} \bigg),
\end{aligned}
\end{equation}
where $C_r$ depending on $\eps,\,\eps_\Gamma$ and $r$.
Let us proceed with estimating the terms on the right-hand of \eqref{eq5.1} independently of $n$. First, arguing as in the proof of Proposition \ref{Proposition_first_priori_estimmate}, we have 
\begin{align*}
\|F_{1}(\tilde{\phi}^n)\|_{\mathscr{T}_2(U,L^2(\mathcal{O}))}^2 &\leq |\mathcal{O}| C_1,
  \\
\|F_{\Gamma}(\tilde{\varphi}^n)\|_{\mathscr{T}_2(U_\Gamma,L^2(\Gamma))}^2 &\leq |\Gamma| C_2,
     \\
\|F_1(\tilde{\phi}^n)\|_{\mathscr{T}_2(U,L^2(\Gamma))}^2 &\leq C(\mathcal{O},\Gamma)( C_1 + C_1|\mathcal{O}|), 
        \\
\|\nabla F_{1}(\tilde{\phi}^n)\|_{\mathscr{T}_2(U,\mathbb{L}^2(\mathcal{O}))}^2
&\leq C_1 |\nabla \tilde{\phi}^n|^2, 
           \\
\|\nabla_\Gamma F_{\Gamma}(\tilde{\varphi}^n)\|_{\mathscr{T}_2(U_\Gamma,\mathbb{L}^2(\Gamma))}^2 &\leq C_2 |\nabla_\Gamma \tilde{\varphi}^n|_{\Gamma}^2, 
               \\
\left|\sum_{k=1}^\infty  \int_{\Sigma_t} (F_1(\tilde{\phi}^n) e_{1,k}) (F_2(\tilde{\varphi}^n) e_{2,k})\, \d S \, \d s \right| 
 &\leq \int_0^{t} \|F_1(\tilde{\phi}^n)\|_{\mathscr{T}_2(U,L^2(\Gamma))} \|F_2(\tilde{\varphi}^n)\|_{\mathscr{T}_2(U_\Gamma,L^2(\Gamma))} \, \d s \\
 &\leq C(\mathcal{O},\Gamma,C_1,C_2) t.
\end{align*}
and in turn, using H\"older's inequality, we have
\begin{align*}
\tilde{\mathbb{E}} \left(\int_0^{t} \|\nabla F_1(\tilde{\phi}^n)\|_{L^2(U,\mathbb{L}^2(\mathcal{O}))}^2  \, \d s \right)^\frac{r}{2}
&\leq  C_1^{r/2} t^{\frac{r-2}{2}} \tilde{\mathbb{E}} \int_0^t |\nabla \tilde{\phi}^n|^r\, \d s, 
\\
\tilde{\mathbb{E}} \left(\int_0^{t} \|\nabla_\Gamma F_2(\tilde{\varphi}^n)\|_{L^2(U_\Gamma,\mathbb{L}^2(\Gamma))}^2  \, \d s \right)^\frac{r}{2}
&\leq  C_2^{r/2} t^{\frac{r-2}{2}} \tilde{\mathbb{E}} \int_0^t |\nabla_\Gamma \tilde{\varphi}^n|_\Gamma^r\, \d s, \quad \forall t \in[0,T]. 
\end{align*}
Next, using H\"older's inequality together with the assumption $(H4)$, we obtain
\begin{align*}
& \tilde{\mathbb{E}} \left(\int_0^{t} \left| \int_{\mathcal{O}} M_{\mathcal{O}}(\tilde{\phi}^n) \nabla \tilde{\mu}^n \cdot \nabla \tilde{\phi}^n\, \d x \right| \d s \right)^\frac{r}{2} \\ 
&\leq  (\bar{M}_0)^\frac{r}{4} \tilde{\mathbb{E}} \left(\int_0^{t} \left|\sqrt{M_{\mathcal{O}}(\tilde{\phi}^n)} \nabla \tilde{\mu}^n \right| |\nabla \tilde{\phi}^n| \, \d s \right)^\frac{r}{2} \\
&\leq (\bar{M}_0)^\frac{r}{4} \tilde{\mathbb{E}} \left[ \left(\int_0^{t} \left|\sqrt{M_{\mathcal{O}}(\tilde{\phi}^n)} \nabla \tilde{\mu}^n \right|^2 \d s \right)^\frac{r}{4} \left(\int_0^{t} |\nabla \tilde{\phi}^n|^2 \, \d s \right)^\frac{r}{4} \right] \\
&\leq (\bar{M}_0)^\frac{r}{4} \left[\tilde{\mathbb{E}}  \left(\int_0^{t} \left|\sqrt{M_{\mathcal{O}}(\tilde{\phi}^n)} \nabla \tilde{\mu}^n \right|^2 \d s \right)^\frac{r}{2} \right]^\frac{1}{2} \tilde{\mathbb{E}} \left[\left(\int_0^{t} |\nabla \tilde{\phi}^n|^2 \, \d s \right)^\frac{r}{2}\right]^\frac{1}{2} \\
&\leq (\bar{M}_0)^\frac{r}{4} \left[\tilde{\mathbb{E}}  \left(\int_0^{t} \left|\sqrt{M_{\mathcal{O}}(\tilde{\phi}^n)} \nabla \tilde{\mu}^n \right|^2 \d s \right)^\frac{r}{2} \right]^\frac{1}{2}  \left[\tilde{\mathbb{E}} \int_0^{t} |\nabla \tilde{\phi}^n|^r \, \d s \right]^\frac{1}{2} t^\frac{r-2}{4},
\end{align*}
and in turn, applying Young's inequality, we have
\begin{align*}
&C_r \tilde{\mathbb{E}} \left(\int_0^{t} \left| \int_{\mathcal{O}} M_{\mathcal{O}}(\tilde{\phi}^n) \nabla \tilde{\mu}^n \cdot \nabla \tilde{\phi}^n\, \d x \right| \d s \right)^\frac{r}{2} \\
&\leq \frac{1}{4} \tilde{\mathbb{E}}  \left(\int_0^{t} \left|\sqrt{M_{\mathcal{O}}(\tilde{\phi}^n)} \nabla \tilde{\mu}^n \right|^2 \d s \right)^\frac{r}{2} + C(r,\bar{M}_0) t^\frac{r-2}{2} \tilde{\mathbb{E}} \int_0^{t} |\nabla \tilde{\phi}^n|^r \, \d s.
\end{align*}
We observe that
\begin{align*}
\sum_{k=1}^\infty \int_{Q_t} |F''(\tilde{\phi}^n)| |F_1(\tilde{\phi}^n)e_{1,k}|^2 \, \d x \, \d s
&= \sum_{k=1}^\infty \int_{Q_t} |F''(\tilde{\phi}^n)| |\sigma_k(\tilde{\phi})|^2 \, \d x\, \d s \\
&\leq \int_0^t \|F''(\tilde{\phi}^n)\|_{L^1(\mathcal{O})} \, \d s \sum_{k=1}^\infty \|\sigma_k(\tilde{\phi}^n)\|_{L^\infty(Q_t)}^2 \\
&\leq |Q_t| C_1 \int_0^t \|F''(\tilde{\phi}^n)\|_{L^1(\mathcal{O})} \, \d s.
\end{align*}
This, together with \eqref{condition_ F'_and_F''} gives
\begin{equation*}
\sum_{k=1}^\infty \int_{Q_t} |F''(\tilde{\phi}^n)| |F_1(\tilde{\phi}^n)e_{1,k}|^2 \, \d x \, \d s
\leq C t + C \int_0^t \|F(\tilde{\phi}^n)\|_{L^1(\mathcal{O})} \, \d s,
\end{equation*}
with $C$ being independent of $n$. Furthermore, by H\"older's inequality, we have
\begin{align*}
 \tilde{\mathbb{E}} \left(\sum_{k=1}^\infty \int_{Q_t} |F''(\tilde{\phi}^n)| |F_1(\tilde{\phi}^n)e_{1,k}|^2 \, \d x \,\d s \right)^\frac{r}{2} 
 &\leq C(r)\left[1 + \tilde{\mathbb{E}} \left(\int_0^t \|F(\tilde{\phi}^n)\|_{L^1(\mathcal{O})} \, \d s \right)^{\frac{r}{2}}\right]  \\
 &\leq C(r)\left[t^\frac{r}{2} + t^\frac{r-2}{2} \tilde{\mathbb{E}} \int_0^t \|F(\tilde{\phi}^n)\|_{L^1(\mathcal{O})}^\frac{r}{2} \, \d s \right].
\end{align*}
Similarly, 
\begin{equation*}
\tilde{\mathbb{E}} \left(\sum_{k=1}^\infty \int_{\Sigma_t} |G''(\tilde{\varphi}^n)| |F_2(\tilde{\varphi}^n)e_{2,k}|^2 \, \d S \,\d s \right)^\frac{r}{2} 
 \leq C(r)\left[t^\frac{r}{2} + t^\frac{r-2}{2} \tilde{\mathbb{E}} \int_0^t \|G(\tilde{\varphi}^n)\|_{L^1(\Gamma)}^\frac{r}{2} \, \d s \right],
\end{equation*}
for any $r\geq 2$ and, where $C=C(r)$ is a positive constant independent of $n$.

\noindent
Next, using the Burkholder-Davis-Gundy inequality, we have
\begin{align*}
 &\tilde{\mathbb{E}} \sup_{\tau \in[0,t]} \left|\int_0^{\tau} \left(\tilde{\mu}^n(s),F_{1}(\tilde{\phi}^n(s)) \,\d \tilde{W}^n(s)\right) \right|^{r/2} \\
 &\leq C(r) \tilde{\mathbb{E}} \left(\int_0^{t} |\tilde{\mu}^n(s)|^2 \|F_1(\tilde{\phi}^n(s))\|_{\mathscr{T}_2(U,L^2(\mathcal{O}))}^2\,\d s \right)^{r/4} \\
 &\leq C(C_1,\mathcal{O},r) \tilde{\mathbb{E}} \left(\int_0^{t} |\tilde{\mu}^n(s)|^2 \,\d s \right)^{r/4},
\end{align*}
where $C\coloneq C(C_1,\mathcal{O},r)$. Moreover, we have 
   \begin{equation}
     |\tilde{\mu}^n|^2
      \leq 2 |\tilde{\mu}^n - \langle \tilde{\mu}^n\rangle_{\mathcal{O}}|^2 + 2 |\langle \tilde{\mu}^n\rangle_{\mathcal{O}}|^2 
      \leq 2 \tilde{C}_{\mathcal{O}}^2 |\nabla \tilde{\mu}^n|^2 + 2 |\langle \tilde{\mu}^n\rangle_{\mathcal{O}}|^2,
  \end{equation}
where $\tilde{C}_{\mathcal{O}}$ is the Poincar\'e constant depending only on $\mathcal{O}$. Now, taking $(\psi,\psi \lvert_\Gamma)=(1,0)$ as test function in \eqref{eq3.26}, using \eqref{condition_ F'_and_F''} and H\"older's inequality, we find
\begin{align*}
|\langle \tilde{\mu}^n \rangle_{\mathcal{O}}|
&= \left|-\eps [K_n] |\mathcal{O}|^{-1} \int_\Gamma (\tilde{\varphi}^n - \tilde{\phi}^n)\, \d S + \eps^{-1} |\mathcal{O}|^{-1} \int_{\mathcal{O}} F^\prime (\tilde{\phi}^n) \, \d x \right| \\
&\leq \eps [K_n] |\mathcal{O}|^{-1} |\Gamma|^\frac{1}{2} |\tilde{\varphi}^n - \tilde{\phi}^n|_\Gamma + \eps^{-1} |\mathcal{O}|^{-1} \int_{\mathcal{O}} |F^\prime (\tilde{\phi}^n)|\, \d x \\
&\leq \eps^{-1} C_F + \eps [K_n] |\mathcal{O}|^{-1} |\Gamma|^\frac{1}{2} |\tilde{\varphi}^n - \tilde{\phi}^n|_\Gamma  + \eps^{-1} C_F |\mathcal{O}|^{-1} \|F(\tilde{\phi}^n)\|_{L^1(\mathcal{O})}
\end{align*}
and in turn, we infer that
  \begin{equation*}
   |\tilde{\mu}^n|^2
    \leq C(1 + |\nabla \tilde{\mu}^n|^2 + [K_n]^2 |\tilde{\varphi}^n - \tilde{\phi}^n|_\Gamma^2 + \|F(\tilde{\phi}^n)\|_{L^1(\mathcal{O})}^2), 
  \end{equation*}
with $C$ depending only on $\mathcal{O},\,C_F,\,\eps,\,\Gamma$. Besides, we have
\begin{align*}
\tilde{\mathbb{E}} \left(\int_0^t |\tilde{\mu}^n|^2 \, \d s \right)^\frac{r}{4}
&\leq C t^\frac{r}{4} + C \tilde{\mathbb{E}} \left(\int_0^t |\nabla \tilde{\mu}^n|^2 \, \d s \right)^\frac{r}{4} + C \tilde{\mathbb{E}} \left(\int_0^t [K_n]^2 |\tilde{\varphi}^n - \tilde{\phi}^n|_\Gamma^2 \, \d s \right)^\frac{r}{4} \\
&\quad + C \tilde{\mathbb{E}} \left(\int_0^t \|F(\tilde{\phi}^n)\|_{L^1(\mathcal{O})}^2\, \d s \right)^\frac{r}{4}.
\end{align*}
Notice that
\begin{align*}
\tilde{\mathbb{E}} \left(\int_0^t 
\|F(\tilde{\phi}^n(s))\|_{L^1(\mathcal{O})}^2 \, \d s \right)^\frac{r}{4} \leq
  \begin{cases}
   & t^\frac{r - 4}{4} \tilde{\mathbb{E}} \int_0^t \|F(\tilde{\phi}^n(s))\|_{L^1(\mathcal{O})}^{\frac{r}{2}} \, \d s \quad \forall r>4, \\
   & \tilde{\mathbb{E}} \int_0^t 
   \|F(\tilde{\phi}^\delta(s))\|_{L^1(\mathcal{O})}^2 \, \d s \quad \text{if} \quad r=4,
 \end{cases}
\end{align*}

\begin{align*}
\tilde{\mathbb{E}} \left(\int_0^t |\nabla \tilde{\mu}^n(s)|^2 \, \d s \right)^{r/4} 
\leq \frac{1}{M_0^{r/4}} \left[\tilde{\mathbb{E}} \left(\int_0^t \left|\sqrt{M_{\mathcal{O}}(\tilde{\phi}^n(s))} \nabla \tilde{\mu}^n(s) \right|^2 \, \d s \right)^\frac{r}{2} \right]^{1/2}
\end{align*}

        \begin{equation*}
          \tilde{\mathbb{E}} \left(\int_0^t [K_n]^2 |\tilde{\varphi}^n - \tilde{\phi}^n|_\Gamma^2 \, \d s \right)^\frac{r}{4} 
           \leq \left[\tilde{\mathbb{E}} \int_0^t [K_n]^r |\tilde{\varphi}^n - \tilde{\phi}^n|_{\Gamma}^{r} \, \d s \right]^\frac{1}{2} t^\frac{r - 2}{4} \quad \forall r \geq 2.
        \end{equation*}
Hence,
\begin{align*}
& C_r \tilde{\mathbb{E}} \sup_{\tau \in[0,t]} \left|\int_0^{\tau} \left(\tilde{\mu}^n(s),F_{1}(\tilde{\phi}^n(s)) \,\d \tilde{W}^n(s)\right) \right|^{r/2} \\
&\leq C t^\frac{r}{4} + C \left[\tilde{\mathbb{E}} \left(\int_0^t \left|\sqrt{M_{\mathcal{O}}(\tilde{\phi}^n(s))} \nabla \tilde{\mu}^n(s) \right|^2 \, \d s \right)^\frac{r}{2} \right]^{1/2} + C t^\frac{r - 4}{4} \tilde{\mathbb{E}} \int_0^t \|F(\tilde{\phi}^n(s))\|_{L^1(\mathcal{O})}^{\frac{r}{2}} \, \d s \\
&\quad  + C \left[\tilde{\mathbb{E}} \int_0^t [K_n]^r |\tilde{\varphi}^n(s) - \tilde{\phi}^n(s)|_{\Gamma}^{r} \, \d s \right]^\frac{1}{2} t^\frac{r - 2}{4},
\end{align*}
for all $r \geq 4$, where the positive constant $C$ may depend on $M_0,\,C_1,\,\mathcal{O},\,C_F,\,\eps,\,\Gamma$ and $r$. Furthermore, by applying Young's inequality, we arrive at
\begin{equation}\label{eq5.3}
\begin{aligned}
& C_r \tilde{\mathbb{E}} \sup_{\tau \in[0,t]} \left|\int_0^{\tau} \left(\tilde{\mu}^n(s),F_{1}(\tilde{\phi}^n(s)) \,\d \tilde{W}^n(s)\right) \right|^{r/2} \\
&\leq C ( 1 + t^{r/4} + t^{(r - 2)/2}) + \frac{1}{4} \tilde{\mathbb{E}} \left(\int_0^{t} \left|\sqrt{M_{\mathcal{O}}(\tilde{\phi}^n(s))} \nabla \tilde{\mu}^n(s) \right|^2 \d s \right)^{r/2} \\
&\quad + C t^\frac{r - 4}{4} \tilde{\mathbb{E}} \int_0^t \|F(\tilde{\phi}^n(s))\|_{L^1(\mathcal{O})}^{\frac{r}{2}} \, \d s + C \tilde{\mathbb{E}} \int_0^t [K_n]^r |\tilde{\varphi}^n(s) - \tilde{\phi}^n(s)|_{\Gamma}^{r} \, \d s, \quad \forall r \geq 4, 
\end{aligned}
\end{equation}
and for any $t\in[0,T]$, where $C$ may depend on $M_0,\,C_1,\,\mathcal{O},\,C_F,\,\eps,\,\Gamma$ and $r$.

\noindent
For the second stochastic integral in \eqref{eq5.1}, we have
\begin{align*}
\tilde{\mathbb{E}} \sup_{\tau \in[0,t]} \left|\int_0^{\tau} (\tilde{\theta}^n(s),(F_2(\tilde{\varphi}^n(s)))\, \d \tilde{W}^\Gamma(s))_\Gamma\right|^{r/2} 
\leq C \tilde{\mathbb{E}} \left(\int_0^{t} |\tilde{\theta}^n(s)|_\Gamma^2 \, \d s \right)^\frac{r}{4},
\end{align*}
where $C$ depends only on $C_2,\,\Gamma$ and $r$.

\noindent
Now, taking $(\psi,\psi \lvert_\Gamma)=(0,1)$ as test function in \eqref{eq3.26}, using \eqref{condition_ G'_and_G''} and H\"older's inequality, we get
\begin{align*}
|\langle \tilde{\theta}^n \rangle_{\Gamma}| 
&= \left| \eps [K_n] |\Gamma|^{-1} \int_\Gamma (\tilde{\varphi}^n - \tilde{\phi}^n)\, \d S + \eps_\Gamma^{-1} \int_\Gamma  G^\prime(\tilde{\varphi}^n) \, \d x \right| \\
&\leq \eps [K_n] |\Gamma|^{-1/2} |\tilde{\varphi}^n - \tilde{\phi}^n|_\Gamma + \eps_\Gamma^{-1} |\Gamma|^{-1} \int_{\Gamma} |G^\prime(\tilde{\varphi}^n)|\, \d S \\
&\leq \eps_\Gamma^{-1} C_G + \eps [K_n] |\Gamma|^{-1/2} |\tilde{\varphi}^n - \tilde{\phi}^n|_\Gamma + \eps_\Gamma^{-1} |\Gamma|^{-1} C_G \|G^\prime (\tilde{\varphi}^n)\|_{L^1(\Gamma)}.
\end{align*}
Moreover, using the Poincar\'e inequality on $\Gamma$, we obtain
\begin{align*}
|\tilde{\theta}^n|_\Gamma^2
&\leq 2 |\tilde{\theta}^n - \langle \tilde{\theta}^n\rangle_{\Gamma}|_{\Gamma}^2 + 2 |\langle \tilde{\theta}^n \rangle_{\Gamma}|_{\Gamma}^2 \\
&\leq C( 1 +  |\nabla_\Gamma \tilde{\theta}^n|_{\Gamma}^2 + [K_n]^2  |\tilde{\varphi}^n - \tilde{\phi}^n|_\Gamma^2 + \|G^\prime (\tilde{\varphi}^n)\|_{L^1(\Gamma)}^2).
\end{align*}
Here $C$ may depends on $\eps,\,\Gamma,\,C_G$ and $\eps_\Gamma$. Besides, by following the last steps of the proof of \eqref{eq5.3}, we further obtain
\begin{equation}\label{eq5.4}
\begin{aligned}
& C_r \tilde{\mathbb{E}} \sup_{\tau \in[0,t]} \left|\int_0^{\tau} \left(\tilde{\theta}^n(s),F_2(\tilde{\varphi}^n(s)) \,\d \tilde{W}^\Gamma(s)\right) \right|^{r/2} \\
&\leq C ( 1 + t^{r/4} + t^{(r - 2)/2}) + \frac{1}{2} \tilde{\mathbb{E}} \left(\int_0^{t} \left|\sqrt{M_{\Gamma}(\tilde{\varphi}^n(s))} \nabla_\Gamma \tilde{\theta}^n(s) \right|_\Gamma^2 \d s \right)^{r/2} \\
&\quad + C t^\frac{r - 4}{4} \tilde{\mathbb{E}} \int_0^t \|G(\tilde{\varphi}^n(s))\|_{L^1(\Gamma)}^{\frac{r}{2}} \, \d s + C \tilde{\mathbb{E}} \int_0^t [K_n]^r |\tilde{\varphi}^n(s) - \tilde{\phi}^n(s)|_{\Gamma}^{r} \, \d s, \quad \forall r \geq 4, 
\end{aligned}
\end{equation}
and for any $t\in[0,T]$, where $C$ may depend on $N_0,\,C_2,\,\mathcal{O},\,C_G,\,\eps,\,\eps_\Gamma,\,\Gamma$ and $r$.

\noindent
Using the Burkholder-Davis-Gundy,  H\"older's and Young's inequalities together with the assumption $(H2)$, we obtain
\begin{align*}
  & C_r \tilde{\mathbb{E}} \sup_{\tau \in [0,t]} \left|\int_0^{\tau}(\tilde{\phi}^n(s), F_1(\tilde{\phi}^n(s))\,\d \tilde{W}^n(s))\right|^\frac{r}{2} \\
  &\leq C_r \tilde{\mathbb{E}} \left(\int_0^t |\tilde{\phi}^n(s)|^2 \|F_1(\tilde{\phi}^n(s))\|_{\mathscr{T}_2(U,L^2(\mathcal{O}))}^2 \, \d s \right)^\frac{r}{4} \\
  &\leq C(\mathcal{O}, C_1,r) \tilde{\mathbb{E}} \left(\int_0^t |\tilde{\phi}^n(s)|^2 \, \d s \right)^\frac{r}{4} \\
  &\leq C(\mathcal{O}, C_1,r) \left[\tilde{\mathbb{E}} \int_0^t |\tilde{\phi}^n(s)|^{r} \, \d s \right]^\frac{1}{2} t^\frac{r - 2}{4} \\
  &\leq C(r,\mathcal{O},C_1) + t^\frac{r - 2}{2} \tilde{\mathbb{E}} \int_0^t |\tilde{\phi}^n(s)|^{r} \, \d s.
\end{align*}
Collecting now all the previous estimates, inserting all of them on the right-hand side of \eqref{eq5.1}, we obtain after straightforward transformation, that
\begin{equation}\label{eq5.5}
\begin{aligned}
&\tilde{\mathbb{E}} \sup_{s \in [0,t]} [\mathcal{E}_{n}(\tilde{\phi}^n(s), \tilde{\varphi}^n(s))]^\frac{r}{2} + \tilde{\mathbb{E}} \left(\int_{Q_t} \nu(\tilde{\phi}^n) |D\tilde{\bu}^n|^2  \,\d x \,\d s \right)^\frac{r}{2} \\
&\quad + \tilde{\mathbb{E}} \left(\int_{Q_t} \lambda(\tilde{\phi}^n) |\tilde{\bu}^n|^2 \,\d x \,\d s \right)^\frac{r}{2} + \tilde{\mathbb{E}} \left(\int_{\Sigma_t} \gamma(\tilde{\varphi}^n) |\tilde{\bu}^n|^2 \,\d S\,\d s\right)^\frac{r}{2} \\
&\quad  + \frac{1}{2} \tilde{\mathbb{E}} \left(\int_{Q_t} M_{\mathcal{O}}(\tilde{\phi}^n) |\nabla \tilde{\mu}^n|^2 \, \d x \,\d s \right)^\frac{r}{2} + \frac{1}{2} \tilde{\mathbb{E}} \left(\int_{\Sigma_t} M_\Gamma(\tilde{\varphi}^n) |\nabla_\Gamma \tilde{\theta}^n|^2 \,\d S \,\d s \right)^\frac{r}{2}  
 \\
&\leq C\left(1 + (\|\phi_0\|_{V_1}^2 + \|F(\phi_0)\|_{L^1(\mathcal{O})}  + |\nabla_\Gamma \varphi_0|_\Gamma^2 + \|G(\varphi_0)\|_{L^1(\Gamma)})^\frac{r}{2} + (1 + |\varphi_0 - \phi_0|_\Gamma^r) [K_n]^{r/2} \right. \\
&\quad \left. + (1 + [K_n]^{r/2}) \tilde{\mathbb{E}} \int_0^{t} \sup_{0\leq s \leq r} [\mathcal{E}_{n}(\tilde{\phi}^n(s), \tilde{\varphi}^n(s))]^{r/2} \, \d r \right),
\end{aligned}
\end{equation}
for all $t\in[0,T]$ and all $r \geq 4$, where $C$ may depends only on $T,\,\mathcal{O},\,C_1,\,\Gamma,\,C_2,\,M_0$, $C_F,\,\eps,\,N_0,\, C_G,\,\eps_\Gamma,\,\bar{M}_0$ and $r$. It then follows by application of the Gronwall lemma that
\begin{equation}\label{eq5.6}
\begin{aligned}
&\tilde{\mathbb{E}} \sup_{s \in [0,t]} [\mathcal{E}_{n}(\tilde{\phi}^n(s), \tilde{\varphi}^n(s))]^\frac{r}{2} + \tilde{\mathbb{E}} \left(\int_{Q_t} \nu(\tilde{\phi}^n) |D\tilde{\bu}^n|^2  \,\d x \,\d s \right)^\frac{r}{2} \\
&\quad + \tilde{\mathbb{E}} \left(\int_{Q_t} \lambda(\tilde{\phi}^n) |\tilde{\bu}^n|^2 \,\d x \,\d s \right)^\frac{r}{2} + \tilde{\mathbb{E}} \left(\int_{\Sigma_t} \gamma(\tilde{\varphi}^n) |\tilde{\bu}^n|^2 \,\d S\,\d s\right)^\frac{r}{2} \\
&\quad  + \frac{1}{2} \tilde{\mathbb{E}} \left(\int_{Q_t} M_{\mathcal{O}}(\tilde{\phi}^n) |\nabla \tilde{\mu}^n|^2 \, \d x \,\d s \right)^\frac{r}{2} + \frac{1}{2} \tilde{\mathbb{E}} \left(\int_{\Sigma_t} M_\Gamma(\tilde{\varphi}^n) |\nabla_\Gamma \tilde{\theta}^n|^2 \,\d S \,\d s \right)^\frac{r}{2}  
 \\
&\leq C\left(1 + (\|\phi_0\|_{V_1}^2 + \|F(\phi_0)\|_{L^1(\mathcal{O})}  + |\nabla_\Gamma \varphi_0|_\Gamma^2 + \|G(\varphi_0)\|_{L^1(\Gamma)})^\frac{r}{2}  \right. \\
&\left. \hspace{1cm} + (1 + |\varphi_0 - \phi_0|_\Gamma^r) [K_n]^{r/2}\right) e^{(1 + [K_n]^{r/2}) t},
\end{aligned}
\end{equation}
for all $t \in[0,T]$, $n \in \mathbb{N}$, and all $r \geq 4$, with $C$ independent of $K_n$. 

\noindent
As a direct consequence of \eqref{eq5.6}, the definition of the functional $\mathcal{E}_{n}$ and the positivity of the maps $F$ and $G$, we infer that there exists a positive  constant $C$ independent of $K_n$ such that
\begin{equation}
\begin{aligned}
\tilde{\mathbb{E}} \sup_{s \in [0,t]} |\tilde{\varphi}^n(s) - \tilde{\phi}^n(s)|^r 
&\leq C  K_n^{r/2} \left(1 + (\|\phi_0\|_{V_1}^2 + \|F(\phi_0)\|_{L^1(\mathcal{O})}  + |\nabla_\Gamma \varphi_0|_\Gamma^2 + \|G(\varphi_0)\|_{L^1(\Gamma)})^\frac{r}{2}  \right. \\
&\quad \left. \hspace{1cm} + (1 + |\varphi_0 - \phi_0|_\Gamma^r) [K_n]^{r/2}\right) e^{(1 + [K_n]^{r/2}) t}
\end{aligned}
\end{equation}
for all $t\in[0,T]$, $n\in \mathbb{N}$, and all $r \geq 4$.
\end{proof}

\textbf{Brouillon.}

\noindent
Taking $(\psi,\psi \lvert_\Gamma)=(\mathbb{A}_\delta(\bar{\phi}_{\delta}),0)$ as test function in \eqref{eq4.130d}, we have
\begin{align*}
&  \eps \int_{\mathcal{O}} \Psi^\prime_\delta(\bar{\phi}_{\delta}) |\nabla \bar{\phi}_{\delta}|^2 \, \d x + \int_{\mathcal{O}} \frac1\eps |F_\delta^\prime(\bar{\phi}_{\delta})|^2 \, \d x  \\ 
&=   \int_{\mathcal{O}}  \bar{\mu}_{\delta} F_\delta^\prime(\bar{\phi}_{\delta}) \, \d x + \tilde{c}_F \int_{\mathcal{O}} \bar{\mu}_{\delta} \bar{\phi}_{\delta} \, \d x - \frac{\tilde{c}_F}{\eps} \int_{\mathcal{O}} F_\delta^\prime(\bar{\phi}_{\delta}) \bar{\phi}_{\delta} \, \d x \\
&\quad + \eps [K] \int_{\Gamma} (\bar{\varphi}_{\delta} - \bar{\phi}_{\delta}) F_\delta^\prime(\bar{\phi}_{\delta}) \, \d S  + \eps [K] \tilde{c}_F \int_{\Gamma} (\bar{\varphi}_{\delta} - \bar{\phi}_{\delta})  \bar{\phi}_{\delta} \, \d S.
\end{align*}

\begin{align*}
&\int_{\mathcal{O}} \bar{\mu}_{\delta} \mathbb{A}_\delta(\bar{\phi}_{\delta}) \, \d x    
- \eps \int_{\mathcal{O}} \nabla \bar{\phi}_{\delta}(s) \cdot \nabla \mathbb{A}_\delta(\bar{\phi}_{\delta}) \, \d x  - \int_{\mathcal{O}} \frac1\eps F_\delta^\prime(\bar{\phi}_{\delta}) \mathbb{A}_\delta(\bar{\phi}_{\delta}) \, \d x  \notag \\ 
&\quad=  - \eps [K] \int_{\Gamma} (\bar{\varphi}_{\delta} - \bar{\phi}_{\delta}) \mathbb{A}_\delta(\bar{\phi}_{\delta}) \, \d S 
\end{align*}

\begin{equation*}
\int_{\mathcal{O}} \bar{\mu}_{\delta} \mathbb{A}_\delta(\bar{\phi}_{\delta}) \, \d x
= \int_{\mathcal{O}}  \bar{\mu}_{\delta} F_\delta^\prime(\bar{\phi}_{\delta}) \, \d x + \tilde{c}_F \int_{\mathcal{O}}  \bar{\mu}_{\delta} \bar{\phi}_{\delta} \, \d x
\end{equation*}

\begin{equation*}
- \eps \int_{\mathcal{O}} \nabla \bar{\phi}_{\delta} \cdot \nabla \mathbb{A}_\delta(\bar{\phi}_{\delta}) \, \d x
= - \eps \int_{\mathcal{O}} \Psi^\prime_\delta(\bar{\phi}_{\delta}) |\nabla \bar{\phi}_{\delta}|^2 \, \d x
\end{equation*}

\begin{equation*}
- \int_{\mathcal{O}} \frac1\eps F_\delta^\prime(\bar{\phi}_{\delta}) \mathbb{A}_\delta(\bar{\phi}_{\delta}) \, \d x
= - \int_{\mathcal{O}} \frac1\eps |F_\delta^\prime(\bar{\phi}_{\delta})|^2 \, \d x - \frac{\tilde{c}_F}{\eps} \int_{\mathcal{O}} F_\delta^\prime(\bar{\phi}_{\delta}) \bar{\phi}_{\delta} \, \d x
\end{equation*}

\begin{align*}
- \eps [K] \int_{\Gamma} (\bar{\varphi}_{\delta} - \bar{\phi}_{\delta}) \mathbb{A}_\delta(\bar{\phi}_{\delta}) \, \d S
= - \eps [K] \int_{\Gamma} (\bar{\varphi}_{\delta} - \bar{\phi}_{\delta}) F_\delta^\prime(\bar{\phi}_{\delta}) \, \d S  - \eps [K] \tilde{c}_F \int_{\Gamma} (\bar{\varphi}_{\delta} - \bar{\phi}_{\delta})  \bar{\phi}_{\delta} \, \d S.
\end{align*}

\newpage
$\mathbb{A}_\delta(\bar{\phi}_{\delta})= F_\delta^\prime(\bar{\phi}_{\delta}) + \tilde{c}_F \bar{\phi}_{\delta}$

\begin{align*}
&\int_{Q_t} \bar{\mu}_{\delta}(s) \psi \, \d x \, \d s + \int_{\Sigma_t} \bar{\theta}_{\delta}(s) \psi \lvert_\Gamma \, \d S\, \d s  
- \eps \int_{Q_t} \nabla \bar{\phi}_{\delta}(s) \cdot \nabla \psi \, \d x \, \d s - \int_{Q_t} \frac1\eps F_\delta^\prime(\bar{\phi}_{\delta}(s)) \psi \, \d x \, \d s \notag \\ 
&\quad= \eps_\Gamma \int_{\Sigma_t} \nabla_\Gamma \bar{\varphi}_{\delta}(s) \cdot \nabla_\Gamma \psi \lvert_\Gamma \, \d S \, \d s +  \frac{1}{\eps_\Gamma} \int_{\Sigma_t} G_\delta^\prime(\bar{\varphi}_{\delta}(s)) \psi \lvert_\Gamma \, \d S \, \d s \\
&\qquad  + [K] \int_{\Sigma_t} (\bar{\varphi}_{\delta}(s) - \bar{\phi}_{\delta}(s)) (\eps \psi \lvert_\Gamma -\eps \psi)\, \d S \, \d s,
\end{align*}

 Now, taking $(\psi,\psi \lvert_\Gamma)=(1,0)$ and $(\psi,\psi \lvert_\Gamma)=(0,1)$ in \eqref{eq4.130d}, we obtain
\begin{equation}
\begin{aligned}
- |\mathcal{O}| \langle f_\delta \rangle_{\mathcal{O}}
= \eps [K] \int_\Gamma (\bar{\varphi}_{\delta} - \bar{\phi}_{\delta}) \, \d S, 
  \\
|\Gamma| \langle g_\delta \rangle_\Gamma
= \eps [K] \int_\Gamma (\bar{\varphi}_{\delta} - \bar{\phi}_{\delta}) \, \d S.
\end{aligned}
\end{equation}

}

\section*{Acknowledgments}
 The second named author acknowledges financial support from the European Union’s Horizon 2023 Marie Sk{\l}odowska-Curie Action Postdoctoral Fellowship  No. 101151937-SNSCHEs (via URRIP guarantee). 


\section*{Declarations} 
	
\noindent 	\textbf{Ethical Approval:}   Not applicable 
	
	
\noindent  \textbf{Conflict of Interest: } On behalf of all authors, the corresponding author states that there is no conflict of interest.
	
\noindent 	\textbf{Authors' contributions:} All authors have contributed equally.

\noindent \textbf{Data Availability Statement and materials:} Not applicable.

\end{document}